\documentclass[10pt]{amsart}

\usepackage{amssymb,amsmath}
\usepackage{mathrsfs, calligra}
\usepackage{tikz-cd}
\usepackage{bm}
\usepackage{amsthm}
\usepackage{ifpdf}
\usepackage{enumitem}
\usepackage{verbatim}
\usepackage{graphicx,color}
\usepackage{stackengine}

\usepackage[marginratio=1:1]{geometry}
\numberwithin{equation}{section}

\usepackage{cite}

\newtheorem{theorem}{Theorem}[section]
\newtheorem{definition}[theorem]{Definition}
\newtheorem{lemma}[theorem]{Lemma}
\newtheorem{corollary}[theorem]{Corollary}
\newtheorem{proposition}[theorem]{Proposition}

\newtheorem{assumption}[theorem]{Assumption}

\theoremstyle{definition}
\newtheorem{example}[theorem]{Example}

\theoremstyle{remark}
\newtheorem*{remark}{Remark}

\newcommand{\coker}{\textnormal{coker}}
\newcommand{\Cone}{\textnormal{Cone}}

\newcommand{\eps}{\varepsilon}
\newcommand{\rank}{\textnormal{rank}}
\newcommand{\id}{\textnormal{id}}

\newcommand{\Spec}{\textnormal{Spec}}
\newcommand{\im}{\textnormal{im}}

\newcommand{\Sym}{\textnormal{Sym}}
\newcommand{\tensor}{\otimes}
\newcommand{\Ext}{\textnormal{Ext}}
\newcommand{\Tor}{\textnormal{Tor}}
\newcommand{\End}{\textnormal{End}}

\newcommand{\Crit}{\textnormal{Crit}}
\newcommand{\inv}{\textnormal{inv}}

\newcommand{\ind}{\textnormal{ind}}
\newcommand{\grad}{\textnormal{grad}}
\newcommand{\specflow}{\textnormal{sf}}

\newcommand{\pt}{\textnormal{pt}}

\newcommand{\cal}{\mathcal}
\newcommand{\scr}{\mathscr}
\newcommand{\bb}{\mathbb}
\renewcommand{\bf}[1]{\textbf{\textup{#1}}}

\usepackage{tikz-cd}
\usepackage{graphicx,color}

\title{Equivariant Floer theory and double covers of three-manifolds}
\author{Tim Large}
\date{}

\begin{document}

\maketitle

\begin{abstract}Inspired by Kronheimer and Mrowka'€™s approach to monopole Floer homology, we develop a model for $\bb{Z}/2$-equivariant symplectic Floer theory using equivariant almost complex structures, which admits a localization map to a twisted version of Floer cohomology in the invariant set. We then present applications to Steenrod operations on Lagrangian Floer cohomology, and to the Heegaard Floer homology of double covers of three-manifolds.\end{abstract}

\tableofcontents

\section{Introduction}

\subsection{Background.} Suppose $\tilde{M}$ is a symplectic manifold which is exact and convex at infinity (such as the completion of a Liouville domain), which has a symplectic action of the group with two elements $\bb{Z}/2$. Further suppose that $\tilde{L}_0, \tilde{L}_1$ are exact Lagrangians in $M$, compact or convex at infinity, which are invariant under the $\bb{Z}/2$ action. The fixed point set $M$ is then also exact symplectic (possibly with components of different dimensions), with $L_i = \tilde{L}_i \cap M^{\inv}$ as Lagrangian submanifolds. In this setting, Seidel-Smith \cite{SeidelSmith10} asked for a relationship between the $\bb{Z}/2$-equivariant Floer cohomology $HF_{\bb{Z}/2}(\tilde{L}_0, \tilde{L}_1)$ and the ordinary Floer cohomology $HF(L_0, L_1)$ in the fixed point set, where both these groups are taken with $\bb{F}_2$ coefficients.

There is an immediate topological difficulty. To compare Floer cohomology in $M$ and $\tilde{M}$, one would like to consider curves that are holomorphic for an equivariant almost complex structure $J$. However for such $J$, holomorphic curves entirely in $M$ can be of different index when considered in $M$ or in $\tilde{M}$; consequently the moduli spaces of such curves are usually not of the expected dimension. This yields a topological obstruction to simultaneously achieving transversally cut-out and equivariant moduli spaces, stemming from the possible non-triviality of the normal bundle $N_{M \subset \tilde{M}}$ together with the Lagrangian sub-bundles $N_{L_i \subset \tilde{L}_i}$ of $N_{M}|_{L_i}$.

Seidel-Smith showed that one can achieve equivariant transversality in the presence of what they called a stable normal trivialization: an isomorphism $\alpha: N_{M} \oplus \bb{C}^{\ell} \cong \bb{C}^N$, together with homotopies between $\alpha(N_{L} \oplus \bb{R}^{\ell})$ and $\bb{R}^N$ through Lagrangian sub-bundles of $N_{M}|_{L_i} \oplus \bb{C}^{\ell}$. Using an equivariant almost complex structure thus obtained, they constructed a comparatively simple model for $HF_{\bb{Z}/2}(L_0, L_1)$ as an $H^*(B\bb{Z}/2, \bb{F}_2) = \bb{F}_2[[t]]$-module, together with a \emph{localization map}, an $\bb{F}_2[[t]]$ homomorphism
\begin{equation}HF_{\bb{Z}/2}(\tilde{L}_0, \tilde{L}_1) \to HF(L_0, L_1) \tensor_{\bb{F}_2} \bb{F}_2((t)) \end{equation}
which is an isomorphism after inverting the equivariant parameter $t$ in the former group. Using this, they obtained Smith-type inequalities: since there is a spectral sequence from the ordinary Floer cohomology $HF(L_0, L_1) \tensor \bb{F}_2[[t]]$ converging to $HF_{\bb{Z}/2}(L_0, L_1)$, in view of the localization isomorphism above we have
\[\dim_{\bb{F}_2} HF(\tilde{L}_0, \tilde{L}_1) \ge \dim_{\bb{F}_2((t))} HF_{\bb{Z}/2}(\tilde{L}_0, \tilde{L}_1) [t^{-1}] = \dim_{\bb{F}_2} HF(L_0, L_1).\]

\subsection{Twistings and couplings.} In the absence of a stable normal trivialization, Seidel-Smith demonstrated that this inequality fails in even relatively simple cases. However, Seidel-Smith suggested that one might be able to modify the Floer complex for $L_0, L_1$ to incorporate certain characteristic classes coming from the normal data, to produce a twisted version of Floer cohomology that would then be the target of the localization map. This paper is an analytic realization of that idea, taking much inspiration from Kronheimer-Mrowka's construction of $S^1$-equivariant Floer theories in the setting of monopole Floer homology \cite{KronheimerMrowka07}.

Throughout this paper, we refer the triple $\frak{p} = (N_{M}, N_{L_0}, N_{L_1})$ as an example of \emph{polarization data}; the terminology comes from Pressley and Segal's notion of a polarized Hilbert space \cite{PressleySegal86}. The main idea of this paper is to combine the ordinary Floer equation in $M$ with a linear Cauchy-Riemann equation for sections of $N_{M}$. More explicitly, in addition to choosing a time-dependent family of almost complex structures $J_t$ on $M$, we choose a time-dependent complex structure $I_t$ on $N_{M}$, as well as a symplectic connection $\nabla$. We then consider the equations for a map $u: \bb{R}\times[0,1] \to M$ with boundary conditions on $L_0, L_1$, together with a section $\phi$ of $u^*N_M$ with boundary conditions on $N_{L_0}, N_{L_1}$:
\[\frac{\partial u}{\partial s} + J_t(u) \frac{\partial u}{\partial t} = 0, \qquad \nabla_{\partial/\partial s} \phi + I_t(u) \nabla_{\partial/\partial t} \phi = 0.\]
We refer to these as the \emph{twisted Floer equations}, and from them we construct the $\frak{p}$-\emph{twisted Floer cohomology}
\begin{equation}HF_{tw}(L_0, L_1; \frak{p}).\end{equation}
This will naturally be an $\bb{F}_2[t, t^{-1}]$-module, and in the presence of additional grading data, it is the final page of a spectral sequence starting at $HF(L_0, L_1)\tensor \bb{F}_2[t,t^{-1}]$ (in the absence of such grading data, we have to take Floer cohomology in a local system). A stable trivialization of $\frak{p}$ induces an isomorphism
\[HF_{tw}(L_0, L_1; \frak{p}) \cong HF(L_0, L_1) \tensor \bb{F}_2[t,t^{-1}]\]
and thus a degeneration of the above spectral sequence.

The construction of twisted Floer cohomology is entirely local around the invariant set $M$. The next step is to couple the twisted equations on $M, N_M$ to the ordinary Floer equation in $\tilde{M}$ away from the invariant set. Indeed, given a $\bb{Z}/2$-equivariant almost complex structure $\tilde{J}_t$, its restriction to the invariant set splits as an almost complex structure $J_t$ on $M$ and a complex structure $I_t$ on the normal bundle $N_M$.

The key observation is that from a sequence $(u_{\alpha})$ of non-invariant $\tilde{J}_t$-holomorphic curves converging locally to curve contained entirely in the fixed point set, we can extract a non-zero solution of the twisted equations on $M$. We can then modify the traditional compactification of the moduli space of non-invariant Floer trajectories by considering broken trajectories where each component is either a non-invariant Floer trajectory, or a pair $(u, \phi)$ of an invariant Floer trajectory and a solution $\phi$ of the twisted equation.

The main upshot of considering the twisted equations in this problem is that we can achieve transversely cut out moduli spaces with an equivariant almost complex structure: the additional equation in $N_{M}$ precisely accounts for the index difference between the Floer equations in $M$ and $\tilde{M}$. By counting the zero dimensional components of these moduli spaces, we then build in a new model for the equivariant Floer cohomology $HF_{\bb{Z}/2}(\tilde{L}_0, \tilde{L}_1)$, which is naturally a module over $\bb{F}_2[t]$ (rather than its completion, as in the Seidel-Smith model). Moreover, in this model a localization theorem turns out to be an almost immediate consequence of the construction:

\begin{theorem}\label{localization_theorem}There is an $\bb{F}_2[t]$-linear homomorphism
\begin{equation}HF_{\bb{Z}/2}(\tilde{L}_0, \tilde{L}_1) \to HF_{tw}(L_0, L_1; \frak{p})\end{equation}
which is an isomorphism after inverting $t$.\end{theorem}

The relevant finite-dimensional toy model is Morse theory for a manifold $X$ with boundary, where the gradient flow, rather than being everywhere inward or outward pointing, is \emph{everywhere tangent} to the boundary $\partial X$. As developed by Kronheimer-Mrowka \cite{KronheimerMrowka07}, in addition to those critical points in the \emph{interior} of $X$, there are two types of boundary critical points: those where, in the normal to $\partial X$, the flow attracts points into the boundary, and those where the flow repels points away. These two cases are known as the \emph{boundary-stable} and \emph{boundary-unstable} cases respectively, and from these three classes of critical points, we can build three Morse complexes computing each of the singular cohomology groups $H^*(X, \partial X), H^*(X)$ and $H^*(\partial X)$, together with the natural maps between them.

The relevance of this model to equivariant cohomology is this: if $\tilde{X}$ is a manifold with $\bb{Z}/2$-action, we can perform an oriented blow up the fixed point set $X$ to obtain a manifold $\tilde{Y}$ with a free $\bb{Z}/2$ action. The quotient $Y = \tilde{Y}/(\bb{Z}/2)$ is a manifold with boundary given by the projectivization $\bb{R}P(N)$ of the normal bundle $N$ of $X$. Moreover, it is a finite-dimensional approximation to the Borel quotient $\tilde{X} \times_{\bb{Z}/2} E\bb{Z}/2$, and the restriction map
\[H^*(Y) \to H^*(\partial Y) = H^*(\bb{R}P(N))\]
can be seen as an approximation to the localization map
\[H^*_{\bb{Z}/2}(\tilde{X}) \to H^*(X) \tensor H^*(B\bb{Z}/2).\]

Formally speaking, in our setting we have an infinite dimensional real Hilbert manifold
\[\tilde{X} = \cal{P}_{\tilde{M}}(\tilde{L}_0, \tilde{L}_1)\]
of paths between $\tilde{L}_0$ and $\tilde{L}_1$ in $\tilde{M}$. This inherits a $\bb{Z}/2$-action, with fixed point set given by the path space in $M$
\[X = \cal{P}_M(L_0, L_1).\]
The symplectic action gives an equivariant Morse function on $\tilde{X}$. Just as ordinary Floer theory is formally the Morse theory of the symplectic action function, our construction of equivariant Floer cohomology and the localization map is via the formal gradient flow on the blow-up of $\tilde{X}$ along $X$.

The new feature in the infinite-dimensional case is that whilst, in view of Kuiper's theorem, the formal normal bundle $N_{X \subset \tilde{X}}$ is trivial as a Hilbert bundle, the Hessian of the action functional equips $N$ with a \emph{polarization} in the sense of \cite{PressleySegal86}, \cite{CohenJonesSegal}. This structure can be highly non-trivial homotopically: it is the underlying origin of the need to consider polarization-twisted Floer cohomology as the target of the localization map.

The first half or so of this paper is concerned with the analytic theory required to construct the modules $HF_{tw}(L_0, L_1; \frak{p})$, $HF_{\bb{Z}/2}(\tilde{L}_0, \tilde{L}_1)$ and the map between them. Following this,  we prove some basic algebraic properties of this model of equivariant Floer cohomology: the key step is the proof of a K\"{u}nneth theorem, first for twisted Floer cohomology, and then for $\bb{Z}/2$-equivariant Floer cohomology.

Using the K\"{u}nneth formula, we then explain how to extract Seidel-Smith's model of equivariant Floer cohomology, which uses a family of equations parametrized over $\bb{R}P^{\infty}$, from ours. Through this comparison, we also obtain an equivalence with related models for $\bb{Z}/2$-equivariant Floer theory, such as that of Hendricks-Lipshitz-Sarkar \cite{HendricksLipshitzSarkar16}. This is an important step: since our model utilizes almost complex structures that are \emph{not regular} in the ordinary sense, there is no natural spectral sequence relating $HF(\tilde{L}_0, \tilde{L}_1)$ and its equivariant analogue. Nevertheless, by comparing our model to the Seidel-Smith construction, we obtain a Smith-type inequality for the \emph{twisted} Floer cohomology:

\begin{theorem}\label{smith_inequality}There is a rank inequality
\begin{equation} \dim_{\bb{F}_2} HF(\tilde{L}_0, \tilde{L}_1) \ge \rank_{\bb{F}_2[t,t^{-1}]} HF_{tw}(L_0, L_1; \frak{p})\end{equation}
between ordinary Floer cohomology on $\tilde{M}$ and Floer cohomology on $M$ twisted by the polarization data $\frak{p} = (N_M, N_{L_0}, N_{L_1})$.\end{theorem}

\subsection{Steenrod operations.} We then turn our attention to applications, first to Steenrod operations and second to Heegaard Floer homology.

For any exact and convex symplectic manifold $M$, the product $M \times M$ carries a natural symplectic involution exchanging the factors. Moreover, if $L_0, L_1$ are any exact Lagrangians in $M$, then $L_0 \times L_0$ and $L_1 \times L_1$ are natural $\bb{Z}/2$-equivariant Lagrangians of $M \times M$. The fixed point sets are precisely the diagonals $M \subset M \times M$ and $L_i \subset L_i \times L_i$, and it is a well-known fact that in this situation the normal and tangent bundles to the fixed point set coincide: the polarization data is exactly $\frak{p} = (TM, TL_0, TL_1)$. We then accordingly have a localization map
\[HF_{\bb{Z}/2}(L_0 \times L_0, L_1 \times L_1) \to HF_{tw}(L_0, L_1; \frak{p}).\]
Moreover, there is an ``equivariant exterior square'', in the vein of Seidel's equivariant pair-of-pants product \cite{Seidel15}
\[HF(L_0, L_1) \to HF_{\bb{Z}/2}(L_0 \times L_0, L_1 \times L_1)\]
which doubles the degree in the event that $L_0, L_1$ are equipped with gradings, and after tensoring each side with $\bb{F}_2[t,t^{-1}]$ is in fact a linear isomorphism.

\begin{definition}The Floer-theoretic total Steenrod square is the composite map
\begin{equation}Sq: HF(L_0, L_1) \to HF_{tw}(L_0, L_1; \frak{p})\end{equation}
of the two maps above. Given a trivialization of the polarization data $\frak{p} = (TM, TL_0, TL_1)$, the Lagrangians acquire gradings so as to make $Sq$ double the degree, and by comparing the total square $Sq$ with the isomorphism
\[HF_{tw}(L_0, L_1; \frak{p}) \cong HF(L_0, L_1) \tensor \bb{F}_2[t,t^{-1}]\] induced by the trivialization, we obtain individual Steenrod operations $Sq^i$ through the formula
\begin{equation}Sq(x) = t^{\deg(x)}(Sq^0(x) + t^{-1} Sq^1(x) + t^{-2} Sq^2(x) + \hdots).\end{equation}\end{definition}

Even without a trivialization of $\frak{p}$, the remarks above ensure $Sq$ always yields a canonical linear isomorphism
\[Sq: HF(L_0, L_1) \tensor \bb{F}_2[t,t^{-1}] \xrightarrow{\sim} HF_{tw}(L_0, L_1; \frak{p}). \]
In particular, the spectral sequence for ``self-twisted'' Floer cohomology --- when the polarization data $\frak{p}$ is taken to just be the tangent bundles of the symplectic manifold $M$ and its Lagrangians $L_i$ --- always degenerates at the $E_2$ page.

\begin{remark}Given smooth structures as manifolds with corners for the moduli spaces of holomorphic strips in $M$, a trivialization of $\frak{p}$ yields the topological data required to construct a \emph{Floer homotopy type} in the sense of Cohen-Jones-Segal \cite{CohenJonesSegal}. In particular, $HF(L_0, L_1)$ is naturally the $\bb{F}_2$-cohomology of a spectrum, and thus acquires Steenrod squares. In forthcoming work, we will show these coincide with the Floer-theoretic operations defined above.\end{remark}

\subsection{Double covers of three-manifolds.} For our second application, we consider the Heegaard Floer homology of a double cover $\tilde{Y} \to Y$ of closed, oriented three-manifolds, in two cases:
\begin{enumerate}\item when $\tilde{Y} \to Y$ is an unbranched double cover which lifts to a $\bb{Z}$-fold cover;
\item when $\tilde{Y} \to Y$ is a branched over a null-homologous knot $K \subset Y$.\end{enumerate}
In case (1), the question is how do the finite-dimensional $\bb{F}_2$-vector spaces $\widehat{HF}(\tilde{Y})$ and $\widehat{HF}(Y)$, as defined by Ozsvath-Szabo \cite{OzsvathSzabo04a}, \cite{OzsvathSzabo04b}. On the other hand, case (2) we will be concerned with the knot Floer homologies $\widehat{HFK}(\tilde{Y}, K)$ and $\widehat{HFK}(Y, K)$ \cite{OzsvathSzabo04c}, \cite{Rasmussen03}.

In either case, the choice of a Heegaard diagram for $Y$ (respectively $(Y,K)$) determines a symplectic manifold $M$ which is a certain symmetric product of the punctured Heegaard surface, together with Lagrangian tori $L_0, L_1$, such that $HF(L_0, L_1) = \widehat{HF}(Y)$ (respectively $\widehat{HFK}(Y,K)$). There is also an induced (multiply-pointed) Heegaard diagram for $\tilde{Y}$ (respectively $(\tilde{Y}, K)$), and the corresponding symplectic manifold $\tilde{M}$ carries Lagrangian tori $\tilde{L}_0, \tilde{L}_1$ computing $\widehat{HF}(\tilde{Y})$ (respectively $\widehat{HFK}(Y,K)$).

The deck transformations on $\tilde{Y}$ then give rise to a $\bb{Z}/2$-action on $\tilde{M}$, with $M$ as its fixed point set. This action preserves the Lagrangians $\tilde{L}_i$ set-wise, with fixed point sets $L_i$.

The key geometric fact we will prove is that the resulting normal polarization data $\frak{p} = (N_{M}, N_{L_0}, N_{L_1})$ is in fact isomorphic to its tangent polarization data $(TM, TL_0, TL_1)$. In particular, the inverse Steenrod square can be used to trivialize the twisted Floer cohomology group
\[Sq^{-1}: HF_{tw}(L_0, L_1; \frak{p}) \xrightarrow{\sim} HF(L_0, L_1) \tensor \bb{F}_2[t,t^{-1}].\]
From this and Theorem \ref{smith_inequality} we immediately obtain
\begin{theorem}\label{HF_inequality_doublecover}If $\tilde{Y} \to Y$ is an unbranched double cover of closed oriented three-manifolds which lifts to a $\bb{Z}$-fold cover, then
\[2 \ \dim_{\bb{F}_2} \widehat{HF}(\tilde{Y}) \ge \dim_{\bb{F}_2} \widehat{HF}(Y).\]\end{theorem}
The factor of two on the left hand side accounts for the fact that the relevant Heegaard diagram is doubly-pointed. On the other hand, in case (2), we have
\begin{theorem}\label{HF_inequality_branched}If $K \subset Y$ is a null-homologous knot, and $\tilde{Y} = \Sigma(K)$ is the branched double cover of $Y$ obtained from a Seifert surface for $K$, then
\[\dim_{\bb{F}_2} \widehat{HFK}(\tilde{Y}, K) \ge \dim_{\bb{F}_2} \widehat{HFK}(Y, K).\]\end{theorem}
These inequalities extend the results of Hendricks \cite{Hendricks12} and Lipshitz-Treumann \cite{LipshitzTreumann16}, and are similar if somewhat orthogonal to results for rational homology spheres obtained by Lidman-Manolescu \cite{LidmanManolescu} using the Seiberg-Witten-Floer stable homotopy type. Our method essentially follows that of Hendricks, who proved the inequality in the case of knots in $Y = S^3$. In this case, one can take a genus zero Heegaard splitting, at the cost of taking a large number of base-points; the upshot is that the normal polarization in this case is stably trivial and so Seidel-Smith's localization theorem applies.

On the other hand, Lipshitz-Treumann obtained their results by using an algebraic incarnation of localization in bordered Heegaard Floer theory. While they did prove results for three-manifolds other than $S^3$, they were only able to prove the inequality for torsion spin$^c$ structures and, in the branched case, for certain Alexander gradings. Interestingly, one major input in their result is the degeneration of a spectral sequence forced by the existence of a degree-doubling, Frobenius-linear isomorphism from its $E_2$ to its $E_{\infty}$ page--- whose place in our approach is taken by the Steenrod square.

\subsection{Outline of the paper.} After a short diversion in Section 2 to explain the specific incarnation of group homology used in our construction, Sections 3 and 5 develop the polarization-twisted Floer cohomology and our model of equivariant Floer cohomology, whilst sections 4 and 6 cover the analytical back end of each construction. Section 7 proves the necessary K\"{u}nneth theorems, while Section 8 relates our construction to that of Seidel-Smith. Sections 9 and 10 then cover the applications of our theory to Steenrod operations and Heegaard Floer theory respectively.

\subsection{Acknowledgements.} The author would like to thank Paul Seidel for initially suggesting this problem. He would also like to thank Kristen Hendricks for interesting discussions, and sharing her calculations of the Chern classes of the normal and tangent bundles of the symmetric products of a double cover of punctured Riemann surfaces. This research was partly supported by the Simons Foundation (through a Simons Investigator award) and by the National Science Foundation (through awards DMS-1500954 and DMS-1904997).

\section{Algebraic preliminaries}

To begin, we will explain the basic algebra behind our main constructions; this will be crucial in understanding the module structures on polarization-twisted and equivariant Floer cohomology. Most of this material is probably familiar in some form to experts, but let us explain it in the particularly concrete setting we will work in.

Throughout this paper, we will always denote the group with two elements as $\bb{Z}/2 = \{1, \iota\}$, and as a coefficient ring we will take the field with two elements $\bb{F}_2$.

Let $A$ be a free, ungraded chain complex over the group ring $\bb{F}_2[\bb{Z}/2]$; more explicitly for some (possibly infinite) indexing set $I$ we have
\begin{equation*}A = \bigoplus_{i \in I}\bb{F}_2[\bb{Z}/2]\cdot x_i\end{equation*}
for some basis $x_i$ of $A$, together with a map
\begin{equation*}d : A \to A, \qquad x_i \mapsto \sum\limits_{j}(a_{ij} + b_{ij}\iota)x_j\end{equation*}
satisfying $d^2 = 0$. From this complex, we can produce $A_{\bb{F}_2}$, an ungraded chain complex over $\bb{F}_2$, which we think of as the ``quotient'' of $A$ by $\bb{Z}/2$:
\begin{equation}A_{\bb{F}_2} = \bigoplus\limits_{i \in I} \bb{F}_2 \cdot x_i, \qquad d_{\bb{F}_2}(x_i) = \sum\limits_{j}(a_{ij} + b_{ij})x_j.\end{equation}
This carries a map $T: A_{\bb{F}_2} \to A_{\bb{F}_2}$ defined by
\begin{equation}T(x_i) = \sum_j b_{ij}x_j\end{equation}
which is automatically a chain map, giving $A_{\bb{F}_2}$ the structure of a complex over polynomial ring $\bb{F}_2[t]$.

The construction of $A_{\bb{F}_2}$ is dependent on a choice of basis for $A$. However, if $\{x_i'\}$ were another basis on the same indexing set, such that for each $i$, we have either $x_i' = x_i$ or $\iota x_i$, then the two constructions of $A_{\bb{F}_2}$ would be canonically identified. The two maps $T$, $T'$ constructed from each basis would not be the same under this identification, but would instead be related by a chain homotopy: if we define $H: A_{\bb{F}_2} \to A_{\bb{F}_2}$ by $H(x_i) = x_i$ if $x_i' = \iota x_i$, and $H(x_i) = 0$ if $x_i' = x_i$, then $T + T' = d_{\bb{F}_2}H + Hd_{\bb{F}_2}$; in particular the cohomology $H(A_{\bb{F}_2}) = \ker(d_{\bb{F}_2})/\im(d_{\bb{F}_2})$ has the same $\bb{F}_2[t]$-module structure. In fact, we can upgrade this to an isomorphism of $A_{\infty}$-modules:

\begin{proposition}\label{Ainftymodulemap}Suppose $B_1, B_2$ are two (ungraded) chain complexes over $\bb{F}_2$, with endomorphisms $T_1 : B_1 \to B_1$ and $T_2: B_2 \to B_2$ giving them the structure of chain complexes over $\bb{F}_2[t]$. Suppose we are given $F: B_1 \to B_2$ an $\bb{F}_2$-linear chain map, together with a chain homotopy $H : B_1 \to B_2$ such that
\begin{equation*}T_1 F + F T_2 = d_2 H + H d_1.\end{equation*}
Then, thinking of $\bb{F}_2[t]$ as a formal $A_{\infty}$-algebra, and $B_1$, $B_2$ as $A_{\infty}$-modules with no higher products, there is an $A_{\infty}$-$\bb{F}_2[t]$-module map $\{F^i: \bb{F}_2[t]^{\tensor i} \tensor B_1 \to B_2\}_{i \ge 1}$ such that $F^1 = F$.\end{proposition}
\begin{proof}This is given explicitly as $F^2(1, b) = 0$, $F^2(t, b) = H(b)$, and for $n \ge 2$
\begin{equation}F^2(t^n, b) = T_1^{n-1}H(b) + T_1^{n-2}HT_2(b) + \hdots + HT_2^{n-1}(b);\end{equation}
with all the higher maps $F^k$ vanishing for $k > 2$.\end{proof}

\begin{remark}The construction of $A_{\bb{F}_2}$ is precisely the usual construction of group homology as the derived invariants
\[H(\bb{Z}/2, A) = \Tor_{\bb{F}_2[\bb{Z}/2]}(A, \bb{F}_2)\]
applied in this particularly simple setting. In particular, this is a right module for the derived endomorphism algebra
\[\Ext_{\bb{F}_2[\bb{Z}/2]}(\bb{F}_2, \bb{F}_2) \cong \bb{F}_2[t]\]
and indeed our next step is to prove that these module structures coincide in cases of interest.\end{remark}

More familiarly, from an ungraded chain complex $A$ over $\bb{F}_2[\bb{Z}/2]$ we can form the ``Borel complex''
\begin{equation}A[t] = A \tensor_{\bb{F}_2} \bb{F}_2[t]\end{equation}
with differential
\begin{equation}d_{borel} = d_A + t(1+\iota).\end{equation}
The cohomology of this complex computes the group homology $H(\bb{Z}/2, A)$. For a free $\bb{F}_2[\bb{Z}/2]$-complex with basis $\{x_i\}$ as above, there is a natural map
\begin{equation}F : A_{\bb{F}_2} \to A[t], \qquad x_i \mapsto (1 + \iota)x_i\end{equation}
which is easy to check is a chain map. Moreover, it is not difficult to check that the map
\begin{equation*}H : A_{\bb{F}_2} \to A[t], \qquad x_i \mapsto x_i\end{equation*}
satisfies
\begin{equation*}tF + FT = d_{borel}H + Hd_{\bb{F}_2}\end{equation*}
so by Proposition \ref{Ainftymodulemap}, $F$ can be upgraded to an $A_{\infty}$-$\bb{F}_2[t]$-module map.

In an important class of examples, $F$ is a quasi-isomorphism. Consider the following four free complexes over $\bb{F}_2[\bb{Z}/2]$:
\begin{equation}\label{Bzero}B_0 = \bb{F}_2[\bb{Z}/2]\cdot x_0\end{equation}
with zero differential, and
\begin{align}\label{Bplus}B_+ &= \bigoplus\limits_{i \in \bb{Z}_{\ge 0}}\bb{F}_2[\bb{Z}/2] \cdot x_i; \\
\label{Bminus}B_- &= \bigoplus\limits_{i \in \bb{Z}_{< 0}} \bb{F}_2[\bb{Z}/2] \cdot x_i; \\
\label{Binfty}B_{\infty} &= \bigoplus\limits_{i \in \bb{Z}} \bb{F}_2[\bb{Z}/2] \cdot x_i\end{align}
where in each case the differential is given by $d(x_i) = (1+\iota)x_{i+1}$ for each $i$. In each case the differential vanishes on $(B_0)_{\bb{F}_2}, (B_+)_{\bb{F}_2}, (B_-)_{\bb{F}_2}, (B_{\infty})_{\bb{F}_2}$, and the respective $\bb{F}_2[t]$-module structures are just $\bb{F}_2, \bb{F}_2[t], t^{-1}\bb{F}_2[t^{-1}]$ and $\bb{F}_2[t,t^{-1}]$ respectively. It is not hard to compute that the same is true for the Borel complexes $B_0[t], B_+[t], B_-[t], B_{\infty}[t]$, and indeed $F$ is a quasi-isomorphism in each case.

\begin{definition}\label{finitetype}A free $\bb{F}_2[\bb{Z}/2]$-complex $A = \bigoplus_{i \in I} \bb{F}_2[\bb{Z}/2] \cdot x_i$ is said to be \emph{finite type} if there is a finite filtration of the index set
\begin{equation*}I_1 \subset I_2 \subset \hdots \subset I_n = I \end{equation*}\
such that the differential is upper triangular with respect to this filtration, meaning that if $A_k = \bigoplus_{i \in I_k} \bb{F}_2[\bb{Z}/2] \cdot x_i$, we have $d(A_k) \subset A_k$, and such that each of the quotient complexes $A_k / A_{k-1}$ is isomorphic to one of the $\bb{F}_2[\bb{Z}/2]$-complexes $B_0, B_+, B_-$ or $B_{\infty}$ described above.\end{definition}

It follows that for any finite type free $\bb{F}_2[\bb{Z}/2]$-complex $A$, the map $F : A_{\bb{F}_2} \to A[t]$ as an $A_{\infty}$-quasi-isomorphism of $\bb{F}_2[t]$-modules.

These playing off these two perspectives will be primarily useful in proving a K\"{u}nneth theorem for $\bb{Z}/2$-equivariant Floer cohomology. Given two free $\bb{F}_2[\bb{Z}/2]$-complexes $A$ and $A'$, their $\bb{F}_2$-tensor product $A \tensor A'$ is also a free $\bb{F}_2[\bb{Z}/2]$-complex, with the $\bb{Z}/2$-action given by $\iota(a \tensor a') = \iota a \tensor \iota a'$ (note that if $\{x_i\}, \{x'_j\}$ are $\bb{F}_2[\bb{Z}/2]$-bases for $A, A'$ respectively, then $\{x_i \tensor x'_j\} \cup \{x_i \tensor \iota x'_j\}$ is a basis for $A \tensor A'$).

On the other hand, given two chain complexes $B, B'$ over $\bb{F}_2[t]$, we can form the derived tensor product
\begin{equation*}B \tensor_{\bb{F}_2[t]}^{\bb{L}} B'.\end{equation*}
This is usually defined in the setting of graded chain complexes; in our setting we will be working in any case with a particularly explicit model coming from the Koszul resolution of $\bb{F}_2[t]$. Indeed, $B \tensor_{\bb{F}_2} B'$ is an $\bb{F}_2[t_1] \tensor \bb{F}_2[t_2] = \bb{F}_2[t_1, t_2]$-module in the obvious way, and there is an equivalence
\begin{equation*}B \tensor_{\bb{F}_2[t]}^{\bb{L}} B' \cong \bb{F}_2[t] \tensor^{\bb{L}}_{\bb{F}_2[t_1, t_2]} \left( B \tensor B' \right)\end{equation*}
covering the ring map $\bb{F}_2[t_1, t_2] \to \bb{F}_2[t]$ sending $t_1, t_2 \mapsto t$. The complex
\begin{equation}\bb{F}_2[t_1, t_2] \xrightarrow{t_1 - t_2} \bb{F}_2[t_1, t_2]\end{equation}
gives a free resolution of $\bb{F}_2[t]$ as a $\bb{F}_2[t_1, t_2]$-module; and so we have 
\begin{equation*}B \tensor_{\bb{F}_2[t]}^{\bb{L}} B' \cong \Cone\left( B \tensor_{\bb{F}_2} B' \xrightarrow{t \tensor \id - \id \tensor t} B \tensor_{\bb{F}_2} B' \right).\end{equation*}

\begin{proposition}\label{monoidal_tensor_product}If $A, A'$ are free $\bb{F}_2[\bb{Z}/2]$-complexes as before, with $A \tensor A'$ the tensor product $\bb{F}_2[\bb{Z}/2]$-complex, there is a quasi-isomorphism of $\bb{F}_2[t]$-complexes
\begin{equation*}\left((A \tensor A')[t], d_{borel} \right) \cong A[t] \tensor^{\bb{L}}_{\bb{F}_2[t]} A'[t].\end{equation*}\end{proposition}

\begin{proof}On $(A \tensor A')[t]$, the differential $d_{borel}$ is given by
\begin{equation*}d_{borel} = d_A \tensor \id + \id \tensor d_{A'} + t(\id \tensor \id + \iota \tensor \iota).\end{equation*}
Using the automorphism $A \tensor A' \xrightarrow{\id \tensor \iota} A \tensor A'$, we obtain an isomorphism of chain complexes
\begin{equation*}\left((A \tensor A')[t], d_{borel}\right) \cong \left((A \tensor A')[t], d_{tensor}\right)\end{equation*}
where the new differential $d_{tensor}$ is given by
\begin{equation*}d_{tensor} = d_A \tensor \id + \id \tensor d_{A'} + t(\id \tensor \iota + \iota \tensor \id).\end{equation*}
We claim that $((A \tensor A')[t], d_{tensor})$ is a model for the tensor product $A[t] \tensor^{\bb{L}}_{\bb{F}_2[t]} A'[t]$. Indeed there is a short exact sequence of $\bb{F}_2[t_1, t_2]$-modules
\begin{equation*}0 \rightarrow A[t_1] \tensor_{\bb{F}_2} A'[t_2] \xrightarrow{t_1 \tensor \id - \id \tensor t_2} A[t_1] \tensor_{\bb{F}_2} A'[t_2] \xrightarrow{t_1, t_2 \mapsto t} (A \tensor A')[t] \rightarrow 0\end{equation*}
and moreover each map is a chain map once $A[t_1] \tensor_{\bb{F}_2} A'[t_2]$ is equipped with the tensor product of the Borel differentials on $A[t_1]$ and $A'[t_2]$, and $(A \tensor A')[t]$ is equipped with $d_{tensor}$; in particular $((A \tensor A')[t], d_{tensor})$ is quasi-isomorphic to the desired cone.\end{proof}

\begin{remark}Indeed, if $G$ is a finite group and $k$ is a field, by a version of the Rothenberg-Steenrod spectral sequence the algebra of cochains on the classifying $BG$ can be computed by
\[C^*(BG, k) \simeq RHom_{k[G]}(k, k)\]
where $k[G]$ is the group ring. Writing $k[G]$-mod of the category of chain complexes of $G$-modules, there is thus a functor
\begin{equation}\label{koszul_duality}RHom_{k[G]}(k, -) : k[G]\text{-mod} \to C^*(BG)\text{-mod}\end{equation}
The group structure on $G$ then endows $C^*(BG)$ with the structure of a differential graded Hopf algebra; in particular its category of modules has a tensor product. Likewise, the tensor product of representations defines a monoidal structure on $k[G]$-mod, and one would hope that the above map \eqref{koszul_duality} is monoidal at least on some large class of quasi-free $k[G]$-modules.

In the concrete setting we work in, we have given an explicit model for the derived tensor product, and we have checked by hand the desired monoidal property. The advantage of this more hands-on approach is to produce $\bb{F}_2[t]$-modules right off the bat, sidestepping formality questions about $C^*(\bb{R}P^{\infty}, \bb{F}_2)$.\end{remark}

\section{Polarization-twisted Floer theory}

\subsection{A finite dimensional model}
In \cite{KronheimerMrowka07}, Kronheimer and Mrowka introduced ``coupled Morse homology'' associated to a closed manifold $X$ together with a map $\frak{p} : X \to U$ to the infinite unitary group, depending only on the homotopy class of $\frak{p}$. In this section, we will consider an $\bb{F}_2$-valued version of the same construction for maps $\frak{p}: X \to U/O$, which we will call the ``polarization-twisted'' Morse theory. Since the application will be to Lagrangian Floer theory, the analytical implementation of the construction will be more explicit (and depend on ``polarization data'', rather than a map to $U/O$); but for ease of exposition let us first outline the general construction.

Let $H$ be a separable real Hilbert space, equipped with a compact operator $K: H \to H$ with image $H^1 = K(H) \subset H$, such that $K$ has infinite dimensional positive and negative eigenspaces, and has mild spectrum in the sense of \cite{KronheimerMrowka07}: the example to keep in mind is $H = L^2(M, E)$ of $L^2$ sections of a vector bundle $E$ over a manifold $M$, with $H^1$ the Sobolev space $W^{1,2}(M,E) \subset L^2(M,E)$. Consider the space of unbounded, self-adjoint Fredholm operators on $H$ with domain $H^1$:
\begin{equation}\cal{S}(H) := \{A: H^1 \to H \text{ Fredholm }  | \quad \ind(A) = 0, \langle v, Aw \rangle = \langle Av, w \rangle \quad \forall v,w  \in H^1\} \end{equation}
with its topology induced from the operator norm. Each $A \in \cal{S}(H)$ has an orthonormal basis of eigenvectors, and the eigenvalues have no accumulation point. $\cal{S}(H)$ splits into three components
\begin{equation} \cal{S}(H) = \cal{S}_+(H) \sqcup \cal{S}_*(H) \sqcup \cal{S}_-(H) \end{equation}
where $\cal{S}_{\pm}(H)$ consist of those $A$ such that the spectrum is respectively bounded below or above, and $\cal{S}_*(H)$ is those with spectrum unbounded in both positive and negative directions. Each $\cal{S}_{\pm}(H)$ is contractible; of interest to us is $\cal{S}_*(H)$: following \cite{AtiyahSinger69} the homotopy type of $\cal{S}_*(H)$ can be identified with $U/O$.

\begin{remark}Kronheimer-Mrowka work with complex Hilbert spaces in \cite{KronheimerMrowka07}; in this setting $\cal{S}_*(H)$ is homotopy equivalent to $U$.\end{remark}

Accordingly, we can identify homotopy classes of maps $\frak{p}: X \to U/O$ with continuous families $\{A_x\}_{x \in X}$ of self-adjoint operators parametrized by $X$: we refer to either as a \emph{polarization}. Suppose that $X$ is a manifold, equipped with a Morse function $f$, which has finitely many critical points. By perturbing the family $\{A_x\}$, we can arrange so that at each critical point $x \in X$ of $f$, the operator $A_x$ is invertible, with one-dimensional eigenspaces. Accordingly we can label the eigenvalues $\lambda_i(x)$ for $i \in \bb{Z}$ in increasing order, such that
\begin{equation*} \hdots \lambda_{-2}(x) < \lambda_{-1}(x) < 0 < \lambda_0(x) < \lambda_1(x) < \lambda_2(x) < \hdots \end{equation*}
and choose corresponding unit eigenvectors $\psi_i(x)$. We will often write $i(\lambda)$ for the numbering of such an eigenvalue $\lambda$ as above. We will write
\begin{equation}\frak{C} = \{\bf{x} = (x, \lambda_i): x \in \Crit(f), \lambda_i \in \Spec(A_x)\}\end{equation}
referring to elements of $\frak{C}$ by the bold-face version of the corresponding element of $\Crit(f)$.

Given any two critical points $x_-, x_+$ and a continuous path $u: \bb{R} \to X$ such that $u(s) \to x_{\pm}$ exponentially as $s \to \pm \infty$, the operator
\begin{equation}\frac{d}{ds} + A_{u(s)} : W^{1,2}(\bb{R}, H) \cap L^2(\bb{R}, H^1) \to L^2(\bb{R}, H)\end{equation}
is Fredholm, since each $A_{x_-}$ and $A_{x_+}$ are by the hypothesis invertible. Its index depends only on $u$ up to homotopy, and is known as the spectral flow $\specflow(u)$. The spectral flow is additive under concatenation of paths: accordingly, if
\begin{equation}\pi_1(X) \to \pi_1(U/O) \cong \bb{Z}\end{equation}
is zero, the spectral flow $\specflow(u)$ depends only on endpoints $x_{\pm}$, not the choice of path $u$. In this case, we say that the polarization can be \emph{graded}, and we define a \emph{grading} to be a choice of function $s: \Crit(f) \to \bb{Z}$ such that
\begin{equation}\label{grading}\specflow(u) = s(x_-) - s(x_+)\end{equation}
for any path $u$ connecting $x_-$ and $x_+$.

Consider the free $\bb{F}_2$-vector space
\begin{equation}CM(X, \frak{p}) = \bigoplus_{\bm{x} \in \frak{C}} \bb{F}_2 \cdot \bf{x}\end{equation}
generated by $\frak{C}$. A priori this is ungraded, however in the presence of a grading $s$ described above, we define the degree of $\bf{x} = (x, \lambda)$ to be
\begin{equation*}|\bf{x}| = \ind(x) + s(x) + i(\lambda)\end{equation*}

To define a differential on $CM(X, \frak{p})$, choose a metric on $X$, and consider the equations for a path $u : \bb{R} \to X$ together with a map $\phi : \bb{R} \to H^1$:
\begin{equation}\label{morseeqn}\frac{d}{ds}u(s) + \grad_{u(s)}(f) = 0\end{equation}
\begin{equation}\label{twistedmorse}\frac{d}{ds}\phi(s) + A_{u(s)}\phi(s) = 0.\end{equation}
Of course, the first equation is the usual Morse flow; the latter is an auxiliary linear equation we will refer to as the \emph{twisting equation}. Solutions $(u, \phi)$ have an $\bb{R}$-action by translation in $s$, and an $\bb{R}^*$-action by rescaling $\phi$. Observe that for constant solutions $u(s) = x \in \frak{C}(f)$ of the Morse equation, and the twisting equation is precisely the Morse flow equation for the function $\frac{1}{2}\langle v, A_x v\rangle$ on $H^1$. Modulo the rescaling, we can view this as Morse theory on a ``polarized $\bb{RP}^{\infty}_{-\infty}$''.

For each $\bf{x}_- =(x_-, \lambda_-)$ and $\bf{x}_+ = (x_+, \lambda_+) \in \frak{C}$, consider those solutions $(u, \phi)$ with asymptotics
\begin{equation*}u(s) \to x_{\pm} \qquad \text{as } s \to \pm \infty\end{equation*}
\begin{equation*}\phi(s) \sim C_{\pm} e^{-\lambda_{\pm} s} \psi_{\pm} \qquad \text{as } s \to \pm \infty\end{equation*}
for some nonzero constants $C_{\pm}$, where $\psi_{\pm}$ are a choice of unit norm eigensolutions corresponding to $\lambda_{\pm}$. More precisely, writing $|| \cdot ||_H$ for the norm on $H$, we ask that
\begin{equation*}\phi(s)/||\phi(s)||_H \to \psi_{\pm} \text{ or } -\psi_{\pm}\end{equation*}
as $s \to \pm \infty$. We will often say that such a solution $(u, \phi)$ has ``asymptotics of type'' $\lambda_{\pm}$ at $\pm \infty$.

There is a slightly different perspective on the equations \eqref{morseeqn}, \eqref{twistedmorse} that we should also mention, which is referred to in \cite{KronheimerMrowka07} as \emph{the $\tau$-model}. Namely, if we assume that $\phi(s) \not= 0$ for all $s$ (for instance, by a unique continuation result), and normalize to unit length by setting $\varphi(s) = \phi(s)/||\phi(s)||_H$, then the equation for $\varphi(s)$ is
\begin{equation}\label{tautwistedeqn}\frac{d\varphi}{ds} + A_{u(s)}\varphi - \Lambda(\varphi)\varphi = 0\end{equation}
where $\Lambda(\varphi) : \bb{R} \to \bb{R}$ is the real function
\begin{equation}\label{Lambdamorse}\Lambda(\varphi) = \langle \varphi, A \varphi \rangle = \frac{\langle \phi, A \phi \rangle}{\langle \phi, \phi \rangle}.\end{equation}
This no longer has a rescaling action by $\bb{R}$, but rather an involution $\varphi \mapsto -\varphi$. The asymptotic conditions for $\varphi$ are then that
\begin{equation*}\varphi(s) \to \psi_{\pm}(s) \text{ or } -\psi_{\pm}(s)\end{equation*}
as $s \to \infty$. Moreover, $\phi$ can be reconstructed from $\varphi$ by solving the first order linear differential equation
\begin{equation}\frac{d}{ds}r(s) + \Lambda(\varphi)r(s) = 0\end{equation}
and setting $\phi(s) = r(s) \varphi(s)$. In other words, we can think of the twisted equations as being Morse theory coupled to either a linear flow on $H$ with a $\bb{R}^*$ action, or a non-linear flow on $S(H)$ coupled to a $\bb{Z}/2$-action; Kronheimer-Mrowka analogously refer to the first model with linear equations as \emph{the $\sigma$ model}. At various points in the analytical development of polarization twisted Floer theory we will carry out later, it will be advantageous to switch between these two pictures.

\begin{definition}The moduli space of \emph{twisted trajectories} $\cal{M}(\bf{x}_-, \bf{x}_+)$ between $\bf{x}_- = (x_-, \lambda_-)$ and $\bf{x}_+ = (x_+, \lambda_+)$ is is the space of solutions $(u, \phi)$ to the twisted Morse equations with the above asymptotics, modulo the $\bb{R}$-translation and $\bb{R}^*$-rescaling actions. Alternatively in the $\tau$-model, it is the space of solutions $(u, \varphi)$ to the $\tau$-model twisted equations $\eqref{tautwistedeqn}$ modulo the $\bb{R}$-translation and $\bb{Z}/2$-actions.\end{definition}

After a perturbation of the metric on $X$ and the family of operators $\{A_x\}_{x \in X}$, this is a manifold; the component $\cal{M}_{[u]}(\bf{x}_-, \bf{x}_+)$ corresponding to a homotopy class $[u]$ of paths connecting $x_-$ to $x_+$ has dimension
\begin{equation}\dim \cal{M}_{[u]}(\bf{x}_-, \bf{x}_+) = \ind(x_-) - \ind(x_+) + \specflow(u) + i(\lambda_-) - i(\lambda_+) - 1\end{equation}
where $\bf{x}_{\pm} = (x_{\pm}, \lambda_{\pm})$.

\begin{definition}The polarization-twisted Morse cohomology $HM(X, \frak{p})$ is the cohomology of $CM(X, \frak{p})$ with the differential given by
\begin{equation}d \ \bf{x}_+ = \sum\limits_{\bf{x}_-, [u] : \dim \cal{M}_{[u]}(\bf{x}_-, \bf{x}_+) = 0} \#\cal{M}_{[u]}(\bf{x}_-, \bf{x}_+) \cdot \bf{x}_-\end{equation}\end{definition}
Although $CM(X, \frak{p})$ is clearly infinite-dimensional, given $\bf{x}_+ \in \frak{C}$ and $x_- \in \Crit(f)$, for each homotopy class of Morse flows $[u]$ there is exactly one $\bf{x}_- = (x_-, \lambda_-)$ such that $\cal{M}_{[u]}(\bf{x}_-, \bf{x}_+)$ is zero-dimensional, and thus the above is a finite sum. It can be seen to be a differential through the usual method compactifying $\cal{M}$ by broken flows, and inspecting resulting boundary of the one-dimensional moduli spaces $\cal{M}_{[u]}(\bf{x}_-, \bf{x}_+)$. Moreover, using the appropriate continuation map equations, $HM(X, \frak{p})$ can be seen to be dependent on the family of operators $\{A_x\}$ only up to homotopy; in the case that $X$ is closed, it is furthermore independent of the Morse function $f$ and the metric.

In the case that $\frak{p}$ has a grading in the sense of \eqref{grading}, the differential has degree one, and $HM(X, \frak{p})$ becomes a graded $\bb{F}_2$-vector space.

$HM(X, \frak{p})$ carries a number of additional structures. Most important for us is that it carries a distinguished endomorphism
\begin{equation}T: HM(X, \frak{p} \to HM(X, \frak{p})\end{equation}
which is moreover degree one in the presence of a grading. Equivalently, $HM(X, \frak{p})$ has the structure of a $\bb{F}_2[t]$-module. In fact, this endomorphism is always invertible, and so this extends to give the structure of a $\bb{F}_2[t,t^{-1}]$-module.

Indeed, by distinguishing between the two unit eigenvectors $\psi$ and $-\psi$ for each eigenvalue $\lambda$ of $A_x$, and remembering, we can define a larger chain complex with generators each $(x, \psi)$ and $(x, -\psi)$:
\begin{equation}\widetilde{CM}(X, \frak{p}) = \bigoplus\limits_{\bf{x} \in \frak{C}} \bb{F}_2(x, \psi) \oplus \bb{F}_2(x, -\psi)\end{equation}
where the differential counts solutions $(u, \varphi)$, in the $\tau$-model say, with a fixed choice of limiting eigenvector $\psi, -\psi$ at either end. These equations no longer carry a $\bb{Z}/2$ action, but rather the complex $\widetilde{CM}(X, \frak{p})$ carries a natural $\bb{Z}/2$ action, for which it is a free $\bb{F}_2[\bb{Z}/2]$-module, interchanging each pair of unit eigenvectors $\psi, -\psi$. In particular, the setting of Section 2.1 applies, and we have
\begin{equation*}CM(X, \frak{p}) \cong \left(\widetilde{CM}(X, \frak{p})\right)_{\bb{F}_2}.\end{equation*}
Consequently, $CM(X, \frak{p})$ is a chain complex over $\bb{F}_2[t]$, in a way independent of the labeling of the eigenvectors up to $A_{\infty}$-isomorphism.

\begin{remark}To unravel the above, we could directly describe the endomorphism $T$ as follows. Choose for each $(x, \lambda)$ a distinguished choice of unit eigenvector $\psi$. Then, for each trajectory $(u, \phi)$ connecting $(x_{\pm}, \lambda_{\pm})$, we certainly have $u(s) \sim C_{\pm} e^{-\lambda_{\pm}} \psi_{\pm}$ as $s \to \pm_{\infty}$ for some nonzero real numbers $C_{\pm}$, where $\psi_{\pm}$ are the distinguished choices of eigenvectors. We then say that the trajectory $(u, \phi)$ is \emph{positive} if $C_-, C_+$ have the same sign, and \emph{negative} if $C_-, C_+$ have different sign. Equivalently, the corresponding $\tau$-model trajectory $(u, \varphi)$ is positive if it can be chosen to have limits $\psi_-, \psi_+$ at $\pm \infty$, and negative if these limits are $-\psi_-, \psi_+$. The chain-level endomorphism $T$ then counts the index one \emph{negative} twisted trajectories.\end{remark}

\begin{example}\label{constant_u_solutions}Let us illustrate this when $X = \{\pt\}$. In this case, $CM(X, \frak{p}) = \oplus_{i \in \bb{Z}} \bb{F}_2 \cdot \bf{x}_i$ has one generator $\bf{x}_i = (\pt, \lambda_i)$ for each solution to the eigenvalue equation
\begin{equation*}A \psi_i = \lambda_i \psi_i\end{equation*}
for a fixed $A \in \cal{S}_*(H)$ with simple spectrum. This has a clear grading, with $|\bf{x}_i| = i$. The solutions $\phi$ of \eqref{twistedmorse} connecting $\bf{x}_j$ to $\bf{x}_i$ for $j > i$ are precisely those of the form
\begin{equation}\label{constantsolutions}\phi(s) = a_j e^{- \lambda_j s} \psi_j + a_{j-1} e^{-\lambda_{j-1} s}\psi_{j-1} + \hdots + a_i e^{-\lambda_i s} \psi_i\end{equation}
for some tuple of real numbers $(a_j, a_{j-1}, \hdots, a_i)$, such that $a_j \not= 0$ and $a_i \not= 0$. The two operations of rescaling $\phi$ by $\kappa \in \bb{R}^*$ and translation in $s$ by some $\sigma \in \bb{R}$ send $(a_j, a_{j-1}, \hdots, a_i)$ to
\[(\kappa a_j, \kappa a_{j-1}, \hdots, \kappa a_i), \qquad (e^{-\lambda_j \sigma} a_j, e^{-\lambda_{j-1} \sigma} a_{j-1}, \hdots, e^{-\lambda_i \sigma} a_i)\]
respectively. In particular, when $j = i+1$, there are precisely two elements of $\cal{M}(\bf{x}(j), \bf{x}(i))$, represented by
\begin{equation*}e^{-\lambda_{i+1}s} \psi_{i+1} + e^{-\lambda_i s}\psi_i, \qquad e^{-\lambda_{i+1}s} \psi_{i+1} - e^{-\lambda_i s}\psi_i.\end{equation*} Accordingly, the differential $d \bf{x}_i = 0$. Moreover, after distinguishing between the two eigenvectors $\psi_i, -\psi_i$ for each $i$, we see that as a $\bb{F}_2[\bb{Z}/2]$-complex $\tilde{CM}(\pt, \frak{p})$ is exactly the complex $B_{\infty}$ of \eqref{Binfty}; we can then identify the twisted Morse cohomology with Laurent polynomials
\begin{equation*}HM(\pt, \frak{p}) \cong \bb{F}_2[t, t^{-1}].\end{equation*}\end{example}

The Morse index of $x \in \Crit(f)$ (or alternatively the value of the Morse function $f$) provides a filtration of $CM(X, \frak{p})$. In each piece of the associated graded, the differential only counts those twisted flows $(u, \phi)$ with $u$ constant, which are precisely described by the above. Consequently, $\tilde{CM}(X, \frak{p})$ is finite type in the sense of Definition \ref{finitetype}, and moreover we obtain:
\begin{proposition}The endomorphism $T$ is invertible, thus endowing $HM(X, \frak{p})$ with the structure of an $\bb{F}_2[t,t^{-1}]$-module.\end{proposition}

This filtration also induces a spectral sequence converging to $HM(X, \frak{p})$, with $E_1$-page
\[\bigoplus_{x \in \frak{C}(f)} \bb{F}_2[t, t^{-1}]\cdot x.\]
To understand the $d_1$ differential, suppose we are given $x_+, x_-$ with $\ind(x_-) = \ind(x_+) + 1$. There are only finitely many Morse trajectories $u$ up to translation connecting them: pick one such $u$, and a pair of eigenvalues $\lambda_-, \lambda_+$ of the corresponding operators with 
\[\dim \cal{M}_{[u]}(\bf{x}_-(i_-), \bf{x}_+(i_+)) = \specflow(u) + i_- - i_+ = 0.\]
When we later in Section \ref{sec:analysis_I} develop the Fredholm theory for the twisted equations, we shall see in this setting that the twisted equations over $u$ are given by the kernel of a Fredholm operator $\frac{d}{ds} + A_{u(s)}$ of index $1$, on an appropriate weighted Sobolev space. In this case when the underlying Morse moduli space is discrete, the regularity of the equations amounts to the surjectivity of this operator. In particular, it must have precisely one-dimensional kernel (and zero dimensional cokernel). Hence, there is exactly one solution $(u, \phi)$ up to rescaling for each Morse flow $u$.

The spectral flow of $\frak{p}: X \to U/O$ induces a local system $\xi$ on $X$ with fibres $\bb{F}_2[t, t^{-1}]$. The above observation then identifies the differential on the $E_1$ page with the differential on the Morse complex for $X, f$ with coefficients in $\xi$. This proves:

\begin{proposition}There is a spectral sequence, linear over $\bb{F}_2[t,t^{-1}]$, converging to $HM(X, \frak{p})$ with
\[E_2 = HM(X, \xi)\]
the Morse cohomology of $X$ with coefficients in $\xi$. In particular, if $\frak{p}$ is graded, the local system $\xi$ is trivial, and we have a bigraded spectral sequence
\[HM^*(X) \tensor \bb{F}_2[t,t^{-1}] \Rightarrow HM^*(X, \frak{p})\]\end{proposition}

For the rest of this section, we will implement these constructions in the setting of Lagrangian Floer theory. While it is not particularly more difficult to give a more general construction, it is this more concrete situation which we will be used later in the equivariant setting.

\subsection{Polarization-twisted Lagrangian Floer cohomology.} Throughout this section, let $M$ be an exact symplectic manifold, convex at infinity, and let $L_0, L_1$ be two compact, exact Lagrangian submanifolds of $M$, intersecting transversely. Recall that Lagrangian Floer theory is the formal Morse theory of the symplectic action functional on the path space
\[X = \cal{P}_M(L_0, L_1) = \{x : [0,1] \to M : x(0) \in L_0, x(1) \in L(1)\}\]
with critical points precisely at the points of $L_0 \cap L_1$.

Rather than work with an abstract map $\frak{p}: \cal{P}_M(L_0, L_1) \to U/O$, we will consider only those that arise in the following specific geometric situation: suppose there is a symplectic vector bundle $E$ over $M$ of rank $2k$, together with Lagrangian subbundles $F_0, F_1$ on $E|_{L_0}$ and $E|_{L_1}$, such that $F_0|_x$ and $F_1|_x$ are transverse at each $x \in L_0 \cap L_1$. Such data induces a map to $U/O$: indeed, observe this is precisely the data of a map $M \to BU(k)$ with reductions to $BO(k)$ over each $L_i$; since $\cal{P}_M(L_0, L_1)$ is a model for the homotopy fibre product of $L_0, L_1$ over $M$, it inherits a map to $U(k)/O(k)$ and thus $U/O$.

We will refer to the tuple $\frak{p} = (E, F_0, F_1)$ as \emph{polarization data}, hoping that no confusion is caused between thinking of $\frak{p}$ as the tuple or as a map to $U/O$. We will say that $\frak{p} = (E, F_0, F_1), \frak{p}' = (E', F_0', F_1')$ are \emph{isomorphic} if there is an isomorphism of symplectic vector bundles $\alpha: E \xrightarrow{\sim} E'$, and homotopies Lagrangian sub-bundles of $E'|_{L_i}$ between $\alpha(F_i)$ and $F'_i$ for $i = 0, 1$. This is equivalent to the existence of an isomorphism $\beta: \pi^*E \xrightarrow{\sim} \pi^*E'$ where $\pi : M \times [0,1] \to M$, such that $\beta|_{L_i \times \{i\}}$ is an isomorphism $F_i \xrightarrow{\sim} F'_i$ for $i = 0, 1$. Given polarization data $\frak{p}$, we can stabilize it to form polarization data $\frak{p} \oplus \bb{C} = (E \oplus \bb{C}, F_0 \oplus \bb{R}, F_1 \oplus i \bb{R})$; we say that two sets of polarization data $\frak{p}, \frak{p}'$ are \emph{stably isomorphic} if for some large $k, k'$, $\frak{p} \oplus \bb{C}^k$ and $\frak{p}' \oplus \bb{C}^{k'}$ are isomorphic. Observe that a stable isomorphism of polarization data yields a homotopy between the corresponding maps $\cal{P}_M(L_0, L_1) \to U/O$.

\begin{remark}At this point, one could proceed as in the previous section, and define polarization-twisted Floer theory for an abstract continuous map $\frak{p} : \cal{P}_M(L_0, L_1) \to U/O$, by coupling the Floer equation in $M$, thought of as the Morse flow of the action functional, to the twisted equation given by a representative of $\frak{p}$ by a smooth family of self-adjoint operators. This point of view has some advantages, which we will highlight later. On the other hand, the ultimate goal is to relate twisted Floer cohomology to equivariant Floer cohomology in the setting of a $\bb{Z}/2$-action on $\tilde{M}$ with fixed points $M$, by extracting a solution of the twisted equations from a degeneration of a one-parameter family of holomorphic strips in $\tilde{M}$ as the family falls into the invariant set $M$. This necessitates developing a more explicit version of polarization-twisted Floer cohomology when $\frak{p}$ is presented by vector bundles $(E, F_0, F_1)$ as above.\end{remark}

Choose a time-dependent complex structure $(I_t)_{t \in [0,1]}$ on $E$ compatible with the symplectic structure, as well as a time-dependent symplectic connection $\nabla$; recall this amounts to the choice of a complex structure and connection on $\pi^*E$, where $\pi : M \times [0,1] \to M$ is the projection.
We will usually suppress the time-dependence of $\nabla$ from the notation, and write $\nabla_t$ instead for $\nabla_{\partial/\partial t}$; we hope no confusion is caused. It will sometimes be a technical convenience to assume that $\nabla$ is in fact Hermitian (with respect to the inner product $\omega_E(\cdot, I_t \cdot)$ on $\pi^* E$), and that the sub-bundles $F_0 \subset \pi^*E|_{M \times \{0\}}$ and $F_1 \subset \pi^*E|_{M \times \{1\}}$ are parallel with respect to $\nabla$.

In particular, for $x(t) \in \cal{P}(L_0, L_1)$, we can consider the space $C^{\infty}_x(E, F_0, F_1)$ of smooth sections $\psi \in C^{\infty}(x^*E)$ such that $\psi(0) \in F_0|_{x(0)}$ and $\psi(1) \in F_1|_{x(1)}$. This carries an inner product, which we will call $\langle \cdot, \cdot \rangle_H$ in a nod to the finite-dimensional model explained in the previous section:
\begin{equation}\label{H_inner_product}\langle \psi, \psi' \rangle_H = \int_0^1 \langle \psi(t), \psi'(t) \rangle dt = \int_0^1 \omega_E(\psi(t), I_t \psi'(t)) dt\end{equation}
and we will write $|| \cdot ||_H$ for the corresponding norm. The most important feature of this inner product is that the operator $I_t \nabla_t = I_t \nabla_{\partial / \partial t}$ is self-adjoint, since
\[\langle \psi, I_t \nabla_t \psi'\rangle_H - \langle \psi', I_t \nabla_t \psi \rangle_H = \int_0^1 \left(\omega_E(\psi, \nabla_t \psi') + \omega_E(\psi', \nabla_t \psi)\right) dt = \int_0^1 \frac{d}{dt} \omega_E(\psi, \psi') = 0\]
since $F_0, F_1$ are Lagrangian. This in turn induces an inner product the $L^2$-based Sobolev completions of $C^{\infty}_x(E, F_0, F_1)$, which we will use extensively. 

It will be helpful to us to work in a more rigid local geometry around the intersection points $x \in L_0 \cap L_1$. To this end, we choose $I_t$ such that:

\begin{assumption}\label{localgeomI}There is a small neighbourhood $U$ of each $x \in L_0 \cap L_1$, on which the complex structure $I = I_t$ is constant in $t$, such that the following holds.
\begin{itemize}\item There is a trivialization
\begin{equation}\label{Etrivialization}\alpha_U: E|_{U} \cong \bb{C}^k \times U\end{equation}
of $(E|_{U}, I)$ as a Hermitian vector bundle, with $\nabla$ corresponding to the standard connection, such that over $L_0$, $L_1$,
\begin{equation}\alpha_U(F_0|_{L_0 \cap U}) = G_0 \times (L_0 \cap U), \quad \alpha_U(F_1|_{L_1 \cap U}) = G_1 \times (L_1 \cap U)\end{equation}
for fixed linear Lagrangian subspaces $G_0, G_1$ of $\bb{C}^k$;
\item There is some basis $f_1, f_2, \hdots, f_k$ of $G_0$, such that $e^{i\theta_1}f_1, e^{i \theta_2}f_2, \hdots, e^{i \theta_k} f_k$ is a basis of $G_1$, for some real numbers
\begin{equation*}0 < \theta_1 < \theta_2 < \hdots < \theta_k < \pi.\end{equation*}\end{itemize}\end{assumption}

As a consequence of the choice of $I_t, \nabla$, for each smooth $x : [0,1] \to M$ with $x(0) \in L_0, x(1) \in L_1$, there is an operator
\begin{equation}\label{diracoperator}I_t \nabla_t : W^{1,2}([0,1], x^*E) \to L^2([0,1], x^* E)\end{equation}
where it is to be understood that $W^{1,2}([0,1], x^*E)$ refers to those sections $\psi(t)$ of $x^*E$ where $\psi(0) \in F_0, \psi(1) \in F_1$ (however no such restriction is placed on $L^2([0,1], x^*E)$). These operators are unbounded and self-adjoint, of the type $\cal{S}_*$ from the previous section. As such, they can be used to define a polarization-twisted theory.

At each intersection point $x \in L_0 \cap L_1$, consider the eigenvalue equation for this operator
\begin{equation}\label{eqn:evalue}I \frac{d}{dt} \psi(t) = \lambda \psi(t)\end{equation}
for $\psi \in W^{1,2}([0,1], E|_x)$, noting that $I_t$ is constant in $t$. Assumption \ref{localgeomI} ensures that this operator has a simple spectrum, with eigenvalues $(2j + 1)\frac{\pi}{2} + \theta_i$ for each $j \in \bb{Z}$ and $i = 1, ..., k$, where $\theta_i$ are as in Assumption \ref{localgeomI}. Label these eigenvalues
\[\hdots \lambda_{-2}(x) < \lambda_{-1}(x) < 0 < \lambda_0(x) < \lambda_1(x) < \lambda_2(x) < \hdots\]
with corresponding eigensolutions $\psi_i(t)$. We will often write $i(\lambda) \in \bb{Z}$ for the numbering as above of such an eigenvalue $\lambda$. Again we write
\begin{equation}\frak{C} = \{\bf{x} = (x, \lambda_i): x \in L_0 \cap L_1; \lambda_i \in \Spec(I \frac{d}{dt})\}.\end{equation}

As before, polarization-twisted Floer cohomology is computed by a complex
\begin{equation}CF_{tw}(L_0, L_1; \frak{p}) = \bigoplus_{x \in \frak{C}} \bb{F}_2 \cdot \bf{x} \end{equation}
generated by $\frak{C}$.

The differential counts solutions to the usual Cauchy-Riemann equation for pseudoholomorphic strips in $M$, coupled with an auxiliary linear equation for sections of $E$. Namely, writing $Z$ for the Riemann surface $\bb{R}\times [0,1]$, choose an time-dependent, $\omega$-compatible almost complex structure $J_t$ on $M$, and consider pairs $(u, \phi)$ of maps $u : Z \to M$ and sections $\phi \in \Gamma(Z, u^*E)$ satisfying the \emph{twisted Floer equations}:
\begin{equation}\label{floereqn}\frac{\partial u}{\partial s} + J_t(u) \frac{\partial u}{\partial t} = 0\end{equation}
\begin{equation}\label{twistedeqn}\nabla_{\partial/\partial s} \phi + I_t(u) \nabla_{\partial/\partial t} \phi = 0\end{equation}
with boundary conditions $u(\bb{R}\times\{i\}) \subset L_i$ and $\phi(\bb{R}\times\{i\}) \subset F_i$ for $i = 0, 1$. We will also impose asymptotics: for $\bf{x}_{\pm} = (x_{\pm}, \lambda_{\pm}) \in \frak{C}$, we will consider solutions such that
\begin{align*}u(s, t) &\to x_{\pm} \qquad \text{ uniformly as } s \to \pm \infty \\
\phi(s, t) &\sim C_{\pm} e^{-\lambda_{\pm}s}\psi_{\pm}(t) \quad \text{ as } s \to \pm \infty\end{align*}
for some nonzero constants $C_{\pm}$ and eigensolutions $\psi_{\pm}(t)$ to the operator $I\frac{d}{dt}$; more precisely this means
\begin{equation}\frac{\phi(s,t)}{||\phi||_H} \to C_{\pm} \psi_{\pm}(t) \qquad \text{ unifomly as } s \to \pm \infty\end{equation}
As before, we say that $(u, \phi)$ has asymptotics of type $\lambda_{\pm}$ at $\pm \infty$.

We also have the $\tau$-model for the equations. Assuming the identity principle (Proposition \ref{uniquecontinuation}), we can set $\varphi(s,t) = \phi(s)/||\phi(s)||_H$, and the equations for $\varphi$ are
\begin{align}\label{taumodeleqn}\nabla_{\partial/\partial s} \varphi + I_t \nabla_{\partial/\partial t} \varphi - \Lambda(\varphi)\varphi &= 0\\
||\varphi(s)||_H &= 1\end{align}
where $\Lambda(\varphi)$ as before is
\begin{equation}\label{Lambda} \Lambda(\varphi)(s) = \langle \varphi, I_t \nabla_{\partial / \partial t} \varphi \rangle_H = \frac{\langle \phi, I_t \nabla_{\partial / \partial t} \phi \rangle_H}{||\phi||_H^2}.\end{equation}
Again, by solving the equation
\[ \frac{dr}{ds} + \Lambda(\varphi)r = 0\]
we can recover the original $\sigma$-model equation for $\phi(s,t) = r(s) \varphi(s,t)$. In this set-up, we say $\varphi$ has asymptotics of type $\lambda_{\pm}$ if 
\[\varphi(s, t) \to \psi_{\pm}(t) \text{ or } -\psi_{\pm}(t)\]
as $s \to \pm \infty$, where $\psi_{\pm}(t)$ are a choice of \emph{unit eigenvectors} for $\lambda_{\pm}$ in the $|| \cdot ||_H$-norm at $x_{\pm}$.

We will also impose that $u$ is of finite energy; the precise finiteness condition on $\phi$ will be spelled out in the next section, either in terms of a weighted Sobolev norm for the $\sigma$-model, or in a more conventional fashion for the $\tau$-model. These solutions come with a natural $\bb{R}$-action by translation in $s$, and either a $\bb{R}^*$-action by rescaling $\phi$ in the $\sigma$-model, or a $\bb{Z}/2$-action in the $\tau$-model.

\begin{definition}For $\bf{x}_{\pm} = (x_{\pm}, \lambda_{\pm}) \in \frak{C}$, the \emph{moduli space of twisted trajectories} $\cal{M}(\bf{x}_-, \bf{x}_+)$ is the space of finite energy solutions (in equivalently the $\sigma$ or $\tau$ models), modulo translation and rescaling, with asymptotics as above. We will distinguish between a solution $(u, \phi)$, and a \emph{trajectory}, which is an equivalence class of solutions. $\bf{x}_{\pm}$ will sometimes be called the \emph{endpoints} of a trajectory, and by analogy to Morse theory we will sometimes say that a trajectory $[(u, \phi)]$ is a flow from $\bf{x}_-$ to $\bf{x}_+$.\end{definition} 

It will be occasionally useful to observe that the twisted Floer equations can be alternatively viewed as the usual Floer equation in the total space of $E$. Indeed, the symplectic connection $\nabla$ defines a splitting of the tangent space at each $(x, \psi) \in E$ as $T_{(x, \psi)} E \cong T_x M \oplus E_x$; accordingly there is a compatible almost complex structure on the total space of $E$ given as $\begin{pmatrix}J_t & 0 \\ 0 & I_t\end{pmatrix}$ in this decomposition. The Floer equation for this almost complex structure is exactly given by the system \eqref{floereqn}, \eqref{twistedeqn}. In particular, this allows us to immediately deduce an identity principle for the twisted equations:
\begin{proposition}\label{uniquecontinuation}Suppose that $(u, \phi)$ is a solution to the twisted equations, and that the set $\{(s,t) \in Z: \phi(s,t) = 0\}$ has an accumulation point. Then $\phi$ is identically zero.\end{proposition}
This could of course be proved directly, but it follows immediately from the identity principle of \cite{FloerHoferSalamon95} applied to $(u, \phi)$ and $(u, -\phi)$ thought of as pseudoholomorphic strips in the total space of $E$.

In the following Section \ref{sec:analysis_I}, we will detail the necessary analytic arguments to show that, after possibly perturbing $I_t$, the spaces $\cal{M}(\bf{x}_-, \bf{x}_+)$ are smooth manifolds, with a compactification by broken flows. Consequently, we can define a differential as before by counting the points of the zero-dimensional components of the moduli space of twisted trajectories, and obtain:
\begin{definition}The polarization-twisted Floer cohomology $HF_{tw}(L_0, L_1, \frak{p})$ is the cohomology of $CF_{tw}(L_0, L_1, \frak{p})$ with the differential given by
\begin{equation}d(\bf{x}_+) = \sum\limits_{\bf{x}_-, [u] : \dim \cal{M}_{[u]}(\bf{x}_-, \bf{x}_+) = 0} \#\cal{M}_{[u]}(\bf{x}_-, \bf{x}_+) \cdot \bf{x}_-.\end{equation}\end{definition}

Many of the observations made earlier about twisted Morse theory carry through to this case. There is a canonical endomorphism $T$, giving $HF_{tw}(L_0, L_1; \frak{p})$ the structure of a $\bb{F}_2[t]$-module. This is defined in exactly the same way as before, by defining a free $\bb{F}_2[\bb{Z}/2]$ complex
\[\widetilde{CF}_{tw}(L_0, L_1, \frak{p})\]
generated by each unit eigensolution $\psi(t)$ of the operator $I\frac{d}{dt}$ for each $\bf{x} = (x, \lambda)$, and using the construction of Section 2 to realise
\begin{equation}\label{chicken_or_egg}CF_{tw}(L_0, L_1, \frak{p}) \cong \left(\widetilde{CF}_{tw}(L_0, L_1, \frak{p})\right)_{\bb{F}_2}.\end{equation}

\begin{remark}There is a chicken-or-the-egg question as to whether one should view the fundamental object as $\widetilde{CF}_{tw}(L_0, L_1; \frak{p})$ and take formula \eqref{chicken_or_egg} as a definition of polarization-twisted Floer cohomology. However, since we believe that the reader is more comfortable thinking about $\bb{F}_2[t]$- rather than $\bb{F}_2[\bb{Z}/2]$-modules, we have opted instead to develop $CF_{tw}(L_0, L_1; \frak{p})$ first as a complex of $\bb{F}_2$-vector spaces, and observe that \eqref{chicken_or_egg} equips it with a canonical module structure up to $A_{\infty}$-equivalence. This decision informs the notation used throughout later sections, including our discussion of equivariant theory, but makes no actual difference to the content: the analytical back-end of Sections \ref{sec:analysis_I} and \ref{sec:analysis_II} can equally be seen as proving that the appropriate complexes of free $\bb{F}_2[\bb{Z}/2]$-modules are in fact complexes, and the natural maps between them are chain maps linear over $\bb{F}_2[\bb{Z}/2]$.\end{remark}

Moreover, the twisted Floer complex is filtered by the symplectic action on $L_0 \cap L_1$. On the associated graded of this filtration, the differential counts only the twisted trajectories $(u, \phi)$ with constant $u$: over $x \in L_0 \cap L_1$, these are exactly the solutions to
\[\frac{d \phi}{ds} + I_{x} \frac{d \phi}{dt} = 0\]
for $\phi: Z \to E_x$ such that $\phi(\cdot, i) \in F_i|_x$ for $i = 0,1$, with asymptotics to given eigensolutions $I_x \frac{d}{dt} \phi = \lambda \phi$. These solutions are then exactly as in \ref{constant_u_solutions}. We likewise deduce from this that the endomorphism $T$ is invertible, meaning $HF_{tw}(L_0, L_1; \frak{p})$ is in fact a $\bb{F}_2[t,t^{-1}]$-module.

Moreover, there is a local system $\xi$ on the path space $\cal{P}_M(L_0, L_1)$ with fibre $\bb{F}_2[t,t^{-1}]$; by augmenting the usual Floer differential by parallel transport in $\xi$, we can define $HF(L_0, L_1; \xi)$ the Floer cohomology of $L_0, L_1$ with coefficients in $\xi$. The action filtration then gives rise to a spectral sequence
\[HF(L_0, L_1, \xi) \Rightarrow HF_{tw}(L_0, L_1, \frak{p})\]
which we will refer to as the \emph{action spectral sequence}.

\subsection{Invariance.}\label{invariance_I}

Just as twisted Morse theory depends only on $\frak{p}: X \to U/O$ up to homotopy, $HF_{tw}(L_0, L_1; \frak{p})$ is invariant of the particular choice of complex structure $I_t$ and connection $\nabla$. Indeed, this can be seen through the usual argument using the continuation equations. Suppose that $I_t'$ and $\nabla'$ were some other choice of complex structure and connection; by the contractibility of the space of complex structures and connections, we could then find a $Z$-dependent complex structure $I''_{s,t}$ and connection $\nabla''$ (by this we mean a compatible complex structure and symplectic connection on the pull-back of $E$ to $M \times Z$ under the projection to $M$) with
\[(I''_{s,t}, \nabla'') = \begin{cases}(I'_t, \nabla') & \text{ for } s \ll 0; \\
(I_t, \nabla) & \text{ for } s \gg 0.\end{cases}\]
We then consider the non-autonomous continuation map equations for $u: Z \to M$ and $\phi$ a section of $u^*E$:
\[\partial_s u + J_t(u) \partial_t u, \qquad \nabla''_s \phi + I''_{s,t}(u) \nabla''_t \phi\]
with boundary conditions $u(\cdot, i) \in L_i$ and $\phi(\cdot, i) \in F_i$ for $i = 0,1$. By counting the zero dimensional moduli spaces of such solutions with asymptotics at $-\infty$ to $\bf{x}' = (x', \lambda')$, for $\lambda'$ an eigenvalue of $I'_{x'}\frac{d}{dt}$ at $x' \in L_0 \cap L_1$, and at $\infty$ to $\bf{x} = (x, \lambda)$ for $\lambda$ an eigenvalue of $I_x \frac{d}{dt}$ at $x \in L_0 \cap L_1$, we can define a continuation map between the twisted Floer complexes defined with each sets of data
\[CF_{tw}(L_0, L_1; \frak{p}; (I, \nabla)) \to CF_{tw}(L_0, L_1; \frak{p}; (I', \nabla').\]
There is likewise a map in the other direction, and the usual argument by concatenation proves that these are quasi-inverse. The map can clearly also be lifted to the level of free $\bb{F}_2[\bb{Z}/2]$-complexes, and hence we have $A_{\infty}$-quasi-isomorphic $\bb{F}_2[t]$-modules. 

Analytically, the work we will carry out in Section \ref{sec:analysis_I} adapts straightforwardly to this setting. Perhaps the key point to emphasise is that the fundamental estimate \eqref{curvaturebound} used to prove Proposition \ref{Lambdabound} holds essentially as stated with a $Z$-dependent complex structure and connection (the only difference is that the inner product $\langle \cdot, \cdot \rangle = \int_0^1 \omega_E(\cdot, I''_{s,t} \cdot) dt$ used must now be $s$-dependent); this is the main estimate used to prove the moduli spaces are compact.

In a similar vein, twisted Floer theory is invariant under the choice of almost complex structure $J_t$ in the base symplectic manifold $M$. It is also invariant under a Hamiltonian perturbation of the equations: for a time-dependent, compactly supported Hamiltonian $H : M \times [0,1] \to \bb{R}$, we could have just as well taken the equations as
\[\partial_s u + J_t(u)(\partial_t u - X_{H_t}) = 0, \qquad \nabla_{\partial / \partial s} \phi + I_t(u) \nabla_{\partial / \partial t} \phi = 0\]
where $X_{H_t}$ is the corresponding Hamiltonian vector field and used these to define 

Twisted Floer theory is also invariant under perturbations of the polarization data $\frak{p}$. Indeed, suppose that we have a homotopies $(F_0)_s, (F_1)_s$ of Lagrangian sub-bundles of $E|_{L_0}, E|_{L_1}$, meaning Lagrangian sub-bundles of the pullback of $E$ to each of $L_0 \times \bb{R}$ and $L_1 \times \bb{R}$, such that for some fixed Lagrangian sub-bundles $F_0, F_0'$ of $E|_{L_0}$ and $F_1, F_1'$ of $E|_{L_1}$, forming two sets of polarization data $\frak{p} = (E, F_0, F_1)$ and $\frak{p}' = (E, F_0', F_1')$ we have
\[(F_0)_s = \begin{cases} F'_0 & \text{ for } s \ll 0;\\ F'_1 & \text{ for } s \gg 0;\end{cases} \qquad (F_1)_s = \begin{cases} F'_0 & \text{ for } s \ll 0;\\ F'_1 & \text{ for } s \gg 0.\end{cases}\]
Likewise we can consider the twisted equations for $(u, \phi)$ but with the moving Lagrangian boundary conditions that
\[\phi(s, 0) \in (F_0)_s, \qquad \phi(s,1) \in (F_1)_s.\]
Assuming that both $F_0|_x, F_1|_x$ are transverse at each $x \in L_0 \cap L_1$ (likewise for $F'_0|_x, F'_1|_x$) we can define a chain map
\[CF_{tw}(L_0, L_1; \frak{p}) \to CF_{tw}(L_0, L_1; \frak{p}')\]
which is moreover again an $A_{\infty}$-quasi-isomorphism of $\bb{F}_2[t]$-modules. One upshot of this is we can then define twisted Floer cohomology $HF_{tw}(L_0, L_1; \frak{p})$ for arbitrary Lagrangian sub-bundles $F_0, F_1$ of $E$, without the additional posit that $F_0, F_1$ are transverse over $L_0 \cap L_1$.

As in ordinary Lagrangian Floer theory, twisted Floer cohomology is also invariant under the choice of almost complex structure $J_t$ in the base symplectic manifold $M$, by a similar continuation map argument. It is also invariant under exact Lagrangian isotopies of $L_0$ and $L_1$ (noting that in the case that $L_0$ or $L_1$ is non-compact, we must restrict ourselves to compactly supported exact isotopies), by a continuation map argument with moving Lagrangian boundary conditions on the base. We must of course specify how to extend $F_0, F_1$ over this isotopy: however in view of the invariance under perturbations of $F_0, F_1$ above, it does not matter which extension we choose; an obvious choice is to just use the symplectic connection $\nabla$ to parallel transport $F_0, F_1$ along the isotopy. In particular, twisted Floer cohomology is invariant under compactly supported Hamiltonian isotopies of $M$, and we can define the twisted Floer cohomology $HF_{tw}(L_0, L_1; \frak{p})$ without the assumption that $L_0, L_1$ intersect transversely.

\begin{remark}While we could have built Hamiltonian isotopy invariance into the story by including a Hamiltonian perturbation term into the twisted equations directly, we have chosen not to for simplicity: we prefer to take fixed trivializations of $E, F_0, F_1$ around each intersection point of $L_0, L_1$, rather than around Hamiltonian chords connecting them.\end{remark}

The new feature here is that even if $\frak{p}' = \frak{p}$, the homotopies $(F_i)_s$ can form a non-contractible loop in the space of Lagrangian sub-bundles, and the induced isomorphism $HF_{tw}(L_0, L_1; \frak{p}) \xrightarrow{\sim} HF_{tw}(L_0, L_1; \frak{p})$ may not be the identity. In fact, we can make a precise conjecture as to the nature of this automorphism. Suppose for a second that our homotopy of polarization data holds $F_1$ constant, while moving $F_0$ through a loop $(F_0)_s$ of Lagrangian sub-bundles of $E|_{L_0}$. Such a loop gives rise to a map
\[L_0 \to \Omega(U(n)/O(n)) \to \Omega(U/O) \simeq \bb{Z} \times BO\]
by Bott periodity; in other words a virtual vector bundle $\eta \in KO^0(L_0)$. Moreover, $HF_{tw}(L_0, L_1; \frak{p})$ should carry natural structure of a $H^*(L_1; \bb{F}_2)$-$H^*(L_0; \bb{F}_2)$-bimodule: this is out of the scope of this paper, but one would construct this either by defining a version of the triangle product for twisted Floer cohomology, or by counting solutions to the twisted equations with marked boundary points constrained to lie in particular cycles of $L_0, L_1$. We then expect that the automorphism of $HF_{tw}(L_0, L_1; \frak{p})$ induced by the homotopy of Lagrangian sub-bundles to be given as multiplication by a version of the total Stiefel-Whitney class of $\eta$:
\[t^{\rank(\eta)} w(\eta) = t^{\rank(\eta)}\left(1 + t^{-1} w_1(\eta) + t^{-2} w_2(\eta) + \hdots\right).\]

Indeed, the fundamental object in calculating twisted Floer theory should not be $\frak{p} = (E, F_0, F_1)$ at all, but rather the map
\[\cal{P}_M(L_0, L_1) \to U/O;\]
in particular, polarization data $\frak{p}, \frak{p}'$ that give rise to homotopic maps to $U/O$ should yield isomorphic twisted cohomology (although the isomorphism itself would depend on the choice of homotopy, as discussed above). In particular, we should expect twisted cohomology to be \emph{stabilization invariant}: replacing $\frak{p}$ by $\frak{p}\oplus \bb{C} = (E \oplus \bb{C}, F_0 \oplus \bb{R}, F_1 \oplus i \bb{R})$ yields isomorphic twisted cohomology. We will in fact prove this as Corollary \ref{stabilization_invariance} of the K\"{u}nneth theorem, but it would be immediate from a continuation map argument if we were to develop a more general theory of twistings by abstract continuous maps to $U/O$.

Stabilization invariance immediately implies that if there is a stable trivialization of $\frak{p}$ (meaning an isomorphism $\frak{p} \oplus \bb{C}^{\ell} \cong (\bb{C}^{k + \ell}, \bb{R}^{k + \ell}, i \bb{R}^{k + \ell})$ for some $\ell$), then there is an induced isomorphism
\begin{equation}\label{untwisting}HF_{tw}(L_0, L_1; \frak{p}) \cong HF(L_0, L_1) \tensor \bb{F}_2[t,t^{-1}]\end{equation}
which of course depends on the choice of stable trivialization.

There are slightly laxer conditions on which one should expect an untwisting as in \eqref{untwisting}, whenever there is a Lagrangian distribution $D \subset E \oplus \bb{C}^{\ell}$ over $M$, and homotopies between each $F_i \oplus \bb{R}^{\ell}$ and $D|_{L_i}$. Indeed, this yields a null-homotopy of the map $\cal{P}_M(L_0, L_1) \to U/O$. However, this is only obvious in our formulation of twisted Floer cohomology in the event that $D$ is equipped with a \emph{flat connection}. In this case, using the canonical complex structure on $E \cong D \tensor \bb{C}$ and the connection induced from $D$, the vanishing of the right hand side of the fundamental estimate \eqref{curvaturebound} means there can be no twisted trajectories $(u, \varphi)$ with Maslov index $\mu(u) \ge 2$. Without a flat connection, the curvature term in \eqref{curvaturebound} is not necessarily zero, and we cannot a priori rule out such twisted trajectories.

\subsection{A topological interpretation.} Before we proceed to give the analytic foundations of polarization-twisted Floer cohomology, we will mention a topological alternative. Let us first describe the toy case of when $L_0 \cap L_1 = \{x_-, x_+\}$ consists of precisely two critical points, with a single homotopy class $[u]$ of holomorphic strips connecting them; hence for a regular almost complex structure we have a closed $n$ manifold
\[\cal{M}(x_-, x_+)\]
of holomorphic strips, where $n + 1= \mu(u)$ is the Maslov index. Given polarization data $\frak{p} = (E, F_0, F_1)$ together with $(I, \nabla)$, the index of the operator $\bar{\nabla}_I: W^{2,k}(u^*E) \to L^{2,k}(u^*E)$ acting on the \emph{unweighted} Sobolev spaces determines a virtual vector bundle
\[\eta \to \cal{M}\]
of rank $k = \specflow(u)$. Let us suppose, for simplicity, $\bar{\nabla}_I$ is surjective for all $u \in \cal{M}$, so that $\eta$ is in fact an honest vector bundle. Take $\bf{x}_+ = (x_+, \lambda_0(x_+))$ where as before $\lambda_0(x_+)$ is the smallest positive eigenvalue of $I_{x_+}\frac{d}{dt}$; we are interested in the coefficient of $\bf{x}_- = (x_-, \lambda_{-n-k}(x_-))$ in the twisted differential $d(\bf{x}_+)$. This coefficient is a count of pairs $(u, \phi)$ of $u \in \cal{M}$ and $\phi \in \eta|_u$ so that $\phi(s,t) \sim C e^{-\lambda_{-n-k}s}\psi_{-n-k}(t)$ as $s \to -\infty$. However from any such $\phi$, we can extract an asymptotic expansion
\[\phi = a_{-1} e^{-\lambda_{-1} s}\psi_{-1} + \hdots + a_{-n-k+1} e^{-\lambda_{-n-k+1} s} \psi_{-n-k+1} + a_{-n-k}e^{-\lambda_{-n-k}s} \psi_{-n-k} + \hdots\]
and thus we wish to count the solutions $(u, \phi)$ where $a_{-1} = \hdots = a_{-n-k+1} = 0$ (for such a solution, generically $a_{-n-k} \not= 0$). However clearly the coefficients of the asymptotic expansion determine map of vector bundles over $\cal{M}$
\[A : \eta \to \bb{R}^{n + k - 1}\]
and thus for generic data, exactly one scalar multiplication class of solutions $(u, \phi)$ occurs at each of the finitely many points where $A$ fails to be injective. It then follows from Porteous' formula that
\begin{proposition}\label{porteous_calculation}The coefficient of $\bf{x}_-$ in $d(\bf{x}_+)$ is given by the top Stiefel-Whitney number
\begin{equation}\label{porteous_formula}\langle w_n(- \eta), [\cal{M}]\rangle\end{equation}
where $-\eta$ is the negative index bundle of $\bar{\nabla}_I$, and $[\cal{M}] \in H_n(\cal{M}, \bb{F}_2)$ is the fundamental class of $\cal{M}$.

In particular, by considering the action spectral sequence, the twisted Floer cohomology $HF_{tw}(L_0, L_1; \frak{p})$ vanishes if $\langle w_n(- \eta), [\cal{M}]\rangle \not= 0$, while in the event that this Stiefel-Whitney number vanishes, the spectral sequence collapses, and
\[HF_{tw}(L_0, L_1; \frak{p}) \cong \bb{F}_2[t,t^{-1}] \oplus \bb{F}_2[t,t^{-1}].\]\end{proposition}

This calculation is the basic link between this work and Seidel-Smith's original proposal for polarization-twisted cohomology in \cite{SeidelSmith10}, which essentially takes \eqref{porteous_formula} as a definition of the twisted differential. When working with more than two critical points, so that $\cal{M}(x,y)$ is a manifold with boundary given by broken strips, care must be taken to choose cochain representatives for the Stiefel-Whitney classes of $-\eta$ and fundamental chains for $\cal{M}$ to be compatible with the strip-breaking structure.

We conjecture that the resulting theory is equivalent to ours, where we manually cut down the moduli spaces using the auxiliary $\bar{\nabla}_I$-equation. One challenge in comparing the two approaches that in our picture, the relevant $\bb{F}_2[t,t^{-1}]$ action on $CF_{tw}$ is \emph{not} given by simply raising and lowering the eigenvalue level (although they do coincide on the action spectral sequence, which is what is needed for the calculation in Proposition \ref{porteous_calculation}).

\section{Analytical aspects I}\label{sec:analysis_I}

\subsection{Weighted Sobolev spaces.} Let us describe a Sobolev space set-up so as to make the equations \eqref{floereqn}, \eqref{twistedeqn} Fredholm, and so that generic perturbations of $I, \nabla$ yield smooth moduli spaces.

Fix $k > 2$ an order of differentiability; we will throughout work with $L^2$-based Sobolev spaces. We apologise in advance for our somewhat non-standard choice of notation: we will write $L^{2,k}$ rather than $L^2_k$ for the $L^2$-Sobolev space of $k$-times differentiable functions, and we will reserve the subscripts for the exponential weights we will eventually impose to control the asymptotic behaviour of the twisted equations.

Let us first recall the set-up for ordinary Floer theory. For $x_-, x_+ \in L_0 \cap L_1$, there is a Banach (in fact, real Hilbert) manifold $\cal{Z}^{2,k}(x_-, x_+)$ of $L^{2,k}_{loc}$-maps $u: Z \to M$ such that $u(s, 0) \in L_0$ and $u(s, 1) \in L_1$ for each $s$, and such that for $s \gg 0$ and $s \ll 0$ respectively, we can write
\begin{equation}u(s,t) = \exp_{x_{\pm}}(\xi_{\pm}(s,t))\end{equation}
where $\xi(s,t) \in T_{x_{\pm}}M$ has finite $L^{2,k}$-norm. The tangent space at $T_u \cal{Z}^{2,k}(x_-, x_+)$ at a map $u$ is the Banach space $W^{2,k}(u^* TM)$ of class $L^{2,k}$ vector fields $\xi$ on $u$ such that $\xi(s,0) \in TL_0$ and $\xi(s, 1) \in TL_1$. This Banach manifold is equipped with a Banach vector bundle $\cal{L}^{2,k-1}(x_-, x_+)$, whose fibre at $u$ is the Banach space $L^{2,k-1}(u^* TM)$ of class $L^{2,k-1}$ vector fields $\eta$ on $u$, with no conditions imposed on $\bb{R} \times \{0\}$ or $\bb{R} \times \{1\}$. We will distinguish throughout this section in our notation between spaces of sections with Lagrangian boundary conditions, denoted by the letter $W$, and spaces without, denoted by $L$.

Given a time-dependent almost complex structure $J_t$ on $M$, the $\bar{\partial}_{J}$ operator then defines a Fredholm section of $\cal{L}^{2,k-1}(x_-, x_+)$. By elliptic regularity, the zeroes of this section are precisely the finite energy smooth $J_t$-holomorphic strips $u$ with endpoints at $x_{\pm}$ and boundary conditions on $L_0, L_1$. Moreover, for generic $J_t$, at each solution the linearized operator
\begin{equation}(D\bar{\partial})_u : T_u \cal{Z}^{2,k} \to L^{2,k-1}(u^*TM)\end{equation}
is surjective, in which case we say the almost complex structure $J_t$ is regular, and so by the implicit function theorem the space of solutions is a finite-dimensional manifold modelled on $\ker D\bar{\partial}$. We will fix, once and for all, such a regular $J_t$.

Returning to the twisted equations for polarization data $\frak{p} = (E, F_0, F_1)$, for $u \in \cal{Z}^{2,k}(x_-, x_+)$, consider the two Banach spaces
\begin{equation}W^{2,k}(u^* E) = \{\phi \in L^{2,k}(Z, u^*E): \phi(\cdot, 0) \in F_0, \phi(\cdot, 1) \in F_1\}, \quad L^{2,k-1}(u^*E) = L^{2,k-1}(Z, u^*E).\end{equation}
The definition of the $L^{2,k}$ norm for sections of $E$ depends both on a metric on $E$ (which can be given say by a complex structure) and a connection; however the class of $L^{2,k}$ sections is independent of these choices. These each form Banach vector bundles $\cal{W}^{2,k}(E), \cal{L}^{2,k-1}(E)$ over $\cal{W}^{2,k}(x_-, x_+)$. Then, given a time-dependent complex structure $I = \{I_t\}_{t \in [0,1]}$ on $E$, together with a (time-dependent) symplectic connection $\nabla$, at each $u$ we have:

\begin{proposition}The operator
\begin{equation}\bar{\nabla}_{I} = \nabla_{\partial/\partial s} + I_t \nabla_{\partial/\partial t} : W^{2,k}(u^*E) \to L^{2,k-1}(u^*E).\end{equation}
is Fredholm. Moreover, its index, which we denote
\begin{equation}\label{specflow_sfu}\specflow(u)\end{equation}
and refer to as the ``spectral flow over $u$'', can be computed as the Maslov index as the loop of Lagrangian subspaces of $E$ along $\partial u$ formed by a concatenation of $F_0|_{u(s, 0)}$ and $F_1|_{u(s,1)}$, where we join the transversely intersecting subspaces $F_0|_{x_-}, F_1|_{x_-}$ and $F_0|_{x_+}, F_1|_{x_+}$ by the canonical ``short path''.\end{proposition}

The proof is entirely standard, since our geometric assumptions ensure that near each endpoint the equation is just the usual $\bar{\partial}$ equation in $E_{x_{\pm}}$.

However, these Sobolev spaces are not good enough to model twisted Floer theory: solutions to $\bar{\nabla}_I \phi = 0$ in $W^{2,k}(u^*E)$ must decay at $\pm \infty$. In particular they can only represent twisted flows $(u, \phi)$ with asymptotics of type $\lambda_- < 0$ at $- \infty$ and $\lambda_+ > 0$ at $+\infty$. To remedy this, we introduce exponential weights along the ends of the strip. 

For a pair of generators $\bf{x}_{\pm} = (x_{\pm}, \lambda_{\pm})$, first choose some $\delta > 0$ such that there are no eigenvalues of $I_{x_-} \frac{d}{dt}$ in the interval $(\lambda_-, \lambda_- + \delta)$, and no eigenvalues of $I_{x_+} \frac{d}{dt}$ in $(\lambda_+ - \delta, \lambda_+)$. Then, choose a smooth function $w: \bb{R} \to \bb{R}$ such that
\begin{equation}\label{weights}w(s) = \begin{cases} (\lambda_- + \delta)s & \text{ for } s \ll 0\\
(\lambda_+ - \delta)s & \text{ for } s \gg 0.\end{cases}\end{equation}

We then consider the weighted Sobolev spaces
\begin{equation}\label{weighted_sobolev_spaces}W^{2,k}_{\lambda_- \lambda_+}(u^*E) = e^{-w(s)}W^{2,k}(u^*E), \qquad L^{2,k-1}_{\lambda_- \lambda_+}(u^*E) = e^{-w(s)} L^{2,k-1}(u^*E).\end{equation}
Observe that different choices of $w$ give rise to equivalent metrics on the same spaces $W^{2,k}_{\lambda_- \lambda_+}, L^{2,k-1}_{\lambda_- \lambda_+}$. 

These then form Banach bundles $\cal{W}^{2,k}_{\lambda_- \lambda_+}(E)$ and $\cal{L}^{2,k-1}_{\lambda_- \lambda_+}(E)$ over $\cal{Z}^{2,k}(x_-, x_+)$.
There is again an operator $\bar{\nabla}_I : W^{2,k}_{\lambda_- \lambda_+}(u^*E) \to L^{2,k-1}_{\lambda_- \lambda_+}(u^*E)$, yielding a bundle map $\cal{W}^{2,k}_{\lambda_- \lambda_+}(E) \to \cal{L}^{2,k-1}_{\lambda_- \lambda_+}(E)$. Indeed, we could alternatively view this an operator on unweighted Sobolev spaces through the isomorphisms of \ref{weighted_sobolev_spaces}:
\begin{equation}\label{unweighted_operator}F_{\lambda_- \lambda_+} : W^{2,k}(u^*E) \to L^{2,k-1}(u^*E), \quad \phi(s,t) \mapsto \nabla_s \phi(s,t) + I_t \nabla_t \phi(s,t) - w'(s)\phi(s,t)\end{equation}
at the price of slightly obfuscating the translation invariance.

The following proposition follows immediately from the foundational work of \cite{LockhartMcOwen85} on the use of weighted Sobolev spaces in the index theory of elliptic operators:

\begin{proposition}\label{sigma_model_index}The operator $\bar{\nabla}_I : W^{2,k}_{\lambda_- \lambda_+}(u^*E) \to L^{2,k-1}_{\lambda_- \lambda_+}(u^*E)$ is Fredholm, of index
\[\specflow(u) + i(\lambda_-) - i(\lambda_+) + 1.\]\end{proposition}

Indeed, the formula for the index can be directly seen by comparing the spectral flow of the operator $F_{\lambda_- \lambda_+}$ of \eqref{unweighted_operator} with the index of $\bar{\nabla}_I$ acting on the unweighted spaces.

By elliptic regularity, if $u$ is smooth (such as if $u$ itself solves $\bar{\partial}_J u = 0$), any solutions $\phi \in W^{2,k}_{\lambda_- \lambda_+}(u^*E)$ to $\bar{\nabla}_I \phi = 0$ are automatically smooth. Moreover, by the geometric assumptions \ref{localgeomI} on $I_t$ and $\nabla$ in a neighbourhood of each intersection point $x \in L_0 \cap L_1$, near the ends of the strip $Z$, $\phi$ can be thought of as an ordinary holomorphic map to $E_{x_{\pm}} \cong \bb{C}^k$ with boundary conditions on linear subspaces $F_0$, $F_1$. Considering for a moment just the positive end of $Z$, $e^{M(s+it)}\phi(s,t)$ is also holomorphic and of finite energy for any $M \le \lambda_+ - \delta$. Choosing $M$ to be an integer multiple of $\pi$, $e^{M(s+it)}\phi(s,t)$ also has boundary conditions on $F_0$, $F_1$. We can apply the exponential decay results of Robbin-Salamon \cite{RobbinSalamon01} in this particularly simple case, to deduce that there is some nonzero eigensolution $I \frac{d}{dt} \psi(t) = \mu \psi(t)$ on $E_{x_+}$ such that
\begin{equation*}\lim\limits_{s \to \infty} e^{\mu s} e^{M(s+it)} \phi(s,t) = \psi(t)\end{equation*}
with this convergence being uniform in $t$, and exponential to all derivatives in $s$. In particular, dividing out by $e^{M(s+it)}$, we obtain an exponential decay result for solutions of the twisted equations:

\begin{proposition}\label{twisted_solution_asymptotics}Suppose $\phi(s,t) \in W^{2,k}_{\lambda_- \lambda_+}(u^*E)$ solves $\bar{\nabla}_I \phi = 0$. Then there is an eigenvalue $\lambda \ge \lambda_+$ of $I_{x_+}\frac{d}{dt}$ and a nonzero eigensolution $\psi$, we have
\begin{equation}\phi(s,t)/||\phi(s,t)||_H \to \psi(t)\end{equation}
uniformly in $t$ and exponentially to all derivatives in $s$. The same conclusion holds over the negative end of the strip as well.\end{proposition}

For fixed $\lambda_-, \lambda_+$, if we temporarily write $\mu_- = \lambda_{i_- - 1} < \lambda_{i_-} = \lambda_-$ for the next eigenvalue of $I \frac{d}{dt}$ at $x_-$ lower than $\lambda_-$, and similarly $\mu_+$ for the next eigenvalue of $I \frac{d}{dt}$ at $x_+$ higher than $\lambda_+$, observe that there are natural inclusions
\begin{equation*}W^{2,k}_{\mu_- \lambda_+}(u^*E) \subset W^{2,k}_{\lambda_- \lambda_+}(u^*E), \quad W^{2,k}_{\lambda_- \mu_+}(u^*E) \subset W^{2,k}_{\lambda_- \lambda_+}(u^*E).\end{equation*}
Then by the above argument, for a fixed $J_t$-holomorphic strip $u \in \cal{M}(x_-, x_+)$, the space of twisted solutions $(u,\phi)$ covering $u$ with asymptotics of type $\lambda_-, \lambda_+$ is the complement
\begin{equation*}\ker \bar{\nabla}_I |_{W^{2,k}_{\lambda_- \lambda_+}} \backslash \left(\ker \bar{\nabla}_I|_{W^{2,k}_{\mu_- \lambda_+}} \cup \ker \bar{\nabla}_I|_{W^{2,k}_{\lambda_- \mu_+}}\right).\end{equation*}
Consequently, the moduli space $\cal{M}_u(\bf{x}_-, \bf{x}_+)$ of twisted trajectories $(u, \phi)$ covering a fixed pseudoholomorphic strip $u$ with asymptotics $\lambda_-, \lambda_+$ is exactly the complement of two projective subspaces in some $\bb{R}P^{\ell}$:
\begin{equation*}\bb{R}P(\ker \bar{\nabla}_I|_{W^{2,k}_{\lambda_- \lambda_+}}) \backslash \left(\bb{R}P(\ker \bar{\nabla}_I|_{W^{2,k}_{\mu_- \lambda_+}}) \cup \bb{R}P(\ker \bar{\nabla}_I|_{W^{2,k}_{\lambda_- \mu_+}})\right).\end{equation*}

\subsection{Transversality.} We now turn to formulating an appropriate notion of regularity for the twisted equations. Note that we do not expect, for a fixed $u$, the operator $\bar{\nabla}_I$ to be surjective for all solutions $(u, \phi)$ of the equations. Rather, we ask for surjectivity of the combined linearization of \eqref{floereqn}, \eqref{twistedeqn} at $(u, \phi)$, which takes the following form:
\begin{equation}\label{linearizedeqn}\begin{pmatrix}D_u & B \\ 0 & \bar{\nabla}_I\end{pmatrix} : W^{2,k}(u^* TM) \oplus W^{2,k}_{\lambda_- \lambda_+}(u^*E) \to L^{2,k-1}(u^*TM) \oplus L^{2,k-1}_{\lambda_- \lambda_+}(u^*E)\end{equation}
where $D_u = (D \bar{\partial}_J)_u: W^{1,p}(u^* TM) \to L^p(u^*TM)$ is the usual linearization of the $\bar{\partial}_J$ operator along the solution $u$, and $B = B_{(u, \phi)} : W^{2,k}(u^* TM) \to L^{2,k-1}_{\lambda_- \lambda_+}(u^*E)$ is given by
\begin{equation}\label{Bmap}B_{(u, \phi)}\xi = F_{\nabla}(\xi, \partial_s u) \phi + I_t F_{\nabla}(\xi, \partial_t u) \phi + (\nabla_{\xi}I_t)\phi\end{equation}
where $F_{\nabla}$ is the curvature of $\nabla$ (note that strictly speaking, $\nabla$ is a connection on the pull-back of $\pi^* E$ to $M \times [0,1]$; in which case we should interpret the terms above as $(\xi, 0)$, $(\partial_s u, 0)$ and $(\partial_t u, 1)$ in the identification $T^*(M \times [0,1]) = \pi^* TM \oplus \bb{R}$).

By assumption, the $J_t$-holomorphic strip $u$ is regular, so $D_u$ is surjective. In particular, we are just asking for the surjectivity of
\begin{equation}B : \ker D_u \to \coker \bar{\nabla}_I.\end{equation}

\begin{definition}\label{regularCS}We say that the pair $I_t, \nabla$ is \emph{regular} for $(\bf{x}_-, \bf{x}_+)$ if the above is surjective for each finite energy solution $(u, \phi)$ with nonzero $\phi$ and these asymptotics. In this case, the moduli space $\cal{M}(\bf{x}_-, \bf{x}_+)$ is a smooth manifold of dimension
\begin{equation}\mu(u) + \specflow(u) + i(\lambda_-) - i(\lambda_+) - 1\end{equation}
where $\mu(u) = \dim \cal{M}(x_-, x_+) + 1$ is the Maslov index of $u$. If this holds for all $(\bf{x}_-, \bf{x}_+)$, we simply say that the pair $I_t, \nabla$ is regular.\end{definition}

\begin{proposition}Generic choices of complex structure $I_t$ and symplectic connection $\nabla$ are regular. Indeed, for a fixed choice of $I_t$, generic $\nabla$ are regular, and vice versa. Moreover, such $I_t, \nabla$ can be chosen while fixing them in a neighbourhood of $L_0 \cap L_1$ to satisfy Assumption \ref{localgeomI}.\end{proposition}

\begin{proof}Observe first that in the case of $x_- = x_+$ and a constant $u$, we have a complete characterization \eqref{constantsolutions} of the solutions $\phi$, and it is not difficult to check these are cut out transversely. So assume that $u$ is nonconstant.

We follow the strategy of Floer-Hofer-Salamon \cite{FloerHoferSalamon95}. Consider the Banach manifold $\scr{A}$ of pairs $(I, \nabla)$ of a class $C^{\ell}$ time-dependent compatible complex structure $I = \{I_t\}$ and a class $C^{\ell}$ time-dependent symplectic connection $\nabla$ on $E$, for some $\ell > k$. Recall we have Banach vector bundles $\cal{W}^{2,k}_{\lambda_- \lambda_+}(E)$, $\cal{L}^{2,k-1}_{\lambda_- \lambda_+}(E)$ over the Banach manifold $\cal{W}^{2,k}(x_-, x_+)$ of possible strips $u$. Consider the map
\begin{align}\label{perturbedmap}\bar{\nabla} : \cal{W}^{2,k}_{\lambda_- \lambda_+}(E) \times \scr{A} &\to \cal{L}^{2,k-1}(x_-, x_+) \oplus \cal{L}^{2,k-1}_{\lambda_- \lambda_+}(E)\\
(u, \phi, I_t, \nabla) &\mapsto (\bar{\partial}_J u, \bar{\nabla}_{I} \phi).\end{align}
It suffices to show that at each zero $(u, \phi, I, \nabla)$ of the above map where $\phi \not= 0$, the linearization of \eqref{perturbedmap} is surjective: applying the implicit function theorem, the space of zeroes $(u, \phi, I, \nabla)$ is then a Banach manifold, and we can apply Sard-Smale to the projection to $\scr{A}$ to obtain that the $(I, \nabla)$ which make \eqref{linearizedeqn} surjective are of second category in $\scr{A}$, and are in particular dense.

The tangent space $T_{(I, \nabla)}\scr{A}$ is precisely the space of $C^{\ell}$ pairs $(K_t, A_t)$, where $K_t$ is a time-dependent section of the endomorphism bundle $\End(E)$, and $A_t$ is a time-dependent section of $\Omega^1(M; \End(E))$, which satisfy
\begin{equation}\label{Ktequations} I_t K_t + K_t I_t = 0, \quad \omega_E(K_t v_1, v_2) + \omega_E(v_1, K_t v_2) = 0\text{ for sections } v_1, v_2 \in E\end{equation}
at each time $t$, where $\omega_E$ is the symplectic form of $E$, and
\begin{equation}\omega_E(A_t(\xi \tensor v_1), v_2) + \omega_E(v_1, A_t(\xi \tensor v_2)) = 0\end{equation}
for all sections $v_1, v_2$ of $E$ and vector fields $\xi$ on $M$.

Since we have already chosen a regular $J_t$, it suffices to check that at each $(u, \phi, I, \nabla)$ with $\phi \not= 0$ the following map is surjective:
\begin{equation}\label{linearperturbedmap}\ker D_u \oplus W^{2,k}_{\lambda_- \lambda_+}(u^*E) \oplus T_{(I,\nabla)}\scr{A} \to L^{2,k-1}_{\lambda_- \lambda_+}(u^*E)\end{equation}
\begin{equation*}(\xi, \theta, K_t, A_t) \mapsto B_{(u, \phi)}\xi + \bar{\nabla}_{I} \theta + K_t \nabla_t \phi + A_t(\partial_s u, \phi) + I_t A_t(\partial_t u, \phi). \end{equation*}
Suppose the above map is not surjective. As its image contains a finite-codimension space $\im \bar{\nabla}_I$, it is closed, so there exists nonzero $\eta \in L^{2,k-1}_{\lambda_- \lambda_+}(u^*E)$ such that the pairing
\begin{equation*}\langle B_{(u, \phi)}\xi + \bar{\nabla}_{I_t} \theta + K_t \nabla_t \phi + A_t(\partial_s u, \phi) + I_t A_t(\partial_t u, \phi), \eta \rangle = 0 \end{equation*}
for each $(\xi, \theta, K_t, A_t)$. Setting $\xi, K_t, A_t$ to all be zero, this implies $\langle \bar{\nabla}_{I} \theta, \eta \rangle = 0$ for each $\theta$, and so $\bar{\nabla}_I^* \eta = 0$ where $\bar{\nabla}_I^*$ is the formal adjoint. This is again an elliptic differential operator with $C^{\ell}$ coefficients; by elliptic regularity $\eta$ must also be $C^{\ell}$ and in particular continuous. Moreover, the set of points where $\eta \not= 0$ must also be dense.

At this point, we recall from \cite{FloerHoferSalamon95} that a point $(s, t)$ in the interior of the strip $Z$ is called regular for the $J_t$-holomorphic strip $u$ if
\begin{equation}\partial_s u(s, t) \not= 0, \qquad u(s, t) \not= x_{\pm}, \qquad u(s,t) \notin u(\bb{R}-\{s\}, t).\end{equation}
The set $R(u)$ of regular points is open and dense in $Z$. In particular, choose a regular point $(s_0,t_0)$ where $\nabla_t \phi \not= 0$ and $\eta \not= 0$. Writing $p = u(s_0,t_0)$, there exists $K \in \End(E_p)$ satisfying \eqref{Ktequations} such that $\langle K \nabla_t \phi(s_0,t_0) , \eta(s_0,t_0) \rangle > 0$. By choosing bump functions on small enough open sets around $t_0 \in (0,1)$ and $p \in M$, we can construct a $K_t$ which is zero except in a small neighbourhood of $(t_0, p) \in [0,1] \times M$. In particular, the integral
\begin{equation*}\int_Z \langle K_t \nabla_t \phi(s,t), \eta(s,t)\rangle \end{equation*}
can be chosen to be necessarily positive: which contradicts the assumption that $\eta$ annihilates the image of \eqref{linearperturbedmap}. On the other hand, we could have also chosen $A \in T_p^* M \tensor \frak{sp}(E_p)$ so that $A(\partial_t u, \phi) = 0$ and $\langle A(\partial_s u(s_0, t_0), \phi(s_0, t_0)), \eta(s_0, t_0)\rangle > 0$, and we would obtain the same contradiction.

This proves that generic $(I, \nabla)$ of class $C^{\ell}$ are regular for $(\bf{x}_-, \bf{x}_+)$, and indeed this can be achieved by holding either one of $I, \nabla$ fixed. We can then follow the standard argument of Taubes, as explained in \cite{McDuffSalamon94}, to upgrade this to smooth $(I, \nabla)$. By taking the intersection of the second category sets of regular $(I, \nabla)$ for each $(\bf{x}_-, \bf{x}_+)$, we see that generic data are regular for all the relevant moduli spaces.

Finally, note that in the proof, we only utilized perturbations of $I, \nabla$ supported in a small neighbourhood of a regular point $(s_0, t_0)$, and we could have further assumed that $u(s_0, t_0)$ was not in fixed sufficiently small neighbourhoods of the points of $L_0 \cap L_1$. In particular, we can choose regular $I, \nabla$ to satisfy Assumption \ref{localgeomI} near $L_0 \cap L_1$.\end{proof}

\begin{remark}Observe that for a solution $(u, \phi)$ with asymptotics $\bf{x}_-, \bf{x}_+$ with the expected dimension of $\cal{M}(\bf{x}_-, \bf{x}_+)$ being zero, the index of the operator $\bar{\nabla}_I$ with $u$ held constant is $2 - \mu(u)$: in particular, nonzero solutions to the equation $\bar{\nabla}_I \phi = 0$ are never cut out transversely for fixed $u$ with $\mu(u) \ge 2$. At first sight, this may seem to be at odds with the above proof, where perturbations of $u$ do not make an explicit appearance. However, for fixed $u$, for generic $I_t$ we do not expect there to be any nonzero solutions $\phi$ at all: these solutions only appear when $u$ is taken over the entire family $\cal{M}(x_-, x_+)$, as in the above. \end{remark}

\subsection{More on the $\tau$-model.}\label{sec:tau_model_I} Before we continue to prove compactness and gluing results for the moduli spaces of twisted trajectories, we will provide an alternative Banach space set-up using the $\tau$-model. This will be particularly important in proving gluing results later, to avoid issues arising from gluing elements of Sobolev spaces with large weights (gluing results in the literature that do use exponential weights, such as in \cite{Abouzaid12}, use at worst weights of the form $e^{\delta |s|}$ where $s$ is small).

For this section, fix some $\bf{x}_{\pm} = (x_{\pm}, \lambda_{\pm})$, and choose eigensolutions $\psi_-(t), \psi_+(t)$ for $I_{x_{\pm}}\frac{d}{dt}$ with these eigenvalues, such that $||\psi_{\pm}||_H = 1$. Given any $u \in \cal{Z}^{2,k}(x_-, x_+)$, we can consider the Banach manifold
\[\cal{W}^{2,k}_{\psi_- \psi_+}(u^* E)\]
of $L^{2,k}_{loc}$ sections $\varphi$ of $u^*E$, such that $\varphi(\cdot, i) \in F_i$ for each $i = 0,1$, and so that for some large $R > 0$, using the trivializations of $E$ near $x_-, x_+$, the differences
\[\psi_-(t) - \varphi(s,t), \qquad \psi_+ - \varphi(s,t)\]
are in $W^{2,k}((-\infty, -R]\times[0,1])$ and $W^{2,k}([R, \infty) \times [0,1])$ respectively. This admits a smooth map
\[\cal{W}^{2,k}_{\psi_- \psi_+}(u^* E) \to 1 + L^{2,k}(\bb{R}, \bb{R}), \qquad \varphi \mapsto ||\varphi||_H\]
and the constant function $1$ is a regular value. In particular, we have a Banach manifold
\begin{equation}\cal{S}^{2,k}_{\psi_- \psi_+}(u^* E) = \{\varphi \in \cal{W}^{2,k}_{\psi_- \psi_+}(u^*E) : ||\varphi||_H = 1\};\end{equation}
the tangent space to $\cal{S}^{2,k}_{\psi_- \psi_+}(u^* E)$ at $\varphi$ is
\[T_{\varphi} \cal{S}^{2,k}_{\psi_- \psi_+}(u^* E) = \{\vartheta \in W^{2,k}(u^* E): \langle \varphi, \vartheta \rangle_H = 0\}.\]
These fit together to form a fibration of Banach manifolds
\[\cal{S}^{2,k}_{\psi_- \psi_+} \to \cal{Z}^{2,k}(x_-, x_+).\]

For $(u, \varphi) \in \cal{S}^{2,k}_{\psi_- \psi_+}$, consider the Banach space
\begin{equation}V^{2,k-1}_{(u, \varphi)} = \{\sigma \in L^{2,k-1}(u^*E): \langle \varphi, \sigma \rangle = 0\}.\end{equation}
We wrute $\cal{V}^{2,k-1}_u$ for the corresponding Banach vector bundle over $\cal{S}^{2,k}_{\psi_- \psi_+}(u^*E)$, and $\cal{V}^{2,k}$ for the Banach bundle this forms over all $\cal{S}^{2,k}_{\psi_- \psi_+}$. Observe that for any $\varphi \in \cal{S}^{2,k}_{\psi_- \psi_+}(u^* E)$, we have
\[ \langle \varphi, \bar{\nabla}_I \varphi - \Lambda(\varphi)\varphi \rangle = \langle \varphi, \nabla_s \varphi \rangle + \langle \varphi, I_t \nabla_t \varphi \rangle - \Lambda(\varphi) = 0.\]
We then have a (nonlinear) map
\begin{equation}\cal{F}_u : \cal{S}^{2,k}_{\psi_- \psi_+}(u^* E) \to \cal{V}^{2,k-1}_u; \qquad \varphi \mapsto \bar{\nabla}_I \varphi - \Lambda(\varphi)\varphi\end{equation}
and by ranging over $u \in \cal{Z}^{2,k}(x_-, x_+)$ we have a map
\begin{equation}\cal{F} : \cal{S}^{2,k}_{\psi_- \psi_+} \to \cal{L}^{2,k-1}(TM) \oplus \cal{V}^{2,k-1}; \qquad (u, \varphi) \mapsto (\bar{\partial}_J u, \cal{F}_u \varphi).\end{equation}

By elliptic regularity, any solution $\varphi$ of $\cal{F}_u(\varphi) = 0$  is smooth. Moreover at either infinite end of the strip $Z$, by additionally taking a nonzero solution the ordinary differential equation
\[\frac{dr}{ds} + \Lambda(\varphi) r = 0,\]
we can identify $\varphi$ with a solution $\phi = r \varphi$ of the $\bar{\partial}_I$-equation in $E_{x_{\pm}}$ up to scalar multiplication. Since $\Lambda(\varphi) \to \lambda_{\pm}$ as $s \to \pm \infty$, we obtain some asymptotic control on $r(s)$ as $s \to \pm \infty$: we can conclude that for any $\delta' > 0$, we must have $e^{(\lambda_- + \delta')s}r(s) \to 0$ with all derivatives as $s \to -\infty$, and likewise $e^{(\lambda_+ - \delta)s}r(s) \to 0$ with all derivatives as $s \to +\infty$.

In particular, passing to a $\sigma$-model solution $\phi(s,t) = r(s) \varphi(s,t)$ of $\bar{\nabla}_I \phi = 0$, it must in fact lie in the weighted Sobolev space
\[\phi(s,t) \in W^{2,k}_{\lambda_- \lambda_+}(u^*E)\]
and in particular we can then apply the exponential decay results of Proposition \ref{twisted_solution_asymptotics}. We obtain

\begin{proposition}\label{tau_model_asymptotics}A solution $\varphi \in \cal{S}^{2,k}_{\psi_- \psi_+}(u^* E)$ of $\cal{F}_u(\varphi) = 0$ is smooth and converges
\[\varphi(s,t) \to \psi_{\pm}(t) \qquad \text{ as } s \to \pm \infty\]
in all derivatives exponentially in $s$ and uniformly in $t$.\end{proposition}

In particular, we can identify, modulo translation, the zero locus of $\cal{F}$ with the components of the moduli space $\cal{M}(\bf{x}_-, \bf{x}_+)$ consisting of those twisted trajectories (in the $\tau$-model) $(u, \varphi)$ such that $\varphi \to \psi_-, \psi_+$ as $s \to \pm \infty$ (for the whole moduli space $\cal{M}(\bf{x}_-, \bf{x}_+)$, we have to also consider the other choices of unit eigensolutions $-\psi_-$ and $-\psi_+$).

For a fixed solution $(u, \varphi)$, let us now consider the index theory of the linearized operator $D (\cal{F}_u)_{\varphi}$. This is given by the formula
\[\vartheta \mapsto \Pi_V \bar{\nabla}_I \vartheta - \Lambda(\varphi) \vartheta\]
where $\Pi_V(\sigma) = \sigma - \langle \varphi, \sigma \rangle_H \varphi$ is the projection $L^{2,k-1}(u^*E) \to V^{2,k-1}_{(u, \varphi)}$. 

We will want to compare this to the operator $\bar{\nabla}_I$ acting on the Sobolev with the ``large'' weights $e^{(\lambda_- + \delta)s}$ and $e^{(\lambda_+ - \delta)s}$ of the previous sections. To do this, we observe that $D (\cal{F}_u)_{\varphi}$ extends to an operator on certain weighted spaces
\[D (\cal{F}_u)^{\delta}_{\varphi}: T_{\varphi} \cal{S}^{2,k,\delta}_{\psi_- \psi_+}(u^* E) \to \cal{V}_{(u, \varphi)}^{2,k-1, \delta}\]
where these are just the earlier spaces but with a small exponential weight $e^{- \delta |s|}$:
\[T_{\varphi} \cal{S}^{2,k,\delta}_{\psi_- \psi_+}(u^* E) = e^{\delta |s|} T_{\varphi} \cal{S}^{2,k}_{\psi_- \psi_+}(u^* E), \qquad \cal{V}_{(u, \varphi)}^{2,k-1, \delta}= e^{\delta |s|} \cal{V}_{(u, \varphi)}^{2,k-1}.\]

Now, choose a nonzero solution $r(s)$ to $\frac{d}{ds} r + \Lambda(\varphi) r = 0$; we must have $r(s) \sim C_{\pm} e^{-\lambda_{\pm}s}$ as $s \to \pm \infty$. Now, take a solution $\vartheta$ of the equation
\[\Pi_V \bar{\nabla}_I \vartheta - \Lambda(\varphi) \vartheta = 0\]
where we allow either $\vartheta \in T_{\varphi} \cal{S}^{2,k}_{\psi_- \psi_+}(u^* E)$ or its weighted analogue $T_{\varphi} \cal{S}^{2,k,\delta}_{\psi_- \psi_+}(u^* E)$. We can then transform $\vartheta$ into a solution $\theta$ of $\bar{\nabla}_I \theta = 0$ as follows: first, find a solution $\rho(s)$ to
\[\frac{d}{ds} \rho + \Lambda(\varphi)\rho + r(s) \langle \varphi, \bar{\nabla}_I \vartheta \rangle_H = 0\]
and then set
\[\theta(s,t) = \rho(s) \varphi(s,t) + r(s) \vartheta(s,t).\]
In particular, the asymptotic behavior of $r(s)$ implies that in fact we must have $\theta \in W^{2,k}_{\lambda_- \lambda_+}$; in particular since $\bar{\nabla}_I \theta = 0$ we see that $\theta$ is smooth, and has asymptotics $\theta \sim C'_{\pm} e^{-\lambda_{\pm} s}\psi_{\pm}(t)$ as $s \to \pm \infty$.

Unwinding this to $\vartheta = \frac{1}{r}\left(\theta - \langle \varphi, \theta \rangle_H \theta\right)$, we see that $\vartheta$ is not only smooth but must also decay exponentially to zero as $s \to \pm \infty$ (more precisely, we must have $e^{\delta' s} \vartheta(s,t) \to 0$ as $s \to +\infty$ for any $\delta' > 0$ smaller than the gap between $\lambda_+$ and the next highest eigenvalue, and similarly at the other limit).

The upshot is that the kernels of $D(\cal{F}_u)_{\varphi}$ and its analogue on weighted spaces $D(\cal{F}_u)_{\varphi}^{\delta}$ coincide, and consist of fast decaying smooth $\vartheta$. Pushing this further, we have

\begin{proposition}\label{tau_index}$D(\cal{F}_u)_{\varphi}$ and $D(\cal{F}_u)^{\delta}_{\varphi}$ are Fredholm, each of index
\[\ind(D(\cal{F}_u)_{\varphi}) = \ind(D(\cal{F}^{\delta}_u)_{\varphi}) = \specflow(u) + i(\lambda_-) - i(\lambda_+).\]\end{proposition}

\begin{proof}To see that the operators are Fredholm and of equal index, we consider the linearized operators at the constant solutions $(u, \varphi = (x_{\pm}, \psi_{\pm})$. Explicitly, consider the spaces
\begin{align*}T_{\psi_{\pm}}\cal{S}^{2,k}(E_{x_{\pm}}) &= \{\vartheta \in L^{2,k}(Z, E_{x_{\pm}}) : \vartheta(\cdot, i) \in F_i|_{x_{\pm}} \text{ for } i = 0,1; \text{ and } \langle \psi_{\pm}, \vartheta \rangle_H = 0\},\\
V^{2,k-1}_{\psi_{\pm}}(E_{x_{\pm}}) &= \{\sigma \in L^{2,k-1}(Z, E_{x_{\pm}}): \langle \psi_{\pm}, \sigma \rangle_H = 1\}\end{align*}
and the analogous weighted spaces $T_{\psi_{\pm}}\cal{S}^{2,k, \delta}(E_{x_{\pm}}), V^{2,k-1, \delta}_{\psi_{\pm}}(E_{x_{\pm}})$. Then take the operators
\[F_{\psi_{\pm}} : T_{\psi_{\pm}}\cal{S}^{2,k}(E_{x_{\pm}}) \to V^{2,k-1}_{\psi_{\pm}}(E_{x_{\pm}}), \quad F_{\psi_{\pm}}^{\delta} : T_{\psi_{\pm}}\cal{S}^{2,k, \delta}(E_{x_{\pm}}) \to V^{2,k-1, \delta}_{\psi_{\pm}}(E_{x_{\pm}})\]
each defined by
\[\vartheta \mapsto \frac{d}{ds} \vartheta + I_{x_{\pm}} \frac{d}{dt} \vartheta - \lambda_{\pm}\vartheta\]
(note there is no need to project to $V$ here, since $\langle \psi_{\pm}, \bar{\nabla}_I \vartheta \rangle_H = 0$ if $\langle \psi_{\pm}, \vartheta \rangle_H = 0$, since $\psi_{\pm}$ is an eigenvector of $I_{x_{\pm}}\frac{d}{dt}$).
The operators $F_{\psi_{\pm}}, F_{\psi_{\pm}}^{\delta}$ are then invertible, from which we deduce that $D(\cal{F}_u)_{\varphi}$, $D(\cal{F}_u)^{\delta}_{\varphi}$ are Fredholm. Moreover, they must be of equal index for sufficiently small $\delta$.

To compute the index, we could do a similar excision argument as before to compare the index to that on solutions with constant $u$. However we can also directly compare $D(\cal{F}_u)^{\delta}_{\varphi}$ to the operator $\bar{\nabla}_I$ in the $\sigma$-model; ultimately this comparison is why we introduced these $\tau$-model exponentially weighted spaces in the first place.

Indeed, choose a nonzero solution $r(s)$ of $\frac{d}{ds}r + \Lambda(\varphi)r = 0$, and a weight function $w : \bb{R} \to \bb{R}$ such that $w(s) = (\lambda_- + \delta)s$ for $s \ll 0$ and $w(s) = (\lambda_+ - \delta)s$ for $s \gg 0$. Consider the space of $L^{2,k}$-real functions
\[L^{2,k}_{\lambda_- \lambda_+}(\bb{R}, \bb{R}) = e^{-w(s)}L^{2,k}(\bb{R}, \bb{R})\]
with weight function $e^{-w(s)}$. There is then an isomorphism of Banach spaces
\begin{align*}L^{2,k}_{\lambda_- \lambda_+}(\bb{R}, \bb{R}) \oplus T_{\varphi} \cal{S}^{2,k,\delta}_{\psi_- \psi_+}(u^*E) &\xrightarrow{\sim} W^{2,k}_{\lambda_- \lambda_+}(u^* E) \\
(\rho, \vartheta) &\mapsto \rho(s) \varphi + r(s) \vartheta\end{align*}
where $r(s)$ is a nonzero solution to $\frac{dr}{ds} + \Lambda(\varphi)r(s) = 0$, and likewise there is an isomorphism
\[L^{2,k-1}_{\lambda_- \lambda_+}(\bb{R},\bb{R}) \oplus V^{2,k-1, \delta}_{(u, \varphi)} \cong L^{2,k-1}_{\lambda_- \lambda_+}(u^*E).\]
According to these isomorphisms, the map $\bar{\nabla}_I: W^{2,k}_{\lambda_- \lambda_+} (u^*E) \to L^{2,k-1}_{\lambda_- \lambda_+}(u^*E)$ is then given by
\[(\rho, \vartheta) \mapsto \left(\frac{d \rho}{ds} + \Lambda(\varphi)\rho + r(s)\langle \varphi, \bar{\nabla}_I \vartheta\rangle, \quad \Pi_V \bar{\nabla}_I \vartheta - \Lambda(\varphi)\vartheta\right).\]
Now, the map
\[L^{2,k}_{\lambda_+ \lambda_-}(\bb{R}, \bb{R}) \to L^{2,k-1}_{\lambda_- \lambda_+}(\bb{R}, \bb{R}), \qquad \rho \mapsto \frac{d \rho}{ds} + \Lambda(\varphi)\rho\]
is Fredholm of index one (indeed, it is surjective, with one-dimensional kernel spanned by $r(s)$); in particular we deduce that
\[\ind\left(D(\cal{F}_u)^{\delta}_{\varphi}\right) + 1 = \ind \left(\bar{\nabla}_I : W^{2,k}_{\lambda_- \lambda_+} (u^*E) \to L^{2,k-1}_{\lambda_- \lambda_+}(u^*E)\right)\]
from which we obtain the result. Clearly, the extra dimension in the $\sigma$-model corresponds exactly to the $\bb{R}^*$ action.\end{proof}

Let us briefly consider the question of transversality in the $\tau$-model, for which we now have to vary $u$ over the moduli space of $J_t$-holomorphic strips. The connection $\nabla$ on $E$ induces a connection on the fibration $\cal{S}^{2,k}_{\psi_- \psi_+} \to \cal{Z}^{2,k}(x_-, x_+)$ (since parallel transport in $E$ preserves the $|| \cdot ||_H$ norm under the assumption that $\nabla$ is Hermitian with respect to $I_t$), identifying the tangent space as
\[T_{(u, \varphi)} \cal{S}^{2,k}_{\psi_- \psi_+} \cong T_u \cal{Z}^{2,k}(x_-, x_+) \oplus T_{\varphi} \cal{S}^{2,k}_{\psi_- \psi_+}(u^* E)\]
At a solution $(u, \varphi)$, we then have the linearization of $\cal{F}$:
\[D\cal{F}_{(u, \varphi)} : W^{2,k}(u^*TM) \oplus T_{\varphi}\cal{S}^{2,k}_{\psi_- \psi_+}(u^*E) \to L^{2,k-1}(u^*TM) \oplus V^{2,k-1}_{(u, \varphi)}\]
which is given by
\begin{equation}D\cal{F}_{(u, \varphi)}(\xi, \vartheta) = (D_u \xi, B^{\tau}_{(u, \varphi)}\xi + \Pi_V \bar{\nabla}_I \vartheta - \Lambda(\varphi) \vartheta)\end{equation}
where $D_u = D(\bar{\partial}_J)_u$ is the linearization of the ordinary $\bar{\partial}$ operator, $\Pi_V(\sigma) = \sigma - \langle \varphi, \sigma \rangle_H \varphi$ is the projection $L^{2,k-1}(u^*E) \to V^{2,k-1}_{(u, \varphi)}$, and $B^{\tau}_{(u, \varphi)}\xi$ is an operator $W^{2,k}(u^*TM) \to V^{2,k-1}_{(u, \phi)}$ given by a similar expression to \eqref{Bmap}:
\begin{equation}\label{Btauoperator}B_{(u, \varphi)}^{\tau} \xi = \Pi_V \left(F_{\nabla}(\xi, \partial_s u) \varphi + I_t F_{\nabla}(\xi, \partial_t u) \varphi +  (\nabla_{\xi}I_t) \varphi\right).\end{equation}

Furthermore, the operator $D \cal{F}_{(u, \varphi)}$ is Fredholm, of index
\[\mu(u) + \specflow(u) + i(\lambda_-) - i(\lambda_+)\]
We say that the initial data $(I, \nabla)$ is \emph{regular} at a solution $(u, \varphi)$ is $D \cal{F}_{(u, \varphi)}$ is Fredholm; this amounts to the surjectivity of
\[B^{\tau}_{(u, \varphi)} : \ker(D_u) \to \coker\left( D(\cal{F}_u)_{\varphi}\right).\]
However, this is notion is equivalent to regularity in the $\sigma$-model:

\begin{proposition}$D\cal{F}_{(u, \varphi)}$ is surjective at a $\tau$-model solution $(u, \varphi)$ if and only if at the corresponding $\sigma$-model solution $(u, \phi = r(s) \varphi)$ the pair $(I, \nabla)$ is regular in the sense of Definition \ref{regularCS}. \end{proposition}

\begin{proof}This follows by inspecting the proof of Proposition \ref{tau_index} and comparing the operators $B_{(u, \phi)}$ and $B^{\tau}_{(u, \varphi)}$ of \eqref{Bmap} and \eqref{Btauoperator}. Indeed, after choosing a solution $r$ of $\frac{d}{ds} r + \Lambda(\varphi) r = 0$, there is an isomorphism
\[ \ker D \cal{F}_{(u, \varphi)} \oplus \bb{R} \to \ker D(\bar{\partial}, \bar{\nabla})_{(u, \phi)}\]
where the additional $\bb{R}$ factor corresponds to a choice of solution $\rho$ to the equation
\[\frac{d}{ds} \rho + \Lambda(\varphi)\rho + r(s)\langle \phi, \bar{\nabla}_I \vartheta \rangle_H + r(s) \langle \phi, (\nabla_{\xi} I) \nabla_t \varphi \rangle_H + r(s) \langle \varphi, IF_{\nabla}(\xi, \partial_t u) \varphi\rangle_H = 0.\]
In particular, since we also have the similar identity on Fredholm indices
\[\ind(D\cal{F}_{(u,\varphi)}) + 1 = \ind(D(\bar{\partial}, \bar{\nabla}))\]
we deduce that one is surjective if and only if the other is.\end{proof}

\subsection{Compactness.}\label{sec:compactness_I} Before we prove compactness results for the moduli spaces of twisted trajectories, we will briefly review the basic compactness results for the ``downstairs'' moduli spaces of pseudoholomorphic strips. Implicit throughout this discussion will be the exactness of $M$ and $L_0, L_1$, which prevents sphere or disc bubbling behavior, and the convexity of $M$ at infinity, which prevents $J_t$-holomorphic strips from escaping to infinity.

Before we discuss broken trajectories, let us recall a local compactness result for solutions $u$ of the $J_t$-holomorphic strip equation with boundary of $L_0, L_1$. Let $u_{\alpha}$ be a sequence of solutions, with uniformly bounded energy
\begin{equation}\cal{E}(u_{\alpha}) = \int_Z |\partial_s u|^2 < \cal{E}_0\end{equation}
Then, following \cite{Floer88a}, there is a subsequence which converges in the $C^{\infty}_{loc}$-topology to another $J_t$-holomorphic strip $u$, meaning on finite substrips $Z_R = [-R, R]\times [0,1]$,
\begin{equation*}u_{\alpha}|_{Z_R} \xrightarrow{C^{\infty}} u|_{Z_R}.\end{equation*}

Recall as well that any finite energy $J_t$-holomorphic strip uniformly converges at either ends of $Z$ to $x_{\pm} \in L_0 \cap L_1$; however, given a sequence $u_{\alpha}$ of solutions with endpoints $x_{\pm}$ converging locally in the above sense to $u$, the endpoints of $u$ need not be $x_{\pm}$.

Recall that a \emph{broken trajectory} $(u^1, \hdots, u^n)$ from $x_-$ to $x_+$ is a tuple of finite energy $J_t$-holomorphic strips $u^i \in \cal{M}(x^{i-1}, x^i)$ for $x_- = x^0, x^1, \hdots, x^n = x_+ \in L_0 \cap L_1$. Appropriately topologized, the moduli space of broken trajectories
\begin{equation}\bar{\cal{M}}(x_-, x_+) = \bigcup\limits_{x^1, \hdots, x^{n-1}} \cal{M}(x_-, x^1) \times \cal{M}(x^1, x^2) \times \hdots \times \cal{M}(x^{n-1}, x_+)\end{equation}
is then compact. Recall that in this topology, \emph{Gromov convergence} for unbroken strips is defined as follows: a sequence $u_{\alpha} \in \cal{M}(x^0, x^n)$ converges to a broken trajectory $(u^1, \hdots, u^n)$ if for each $\alpha$ there are real numbers $\sigma^1_{\alpha} < \hdots < \sigma^n_{\alpha}$ such that:
\begin{itemize}\item the translates $\tau_{-\sigma^i_{\alpha}}^* u_{\alpha}(s,t) = u_{\alpha}(s + \sigma^i_{\alpha}, t)$ converge locally to $u^i$;
\item for any sequence $\rho_{\alpha}$ of real numbers with $\sigma^{i}_{\alpha} < \rho_{\alpha} < \sigma^{i+1}_{\alpha}$ for $i = 1, \hdots, n-1$ and such that both sequences $\rho_{\alpha} - \sigma^{i}_{\alpha}, \sigma^{i+1}_{\alpha} - \rho_{\alpha} \to \infty$, we have the translates $\tau_{- \rho_{\alpha}}^* u_{\alpha}$ converging to the constant strip at $x^i$ locally.
\item Likewise for any sequence $\rho_{\alpha} < \sigma^1_{\alpha}$ and $\sigma^1_{\alpha} - \rho_{\alpha} \to \infty$, we have  $\tau_{- \rho_{\alpha}}^* u_{\alpha}$ converging locally to $x^0 = x_-$; similarly if $\sigma^n_{\alpha} < \rho_{\alpha}$ and $\rho_{\alpha} - \sigma^n_{\alpha} \to \infty$, we have $\tau_{- \rho_{\alpha}}^* u_{\alpha}$ converging locally to $x^n = x_+$.\end{itemize}
Convergence for sequences of broken strips is defined similarly.

We will use a slightly more explicit statement of the above convergence, which follows easily from the definition above. A sequence $u_{\alpha} \in \cal{M}(x_-, x_+)$ converges to a broken trajectory $(u^1, \hdots, u^n)$ if and only if for each $\eps > 0$, for each sufficiently large $\alpha$ we can find a partition of $Z = \bb{R} \times [0,1]$ into substrips
\begin{equation}\label{strippartition}Z = W^0_{\alpha} \cup Z^1_{\alpha} \cup W^1_{\alpha} \cup Z^2_{\alpha} \cup \hdots \cup Z^n_{\alpha} \cup W^n_{\alpha}\end{equation}
where $Z^i_{\alpha} = [a^i_{\alpha}, b^i_{\alpha}] \times [0,1]$ for $i = 1,\hdots, n$, and $W^0_{\alpha} = (-\infty, a^1_{\alpha}]\times[0,1]$, $W^i_{\alpha} = [b^i_{\alpha} \times a^{i+1}_{\alpha}] \times [0,1]$ for $i = 1, \hdots, n-1$ and $W^n_{\alpha} = [b^n_{\alpha}, \infty)\times[0,1]$ for
\[a^1_{\alpha} < b^1_{\alpha} < a^2_{\alpha} < \hdots < a^n_{\alpha} < b^n_{\alpha}\]
which has the properties that:
\begin{itemize}\item each $Z^i_{\alpha}$ has fixed length $b^i_{\alpha} - a^i_{\alpha} = \ell$ independent of $\alpha$;
\item If $\sigma^i_{\alpha} = \frac{a^i_{\alpha} + b^i_{\alpha}}{2}$, the translated strips $\tau^*_{\sigma^i_{\alpha}} u_{\alpha} (s, t) = u_{\alpha}(s + \sigma^i_{\alpha}, t)$ converge locally to the strip $u^i$;
\item on $W^i_{\alpha}$, the strips $u_{\alpha}|_{W^i_{\alpha}}$ are within $\eps$ of the constant strip at $x^i$, meaning for each $(s,t) \in W^i_{\alpha}$, we have $d(u_{\alpha}(s,t), x^i) < \eps$.\end{itemize}

This ensures that the concatenation of $u^1, \hdots, u^n$ has the same homotopy class in $H_2(M, L_0 \cup L_1)$ as $u_{\alpha}$ for sufficiently large $\alpha$; and indeed both the Maslov index and the energy of $u_{\alpha}$ are equal to the sum of those for the $u^i$.

Together with the appropriate gluing results, this suffices to show that the compactified moduli spaces $\bar{\cal{M}}_2(x_-, x_+)$ of Maslov index 2 strips are one-dimensional manifolds, with boundary $\cup_{x^1} \cal{M}_1(x_-, x^1) \times \cal{M}_1(x_1, x_+)$ given by products of Maslov index 1 strips, which proves that the Floer differential squares to zero.

We will now formulate similar results for solutions $(u, \phi)$ of the twisted equations. However, there is no longer a topologically invariant energy for the auxiliary field $\phi$. Instead, as a proxy for the action functional, we will use the $s$-dependent quantity $\Lambda(\phi)$, which is given by
\begin{equation}\Lambda(\phi)(s) = \frac{\langle \phi, I_t \nabla_t \phi \rangle_H}{||\phi||_H^2} = - \frac{d}{ds} \log ||\phi||_H.\end{equation}
Recall as well that if $(u, \phi)$ has asymptotics $\bf{x}_{\pm} = (x_{\pm}, \lambda_{\pm})$, we have $\Lambda(\phi) \to \lambda_{\pm}$ as $s \to \pm \infty$.

\begin{proposition}\label{Lambdabound} $\Lambda(\phi)(s)$ is uniformly bounded amongst all trajectories between $\bf{x}_-$ and $\bf{x}_+$.
\end{proposition}

\begin{proof}We seek control on the derivative of $\Lambda(\phi)$. In the following calculation, we suppress the subscript $H$ from the inner product $\langle \cdot, \cdot \rangle_H$ and norm $|| \cdot ||_H$.
\begin{align*}\frac{d \Lambda(\phi)}{ds} &= \frac{1}{||\phi||^4} \left(\langle \nabla_s \phi, I_t \nabla_t \phi \rangle ||\phi||^2 + \langle \phi, \nabla_s (I_t \nabla_t \phi) \rangle ||\phi||^2 - 2 \langle \phi, I_t \nabla_t \phi \rangle \langle \phi, \nabla_s \phi \rangle \right) \\
&= \frac{1}{||\phi||^4}\left(- ||\nabla_s \phi||^2||\phi||^2 + 2 \langle \phi, \nabla_s \phi \rangle^2 + \langle \phi, \nabla_s(I_t \nabla_t \phi)\rangle||\phi||^2\right)
\end{align*}
using the fact that $\nabla_s \phi + I_t \nabla_t \phi = 0$. As for the third term above, have
\begin{equation*}\langle \phi, \nabla_s(I_t \nabla_t \phi)\rangle = \langle \phi, \nabla_s(I_t)(\nabla_t \phi)\rangle + \langle \phi, I_t F_{\nabla}(\partial_s, \partial_t) \phi \rangle + \langle \phi, I_t \nabla_t \nabla_s \phi \rangle\end{equation*}
where $F_{\nabla}$ is the curvature of $\nabla$. Moreover, we can deal with the third term here by integration by parts:
\begin{align*}\langle \phi, I_t \nabla_t \nabla_s \phi \rangle - \langle I_t \nabla_t \phi, \nabla_s \phi\rangle &= \int_0^1 \left(\omega_E(\phi, \nabla_t \nabla_s \phi) + \omega_E(\nabla_t \phi, \nabla_s \phi)\right) dt \\
&= \int_0^1 \frac{d}{dt} \omega_E(\phi, \nabla_s \phi) dt \\
&= \omega_E(\phi, \nabla_s \phi)|_{t=1} - \omega_E(\phi, \nabla_s \phi)|_{t=0}.\end{align*}
Moreover, we know that $\langle I_t \nabla_t \phi, \nabla_s \phi\rangle = -||\nabla_s \phi||^2$. Putting this all together and applying the Cauchy-Schwarz inequality for $\langle \phi, \nabla_s \phi\rangle$, we obtain
\begin{equation}\label{curvaturebound}\frac{d \Lambda(\phi)}{ds} \le \frac{1}{||\phi||^2}\left(\langle \phi, \nabla_s(I_t) (\nabla_t \phi)\rangle + \langle \phi, I_t F_{\nabla}(\partial_s, \partial_t) \phi \rangle + \omega_E(\phi, \nabla_s \phi)|_{t=1} - \omega_E(\phi, \nabla_s \phi)|_{t=0}\right)\end{equation}

Recall there are small open neighbourhoods $U_x$ of each $x \in L_0 \cap L_1$ on which the symplectic vector bundle $E$, the Lagrangian subspaces $F_0, F_1$, the complex structure $I_t$ and the connection $\nabla$ are all simultaneously trivialized. Thus, for $s \in \bb{R}$ such that $u(s, t) \in U_x$ for each $t \in [0,1]$, the right hand side of \eqref{curvaturebound} must consequently vanish. The fact that $\Lambda(\phi)$ is non-increasing within $U_x$ is a fact we will use repeatedly for the rest of this section.

Moreover, by the earlier compactness results for $\cal{M}(x_-, x_+)$, there is an $L > 0$ such that for each strip $u$ and $s \in \bb{R}$ outside a finite union of intervals of total length less than $L$, $u(s, [0,1]) \subset U_x$ for some $x \in L_0 \cap L_1$. It then suffices to show that the right hand side of \eqref{curvaturebound} is uniformly bounded over all pseudoholomorphic strips $u \in \cal{M}(x_-, x_+)$. This easily follows from the ``downstairs'' compactness results.\end{proof}

Uniform bounds on $\Lambda$ suffice to prove the following ``local'' compactness result.

\begin{proposition}\label{localconvergence}Let $(u_{\alpha}, \phi_{\alpha})$ be some sequence of solutions to the twisted equations on $Z$, with $u_{\alpha}$ of uniformly bounded energy, and a uniform bound on $\Lambda(\phi_{\alpha})(s)$ independent of $\alpha$ and $s$. Then there is some subsequence that converges after rescaling in the $C^{\infty}_{loc}$ topology to a solution $(u, \phi)$ with nonzero $\phi$ and bounded $\Lambda(\phi)$: namely, $u_{\alpha} \to u$ locally as above, and there are scalars $r_{\alpha} \in \bb{R}^*$ such for compact substrips $Z_R = [-R, R] \times [0,1]$,
\begin{equation}r_{\alpha} \phi_{\alpha}|_{Z_R} \rightarrow \phi|_{Z_R}\end{equation}
in the $C^{\infty}$ topology, where strictly speaking, this convergence is after parallel transporting $r_{\alpha} \phi_{\alpha}|_{Z_R}$ to give a section of $u|_{Z_R}^*E$. 

Equivalently, in the $\tau$ model, if $\varphi_{\alpha} = \phi_{\alpha}/||\phi_{\alpha}||_H$ and $\varphi = \phi/||\phi_{\alpha}||$, then $\varphi_{\alpha} \to \varphi$ with all derivatives on compact sets.\end{proposition}

\begin{proof}Pass to a subsequence so that $u_{\alpha}$ converges locally to a solution $u$. Now,
\[\Lambda(\phi_{\alpha})(s) = - \frac{d}{ds} \log ||\phi_{\alpha}(s)||_H\]
is uniformly bounded, so on each finite strip $Z_R = [-R, R] \times [0,1]$, there exists a constant $C_1$ such that
\begin{equation}||\phi_{\alpha}(s_0)||_H \le C_1 ||\phi_{\alpha}(s_1)||_H\end{equation}
for all $s_0, s_1 \in [-R, R]$, and hence there is a constant $C_2$ such that
\begin{equation}\label{finiteL2bound}||\phi_{\alpha}(s)||_H \ge C_2 ||\phi_{\alpha}||_{L^2(Z_{R})}\end{equation}
for all $s \in [-R, R]$.

Now, write $\tilde{\phi_{\alpha}} = \phi_{\alpha}|_{Z_R} / ||\phi_{\alpha}||_{L^2(Z_{R})}$. Then $\tilde{\phi_{\alpha}}$ are a sequence of $L^2$-norm one solutions to the linear elliptic equations
\begin{equation*}\bar{\nabla}_{I_t(u_{\alpha})} \tilde{\phi_{\alpha}} = 0\end{equation*}
on $Z_R$. In particular, for any $\eps > 0$, we can pass to a subsequence of $\tilde{\phi}_{\alpha}$ which converges in the $C^{\infty}$ topology on $Z_{R-\eps} = [R-\eps, R+\eps]\times[0,1]$ to a solution $\tilde{\phi}_{R-\eps}$ of
\begin{equation*}\bar{\nabla}_{I_t(u)} \tilde{\phi}_{R-\eps} = 0.\end{equation*}
The inequality \eqref{finiteL2bound} then ensures that $\tilde{\phi}_{R-\eps}$ is nonzero. Now, for $R = 1, 2, 3, \hdots$, inductively construct subsequences $\phi_{\alpha}$ as above, with limits $\tilde{\phi}_{R-\eps}$, each time passing to a further subsequence of the previously obtained subsequence; by taking the ``diagonal'' of these subsequences, we obtain a subsequence $\phi_{\alpha}$ such that for each $R \in \bb{Z}_{> 0}$,
\begin{equation*}\phi_{\alpha}|_{Z_{R-\eps}}/ ||\phi_{\alpha}||_{L^2(Z_R)} \to \tilde{\phi}_{R-\eps}.\end{equation*}
Now, set $\phi_{R-\eps} = \tilde{\phi}_{R-\eps}/ ||\tilde{\phi}_{R-\eps}||_{L^2(Z_1)}$. Observe that for each $R$,
\begin{equation*}\phi_{\alpha}|_{Z_1}/ || \phi_{\alpha}||_{L^2(Z_1)} \to \phi_{R-\eps}|_{Z_1}\end{equation*}
so in particular the $\phi_{R-\eps}$ all agree on $Z_1$. By unique continuation, they then patch together to form a solution $\phi$ on the entire strip $Z$; moreover setting $r_{\alpha} = 1/||\phi_{\alpha}||_{L^2(Z_1)}$, we have
\begin{equation}r_{\alpha} \phi_{\alpha} \to \phi\end{equation}
over each $Z_{R-\eps}$. Finally, observe that $\Lambda(\phi)(s) = \lim_{\alpha} \Lambda(\phi_{\alpha})(s)$, which is clearly bounded.\end{proof}

\begin{remark}It is possible in the above for $u_{\alpha}$ to locally converge to a constant trajectory $u(s,t) = x \in L_0 \cap L_1$. In this case, $\phi_{\alpha}$ locally converges to a nonzero solution of the equation $\bar{\partial}_{I} \phi = 0$ on $E_{x}$. This solution itself can be a ``constant solution'' $\phi(s,t) = C e^{-\lambda s}\psi(t)$ for some fixed eigenvalue $\lambda$ of $I_x \frac{d}{dt}$, or it can represent a twisted trajectory from some $(x, \lambda_-)$ to $(x, \lambda_+)$ for $\lambda_- > \lambda_+$.\end{remark}

\begin{remark}This proof crucially used the linearity of the twisted equations, and thus requires working with the $\sigma$-model solutions; this is in spite of convergence being easier to state in the $\tau$-model. For the remainder of this section, we will work in the $\tau$-model.\end{remark}

To upgrade this local convergence result to a convergence to broken trajectories, it will be convenient to control the total variation of $\Lambda$, to play the same role as energy does in the study of convergence downstairs: for a $\tau$-model solution $(u, \varphi)$ (or equivalently, a solution $(u, \phi)$ in the $\sigma$-model), consider
\begin{equation}K(u, \varphi) = \int_{\bb{R}} \left|\frac{d \Lambda(\varphi)}{ds}\right| ds, \qquad K^+(u, \varphi) = \int_{\bb{R}} \left( \frac{d \Lambda(\varphi)}{ds} \right)^+ ds\end{equation}
where $f^+ = \max\{0, f\}$ is the positive part of a real number $f$. Whenever $Z' \subset Z$ is a substrip, we will write $K_{Z'}, K^+_{Z'}$ for the corresponding quantities defined over this substrip. For a flow between $(x_-, \lambda_-)$ and $(x_+, \lambda_+)$, these quantities satisfy
\begin{equation}2 K^+(u, \varphi) - K(u, \phi) = \int_{\bb{R}} \frac{d \Lambda(\varphi)}{ds} ds = \lambda_{-} - \lambda_{+}.\end{equation}
so in particular, since $\Lambda(\varphi)$ is decreasing whenever $u(s,t)$ is sufficiently close to $L_0 \cap L_1$, both these quantities are finite. We can define $K, K^+$ for broken trajectories as well, as the sum over the unbroken components.

\begin{proposition}\label{KboundI}$K(u, \varphi)$ and $K^+(u, \varphi)$ are uniformly bounded for all possibly broken trajectories from $\bf{x}_-$ to $\bf{x}_+$.\end{proposition}

\begin{proof}It suffices to prove the proposition for $K^+$. Since $\Lambda(s)$ is strictly decreasing on trajectories with constant $u$, it suffices to consider just the components of a broken trajectory with non-constant $u$. There is an a priori bound on the number of such components, so we can reduce to showing that $K^+$ is bounded over all unbroken trajectories with non-constant $u$.

This follows from compactness results downstairs. Indeed, suppose for a contradiction there were a sequence of unbroken flows $(u_{\alpha}, \varphi_{\alpha})$ such that $K^+(u_{\alpha}, \varphi_{\alpha})$ increases without bound. After passing to a subsequence, the strips $u_{\alpha}$ limit to a broken strip $(u^1, \hdots u^n)$. In particular, as above we partition $Z$ as
\begin{equation*}Z = W^0_{\alpha} \cup Z^1_{\alpha} \cup W^1_{\alpha} \cup Z^2_{\alpha} \cup \hdots \cup Z^n_{\alpha} \cup W^n_{\alpha}\end{equation*}
such that $u(W^i_{\alpha})$ is contained in an arbitrarily small neighbourhood of some intersection point $x \in L_0 \cap L_1$, and so that the translations $\tau^*_{\sigma^i_{\alpha}} u_{\alpha}$ centred on $Z^i_{\alpha}$ converge locally to $u^i$. In particular, $K^+_{W^i_{\alpha}}(\phi_{\alpha}) = 0$.

By the local convergence result Proposition \ref{localconvergence}, $\tau^*_{\sigma^i_{\alpha}} \varphi_{\alpha}$ converge locally to a solution $\varphi^i$. Note that $u^i \to x^{i-1}$ as $s \to -\infty$ and $u^i \to x^i$ as $s \to +\infty$, in particular $\Lambda_{\varphi^i}$ is decreasing for sufficiently positive or negative $s$. Consequently, $K^+(\varphi^i)$ is finite. In particular, this means that $K^+_{Z^i_{\alpha}}(\varphi_{\alpha})$ is in fact bounded independent of $\alpha$. Summing these contributions over the whole strip $Z$, we see that $K^+(\varphi_{\alpha})$ is bounded, a contradiction.\end{proof}

We now prove a crucial lemma, which allows us to show solutions $(u, \varphi)$ with small energy and small total variation $K$ in $\Lambda(\phi)$ are close to constant solutions: by a constant solution we mean those $(u, \varphi)$ such that $u(s, t) = x$ for $x \in L_0 \cap L_1$, and $\varphi(s,t) = \psi(t)$ for some unit eigenvalue $\psi(t)$ of $I \frac{d}{dt}$ at $x$. Equivalently, in the $\sigma$-model, this means
\[\phi(s,t) = C e^{-\lambda s} \psi(t)\]
for some nonzero constant $C$, where $\lambda$ is the eigenvalue of $\psi$. Note that these are precisely those solutions which satisfy $\cal{E}(u) = 0$ and $K(\varphi) = 0$, since $\frac{d\Lambda(\varphi)}{ds} = 0$ everywhere implies that $\varphi(s)$ and $I \frac{d}{dt}\varphi(s)$ must be everywhere proportional, since this is exactly the equality condition of the Cauchy-Schwarz inequality \eqref{curvaturebound} (which in this case of solutions supported on constant $u$, has vanishing right hand side).

\begin{lemma}\label{smallenergylemmaI}Fix arbitrary compact sub-strips $Z_1, Z_2 \subset Z$, and constants $\eps > 0$, $\cal{E}_0, \Lambda_0 > 0$. Then there exist $\delta_1, \delta_2 > 0$ such that if $(u, \varphi)$ is a $\tau$-model solution which satisfies
\begin{equation}\cal{E}(u) < \cal{E}_0, \qquad |\Lambda(\varphi)(s)| < \Lambda_0 \text{ for all } s, \qquad\mathcal{E}_{Z_2}(u) \le \delta_1, \qquad K_{Z_2}(u, \phi) \le \delta_2\end{equation}
then $(u, \varphi)|_{Z_1}$ is in an $\eps$-neighbourhood of a constant solution in the $L^{2,k}(Z_1)$-norm.\end{lemma}

\begin{proof}Suppose, for a contradiction, there is a sequence of flows $(u_{\alpha}, \varphi_{\alpha})$ with $\mathcal{E}_{Z_2}, K_{Z_2} \to 0$, and uniformly bounded energy and $\Lambda(\varphi_{\alpha})$, none of which have $(u_{\alpha}, \varphi_{\alpha})|_{Z_1}$ in an $\eps$-neighbourhood of a constant solution. By the local convergence result Proposition \ref{localconvergence}, $(u_{\alpha}, \varphi_{\alpha})$ converges locally to $(u, \varphi)$. This limit must have $\mathcal{E}_{Z_2}(u) = 0$ and $K_{Z_2}(\varphi) = 0$, which implies that $(u, \varphi)|_{Z_2}$ is actually a constant solution. By unique continuation is must be constant on all of $Z$, so in particular on $Z_1$, which is a contradiction. \end{proof}

In particular, this lemma allows us to deduce the following, a sort of converse to Proposition \ref{Lambdabound}, the downstairs version of which is well known (and we have indeed already used it extensively):

\begin{corollary}Suppose $(u, \varphi)$ is a solution to the twisted equations on $Z$ with finite energy $\mathcal{E}(u)$ and bounded $\Lambda(\varphi)$. Then $(u, \varphi)$ is in fact a flow from $\bf{x}_-$ to $\bf{x}_+$ for some $\bf{x}_{\pm} \in \frak{C}$.\end{corollary}

\begin{proof}Fix some finite $Z_1$, and consider the translated solutions $\tau^*_s (u, \varphi)|_{Z_1}$. These must satisfy $\mathcal{E}_{Z_1}, K_{Z_1} \to 0$ as $s \to \infty$; the above lemma implies that $\tau_s (u, \varphi)|_{Z_1}$ must converge to a constant solution. However $Z_1$ was arbitrary, so in fact $\tau_s (u, \varphi)$ converges locally to a constant solution as $s \to \infty$, say $u(s,t) = x$ and $\varphi(s, t) = \psi(t)$ an eigensolution with eigenvalue $\lambda$: this implies $(u, \varphi)$ is a flow to $\bf{x} = (x, \lambda)$. The same argument holds in the other limit $s \to -\infty$. \end{proof}

We can now prove a compactness theorem for the twisted flow.

\begin{theorem}Let $(u_{\alpha}, \varphi_{\alpha})$ be a sequence of twisted trajectories between $\bf{x}_-$ and $\bf{x}_+$. Then there is a subsequence that Gromov converges to a broken twisted trajectory
\[((u^1, \varphi^1),(u^2, \varphi^2), ..., (u^n, \varphi^n)) \in \cal{M}(\bf{x}_-, \bf{x}^1) \times \hdots \times \cal{M}(\bf{x}^n, \bf{x}_+);\]
namely there are real numbers $\sigma_{\alpha}^1 < \hdots < \sigma_{\alpha}^n$ such that each translate $\tau_{-\sigma_{\alpha}^i}^*(u_{\alpha}, \varphi_{\alpha})$ converges locally to $(u^i, \varphi^i)$, and for any sequences of real numbers $\rho_{\alpha}$ with $\sigma^i_{\alpha} - \rho_{\alpha} \to -\infty$ and $\sigma^{i+1}_{\alpha} - \rho_{\alpha} \to +\infty$ for some $i = 0, \hdots, n$ (where by convention $\sigma^0 = -\infty, \sigma^{n+1} = +\infty$), the translates $\tau_{-\rho_{\alpha}}^*(u_{\alpha}, \varphi_{\alpha})$ converge locally to the constant solution $(x^i, \psi^i)$ at $\bf{x}^i$.\end{theorem}

\begin{proof} It suffices to show that there is a subsequence and a broken trajectory $(u^1, \varphi^1), \hdots, (u^n, \varphi^n)$, so that for every $\eps > 0$, there are partitions of $Z$ for each $\alpha$
\[Z = W^0_{\alpha} \cup Z^1_{\alpha} \cup W^1_{\alpha} \cup Z^2_{\alpha} \cup \hdots \cup Z^n_{\alpha} \cup W^n_{\alpha}\]
satisfying the properties of \eqref{strippartition}, where the translated solutions centred on $Z^i_{\alpha}$, $\tau^*_{\sigma^i_{\alpha}} (u_{\alpha}, \varphi_{\alpha})$ converge locally to $(u^i, \varphi^i)$, and such that $(u_{\alpha}, \varphi_{\alpha})|_{W^i_{\alpha}}$ are within $\eps$ of a constant solution $(x^i, \psi^i)$ in the $L^{2,k}$ norm.

Now, there exist uniform bounds $\mathcal{E}_0$, $K_0$ be uniform bounds for $\mathcal{E}(u_{\alpha})$ and $K(u_{\alpha}, \phi_{\alpha})$ respectively. Choose some small $\eps > 0$, and take the finite strip $Z_1 = [-1, 1] \times [0,1]$. The lemma above yields $\delta_1, \delta_2 > 0$ such that any solutions on $Z_1$ with $\mathcal{E}_{Z_1} < \delta_1$, $K_{Z_1} < \delta_2$ are within $\eps$ of a constant solution. In particular, for each $\alpha$, there are at most $2\mathcal{E}_0/\delta_1 + 2K_0/\delta_2$ integers $p$ such that $\tau^*_p (u_{\alpha}, \varphi_{\alpha})$ is not $\eps$-close to a constant solution.

We can then find, for some fixed $n$, partitions of the strip
\[Z = W^0_{\alpha} \cup Z^1_{\alpha} \cup W^1_{\alpha} \cup Z^2_{\alpha} \cup \hdots \cup Z^n_{\alpha} \cup W^n_{\alpha}\]
where each $Z^i_{\alpha}$ is of bounded length, at least one of $\cal{E}_{Z^i_{\alpha}} \ge \delta_1$ or $K_{Z^i_{\alpha}} \ge \delta_2$ holds, and such that $(u_{\alpha}, \varphi_{\alpha})|_{W^i_{\alpha}}$ is within $\eps$ of a constant solution.

Now, there are only finitely many intersection points of $L_0 \cap L_1$, and since $\Lambda(\varphi_{\alpha})$ is uniformly bounded independently of $\alpha$, for each intersection point only finitely many eigenvalues $\lambda$ can appear as the constant solutions proximate to $(u_{\alpha}, \varphi_{\alpha})|_{W^i_{\alpha}}$. Thus, pass to a subsequence so that each $(u_{\alpha}, \varphi_{\alpha})|_{W^i_{\alpha}}$ is within $\eps$ of the same constant solution at $\bf{x}^i$ independent of $\alpha$.

Pass to a further subsequence so that the translates $\tau^*_{\sigma^i_{\alpha}} (u_{\alpha}, \varphi_{\alpha})$ centred on $Z^i_{\alpha}$ converge locally to some solution $(u^i, \varphi^i)$. This solution must then itself be a non-constant trajectory between $\bf{x}^{i-1}$ and $\bf{x}^i$; this gives the desired broken limit.\end{proof}

\begin{corollary}The moduli space of broken twisted flows
\[\bar{\cal{M}}(\bf{x}_-, \bf{x}_+) = \bigcup\limits_{n \ge 0} \bigcup\limits_{\bf{x}^1, \hdots, \bf{x}^{n-1}} \cal{M}(\bf{x}_-, \bf{x}^1) \times \hdots \times \cal{M}(\bf{x}^{n-1}, \bf{x}_+)\]
with the topology of Gromov convergence is compact.\end{corollary}

We should finally note that the gluing properties of both the Maslov index $\mu(u)$ of a strip $u$, as well as the spectral flow $\specflow(u)$ over $u$ of \eqref{specflow_sfu} (which, in any case, is itself a Maslov index in the bundle $E$), as well as the proof above, imply the following. If a sequence $(u_{\alpha}, \varphi_{\alpha})$ of trajectories Gromov converges to a broken trajectory $((u^1, \varphi^1), \hdots, (u^n, \varphi^n))$, then for large enough $\alpha$,
\[\mu(u_{\alpha}) = \sum\limits_{i=1}^n \mu(u^i); \qquad \specflow(u_{\alpha}) = \sum\limits_{i=1}^n \specflow(u^i).\]
In particular, assuming all moduli spaces are regular, the dimension of the stratum of $\bar{\cal{M}}(\bf{x}_-, \bf{x}_+)$ consisting of broken trajectories with $n$ components is
\[\dim \cal{M}(\bf{x}_-, \bf{x}_+) - n + 1\]
as in ordinary Floer theory.

\subsection{Gluing.} We now prove a gluing theorem for moduli spaces of twisted trajectories, from which we conclude that $CF_{tw}(L_0, L_1; \frak{p})$ is indeed a complex with a differential of square zero. To prove this, along with the invariance and naturality results we wish, it suffices to only prove gluing results for the boundaries of the one-dimensional moduli spaces, in other words we only need to glue together twisted trajectories that are in the discrete (zero-dimensional) components of the moduli space.

\begin{remark}Whereas for the proof of compactness it was most convenient to use the $\sigma$ model for the twisted equations, for gluing we will find it much easier to use the $\tau$ model. The essential reason is that, in order to glue two $\sigma$-model solutions $(u^1, \phi^1) \in \cal{M}(x^0, x^1)$ and $(u^2, \phi^2) \in \cal{M}(x^1, x^2)$, we would first have to rescale $\phi^1, \phi^2$ to match the values of $\phi^1(T, \cdot)$ and $\phi^2(-T, \cdot)$, however since $||\phi^1(s,t)||_H = O(e^{-\lambda^1 s})$ as $s \to +\infty$, whereas $||\phi^2(s,t)||_H = O(e^{-\lambda^1 s})$ as $s \to -\infty$, the choice of rescaling factors cannot be bounded in $T$. Working in the $\tau$ picture avoids this issue by normalizing the solutions; it has the added advantage that we avoid any need to glue Sobolev spaces with exponential weights.\end{remark}

Working in the $\tau$ picture, fix $\bf{x}^0 = (x^0, \lambda^0), \bf{x}^1 = (x^1, \lambda^1), \bf{x}^2 = (x^2, \lambda^2) \in \frak{C}$ and twisted trajectories $(u^1, \varphi^1) \in \cal{M}(\bf{x}^0, \bf{x}^1), (u^2, \varphi^2) \in \cal{M}(\bf{x}^1, \bf{x}^2)$ which are regular and in the zero-dimensional components of the moduli space. Pick unit eigensolutions $\psi^0, \psi^1, \psi^2$ corresponding to $\lambda^0, \lambda^1, \lambda^2$ at $x^0, x^1, x^2$; we will also assume that $(u^1, \varphi^1), (u^2, \varphi^2)$ satisfy $\varphi^1 \to \psi^0, \psi^1$ at $\pm \infty$ and $\varphi^2 \to \psi^1, \psi^2$ at $\pm \infty$ respectively. We also fix slices of the translation action, for instance by requiring $u^1, u^2$ to have equal energy on $(-\infty, 0] \times[0,1]$ and $[0, \infty)\times[0,1]$.

We will follow the usual strategy of constructing for $T$ sufficiently large an approximate solution
\[(u_T, \varphi_T)\]
where $u_T: Z \to M$ is a smooth strip with boundary on $L_0, L_1$ with asymptotics to $x^0, x^2$ at $\pm \infty$, and $\varphi_T$ is a smooth section of $u_T^*E$, with boundary conditions on $F_0, F_1$, normalized to $||\varphi_T|| = 1$, and with $\varphi_T \to \psi^0, \psi^2$ respectively at $\pm \infty$. We will then construct an actual solution $(\hat{u}_T, \hat{\varphi}_T)$ of the twisted equations near $(u_T, \varphi_T)$, by using the implicit function theorem on the Banach manifold $\cal{S}^{2,k}_{\psi_- \psi_+}$.

Choose $T$ sufficiently large so that $u^1([T, \infty) \times[0,1])$ and $u^2((-\infty, -T]\times[0,1])$ are contained in a sufficiently small neighbourhood $U$ of $x^1$ on each $E$ is trivialized. Then, for $s \ge T$ and $s \le -T$ respectively we can write
\begin{equation}\label{etaerror}\varphi^1(s,t) = \psi^1(t) + \eta^1(s,t), \qquad \varphi^2(s,t) = \psi^1(t) + \eta^2(s,t)\end{equation}
where $\eta^1, \eta^2 \to 0$ uniformly in $t$ and exponentially in all derivatives as $s \to \infty$.

Now, choose a smooth, non-decreasing cut-off function $\beta_+ : \bb{R} \to [0,1]$ such that $\beta(s) = 0$ for $s \le -1$, and $\beta(s) = 1$ for $s \ge 1$; set $\beta_-(s) = 1 - \beta_+(s)$. We define the pre-glued map $(u_T, \varphi_T)$ as
\begin{align}u_T(s,t) &= \begin{cases}u^1(s+T, t) & \text{ for } s \le -1 \\
\exp_{x^1}\left(\beta_-(s) \exp_{x^1}^{-1}(u^1(s+T, t)) + \beta_+(s)\exp_{x^2}^{-1}(u^2(s-T, t))\right) & \text{ for } -1 \le s \le 1\\
u^2(s-T, t) & \text{ for } s \ge 1\end{cases}\\
\varphi_T(s,t) &= \begin{cases}\varphi^1(s + T, t) & \text{ for } s \le -1 \\
\frac{\psi^1(t) + \beta_-(s) \eta^1(s+T,t) + \beta_+(s)\eta^2(s-T,t)}{||\psi^1(t) + \beta_-(s) \eta^1(s+T,t) + \beta_+(s)\eta^2(s-T,t)||_H} & \text{ for } -1 \le s \le 1\\
\varphi^2(s - T, t) & \text{ for } s \ge 1\end{cases}\end{align}
where $\eta^1, \eta^2$ are the error terms of \eqref{etaerror}. This defines an element of the $\cal{S}^{2,k}_{\psi^0 \psi^2}$.

Observe that both $\bar{\partial}_J u_T$ and $\cal{F}_u(\varphi_T) = \bar{\nabla}_I \varphi_T - \Lambda(\varphi_T) \varphi_T$ are supported in $[-1,1]\times[0,1]$. Moreover, the exponential decay of $u^1, \varphi^1$ and $u^2, \varphi^2$ imply that for some constants $C$ and $\delta > 0$ independent of $T$ we have
\begin{equation}\label{pregluing_error_I}||\bar{\partial}_J u||_{L^{2,k-1}} + ||\cal{F}_u \varphi_T||_{L^{2,k-1}} \le C e^{- \delta T}.\end{equation}

We seek to construct a zero of the map
\begin{equation}\cal{F} : \cal{S}^{2,k}_{\psi^0 \psi^2} \to \cal{L}^{2,k-1}(TM) \oplus \cal{V}^{2,k-1}, \quad (u, \varphi) \mapsto (\bar{\partial}_J u, \bar{\nabla}_I \varphi - \Lambda(\varphi)\varphi)\end{equation}
near $(u_T, \phi_T)$. Following the strategy of Floer \cite{Floer88b}, we first extend the linearization of $(\bar{\partial}, \bar{\nabla})$ away from the zero locus, by fixing a metric on $M$ and using the Levi-Civita connection on $TM$ and the connection $\nabla$ on $E$. As usual, it will be convenient to choose a time-dependent metric $g_t$, so that each $L_i$ is totally geodesic for $g_i$. Indeed, once we have chosen a linearization $D_u$ of the ordinary $\bar{\partial}$-operator, for $(u, \varphi) \in \cal{S}^{2,k}_{\psi^0 \psi^2}$ we obtain a linear map 
\begin{align}D_{(u, \varphi)}: T_u \cal{Z}^{2,k}(x^0, x^2) \oplus T_{\varphi}\cal{S}^{2,k}_{\psi^0 \psi^2}(u^* E) &\to L^{2,k-1}(u^*TM) \oplus \cal{V}^{2,k-1}_{(u, \varphi)} \\
(\xi, \vartheta) &\mapsto (D_u \xi, B^{\tau}_{(u, \phi)}\xi + \Pi_V \bar{\nabla}_I \vartheta - \Lambda(\varphi)\vartheta)\end{align}
where $B^{\tau}_{(u,\phi)}\xi$ is given by \eqref{Btauoperator}.

\begin{remark}Notice that the $\cal{V}^{2,k-1}$ term of this operator is identical to the original calculation of $D\cal{F}_{(u, \varphi)}$ at a solution $(u, \varphi)$. If we were to work instead in the $\sigma$ picture, this would be obvious, since the equations are linear in $\phi$. In the $\tau$ picture this is not so clear: indeed were we to instead linearize $\cal{F}$ as a map taking values in all of $L^{2,k-1}(u^* E)$, there would be additional terms such as $\langle \bar{\nabla}_I \varphi, \vartheta \rangle \varphi$ in the above map. However, these terms vanish after using $\Pi_V$ to project back to $\cal{V}^{2,k-1}$. Indeed, one could think of $\Pi_V$ as being a sort of parallel transport map for $\cal{V}^{2,k-1}$.\end{remark}

For sufficiently small $\xi \in T_{u_T}\cal{Z}^{2,k}(x^0, x^2)$, the exponential map $\xi \mapsto \exp_{u_T}(\xi)$ yields charts for $\cal{Z}^{2,k}(x^0, x^2)$ around $u_T$. Moreover, if we assume that $\nabla$ is Hermitian for $I_t, \omega_E$, the parallel transport map $\Pi^{\exp_{u_T}(\xi)}_{u_T}$ for sections of $E$ preserves $\langle \cdot, \cdot \rangle_H$. Thus we have charts
\begin{equation}T_{u_T} \cal{Z}^{2,k}(x^0, x^2) \oplus T_{\varphi_T}\cal{S}^{2,k}_{\psi^0 \psi^2}(u_T^* E) \ni (\xi, \vartheta) \mapsto (\exp_{u_T}(\xi), \Pi^{\exp_{u_T}(\xi)}_{u_T} \vartheta) \in \cal{S}^{2,k}_{\psi^0\psi^2}.\end{equation}
We then take the non-linear map
\[\cal{F}_{(u_T, \phi_T)} : T_{u_T} \cal{Z}^{2,k}(x^0, x^2) \oplus T_{\varphi_T} \cal{S}^{2,k}_{\psi^0 \psi^2}(u_T^* E) \to L^{2,k-1}(u_T^*TM) \oplus V^{2,k-1}_{(u_T, \varphi_T)} \]
\[(\xi, \vartheta) \mapsto \left(\Pi^{\exp_{u_T}(\xi)}_{u_T}\right)^{-1} \cal{F}\left(\exp_{u_T}(\xi), \Pi^{\exp_{u_T}(\xi)}_{u_T} \vartheta\right).\]
If we write this as
\begin{equation}\cal{F}_{(u_T, \varphi_T)}(\xi, \vartheta) = \cal{F}(u_T, \varphi_T) + D_{(u_T, \phi_T)}(\xi, \vartheta) + N(\xi, \vartheta)\end{equation}
we then have the following standard quadratic estimate on the error term $N(\xi, \vartheta)$:
\begin{lemma}For some constant $C$ independent of $T$, if $(\xi, \vartheta), (\xi, \vartheta)' \in T_{u_T} \cal{Z}^{2,k}(x^0, x^2) \oplus T_{\varphi_T}\cal{S}^{2,k}_{\psi^0 \psi^2}(u_T^* E)$ such that $||(\xi, \vartheta)||_{2, k}, ||(\xi', \vartheta')||_{2,k} \le C$, then
\begin{equation}||N(\xi, \vartheta) - N(\xi', \vartheta')||_{2,k-1} \le C ||(\xi, \vartheta) - (\xi', \vartheta')||_{2,k} ||(\xi, \vartheta) + (\xi', \vartheta')||_{2,k}.\end{equation}\end{lemma}

The final ingredient for the implicit function theorem on $(u_T, \phi_T)$ is a uniformly bounded right inverse for the linearized operator $D_{(u_T, \phi_T)}$. We will later use a more explicit choice, but for now let us prove one exists:

\begin{lemma}\label{bounded_right_inverse}The operator $D_{(u_T, \varphi_T)}$ is Fredholm and surjective, of index $\mu(u_T) + \specflow(u_T) + i(\lambda^0) - i(\lambda^2) = 2$. Moreover, it admits a smooth family of right inverses $Q_{(u_T, \varphi_T)}$, with a uniform bound on the operator norm
\begin{equation}||Q_{(u_T, \varphi_T)}|| \le C\end{equation}
independent of $T$.\end{lemma}

\begin{proof}We follow the usual strategy of constructing an approximate right inverse first. By the regularity of the moduli spaces $\cal{M}(\bf{x}^0, \bf{x}^1)$ and $\cal{M}(\bf{x}^1, \bf{x}^2)$, we choose bounded right inverses $Q_{(u^1, \varphi^1)}, Q_{(u^2, \varphi^2)}$ for $D_{(u^1, \varphi^1)}, D_{(u^2, \varphi^2)}$.

Given $(\zeta, \sigma) \in L^{2,k-1}(u_T^*TM) \oplus V^{2,k-1}_{(u_T, \varphi_T)}$, we construct $\beta^1(\zeta, \sigma) \in L^{2,k-1}((u^1)^*TM) \oplus V^{2,k-1}_{(u^1, \varphi^1)}$, noting that $u^1(s + T,t) = u_T(s, t)$ for $s \le - 1$, and for $-1 \le s \le +1$, both $u^1(s + T,t)$ and $u_T(s, t)$ are contained in a small neighbourhood of $x^1$, on which $E$ is trivialized:
\begin{equation}\beta^1(\zeta, \sigma)(s,t) = \begin{cases}(\zeta(s, t), \sigma(s, t)) & \text{ for } s \le - 1;\\
\beta_-(s -1)\left(\Pi^{u_1(s+T,t)}_{u_T(s, t)} \zeta(s, t), \Pi_{V_{(u^1, \varphi^1)}} \sigma(s, t)\right) & \text{ for } -1 \le s \le +1;\\
0 & \text{ for } 1 \le s.\end{cases}\end{equation}
where $\Pi^{u^1(s+T, t)}_{u(s,t)}$ is again the parallel transport map, and $\Pi_{V_{(u^1, \varphi^1)}}$ is the $s$-dependent linear projection given by
\[\sigma(s,t) \mapsto \sigma(s,t) - \langle \varphi^1(s+T), \sigma(s) \rangle_H \varphi(s+T, t)\]
which is well defined by the trivialization of $E$. We hope no confusion is caused between the two; in any case the linear projection can be thought of as a kind of parallel transport between subspaces of the space of sections of $E$.

Likewise, we can also construct $\beta^2(\zeta, \sigma) \in L^{2,k-1}((u^2)^*TM) \oplus V^{2,k-1}_{(u^2, \varphi^2)}$ as
\begin{equation}\beta^2(\zeta, \sigma) = \begin{cases}0 & \text{ for } s \le - 1; \\
\beta_+(s +1)\left(\Pi^{u_2(s-T,t)}_{u_T(s, t)} \zeta(s, t), \Pi_{V_{(u^2, \varphi^2)}}\sigma(s, t)\right) & \text{ for } -1 \le s \le +1;\\
(\zeta(s, t), \sigma(s, t)) & \text{ for } 1 \le s.\end{cases}\end{equation}
Together these define the ``breaking map''
\begin{equation}\beta_T : L^{2,k-1}(u_T^*TM) \oplus V^{2,k-1}_{(u_T, \varphi_T)} \to \bigoplus\limits_{i = 1,2} \left(L^{2,k-1}((u^i)^*TM) \oplus V^{2,k-1}_{(u^i, \varphi^i)}\right).\end{equation}

Similarly, we can define the ``linear pregluing map''
\begin{equation}\gamma_T: \bigoplus\limits_{i = 1,2}W^{2,k}((u^i)^* TM) \oplus T_{\varphi^i} \cal{S}^{2,k}_{\psi^{i-1}\psi^i}((u^i)^*E) \to W^{2,k}(u_T^* TM) \oplus T_{\varphi_T} \cal{S}^{2,k}_{\psi^{0}\psi^2}(u_T^*E)\end{equation}
as follows: given $(\xi^i, \vartheta^i) \in W^{2,k}((u^i)^* TM) \oplus T_{\varphi^i} \cal{S}^{2,k}_{\psi^{i-1}\psi^i}((u^i)^*E)$ for $i = 1,2$, we set
\begin{equation*}\gamma_T(\xi^1, \vartheta^1, \xi^2, \vartheta^2) = (\xi_T, \vartheta_T)\end{equation*}
where $\xi_T, \vartheta_T$ are given by
\begin{align*}\xi_T &= \begin{cases}\xi^1(s+T, t) & \text{ for } s \le -1 \\
\beta_-(s) \Pi^{u_T(s,t)}_{u^1(s+T, t)} \xi^1(s+T, t) + \beta_+(s) \Pi^{u_T(s,t)}_{u^2(s-T, t)} \xi^2(s-T, t) & \text{ for } -1 \le s \le 1\\
\xi^2(s-T, t) & \text{ for } s \ge 1\end{cases}\\
\vartheta_T &= \begin{cases}\vartheta^1(s+T, t) & \text{ for } s \le -1\\
\Pi_{V_{(u_T, \varphi_T)}}\left(\beta_-(s) \vartheta^1(s+T, t) + \beta_+(s) \vartheta^2(s-T, t)\right) & \text{ for } -1 \le s \le 1\\
\vartheta^2(s-T, t) & \text{ for } 1 \le s\end{cases}\end{align*}
where again $\Pi_{V_{(u_T, \varphi_T)}}$ is the $s$-dependent projection onto sections which are $\langle \cdot, \cdot \rangle_H$-orthogonal to $\varphi_T$, which is well defined since all the sections in question can be thought of as sections of a fixed vector space $E_{x^1}$.

We can then define an approximate right inverse to $D_{(u_T, \varphi_T)}$ as the composition
\begin{equation}Q_{(u_T, \varphi_T)}^{approx} = \gamma_T \circ \begin{pmatrix}Q_{(u^1, \varphi^1)}& 0 \\ 0 & Q_{(u^2, \varphi^2)}\end{pmatrix} \circ \beta_T.\end{equation}
It is then a straightforward adaptation of the standard gluing estimates as in \cite{Floer88b}, using the exponential decay for $u^1, u^2$ as well as the exponential convergence of $\varphi^1, \varphi^2$ to $\psi^1$, to deduce:
\begin{itemize}\item The breaking map $\beta_T$ and the linear pregluing map $\gamma_T$ are both uniformly bounded in $T$;
\item hence the composite map $Q_{(u_T, \varphi_T)}^{approx}$ is also uniformly bounded in $T$;
\item the difference
\[D_{(u_T, \varphi_T)} \circ Q_{(u_T, \varphi_T)}^{approx} - \id\]
has operator norm exponentially decaying in $T$.\end{itemize}
Hence, we obtain a uniformly bounded right inverse for $D_{(u_T, \varphi_T)}$ as
\begin{equation}Q_{(u_T, \varphi_T)} = Q_{(u_T, \varphi_T)}^{approx} \circ \left(D_{(u_T, \varphi_T)} \circ Q_{(u_T, \varphi_T)}^{approx}\right)^{-1}\end{equation}
which is well defined for sufficiently large $T$. \end{proof}

It will be helpful, especially when we prove gluing results for equivariant Floer theory, to make a somewhat more explicit choice of such an inverse, following \cite{FOOO}. Choose marked points $(s_1, t_1) \in Z$ and $(s_2, t_2) \in Z$ which are regular for each of $u^1, u^2$ in the sense of \cite{FloerHoferSalamon95}, and at each $u^1(s_1, t_1)$ and $u^2(s_2, t_2)$, choose small codimension one hypersurfaces $H^i \subset M$ so that:
\begin{itemize}\item $u^i$ is tranverse to $H^i$, and moreover $(u^i)^{-1}(H^i)$ is transverse to $\bb{R} \times \{t_i\} \subset Z$, at $(s_1, t_1)$;
\item for $p \in H^i$ sufficiently close to $u^i(s_i, t_i)$, the short geodesic between them is contained in $H^i$, meaning $\exp_{u^i(s_i, t_i)}(B_{\eps}(0)\cap TH^i) \subset H^i$ for some $\eps > 0$.\end{itemize}
The first condition ensures that if we define the codimension one subspace
\[W^{2,k}((u^i)^*TM, TH^i) \subset W^{2,k}((u^i)^*TM, TH^i)\]
of vector fields $\xi$ with $\xi(s_i, t_i) \in TH^i$, then
\[W^{2,k}((u^i)^*TM, TH^i)\oplus T_{\varphi^i}\cal{S}_{\psi^{i-1} \psi^{i}}((u^i)^*E)\]
is complementary to the one-dimensional space $\ker(D_{(u^i, \varphi^i)})$; there is thus a unique right inverse $Q_{(u^i, \varphi^i)}$ for $D_{(u^i, \varphi^i)}$ with $W^{2,k}((u^i)^*TM, TH^i)$ as its image.

Moreover, we could similarly define a codimension two subspace $W^{2,k}(u_T^*TM, TH^1, TH^2)$ of vector fields along the pregluing $u_T$. It is then not difficult to see that the approximate right inverse $Q^{approx}_{(u_T, \varphi_T)}$ constructed earlier, and thus also the true right inverse $Q_{(u_T, \varphi_T)}$, has exactly $W^{2,k}(u_T^*TM, TH^1, TH^2) \oplus T_{\varphi_T} \cal{S}^{2,k}_{\psi^0 \psi^2}(u_T^*E)$ as its image. 

As a direct application of Floer's quantitative implicit function theorem of \cite{Floer88b}, \cite{Floer95}, we then have the following:

\begin{proposition}For some sufficiently small $\eps_0 > 0$, independent of $T$, there a unique element $(\xi_T, \vartheta_T) \in T_{u_T}\cal{Z}^{2,k}(x^0, x^2) \oplus T_{\varphi_T} \cal{S}^{2,k}_{\psi^0 \psi^2}(u_T^*E)$ such that $||(\xi_T, \vartheta_T)|| < \eps_0$ and each $\xi_T(s_i, t_i) \in TH^i$, and solves
\[\cal{F}_{(u_T, \varphi_T)}(\xi_T, \vartheta_T) = 0.\]
Moreover, $(\xi_T, \vartheta_T)$ is bounded by
\begin{equation}\label{gluing_estimate_I}||(\xi_T, \vartheta_T)|| \le ||Q_{(u_T, \varphi_T)} \cal{F}(u_T, \varphi_T)||.\end{equation}\end{proposition}

This thus yields the gluing map $G$:

\begin{theorem}\label{gluing_theorem_I}For some sufficiently large $T_0$, for each $T \ge T_0$ the resulting solution of the twisted equations
\[(\hat{u}_T, \hat{\varphi}_T) = (\exp_{u_T}(\xi_T), \Pi_{u_T}^{\exp_{u_T}(\xi_T)} \vartheta_T)\]
defines a smooth embedding
\[G: [T_0, \infty) \to \cal{M}(\bf{x}^0, \bf{x}^2).\]
Moreover, as $T \to \infty$, $(\hat{u}_T, \hat{\varphi}_T)$ converges in the Gromov topology to the broken trajectory $(u^1, \varphi^1), (u^2, \varphi^2)$, and by writing $G(\infty) = ((u^1, \varphi^1), (u^2, \varphi^2))$ in the compactification $\bar{\cal{M}}(\bf{x}^0, \bf{x}^1)$, the map $G$ is a surjection onto some open neighbourhood of $((u^1, \varphi^1), (u^2, \varphi^2))$.\end{theorem}

The smoothness of $G$ in $T$ follows from the smoothness of the pregluings $u_T, \varphi_T$ and the smoothness of the right inverses $Q_{(u_T, \varphi_T)}$. The convergence of $(\hat{u}_T, \hat{\varphi}_T)$ to the broken trajectory is an immediate consequence of the bound \eqref{gluing_estimate_I} and the exponential decay \eqref{pregluing_error_I} of $\cal{F}(u_T, \varphi_T)$.

For the injectivity, observe that if the hypersurfaces $H^i$ are chosen to be sufficiently small, for large enough $T$, each glued curve $\hat{u}_T$ intersects each $H^i$ at precisely one point $(s_i', t_i)$ with the same $t$-coordinate as the original marked point $(s_i, t_i)$; the difference $s_2' - s_1'$ is then a translation invariant quantity. But by the second assumption on $H^i$, this quantity is actually just $2T + s_2 - s_1$; hence $G$ is injective.

\begin{remark}This approach is similar in spirit to that in \cite{Schwarz}, with local codimension one hypersurfaces replacing the level sets of a Morse function. While an approach using the symplectic action could also work here, this approach will be more helpful for proving the gluing results of Theorems \ref{gluing_theorem_II}, \ref{gluing_theorem_III} in the setting of equivariant cohomology.\end{remark}

Proving the surjectivity is also made significantly simpler by our choice of local hypersurfaces. Indeed, given some sequence of twisted solutions $(u_{\alpha}, \varphi_{\alpha})$ converging to the broken trajectory $((u^1, \varphi^1), (u^2, \varphi^2))$, for sufficiently large $\alpha$ the curve $u_{\alpha}$ intersects each $H^i$ at precisely one point $(s_i', t_i)$ with the earlier fixed choice of $t_i$. Then letting $T = \frac{1}{2}(s_2' - s_1' - s_2 + s_1)$, after a translation we can write $(u_{\alpha}, \varphi_{\alpha}) = (\exp_{u_T}(\xi_{\alpha}), \Pi_{u_T}^{\exp_{u_T}(\xi_{\alpha})} \vartheta_{\alpha})$. We would like to conclude $\xi_{\alpha} = \xi_T$ and $\vartheta_{\alpha} = \vartheta_T$ by using the uniqueness part of the implicit function theorem. We have a priori $C^1$ bounds on $\xi_{\alpha}, \vartheta_{\alpha}$ for large enough $\alpha$ from the definition of Gromov convergence, which we would like to upgrade to an $L^{2,k}$ bound.

The final ingredient is then an estimate on the $L^{2,k}$-norm of a twisted solution $(u, \varphi)$ defined on a finite but ``long'' strip, which is close to a constant solution $(x, \psi)$. This allows us to deduce that for sufficiently large $\alpha$, $(u_{\alpha}, \varphi_{\alpha})$ is in the range of the gluing map.

\begin{proposition}\label{longstrip_phicontrol}Let $(x, \psi)$ be a critical point, and fix some $\eta > 0$. Then, there exists some $\eps > 0$, such that for all $R > 1$, whenever $(u, \varphi)$ is a solution of the twisted equations on $Z_R = [-R,R]\times[0,1]$ so that we have $d(u(s,t), x) < \eps$ and the $C^1$-bound $||\varphi - \psi||_{C^1(Z_R)} < \eps$, then over the slightly smaller strip $Z_{R-1} = [-R+1, R-1]\times[0,1]$:
\[||\exp^{-1}_{x}(u(s,t))||_{L^{2,k}(Z_{R-1})} \le \eta; \qquad ||\varphi(s,t) - \psi(t)||_{L^{2,k}(Z_{R-1})} \le \eta.\]\end{proposition}

The first inequality is a standard estimate in ordinary Floer theory; let us explain the bound on $\varphi(s,t) - \psi(t)$. First, we will always take $\eps$ small enough so that Assumption \ref{localgeomI} holds over an $\eps$-neighbourhood of $x$; in particular we can think of $\varphi(s,t)$ as solving the twisted equations in $(E_x, F_0|_x, F_1|_x)$, and from \eqref{curvaturebound} $\Lambda(\varphi)$ is a decreasing function. The main ingredient is the following estimate, which allows us to control $||\varphi(s,t) - \psi(t)||_{L^{2,k}}$ through the drop in $\Lambda(\varphi)$:

\begin{lemma}\label{longstrip_Lambda_drop}Take a fixed finite strip $Z_1 = [s_1, s_2] \times [0,1]$, a proper substrip $Z_1' = [s_1', s_2'] \times [0,1] \subset Z_1$, and a unit eigensolution $\psi(t)$ of $I_x \frac{d}{dt}$ with eigenvalue $\lambda$. Then, there is $\eps > 0$ and a constant $C$ such that whenever $\varphi$ solves the twisted equations in $(E_x, F_0|_x, F_1|_x)$ and satisfies the $C^1$ bound $||\varphi(s,t) - \psi(t)||_{C^1(Z_1)} < \eps$ over $Z_1$, we have
\[||\varphi(s,t) - \psi(t)||_{L^{2,k}(Z_1')} \le C\left(\Lambda(\varphi)(s_1) - \Lambda(\varphi)(s_2)\right).\] \end{lemma}

\begin{proof}It suffices to prove the $L^{2,1}$-bound over the whole strip $Z_1$: the bounds for higher $k$ over slightly smaller strips can then be obtained by bootstrapping.

Let $\vartheta = \varphi - \psi$. Since zero is not an eigenvalue of $I_x \frac{d}{dt}$, there is a constant $C_1$ so that
\begin{equation}\label{nondegeneracy_estimate}||\vartheta(s)||^2_{L^{2,1}([0,1])} \le C_1 || I_x \frac{\partial \vartheta}{\partial t}(s)||_H^2\end{equation}
at each point $s \in [s_1, s_2]$, where recall $|| \cdot ||_H$ is the $L^2(\times [0,1])$ norm. In particular, we have
\[||\vartheta(s,t)||^2_{L^{2,1}(Z_1)} \le \int_{s_1}^{s_2} ||\frac{\partial \vartheta}{\partial s}||^2_H + C_1 \int_{s_1}^{s_2} || I_x \frac{\partial \vartheta}{\partial t}(s)||_H^2 \]
For the second term above, observe that
\[|| I_x \frac{\partial \vartheta}{\partial t} ||^2_H = || I_x \frac{\partial \varphi}{\partial t}||^2_H - 2 \langle I_x \frac{\partial \varphi}{\partial t}, \psi\rangle + \lambda^2 ||\psi||^2_H = ||\frac{\partial \varphi}{\partial s}||_H^2 + \Lambda^2 - 2\lambda^2 \langle \varphi, \psi \rangle + \lambda^2\]
since $||\varphi||_H = ||\psi||_H = 1$ and $I_x \frac{d}{dt}$ is self-adjoint; however the $C^1$ control on $\varphi - \psi$ ensures that $\Lambda^2 - 2\lambda^2 \langle \varphi, \psi \rangle + \lambda^2$ is bounded by a constant multiple of $\Lambda - \lambda$. For the remaining terms, we observe that for a solution of the equations on $E_x$ with constant $I, \nabla$ we have
\[\frac{d}{ds}\Lambda(\varphi) = -2 ||\frac{\partial \varphi}{\partial s}||^2_H\]
which after integration yields the desired bound. \end{proof}

Proposition \ref{longstrip_phicontrol} then follows by dividing $Z_R$ into smaller strips of equal bounded length independent of $R$, and summing the contributions from the above lemma.

\section{Equivariant cohomology}

\subsection{Manifolds with boundary}

Let us first quickly review the set-up for Morse theory on manifolds $Y$ with boundary $\partial Y$ as developed in Kronheimer-Mrowka \cite{KronheimerMrowka07}. Suppose that a Morse function $f$ and metric $g$ are chosen on $Y$ so that the gradient vector field $V$ which is \emph{everywhere tangent} to the boundary $\partial Y$. Thus, in addition to critical points in the interior of $Y$, there are also critical points on the boundary. For such $y \in \partial Y$, the inward pointing normal $N_y$ is always an eigenvector of the Hessian $\nabla^2(f)$, with either positive or negative eigenvalue. Accordingly we have three types of critical points:
\begin{itemize}\item $\frak{C}_o(f)$, the interior critical points;
\item $\frak{C}_s(f)$, the boundary-stable $y \in \partial Y$, where the eigenvalue of $N_y$ is positive;
\item $\frak{C}_u(f)$, the boundary-unstable $y \in \partial Y$, where the eigenvalue of $N_y$ is negative.\end{itemize}
Observe that for a stable critical point $y \in \frak{C}_s(f)$, the Morse indices $\ind_Y(y)$, $\ind_{\partial Y}(y)$ of $y$ thought of as a critical point of $f$ and $f|_{\partial Y}$ respectively coincide, whereas in the unstable case $y \in \frak{C}_u(f)$ we have
\begin{equation}\ind_Y(y) = \ind_{\partial Y}(y) + 1\end{equation}
For critical points $y_-, y_+$, write $\cal{M}(y_-, y_+)$ for the moduli space of Morse flows $u(s)$ between $y_-$ and $y_+$, modulo translation. Observe that in the case when
\begin{equation*}y_- \in \frak{C}_s(f), \qquad y_+ \in \frak{C}_u(f)\end{equation*}
the descending disc of $y_-$ as well as the ascending disc of $y_+$ are both contained in $\partial Y$: in particular they can never intersect transversely when considered as submanifolds of $Y$. Consequently, the Palais-Smale transversality conditions can never be satisfied. Instead, we proceed with modified hypothesis that these intersect transversely \emph{as submanifolds of $\partial Y$}. We call this the ``boundary-obstructed'' case. Together with the usual Palais-Smale asumptions for other pairs of critical points $y_-, y_+$, the moduli spaces $\cal{M}(y_-, y_+)$ are all smooth manifolds, of dimension
\begin{equation}\dim \cal{M}(y_-, y_+) = \begin{cases}\ind(y_-) - \ind(y_+) & \text{ in the boundary obstructed case}\\
\ind(y_-) - \ind(y_+) - 1 & \text{ otherwise}\end{cases}\end{equation}
Observe that $\cal{M}(y_-, y_+)$ is automatically empty in the two cases
\begin{equation*}(a): \quad y_- \in \frak{C}_o, y_+ \in \frak{C}_u; \qquad (b): \quad y_- \in \frak{C}_s, y_+ \in \frak{C}_o.\end{equation*}
In the case that both $y_-, y_+ \in \frak{C}_u$, or $y_-, y_+ \in \frak{C}_s$, observe that all Morse trajectories connecting them must lie entirely within the boundary $\partial Y$. Moreover, in the case that $y_- \in \frak{C}_u$, $y_+ \in \frak{C}_s$, the moduli space can be partitioned
\begin{equation}\cal{M}(y_-, y_+) = \cal{M}^o(y_-, y_+) \sqcup \cal{M}^{\partial}(y_-, y_+)\end{equation}
between those Morse trajectories $u(s)$ which respectively pass through the interior of $Y$, and are contained entirely in the boundary $\partial Y$. These have dimension $\ind(y_-) - \ind(y_+) - 1$ and $\ind(y_-) - \ind(y_+) - 2$ respectively; indeed $\cal{M}(y_-, y_+)$ can be given the structure of a manifold with boundary $\cal{M}^{\partial}(y_-, y_+)$.

We then define free graded $\bb{F}_2$-vector spaces
\[C_o^* = \bigoplus\limits_{\frak{C}_o(f)}\bb{F}_2 \langle y \rangle, \qquad C_s^* = \bigoplus\limits_{\frak{C}_s(f)}\bb{F}_2 \langle y \rangle, \qquad C_u^* = \bigoplus\limits_{\frak{C}_u(f)}\bb{F}_2 \langle y \rangle\]
with grading given by the Morse index. By counting zero-dimensional moduli spaces of flows, we define a homomorphism
\[d_{oo} : C_o^* \to C_o^{*+1}, \qquad d_{oo}(y_+) = \sum\limits_{\ind(y_-) = \ind(y_+) +1} \# \cal{M}(y_-, y_+) y_-;\]
similarly we have other counts of flows from counting interior flows
\[d_{os} : C_s^* \to C_o^{*+1}, \qquad d_{uo}: C_o^* \to C_u^{*+1}, \qquad d_{us} : C_s^* \to C_u^{*+1}.\]

We then also have similar homomorphisms which this time count zero-dimensional moduli spaces of flows entirely contained in $\partial Y$, noting that for unstable critical points $\ind_{\partial Y}(y) = \ind_Y(y) - 1$:
\[\begin{matrix}d^{\partial}_{ss}: C_s^* \to C_s^{*+1}, & & & d^{\partial}_{us}: C_s^* \to C_u^{*+2}\\
d^{\partial}_{su}: C_u^* \to C_s^*, & & &  d^{\partial}_{uu}: C_u^* \to C_u^{*+1}.\end{matrix}\]

From this, we can define three different Morse cochain complexes:
\begin{align}\label{checkC}\check{C}^*(Y, f) = C_o^* \oplus C_s^*, \qquad & \text{differential } \check{d} = \begin{pmatrix}d_{oo} & d_{os} \\ d^{\partial}_{su} d_{uo} & d^{\partial}_{ss} + d^{\partial}_{su}d_{us}\end{pmatrix}\\
\label{hatC}\hat{C}^*(Y, f) = C_o^* \oplus C_u^*, \qquad & \text{differential } \hat{d} = \begin{pmatrix}d_{oo} & d_{os} d^{\partial}_{su} \\ d_{uo} & d^{\partial}_{uu} + d_{us} d^{\partial}_{su}\end{pmatrix}\\
\label{barC}\bar{C}^*(Y, f) = C_s^*\oplus C_u^{*+1}, \qquad &  \text{differential } \bar{d} = \begin{pmatrix}d^{\partial}_{ss} & d^{\partial}_{su} \\ d^{\partial}_{us} & d^{\partial}_{uu}\end{pmatrix}\end{align}
giving cohomology groups $\check{H}^*(Y,f), \hat{H}^*(Y,f), \bar{H}^*(Y,f)$, which model the cohomology of $Y$, the pair $(Y, \partial Y)$ and the boundary $\partial Y$ respectively.

The complex $(\bar{C}, \bar{d})$ is clearly just the Morse complex for $\partial Y$ with the Morse function $f|_{\partial Y}$. For the other two complexes, the differential counts the index one \emph{broken trajectories}: observe there can be two-component broken trajectories connecting critical points $y_-, y_+$ of index difference one so long as one of the components is \emph{boundary obstructed}, since these boundary obstructed trajectories have index zero (and are counted by $d^{\partial}_{su}$ in the above formulae).

The proof that $(\check{C}, \check{d})$ and $(\hat{C}, \hat{d})$ are complexes is analogous to the usual proof that the Morse differential squares to zero, by looking at the one-dimensional moduli spaces of (possibly broken) trajectories connecting $y_-, y_+$ of index difference two. This moduli space as in ordinary Morse theory has a compactification by broken trajectories, where the new feature is that these can have \emph{more than two components} due to the presence of boundary obstructed trajectories. Moreover in the case that $y_-, y_+$ are respectively boundary unstable and stable, the moduli space of \emph{unbroken} boundary trajectories $\cal{M}^{\partial}(y_-, y_+)$ also must be included in the compactification of $\cal{M}^o(y_-, y_+)$.

The presence of boundary obstructed trajectories has serious analytical implications: the gluing results for boundary obstructed trajectories are significantly weaker. The compactified moduli space is still stratified by manifolds, but is \emph{not a naturally a one-manifold with boundary}, but rather carries something which Kronheimer and Mrowka \cite{KronheimerMrowka07} call a \emph{delta structure}. Nevertheless, for such spaces stratified by one- and zero-manifolds, the number of points in the zero-manifold stratum is still even (or zero if counted with appropriate signs), and so we can still conclude that the appropriate algebraic relations hold (we will discuss this in more detail in sections \ref{sec:extended_moduli_space} and \ref{sec:boundary_obstructed_gluing}).

Thus, by counting the number of broken trajectories in the compactification in the four cases when $y_-, y_+$ are respectively: both interior critical points; interior and boundary-stable; boundary-unstable and interior point; and boundary-unstable and boundary-stable; we obtain four relations
\begin{align}d_{oo}^2 + d_{os} d^{\partial}_{su} d_{uo} &= 0; \\
d_{oo} d_{os} + d_{os}d^{\partial}_{ss} + d_{os}d^{\partial}_{su} d_{us} &= 0; \\
d_{uo} d_{oo} + d^{\partial}_{uu} d_{uo} + d_{us} d^{\partial}_{su} d_{uo} &= 0; \\
d^{\partial}_{us} + d_{uo} d_{os} + d^{\partial}_{uu}d_{us} + d_{us} d^{\partial}_{ss} + d_{us} d^{\partial}_{su} d^{us} &= 0.\end{align}
Each of these terms are explained in more detail in \ref{stratalist}. In addition to these four relations, there are another four coming from operators counting just boundary trajectories: these are more simply encapsulated by $\bar{d}^2 = 0$. Together, these relations show that $(\check{C}, \check{d})$ and $(\hat{C}, \hat{d})$ are complexes.

There are also homomorphisms
\begin{align}j^*: \hat{C}^* \to \check{C}^*, \qquad & \begin{pmatrix}\id_{C_o} & 0 \\ 0 & d^{\partial}_{su}\end{pmatrix}: C_o^* \oplus C_u^* \to C_o^* \oplus C_s^*\\
\label{localization_formula}i^*: \check{C}^* \to \bar{C}^*, \qquad & \begin{pmatrix} 0 & \id_{C_s} \\ d_{uo} & d_{us}\end{pmatrix}: C_o^* \oplus C_s^* \to C_s^* \oplus C_u^{*+1}\\
\partial: \bar{C}^* \to \hat{C}^{*+1}, \qquad & \begin{pmatrix} d_{os} & 0 \\ d_{us} & \id_{C_u}\end{pmatrix}: C_s^* \oplus C_u^{*+1} \to C_o^{*+1} \oplus C_u^{*+1}\end{align} 
which by a similar argument are seen to be chain maps. Moreover, these define an exact triangle of chain complexes; so in particular there is a long exact sequence
\begin{equation}\hdots \bar{H}^{*-1} \xrightarrow{\partial} \hat{H}^* \xrightarrow{j^*} \check{H}^* \xrightarrow{i^*} \bar{H}^* \xrightarrow{\partial} \hat{H}^{*+1} \to \hdots\end{equation}
modelling the long exact sequence for the pair $(Y, \partial Y)$.

\subsection{A model for $\bb{Z}/2$-equivariant Morse cohomology.}
Now, suppose $\tilde{X}$ is a smooth manifold with a smooth $\bb{Z}/2$-action given by an involution $\iota: \tilde{X} \to \tilde{X}$. Let $X$ be the fixed point set. It is a priori a manifold with components of different dimension; for simplicity we will assume all components have codimension $k$ inside $\tilde{X}$. Let $N$ be the normal bundle to $X \subset \tilde{X}$; note that at $x \in X$ it can also be characterized as the $-1$-eigenspace of the endomorphism $d \iota : T_x \tilde{X} \to T_x \tilde{X}$.

Consider the oriented real blow-up $\tilde{Y}$ of $X \subset \tilde{X}$: this is a manifold with boundary $\partial \tilde{Y} = S(N)$ the sphere bundle of $N$, coming with a projection
\begin{equation}\pi : \tilde{Y} \to \tilde{X}\end{equation}
which is the identity over $\tilde{X}\backslash X$, and the natural projection $S(N) \to X$ over the fixed point set $X$. Moreover, the $\bb{Z}/2$-action on $\tilde{X}$ lifts to a \emph{free action} on $\tilde{Y}$, which is just the antipodal map
\begin{equation*}-1 : S(N) \to S(N)\end{equation*}
over $\pi^{-1}(X)$. Let $Y$ be the quotient $\tilde{Y}/\bb{Z}/2$; this is a manifold with boundary $\partial Y = \bb{P}(N)$ the real projective bundle of $N$. One can also think of $Y$ as the blow-up of the singular quotient $\tilde{X}/\bb{Z}/2$ along the invariant set $X$.

Observe that the blow-up $\tilde{Y} \to \tilde{X}$ is a $(k-1)$-connected, $\bb{Z}/2$-equivariant map, and the $\bb{Z}/2$-action on $\tilde{Y}$ is free: accordingly, there is an isomorphism in singular cohomology
\begin{equation}H^*(Y; \bb{F}_2) \cong H^*_{\bb{Z}/2}(\tilde{Y}; \bb{F}_2) \cong H^*_{\bb{Z}/2}(\tilde{X}; \bb{F}_2) \quad \text{for } * \le k-2.\end{equation}
Hence, we will build a model for the $\bb{Z}/2$-equivariant cohomology of $\tilde{X}$, at least in low degrees, from the Morse cohomology of $Y$, as a manifold with boundary. If one also wanted to model equivariant cohomology in higher degrees, one could replace $\tilde{X}$ with a stabilization $\tilde{X} \times \bb{R}^{\ell}$, with $\bb{Z}/2$ action given by the product of $\iota$ and $-1: \bb{R}^{\ell} \to \bb{R}^{\ell}$. However, the infinite dimensional Floer-theoretic situations of our eventual application, invariant set will already be of infinite codimension, so this is not necessary.

Take a $\bb{Z}/2$-invariant Morse function $f: \tilde{X} \to \bb{R}$, together with a $\bb{Z}/2$-invariant metric. For points $x \in X \subset \tilde{X}$ in the invariant set, the Hessian $\nabla^2 f: T_x\tilde{X} \to T_x\tilde{X}$ respects the decomposition $T_x\tilde{X} = T_x X \oplus N_x$. Let $A_x : N_x \to N_x$ be the anti-invariant (or normal) part of the Hessian. Let us also assume that the eigenspaces of $A_x$ are all one-dimensional; by the Morse assumption the eigenvalues are all nonzero.

Its gradient vector field $\nabla f$ lifts to a vector field $\tilde{V}$ on the interior $\tilde{Y}\backslash \partial \tilde{Y}$ of the blow-up, which naturally extends to a vector field everywhere tangent to the boundary $\partial \tilde{Y}$: in the fibre $S(N_x)$ of a critical point $x \in X$, $\tilde{V}$ is precisely the Morse flow of the quadratic function
\begin{equation}S(N_x) \ni \psi \mapsto \langle \psi, A_x \psi\rangle\end{equation}
so in particular the boundary critical points of $\tilde{V}$ are exactly the unit eigenvectors of $A_x$. A boundary critical point is stable if the corresponding eigenvalue is positive, and unstable if the eigenvalue is negative.

The vector field $\tilde{V}$ is clearly $\bb{Z}/2$-equivariant, so it descends to a vector field $V$ on the quotient $Y$. Away from the boundary, this vector field is just the gradient $\nabla f$ descended to the quotient, and it is everywhere tangent to the boundary $\partial Y = \bb{P}(N)$. The critical points of $V$ come in three types:
\begin{itemize}\item The interior ones, which correspond to pairs $\{x, \iota(x)\}$ of \emph{non-invariant} critical points in $\tilde{X}$;
\item The boundary-stable ones, which correspond to an invariant critical point $x \in X$ and a \emph{positive} eigenvalue $\lambda > 0$ of $A_x$;
\item The boundary-unstable ones, which correspond to an invariant critical point $x \in X$ and a \emph{negative} eigenvalue $\lambda < 0$ of $A_x$.\end{itemize}
Morse trajectories in the interior of $Y$ correspond to $\bb{Z}/2$-orbits in the space of finite energy solutions $u : \bb{R} \to \tilde{X}$ of the Morse flow equation
\begin{equation*}\frac{d}{ds}u(s) + \nabla f(u(s)) = 0\end{equation*}
on $\tilde{X}$ that are entirely contained in $\tilde{X}\backslash X$. If such a trajectory has a limit at an invariant critical point
\begin{equation*}u(s) \to x_+ \in X \quad \text{ as } s \in \infty\end{equation*}
then for $s \gg 0$, via a choice of local coordinates around $x$, we can think of $u(s)$ as an element of $T_{x_+} \tilde{X} = T_{x_+} X \oplus N_{x_+}$. Writing $\pi_N u(s)$ for the projection of $u(s)$ to $N_{x_+}$, we must have
\begin{equation}\pi_N u(s) \sim C e^{-\lambda s} \psi \quad \text{ as } s \to \infty\end{equation}
for some eigenvector $\psi$ of $A_{x_+}$ with \emph{positive} eigenvalue $\lambda > 0$. Accordingly, we think of $u(s)$ as being a trajectory with positive limit the boundary-stable critical point given by $(x_+, \lambda)$.

Similarly, if $u(s) \to x_- \in X$ has a $\bb{Z}/2$-invariant limit as $s \to -\infty$, then considering the projection $\pi_N u(s)$ to $N_{x_-}$ for $s \ll 0$ we have
\begin{equation}\pi_N u(s) \sim C e^{-\lambda s} \psi \quad \text{ as } s \to -\infty\end{equation}
for an eigenvector $\psi$ of $A_{x_-}$ this time with \emph{negative} eigenvalue $\lambda < 0$; and so we think of $u(s)$ as having a negative limit the boundary-unstable critical point given by $(x_-, \lambda)$.

We also trajectories contained entirely in the boundary $\partial Y \cong \bb{P}(N)$. These correspond to pairs $(u(s), \phi(s))$ of a map $u: \bb{R} \to X$ together with a nonzero section $\phi$ of $u^* N$ satisfying the equations
\begin{align*}&\frac{d}{ds}u(s) + \nabla f(u(s)) = 0 \\
&\nabla_s \phi(s) + A_{u(s)} \phi(s) = 0\end{align*}
considered up to a rescaling action of $\phi$ by $\bb{R}^*$. Note that these are exactly a finite-dimensional version of the twisted equations we studied in the previous section. As before, such a boundary solution has positive and negative limits at boundary critical points $(x_{\pm}, \lambda_{\pm})$ if $u(s)$ is a Morse trajectory from $x_-$ to $x_+$ and
\begin{equation}\phi(s) \sim C_{\pm} e^{-\lambda_{\pm} s} \psi_{\pm} \quad \text{ as } s \to \pm \infty\end{equation}
for constants $C_{\pm}$ and eigenvectors $\psi_{\pm}$ of $A_{x_{\pm}}$ with eigenvalues $\lambda_{\pm}$, which can be either positive or negative.

Under the appropriate transversality assumptions, we are now entirely in the previous setting of Morse theory for manifolds with boundary. In particular, we have three complexes
\begin{equation}\check{C}^*(Y) = C_o^* \oplus C_s^*, \qquad \hat{C}^*(Y) = C_o^* \oplus C_u^*, \qquad \bar{C}^*(Y) = C_s^* \oplus C_u^{*+1}.\end{equation}
with corresponding cohomology groups $\check{H}^*, \hat{H}^*, \bar{H}^*$ forming a long exact sequence as before.

To define the $\bb{F}_2[t]$-module structures, observe that $Y$ is a free $\bb{Z}/2$-quotient of $\tilde{Y}$. In particular we could have just as well defined a triple of complexes $\check{C}^*(\tilde{Y}), \hat{C}^*(\tilde{Y}), \bar{C}^*(\tilde{Y})$ as above, whose generators are either interior critical points of $\tilde{Y}$ (i.e. a non-invariant critical point of $\tilde{X}$), or a distinguished choice of unit eigenvector $\psi$ of the self-adjoint operator $A_x$ at an invariant critical point $x$ of $X$. Each of these would be a free $\bb{F}_2[\bb{Z}/2]$-complex, and the natural maps $j^*, i^*, \partial$ are $\bb{F}_2[\bb{Z}/2]$-homomorphisms. The original complexes $\check{C}^*(Y), \hat{C}^*(Y), \bar{C}^*(Y)$ are then obtained by applying the functor $A \mapsto A_{\bb{F}_2}$ of Section 2.1. In particular the cohomology groups $\check{H}^*, \hat{H}^*, \bar{H}^*$ have the structure of $\bb{F}_2[t]$-modules, and the natural maps $j^*, i^*, \partial$ between these groups are $\bb{F}_2[t]$-module homomorphisms.

Finally, note that
\begin{equation}\bar{H}^*(Y) \cong H^*(\bb{P}(N); \bb{F}_2) \cong \frac{H^*(X; \bb{F}_2) \tensor \bb{F}_2[t]}{t^{k} + w_1(N) t^k + \hdots + w_k(N)}\end{equation}
where $w_i(N)$ are the Stiefel-Whitney classes of $N$. In particular, we can then view $i^* : \check{H}^* \to \bar{H}^*(Y)$ as a finite-dimensional replacement for the localization map
\begin{equation}H^*_{\bb{Z}/2}(\tilde{X}; \bb{F}_2) \to H^*(X; \bb{F}_2)\tensor \bb{F}_2[t]\end{equation}
which is the approach we now adopt for constructing this map in the Floer theoretic case.

\subsection{Equivariant Lagrangian Floer theory.}

Suppose that $\tilde{M}^n$ is an exact symplectic manifold, with convex boundary at infinity (for instance, a Liouville manifold), with a symplectic involution $\iota$. Let $\tilde{L}_0$, $\tilde{L_1}$ be exact Lagrangians of $\tilde{M}$, which are preserved by the involution $\iota$, and which we assume are either compact, or conical and disjoint at infinity. The fixed point set, which we will denote $M$, is automatically a symplectic submanifold, of which the fixed parts $L_i = \tilde{L}_i \cap M$ of $\tilde{L}_i$ are (exact) Lagrangian submanifolds.

Let $E = N_{M \subset \tilde{M}} \xrightarrow{\pi} M$ be the normal bundle of $M$, or equivalently the bundle of $-1$-eigenspaces of $d \iota : T\tilde{M}|_M \to T\tilde{M}|_M$. This is naturally a symplectic vector bundle. Moreover, for $i = 0,1$, $E|_{L_i}$ has Lagrangian subbundles $F_i = N_{L_i \subset \tilde{L_i}}$; the tuple $\frak{p} = (E, F_0, F_1)$ describes polarization data as before. We will build a model for the equivariant Lagrangian Floer cohomology $HF_{\bb{Z}/2}(\tilde{L}_1, \tilde{L}_2)$ following the outlined Morse-theoretic approach of blowing up the invariant locus, which will couple the normal pseudoholomorphic curve equation away from $M$, with the twisted equations for Floer theory on $M$ twisted by the polarization $\frak{p}$.

Indeed, $\tilde{X} = \cal{P}_{\tilde{M}}(\tilde{L_0}, \tilde{L_1})$ inherits a $\bb{Z}/2$-action, with fixed point set $X = \cal{P}_M(L_0, L_1)$. The symplectic action functional then define a $\bb{Z}/2$-invariant Morse function on $\tilde{X}$. Moreover, the choice of a $\bb{Z}/2$-equivariant almost complex structure $\tilde{J}_t$ on $\tilde{M}$ induces a $\bb{Z}/2$-invariant metric on $\tilde{X}$.

The normal bundle of $X \subset \tilde{X}$ at a path $x : [0,1] \to M$ with $x(i) \in L_i$ for $i = 0, 1$ is
\begin{equation}N_{x} = \{\psi(t) \in \Gamma([0,1], x^* E): \psi(0) \in F_0, \psi(1) \in F_1\}.\end{equation}
There is an induced complex structure $I_t$ on $E$, so after choosing a Hermitian connection, we obtain self-adjoint operators $A_x = I_t \nabla_t$ on $N_x$ at each $x \in X$, which can also be thought of as the Hessians of the action functional when restricted to $N$. Thus, in the formal blow-up of $\tilde{X}/\bb{Z}/2$ along $X$, the boundary flow is given by the twisted flow equations \eqref{twistedeqn}. The stable and unstable boundary critical points correspond to the positive and negative eigenvalues of $I_t \nabla_t$ respectively at $x \in L_0 \cap L_1$. The interior critical points correspond to pairs of non-invariant points of $\tilde{L}_0 \cap \tilde{L}_1$, whilst the interior flow is just the ordinary pseudoholomorphic curve equation on $\tilde{M}$, considered modulo the additional $\bb{Z}/2$-action.

Thus, choose a $\bb{Z}/2$-equivariant compatible, time-dependent almost complex structure $\tilde{J}_t$ on $\tilde{M}$, of contact type at infinity. Along the invariant set $M$, this necessarily preserves the decomposition $T\tilde{M} = TM \oplus E$, and so we obtain an almost complex structure $J_t$ on $M$, and a complex structure $I_t$ on $E$. For simplicity, we will always assume $\tilde{L}_0, \tilde{L}_1$ intersect everywhere transversely; if not we could work with equivariant Hamiltonian perturbations.

As before, it will be helpful to make some strong assumptions on the local form of $\tilde{J}_t$ near the invariant set, as well as near the invariant intersection points of $\tilde{L}_0, \tilde{L}_1$. These can be obtained from Weinstein's symplectic tubular neighbourhood theorem, adapted for this $\bb{Z}/2$-equivariant setting:

\begin{assumption}\label{localgeomII}There is an equivariant open neighbourhood $N$ of the invariant set $M \subset \tilde{M}$, identified with the disc bundle of $E$ by a symplectomorphism
\begin{equation}\beta_N : N \xrightarrow{\sim} D(E)\end{equation}
which is equivariant for the $\bb{Z}/2$-action on $N$ by $\iota$, and on $D(E)$ by multiplying the fibres by $-1$, with the additional properties:
\begin{itemize}\item The symplectomorphism $\beta_N$ sends the $\tilde{L}_i$ to the linear subbundles $F_i$, in other words
\begin{equation}\beta_N(\tilde{L_0}\cap N) = D(F_0); \qquad \beta_N(\tilde{L}_1 \cap N) = D(F_1);\end{equation}
\item there is a symplectic connection $\nabla$ on $E$ such at each $(x,\psi) \in D(E)$, after using the connection to split the tangent space as $T_{(x,\psi)} E \cong T_x M \oplus E_x$, the almost complex structure on $D(E)$ given by
\begin{equation}\begin{pmatrix}J_t & 0 \\ 0 & I_t \end{pmatrix} : T_x M \oplus E_x \to T_x M \oplus E \end{equation}
agrees, via $\beta_N$, with $\tilde{J}_t$;
\item Near each $x \in L_0 \cap L_1$, there conditions of Assumption \ref{localgeomI} are satisfied; in particular $\tilde{M}$ is locally modelled on $\bb{C}^{n-k} \times \bb{C}^k$ with the involution $\begin{pmatrix}1 & 0 \\ 0 & -1\end{pmatrix}$, such that $\tilde{L}_0, \tilde{L}_1$ are locally given by $L_0 \times G_0, L_1 \times G_1$ for linear Lagrangian subspaces $G_i$ of $\bb{C}^k$.\end{itemize}\end{assumption}

Once again, these assumptions ensure that at each $x \in L_0 \cap L_1$, the spectrum of the operator $I \frac{d}{dt}$ is simple, and can be labelled
\begin{equation*} \hdots \lambda_{-2}(x) < \lambda_{-1}(x) < 0 < \lambda_0(x) < \lambda_1(x) < \lambda_2(x) < \hdots \end{equation*}

For generators of our Floer complexes, we then have
\begin{align}\frak{C}_o &= \{\bf{x} = \text{ pairs } \{x, \iota(x)\} \subset \tilde{L}_0 \cap \tilde{L}_1 : x \not= \iota(x)\}; \\
\frak{C}_s &= \{\bf{x} = (x, \lambda_i) : x \in L_0 \cap L_1, \lambda_i \in \Spec_{>0} (I \frac{d}{dt})\}; \\
\frak{C}_u &= \{\bf{x} = (x, \lambda_i): x \in L_0 \cap L_1, \lambda_i \in \Spec_{<0} (I \frac{d}{dt})\} \end{align}
where we hope the notation boldface $\bf{x}$ to refer to any of these three things does not cause unnecessary confusion; we will again write $\frak{C} = \frak{C}_o \cup \frak{C}_s \cup \frak{C}_u$.

Let $C_o$, $C_s$, $C_u$ be the free $\bb{F}_2$-vector spaces generated by each of the above (note that these are a priori ungraded). Once again write
\begin{equation}\check{C} = C_o \oplus C_s, \qquad \hat{C} = C_o \oplus C_u, \qquad \bar{C} = C_s \oplus C_u.\end{equation}

Now, consider finite energy solutions $u: Z \to \tilde{M}$ of the $\tilde{J}_t$-holomorphic strip equation with boundary conditions on $\tilde{L}_0, \tilde{L}_1$:
\begin{equation*}\partial_s u + \tilde{J}_t(u) \partial_t u = 0 \quad \text{ with } u(\cdot, i) \in \tilde{L}_i \text{ for } i = 0, 1.\end{equation*}
In addition to the natural $\bb{R}$-action by translation in $s$, there is now also a $\bb{Z}/2$ action on the space of solutions induced from involution $\iota$. There is an important preliminary observation to make about such solutions:

\begin{proposition}$\tilde{J}_t$-holomorphic strips $u$ are either totally contained in the invariant set $M$, or $u^{-1}(M) \subset Z$ is a discrete set of points.\end{proposition}
\begin{proof}This is a consequence of the identity principle for pseudoholomorphic strips (see \cite{FloerHoferSalamon95}), applied to $u$ and $\iota(u)$.\end{proof}

Solutions $u$ contained in $M$ are exactly $J_t$-holomorphic curves. Furthermore, suppose that for some substrip $Z' \subseteq Z$, a $\tilde{J}_t$ holomorphic strip $u$ with boundaries on $\tilde{L}_i$ is in fact entirely contained in the open neighbourhood $N$ of Assumption \ref{localgeomII}, which is identified with $D(E)$ by $\beta_N$. Then, we can write
\begin{equation}\beta_N u(s,t) = (v(s,t), \phi(s,t))\end{equation}
for some $v: Z' \to M$ and some section $\phi(s,t)$ of $v^* E$, with $v(\cdot, i) \in L_i$ and $\phi(\cdot, i) \in F_i$ for $i = 0,1$. Moreover, under Assumption \ref{localgeomII}, the $\tilde{J}_t$ holomorphic curve equation is then exactly identified with the twisted equations for $(v, \phi)$:
\begin{equation}\partial_s v + J_t(v) \partial_t v = 0, \qquad \nabla_s \phi + I_t(v) \nabla_t \phi = 0.\end{equation}

As before, we could also work in the $\tau$-model: for $s$ such that $\beta_N(u(s \times [0,1])) \subset D(E)$, set
\begin{equation}r(s) = ||\phi||_H, \qquad \varphi(s,t) = \frac{\phi(s,t)}{r(s)}\end{equation}
which then satisfy the equations
\begin{equation}\frac{dr}{ds} + \Lambda(\varphi) r = 0, \qquad \nabla_s \varphi + I_t(v) \nabla_t \varphi - \Lambda(\varphi)\varphi = 0\end{equation}
with the added condition that $r(s) > 0$ where defined. Unlike our earlier setting, it is important that we remember the real function $r(s)$: indeed whereas for the twisted equations on their own there was a $\bb{R}^*$-symmetry, which we reduced to a $\bb{Z}/2$-symmetry in the $\tau$-model, in this setting there is only a $\bb{Z}/2$-symmetry to begin.

We now focus on the geometry of the moduli spaces of solutions not contained within $M$. We will consider the quotients of the moduli spaces these form by the $\bb{Z}/2$ action. Suppose the limits of the strip are $u(s,t) \to x_{\pm}$ as $s \to \pm \infty$. There are now several cases to consider, depending on whether $x_-, x_+$ are $\bb{Z}/2$-invariant or not.

\emph{Case I: If both $x_-, x_+$ are non-invariant,} the involution $\iota$ exchanges the moduli spaces of strips:
\begin{equation*}\cal{M}(x_-, x_+) \longleftrightarrow \cal{M}(\iota(x_-), \iota(x_+)); \qquad \cal{M}(x_-, \iota(x_+)) \longleftrightarrow \cal{M}(\iota(x_-), x_+).\end{equation*}
Moreover, we have elements $\bf{x}_{\pm} \in \frak{C}_o$, and so define $\cal{M}(\bf{x}_-, \bf{x}_+)$ to be the moduli space of strips connecting either of $x_-, \iota(x_-)$ to $x_+, \iota(x_+)$, modulo the $\bb{Z}/2$-action.

\emph{Case II: If $x_-$ is non-invariant, whilst $x_+$ is invariant,} then there is a bijection
\begin{equation*}\cal{M}(x_-, x_+) \longleftrightarrow \cal{M}(\iota(x_-), x_+).\end{equation*}
Moreover, for such a strip $u$, for $s \gg 0$, $u(s,t)$ lies in the tubular neighbourhood neighbourhood $N \cong D(E)$ of $M$ of Assumption \ref{localgeomII}, so in particular we can write
\begin{equation*}\beta_N(u(s,t)) = (v(s,t), \phi(s,t))\end{equation*}
where $(v(s,t), \phi(s,t))$ solve the twisted equations. Moreover, there is a trivialization of $E$
\begin{equation*}\alpha_U : E|_U \cong \bb{C}^k \times U\end{equation*}
in a neighbourhood $U$ of $x_+$, and the finite energy assumption on $u$ means we can think of $\phi(s,t)$ as a finite energy holomorphic map to $\bb{C}^k$ for $s \gg 0$, which is nonzero since $u$ is not contained in $M$. The exponential decay of such maps then imply that
\begin{equation}\phi(s,t) \sim C e^{-\lambda s} \psi(t) \qquad \text{as } s \to \infty \end{equation}
for some eigenvector $\psi(t)$ of the operator $I \frac{d}{dt}$ at $x_+$, with $\lambda > 0$. We can thus think of $u$ as being a trajectory with positive limit at $\bf{x}_+ = (x_+, \lambda) \in \frak{C}_s$, and negative limit at $\bf{x}_- = \{x_-, \iota(x_-)\} \in \frak{C}_o$; and we have a moduli space $\cal{M}(\bf{x}_-, \bf{x}_+)$ of all such trajectories modulo the $\bb{Z}/2$ action.

We could alternatively describe such solutions in the $\tau$ model; for $s \gg 0$ we write $\phi(s,t) = r(s) \varphi(s,t)$ with $|| \varphi(s,t) ||_H = 1$. We must then have
\begin{equation}\varphi(s,t) \to \psi(t)\end{equation}
as $s \to \infty$, where for some unit eigensolution $\psi$ of $I_{x_+} \frac{d}{dt}$ with $\lambda > 0$; the convergence is exponential. Likewise $e^{\lambda s} r(s)$ must have a finite, nonzero limit as $s \to \infty$. Before taking the $\bb{Z}/2$-quotient, the moduli space of such trajectories with asymptotics of type $\lambda$ then is partitioned into four disjoint parts: those trajectories from $x_-$ to $(x_+, \psi)$, those from $x_-$ to $(x_+, -\psi)$, those from $\iota(x_-)$ to $(x_+, \psi)$ and those from $x_-$ to $(x_+, -\psi)$; the $\bb{Z}/2$-quotient identifies pairs of these. 

We will later see that $\cal{M}(\bf{x}_-, \bf{x}_+)$ is a smooth manifold, whose components can be grouped according to the homotopy class of the strip $[u]$, of dimension
\begin{equation}\dim \cal{M}_{[u]}(\bf{x}_-, \bf{x}_+) = \mu(u) - i(\lambda) - 1\end{equation}
where $\mu(u)$ is the Maslov index, and as before $i(\lambda) \in \bb{Z}_{\ge 0}$ is such that $\lambda = \lambda_i(x)$.

\emph{Case III: If $x_-$ is invariant, whilst $x_+$ is non-invariant,} again for such a flow, $u(s,t)$ lies in the tubular neighbourhood $N \cong D(E)$ of $M$ for $s \ll 0$, and thus can be identified with a solution to the twisted equations $(v, \phi)$. In particular we must have asymptotics
\begin{equation}\phi(s,t) \sim C e^{- \lambda s} \psi(t) \qquad \text{as } s \to -\infty \end{equation}
where this time instead $\lambda$ must be negative. Hence we think of this as being a trajectory from $\bf{x}_- = (x_-, \lambda) \in \frak{C}_u$ to $\bf{x}_+ = (x_+, \iota(x_+)) \in \frak{C}_o$.

We can similarly describe the moduli space in the $\tau$ model, identifying $\phi(s,t)$ with the pair $r(s), \varphi(s,t)$. We must then have
\begin{equation}\varphi(s,t) \to \psi(t)\end{equation}
exponentially fast as $s \to -\infty$, where $\psi(t)$ is a unit eigensolution of $I_{x_-} \frac{d}{dt}$ for a negative eigenvalue $\lambda < 0$. Similarly $e^{\lambda s}r(s)$ has a finite, nonzero limit as $s \to -\infty$.

Modulo translation and the $\bb{Z}/2$-action, the solutions in either picture form a moduli space $\cal{M}(\bf{x}_-, \bf{x}_+)$. Again, this a smooth manifold, and the dimension of the components of paths in the same homotopy class $[u]$ is
\begin{equation}\dim \cal{M}_{[u]}(\bf{x}_-, \bf{x}_+) = \mu(u) + i(\lambda).\end{equation}

\emph{Case IV: If both $x_-$ and $x_+$ are invariant,} we have the above asymptotic behaviors at both the positive and the negative ends of the strip. In particular, we can think of $u$ as being a trajectory from $\bf{x}_- = (x_-, \lambda_-) \in \frak{C}_u$ to $\bf{x}_+ = (x_+, \lambda_+) \in \frak{C}_s$. We will call such trajectories ``interior trajectories'' following the analogy with manifolds with boundary, and denote their moduli space $\cal{M}^o(\bf{x}_-, \bf{x}_+)$. Note the difference between these and \emph{twisted flows} between $\bf{x}_-$ and $\bf{x}_+$, which are to be thought of as entirely contained in the boundary of the formal blow-up. Again, this is a manifold, of dimension near $u$
\begin{equation}\dim \cal{M}_{[u]}(\bf{x}_-, \bf{x}_+) = \mu(u) + i(\lambda_-) - i(\lambda_+).\end{equation}
Moreover, working in the $\tau$ model say, if $(v, \varphi)$ is a twisted flow from $\bf{x}_-$ to $\bf{x}_+$, then solution $r : \bb{R} \to \bb{R}$ of $\frac{d}{ds}r + \Lambda(\varphi)r$ satisfy $r(s) \to 0$ as $s \to \pm \infty$. Then, for all sufficiently small positive solutions $r(s)$ (say with $0 < r(s) < 1$ everywhere), $(v, r \varphi)$ defines a holomorphic strip in the disc bundle $D(E)$, and thus a $\tilde{J}_t$-holomorphic strip in $\tilde{M}$ with asymptotics at $x_-, x_+$ of type $\lambda_-, \lambda_+$ respectively. In particular, we can naturally think of $\cal{M}^o(\bf{x}_-, \bf{x}_+)$ as a \emph{manifold with boundary} given by the moduli space of twisted flows from $\bf{x}_-$ to $\bf{x}_+$, which we will in this context refer to as $\cal{M}^{\partial}(\bf{x}_-, \bf{x}_+)$ to avoid confusion.

In all these cases, we will refer to the elements of $\cal{M}(\bf{x}_-, \bf{x}_+)$ as \emph{interior trajectories}, to stress the analogy with the picture of the Morse theory of blow-ups. As in that picture, we also have \emph{boundary trajectories}: for $\bf{x}_-, \bf{x}_+ \in \frak{C}_s \cup \frak{C}_u$, to stress this we will usually write $\cal{M}^{\partial}(\bf{x}_-, \bf{x}_+)$ for the moduli spaces of twisted trajectories between them; these will again be manifolds, of dimension 
\begin{equation}\mu(u) + \specflow(u) + i(\lambda_-) - i(\lambda_+) - 1.\end{equation}
We will say that such a trajectory $(u, \phi)$ has \emph{index in $\tilde{M}$}
\begin{equation}\label{tildeM_index}\ind_{\tilde{M}}(u, \phi) = \begin{cases}\mu(u) + \specflow(u) + i(\lambda_-) - i(\lambda_+) - 1& \text{if } \bf{x}_- \in \frak{C}_s, \bf{x}_+ \in \frak{C}_u \\ \mu(u) + \specflow(u) + i(\lambda_-) - i(\lambda_+) & \text{if both } \bf{x}_-, \bf{x}_+ \in \frak{C}_s \text{ or } \frak{C}_u\\
\mu(u) + i(\lambda_-) - i(\lambda_+) + 1 & \text{if } \bf{x}_- \in \frak{C}_u, \bf{x}_+ \in \frak{C}_s\end{cases}
\end{equation}
and again we refer to the first case as the \emph{boundary obstructed} case. When we wish to place the interior and boundary trajectories on the same footing, we will refer to them as \emph{blown up trajectories}.

By counting possibly broken, blown-up trajectories with total index one, we define differentials $\check{d}, \hat{d}, \bar{d}$ on $\check{C}, \hat{C}, \bar{C}$ in accordance with the formulas \eqref{checkC}, \eqref{hatC}, \eqref{barC} from before. Again, there are natural chain homomorphisms $j^*, i^*, \partial$ between these complexes. Observe that in the case of $(\bar{C}, \bar{d})$ this is just the complex computing the polarization-twisted Floer cohomology of $(L_0, L_1)$.

The $\bb{F}_2[t]$-module structures are defined as before: each of the three complexes lift to free $\bb{F}_2[\bb{Z}/2]$-complexes, by building complexes generated by the non-invariant intersection points of $\tilde{L}_0 \cap \tilde{L}_1$ (rather than pairs of such points), as well as for each $x \in L_0 \cap L_1$ pairs of generators $(x, \psi)$ and $(x, -\psi)$ for the unit $|| \cdot ||_H$-norm eigensolutions of $I_{x} \frac{d}{dt}$. The natural maps $j^*, i^*, \partial$ then lift to maps of $\bb{F}_2[\bb{Z}/2]$, and the original complexes $\check{C}, \hat{C}, \bar{C}$ can be recovered using the functor $A \mapsto A_{\bb{F}_2}$ of Section 2.1.

\begin{definition}We define the equivariant Floer cohomology, and by analogy with cyclic homology, the negative equivariant Floer cohomology
\begin{equation}HF_{\bb{Z}/2}(\tilde{L}_0, \tilde{L}_1), \qquad HF^-_{\bb{Z}/2}(\tilde{L}_0, \tilde{L}_1)\end{equation}
to be the cohomology of the complexes $(\check{C}, \check{d})$ and $(\hat{C}, \hat{d})$ respectively. These are $\bb{F}_2[t]$-modules. Moreover there is a natural long exact sequence of $\bb{F}_2[t]$-modules
\begin{equation}\label{equivariant_LES} \hdots \to HF_{\bb{Z}/2}^-(\tilde{L}_0, \tilde{L}_1) \to HF_{\bb{Z}/2}(\tilde{L}_0, \tilde{L}_1) \to HF_{tw}(L_0, L_1; \frak{p}) \to \hdots \end{equation}
where $\frak{p} = (E, F_0, F_1)$ is the polarization data coming from the normal to the invariant set $M$. The map $HF_{\bb{Z}/2}(\tilde{L}_0, \tilde{L}_1) \to HF_{tw}(L_0, L_1; \frak{p})$ will be referred to as the \emph{localization map}.\end{definition}

In Section \ref{sec:analysis_II}, we detail the necessary analytic arguments to rigorously define the differentials $\hat{d}, \check{d}, \bar{d}$ and the natural maps $j^*, i^*, \partial$, and prove both that these respectively square to zero and are chain maps. Moreover, these same analytical arguments adapted to a continuation equation also show that the modules $HF_{\bb{Z}/2}(\tilde{L}_0, \tilde{L}_1)$ and $HF^-_{\bb{Z}/2}(\tilde{L}_0, \tilde{L}_1)$ are invariant under a perturbation of $\tilde{J}_t$, as well as equivariant exact Lagrangian isotopies of $\tilde{L}_0, \tilde{L}_1$.

Before we do that, observe that each of $\check{C}, \hat{C}$ have a filtration by the symplectic action, compatible with the $\bb{F}_2[t]$-module structure, analogously to the filtration on polarized Floer theory. Moreover, the associated graded modules are precisely
\begin{equation*}\bigoplus\limits_{\bf{x} \in \frak{C}_o} \bb{F}_2\langle \bf{x}\rangle \oplus \bigoplus\limits_{x \in L_0 \cap L_1} \bb{F}_2[t] \langle x \rangle, \qquad \bigoplus\limits_{\bf{x} \in \frak{C}_o} \bb{F}_2\langle \bf{x}\rangle \oplus \bigoplus\limits_{x \in L_0 \cap L_1} t^{-1}\bb{F}_2[t^{-1}] \langle x \rangle\end{equation*}
respectively, with the obvious $\bb{F}_2[t]$-module structures. In particular, the associated graded for $\hat{C}$ is a torsion $\bb{F}_2[t]$-module; thus the same must be true for $\hat{C}$: each $\alpha \in \hat{C}$ must be annihilated by $T^n$ for some $n$. In particular, recalling that $T$ is already invertible on twisted Floer cohomology, we deduce

\begin{proposition}\label{localization_vanishing}The negative equivariant Floer cohomology $HF^-_{\bb{Z}/2}(\tilde{L}_0, \tilde{L}_1)\tensor_{\bb{F}_2[t]} \bb{F}_2[t, t^{-1}]$ vanishes after inverting $t$.\end{proposition}

An immediate corollary, in view of the long exact sequence \eqref{equivariant_LES}, is Theorem \ref{localization_theorem}: the ``localization map''
\[HF_{\bb{Z}/2}(\tilde{L}_0, \tilde{L}_1) \to HF_{tw}(L_0, L_1; \frak{p})\]
is an isomorphism after inverting $t$. This is, however, not enough to deduce the Smith-type inequality of Theorem \ref{smith_inequality}: there is so far no natural spectral sequence relating the ordinary Floer cohomology $HF(\tilde{L}_0, \tilde{L}_1)$ to its equivariant analogue. This will instead by a consequence of Theorem \ref{KMSScomparison}, which compares our model of equivariant cohomology to Seidel-Smith's.

\section{Analytical aspects II}\label{sec:analysis_II}

\subsection{Fredholm theory and equivariant transversality} In Section 2.2, we described Sobolev space frameworks in each of the $\sigma$ and $\tau$ models, so that the twisted equations are Fredholm, and proved that for fixed $J_t$ cutting out the moduli spaces of holomorphic strips in $M$ transversality, and generic complex structures $I_t$ and connections $\nabla$ on $E$, the spaces of solutions are cut out transversely.

We will do the same here for the ``interior'' moduli spaces $\cal{M}(\bf{x}_-, \bf{x}_+)$ of $\tilde{J}_t$-holomorphic strips, with prescribed asymptotics at the fixed point set. As before, we will give both a $\sigma$-model description, as well as an equivalent $\tau$-model description. On the one hand, the index formula will be much more transparent in the $\sigma$-model, while gluing will be much easier in the $\tau$-model.

Let us focus on the case that $\bf{x}_- = (x_-, \iota(x_-)) \in \frak{C}_o$ and $\bf{x}_+ = (x_+, \lambda_+) \in \frak{C}_s$. The case that $\bf{x}_-, \bf{x}_+ \in \frak{C}_o$ is that of standard Lagrangian Floer theory, whilst the remaining two cases where $\bf{x}_- \in \frak{C}_u$ are entirely analogous to the one at hand.

Choose a small $\delta > 0$ so that there are no eigenvalues of $I_{x_+}\frac{d}{dt}$ in between $\lambda_+ - \delta$ and $\lambda_+$.

We begin in the $\sigma$ model. Let $\cal{Z}^{2,k}(x_-, \bf{x}_+)$ be the set of all continuous maps $u: Z \to \tilde{M}$, with boundary on $\tilde{L}_0, \tilde{L}_1$, such that:
\begin{itemize}\item The map $u$ is locally class $L^{2,k}$;
\item for $s \ll 0$, there is some $\xi_-(s,t) \in T_{x_-}\tilde{M}$ of class $W^{2,k}$ on some semi-infinite strip $(-\infty, -S_0] \times [0,1]$ such that
\[u(s,t) = \exp_{x_-}(\xi_-(s,t));\]
\item for $s \gg 0$, $u(s,t)$ is contained in the neighbourhood $N$ of the invariant set $M$, and under $\beta_N : N \cong D(E)$, we have $u(s,t) = (v(s,t), \phi(s,t))$ such that: on some semi-infinite strip $[S_0, \infty) \times [0,1]$ there is $\xi_+(s,t) \in T_{x_+}M$ of class $W^{2,k}$ with
\[v(s,t) = \exp_{x_+}(\xi_+(s,t))\]
and moreover $e^{(\lambda_+ - \delta) s}\phi(s,t)$ is class $W^{2,k}$.\end{itemize}

To see that this is a Banach manifold, we construct a weight function as follows. Choose first a smooth function $\gamma: \tilde{M} \to [0,1]$ such that, writing $D_{r}(E)$ for the radius $r$ disc bundle of $E$, we have:
\begin{equation}\label{gamma_bump}\gamma(x) = \begin{cases}0 & \text{ if } x \notin N_{2/3} = (\beta_N)^{-1} D_{2/3}(E); \\
1 & \text{ if } x \in N_{1/3} = (\beta_N)^{-1} D_{1/3}(E).\end{cases}\end{equation}
Then, choose a smooth function $w: \bb{R} \to \bb{R}$ such that
\[w(s) = \begin{cases}0 &\text{ for } s \ll 0; \\
(\lambda_+ - \delta)s & \text{ for } s \gg 0.\end{cases}\]

Then, for $s \in \bb{R}$ and any tangent vector $\xi \in T_x \tilde{M}$ for any $x \in \tilde{M}$, define the $\gamma, w$-rescaling of this tangent vector to be
\begin{equation}R_{\gamma, w} (\xi) = \begin{cases}\xi & \text{if } x \notin N \\
(\beta_N^{-1})_*(\xi_M, e^{\gamma(x) w(s)} \theta) & \text{if } x \in N, (\beta_N)_*\xi = (\xi_M, \theta) \in T_{\pi(x)}M \oplus E_{\pi(x)}\\\end{cases}\end{equation}
recalling that $\beta_N : N \xrightarrow{\sim} D(E)$, and the tangent bundle of $D(E)$ is split as $\pi^*TM \oplus \pi^*E$ by the choice of connection. Then, define
\begin{equation}W^{2,k}_{\lambda_+}(u) := \{\xi \in C^0(Z, u^*T \tilde{M}): R_{\gamma, w}(\xi) \text{ is class } W^{2,k}, \ \xi(\cdot, i) \in T\tilde{L}_i \text{ for } i = 0,1\}\end{equation}
This is a Banach space via the norm
\begin{equation}||\xi||_{W^{2,k}_{\lambda_+}} = ||R_{\gamma, w}(\xi)||_{W^{2,k}};\end{equation}
and moreover via the exponential map this makes $\cal{Z}^{2,k}(x_-, \bf{x}_+)$ into a Banach manifold with $W^{2,k}_{\lambda_+}(u)$ as its tangent spaces. Moreover, we can similarly rescale sections of $u^* T\tilde{M}$ without Lagrangian boundary conditions, and thus define
\begin{equation}L^{2,k-1}_{\lambda_+}(u) := \{\eta \in C^0(Z, u^*T \tilde{M}): R_{\gamma, w}(\eta) \text{ is class } L^{2,k-1}\}.\end{equation}
These fit into a Banach vector bundle $\cal{L}^{2,k-1}_{\lambda_+}(TM)$ over $\cal{Z}^{2,k}(x_-, \bf{x}_+)$, such that the $\bar{\partial}_{\tilde{J}_t}$-operator defines a Fredholm section
\begin{equation}\bar{\partial}_{\tilde{J}_t}: \cal{Z}^{2,k}(x_-, \bf{x}_+) \to \cal{L}^{2,k-1}_{\lambda_+}(TM)\end{equation}
whose zero set, after quotienting by $\bb{R}$, is precisely the moduli space $\cal{M}(x_-, \bf{x}_+)$.

\begin{proposition}At a solution $u$, the index of the linearized operator
\[D(\bar{\partial}_{\tilde{J}_t}) : W^{2,k}_{\lambda_+}(u) \to L^{2,k-1}_{\lambda_+}(u)\]
is equal to $\mu(u) - i(\lambda_+)$.\end{proposition}

\begin{proof}This is again an argument using the cut-and-paste and homotopy invariance properties of the Fredholm index. Consider the linearized $\bar{\partial}$-operator without exponential weights
\[F: W^{2,k}(u^*T\tilde{M}) \to L^{2,k-1}(u^*T\tilde{M}),\]
which has index $\mu(u)$. Consider also the linear $\bar{\partial}$ operator acting on spaces of sections of the constant vector bundle on $T_{x_+}M, E_{x_+}$:
\[G_+ = (\bar{\partial}_{J}, \bar{\partial}_{I}) : W^{2,k}(Z, T_{x_+}M) \oplus W^{2,k}_{\lambda_+}(Z, E_{x_+}) \to L^{2,k-1}(Z, T_{x_+}M) \oplus L^{2,k-1}_{\lambda_+}(Z, E_{x_+}),\]
which has index $i(\lambda_+)$ by Proposition \ref{sigma_model_index}, where
\[W^{2,k}_{\lambda_+}(Z, E_{x_+}) = e^{-w(s)}W^{2,k}(Z, E_{x_+}), \quad L^{2,k-1}(Z, E_{x_+}) = e^{-w(s)} L^{2,k-1}(Z, E_{x_+})\]
have are Sobolev spaces on the strip with a large exponential weight $e^{w(s)}$ at the positive end, and unweighted at the negative end. Gluing these two operators together, we obtain an operator homotopic to $D(\bar{\partial}_{\tilde{J}_t}) : W^{2,k}_{\lambda_+}(u) \to L^{2,k-1}_{\lambda_+}(u)$, yielding the result.\end{proof}

We make a completely analogous construction in the remaining cases where $\bf{x}_- \in \frak{C}_u$, and either $\bf{x}_+ \in \frak{C}_o$ or $\frak{C}_s$.

In the case when $\bf{x}_- \in \frak{C}_u$ and $\bf{x}_+ \in \frak{C}_o$, we then similarly define a space of strips $\cal{Z}^{2,k}_{\lambda_-}(\bf{x}_-, x_+)$ in the case that $\bf{x}_+ \in \frak{C}_o$ by imposing similar decay conditions on the $E$-coordinate of $u$ as $s \ll 0$, this time with a weight function of the form $e^{(\lambda_- + \delta)s}$. Likewise, when $\bf{x}_- \in \frak{C}_u$ and $\bf{x}_+ \in \frak{C}_s$, we define a space $\cal{W}^{2,k}_{\lambda_-\lambda_+}(\bf{x}_-, \bf{x}_+)$ with decay conditions at both infinite ends. These are Banach manifolds, modelled after Banach spaces $W^{2,k}_{\lambda_-}(u)$ and $W^{2,k}_{\lambda_-\lambda_+}(u)$ respectively, defined by a similar rescaling map $R_{\gamma, w}$ that weights the $E$-coordinate by $e^{(\lambda_- + \delta)s}$ for $s \ll 0$, as well as by $e^{(\lambda_+ - \delta) s}$ for $s \gg 0$ in the second case. Again, there are Banach vector bundles $\cal{L}^{2,k-1}_{\lambda_-}$ and $\cal{L}^{2,k-1}_{\lambda_-\lambda_+}$ of class $L^{2,k-1}$ sections of $u^* T\tilde{M}$ with the same asymptotic conditions and no boundary conditions, such that the $\bar{\partial}_{\tilde{J}_t}$-operator defines a Fredholm section. The zero set of this section is the moduli space of $\tilde{J}_t$-holomorphic strips, and its linearization has index
\begin{equation*}\mu(u) + i(\lambda_-) + 1\end{equation*}
in the case $\bf{x}_- \in \frak{C}_u$ and $\bf{x}_+ \in \frak{C}_o$, and index
\begin{equation*}\mu(u) + i(\lambda_-) - i(\lambda_+) + 1\end{equation*}
in the case $\bf{x}_- \in \frak{C}_u$ and $\bf{x}_+ \in \frak{C}_s$.

\begin{definition}\label{equivarianttransversality} We say that a time-dependent, equivariant almost complex structure $\tilde{J}_t$ is \emph{equivariantly regular} if it satisfies the conditions of Assumption \ref{localgeomII} for some symplectic tubular neighbourhood $N$ of $M$, and
\begin{itemize}\item the moduli spaces of non-invariant $\tilde{J}_t$-holomorphic strips $\cal{M}(\bf{x}_-, \bf{x}_+)$ with prescribed asymptotics are all transversely cut out, meaning that at each $u$, the linearized $\bar{\partial}_{\tilde{J}_t}$ operator acting on the weighted Sobolev spaces as defined above is surjective;
\item the induced almost complex structure $J_t$ on $M$ is regular for all the moduli spaces of $J_t$ holomorphic strips $\cal{M}(x_-, x_+)$ contained entirely within $M$;
\item the induced complex structure $I_t$ on $E$ is regular in the sense of polarization-twisted Floer theory.\end{itemize}
When this holds, the moduli spaces $\cal{M}(\bf{x}_-, \bf{x}_+)$ are then smooth manifolds of the claimed dimensions.\end{definition}

Constructing such a $\tilde{J}_t$ is not hard. First, pick regular $J_t$ on $M$, and $I_t$ on $E$. Then, fix a tubular neighbourhood $N$ of $M$ and $\beta_N : N \cong D(E)$; after choosing a symplectic connection on $E$ this determines the almost complex structure on $N$. We then pick some smooth, $\bb{Z}/2$-equivariant extension $\tilde{J}_t$ of this to all of $\tilde{M}$. For non-invariant $u : Z \to M$ contained entirely within the closure $\bar{N}$, the regularity of the twisted equations means that such $u$ is automatically cut out transversely.

For $u$ not contained entirely within $\bar{N}$, there must exist a \emph{regular point} of $Z$ in the sense of Floer-Hofer-Salamon \cite{FloerHoferSalamon95} whose image lies outside $\bar{N}$; moreover we can also choose it so that its image is disjoint from $\iota(u(Z))$. This then implies that $\tilde{J}_t$ can be equivariantly perturbed outside of $N$ to achieve regularity.

\subsection{The $\tau$-model.}\label{tau_model_II} We will also need an alternative Banach manifold set-up using the $\tau$-model, in order to perform gluing without large exponential weights. However since the only application of this set-up will be to gluing, we will only give a ``local'' Banach manifold structure near a \emph{solution} $u$ to the $\bar{\partial}$ equations.

Again fix a tubular neighbourhood $N$ of the invariant set $M$ and an isomorphism $\beta_N : N \cong D(E)$. 

Let us again focus on the case $\bf{x}_- = (x_-, \iota(x_-)) \in \frak{C}_o$ and $\bf{x}_+ = (x_+, \lambda_+) \in \frak{C}_s$, and fix a unit eigenvector $\psi_+$ for $I_{x_+} \frac{d}{dt}$ with eigenvalue $\lambda_+$.

Recall that for any solution $u$ with limits at $x_-$ and $x_+$, there does not exist an $s \in \bb{R}$ such that $u(s,[0,1])$ is entirely contained in $M$. Moreover, there exists an $S \in \bb{R}$ such that for all $s \ge S$, we have $u(s,t) \in N$. 

We will write $Z_S^+ = [S,\infty) \times [0,1]$ for the semi-infinite strip. For a map $u : Z \to \tilde{M}$ such that $u(Z_S^+) \subset N$, under the identification $\beta_N : N \cong D(E)$, we will write $v = \pi_M u : Z_S^+ \to M$ to be the projection of $u$ to $M$, and $\phi = \pi_E u \in v^*E$ to be the fiber coordinate of $u$.

Given $S \in \bb{R}$, we will define a Banach manifold depending on $S$
\[\tilde{\cal{Z}}^{2,k}_{S}(x_-, \bf{x}_+)\]
to be the set of tuples
\[(u, r, \varphi)\]
where:
\begin{itemize}\item $u : Z \to \tilde{M}$ is a continuous map with boundaries on $\tilde{L}_0, \tilde{L}_1$, such that $u(Z_S^+) \subset N$, and so we can write $v(s,t) = \pi_M u(s,t)$ and $\phi(s,t) = \pi_E u(s,t)$ for $s \ge S$;
\item $r(s)$ is a real function defined for $s \ge S$;
\item $\varphi \in C^0(Z_S^+, v^*E)$ is a section of $v^*E$, with boundary conditions on $F_0, F_1$, and such that $||\varphi||_H = 1$ for all $s \ge S$.\end{itemize}
These must moreover satisfy the conditions:
\begin{itemize}
\item $u$ and $\varphi$ are of class $L^{2,k}_{loc}$, and moreover $r \in L^{2,k}([S, \infty), \bb{R})$;
\item for $s \ll 0$, we have $u(s,t) = \exp_{x_-}(\xi_-(s,t))$ where $\xi_-(s,t) \in T_{x_-}\tilde{M}$ is of class $L^{2,k}$;
\item for $s \gg 0$, we have $v(s,t) = \exp_{x_+}(\xi_+(s,t))$ where $\xi_+(s,t) \in T_{x_+}M$ is of class $L^{2,k}$;
\item for $s \gg 0$ so that $v$ is contained in an open neighbourhood of $x_+ \in M$ of Assumption \ref{localgeomI} on which $E$ is trivialized, the difference $\varphi(s,t) - \psi_+(t)$ is of class $L^{2,k}$;
\item $u$ and $r, \varphi$ match over $Z_S^+$:
\[\phi(s,t) = \pi_E u(s,t) = r(s)\varphi(s,t)\]
and in particular as long as $u(s, [0,1])$ is not contained in $M$, $u(s,t)$ determines $r(s), \varphi(s,t)$ via $r(s) = ||\phi||_H$ and $\varphi = \phi/r$;
\item $u(S, [0,1])$ is not entirely contained in $M$.\end{itemize}

The underlying reason for this at first sight perhaps convoluted definition is that we need to include in our Banach manifold those $u(s,t)$ on which $u(s,[0,1]) \subset M$ for arbitrarily large $s$: for such $s$, we must have $r(s) = 0$, and $u(s,t)$ \emph{does not determine} $\varphi(s,t)$ through the matching condition. Of course, this cannot happen for a $\tilde{J}_t$-holomorphic $u$; but we still need to consider such $(u, r, \varphi)$ in our Banach space: the tuples $(u, r, \varphi)$ for which $r(s) > 0$ everywhere do not form a Banach manifold in any meaningful sense. On the other hand, we will talk about the space
\[\cal{Z}^{2,k}_{S}(x_-, \bf{x}_+) \subset \tilde{\cal{Z}}^{2,k}_{S}(x_-, \bf{x}_+)\]
on which $r(s) \ge 0$ for all $s$ where defined, which will be a closed subset. We will usually abbreviate the entire tuple $(u, r, \varphi)$ as just $u$, since for $\tilde{J}_t$-holomorphic strips, the choice of $u$ does determine $r, \varphi$.

A consequence of the final condition is that any element $(u, r, \varphi)$ of $\tilde{\cal{Z}}^{2,k}_{S}(x_-, \bf{x}_+)$ also determines an element of $\tilde{\cal{Z}}^{2,k}_{S'}$ and for any $S'$ with $|S' - S|$ sufficiently small: this is since there must exist $\eps > 0$ so that for all $S - \eps < S' < S + \eps$, we have that $u(S', [0,1])$ is contained in $N$ but not contained in $M$.

Indeed, once we have defined the Banach manifold structure, we shall see that small neighbourhoods of $(u,r,\varphi)$ in $\tilde{\cal{Z}}^{2,k}_{S}$ and $\tilde{\cal{Z}}^{2,k}_{S'}$ are diffeomorphic; in particular we obtain a canonical Banach manifold structure  (independent of $S$) in a neighbourhood of the moduli space of solutions. However, we will see that the appearance of $S < s < S'$ such that $u(s,[0,1]) \subset M$ means the two Banach manifold structures from $\tilde{\cal{Z}}^{2,k}_{S}$ and $\tilde{\cal{Z}}^{2,k}_{S'}$ around $(u, r, \varphi)$ are no longer compatible; this underlies our failure to describe a global Banach manifold structure as we did for the $\sigma$ model.

Let us now define the Banach manifold structure on $\tilde{\cal{Z}}^{2,k}_{S}(x_-, \bf{x}_+)$. At each $(u, r, \varphi)$, we define the tangent space
\[T_u \tilde{\cal{Z}}^{2,k}_{S}\]
(where we are abbreviating $u$ for $(u,r,\varphi)$) to be the tuples $(\xi, \rho, \vartheta)$ where
\begin{itemize}\item $\xi \in W^{2,k}(Z, u^*T\tilde{M})$ is a class $L^{2,k}$ section of $u^* T \tilde{M}$ along $Z$, with boundary conditions on $F_0, F_1$; in particular over $Z_S^+$ we can write $\xi = (\xi_M, \xi_E)$ where $\xi_M \in W^{2,k}(Z_S^+, v^*TM)$ and $\xi_E \in W^{2,k}(Z_S^+, v^*E)$ after using the connection $\nabla$ to produce a (time-dependent) splitting $TD(E) = TM \oplus E$;
\item $\rho \in L^{2,k}([S,\infty), \bb{R})$;
\item $\vartheta \in T_{\varphi}\cal{S}^{2,k}(Z_S^+, v^* E)$ the space of $L^{2,k}$ sections of $v^*E$ over $Z_S^+$, with boundary conditions on $F_0, F_1$, such that for $s \ge S$ we have
\[\langle \varphi, \vartheta \rangle_H = 0;\]
\item we have the matching condition
\[\xi_E(s,t) = \rho(s,t) \varphi(s,t) + r(s) \vartheta(s,t),\]
in particular we see that so long as $u(s,[0,1])$ is not contained in $M$, then $\xi$ determines $\rho$ and $\vartheta$ by $\rho = \langle \varphi, \xi_E\rangle$ and $\vartheta = \frac{1}{r}(\xi_E - \rho \varphi)$.\end{itemize}

This forms a vector space. In particular, the matching condition ensures that for $S < S'$ such that there is no $s \in (S,S')$ with $u(s, [0,1]) \subset M$, the two vector spaces $T_u \tilde{\cal{Z}}^{2,k}_{S}, T_u \tilde{\cal{Z}}^{2,k}_{S'}$ are identified.

The norm on $T_u \tilde{\cal{Z}}^{2,k}_{S}$ is constructed as follows: writing $\beta_+(s) : \bb{R} \to [0,1]$ for an increasing smooth cut-off function with $\beta_+ = 0$ on $s \le 0$ and $\beta_+ = 1$ on $s \ge 1$, we first define for $\xi \in W^{2,k}(u^*T\tilde{M})$ a similar rescaling as before, depending on $S$:
\[R_S(\xi(s,t)) = \begin{cases}\xi(s,t) & \text{ if } s < S; \\
(\xi_M, (1 - \beta_+(s-S)) \xi_E) \in v^*TM \oplus v^*E \cong u^*T \tilde{M} & \text{ if } s \ge S.\end{cases}\]
We then define the norm as
\begin{equation}\label{tau_2k_norm}||(\xi, \rho, \vartheta)||_{Z^{2,k}_{S}} = ||R_S(\xi)||_{2,k} + ||\beta_+(s - S) \rho||_{2,k} + ||\beta_+(s-S) \vartheta||_{2,k}.\end{equation}

This norm is complete, and moreover the exponential map on $\tilde{M}$ as well as the connection on $E$ induces a map from a sufficiently small open neighbourhood of zero to $\tilde{\cal{Z}}^{2,k}_{S}$. We leave it to the reader to check this defines a Banach manifold. The key point is from the definition of $S$, for $s$ sufficiently close to $S$, $u(s,[0,1])$ is both contained in $N$ and not wholly contained in $M$: the same holds for any image of a small enough $\xi$ under the exponential map. We also leave it to the reader to check that for such $S'$ sufficiently close to $S$, the two norms $|| \cdot ||_{Z^{2,k}_S}$ and $|| \cdot ||_{Z^{2,k}_{S'}}$ are equivalent, and the Banach manifolds locally diffeomorphic (indeed, this is true whenever $S < S'$ and there is no $S < s < S'$ such that $u(s, [0,1]) \subset M$).

Likewise, there is a Banach bundle $\cal{V}^{2,k-1}_{S}$ over $\tilde{\cal{Z}}^{2,k}_{S}(x_-, \bf{x}_+)$ whose fibre over $(u, r, \varphi)$ is the set of tuples $(\zeta, \mu, \sigma)$ where $\zeta \in L^{2,k-1}$, $\mu \in L^{2,k-1}([S, \infty), \bb{R})$ and $\sigma$ is an element of
\[V^{2,k-1}_{(u, \varphi)}(Z_S^+, v^*E) = \{\sigma \in L^{2,k-1}(Z_S^+, v^*E): \langle \varphi, \sigma \rangle_H = 0\}\]
subject to a similar matching condition
\[\pi_E \zeta(s,t) = \mu(s) \varphi(s,t) + r(s) \sigma(s,t)\]
over $Z_S^+$. The norm on this space is given by the same construction:
\begin{equation}||(\zeta, \mu, \sigma)||_{V^{2,k-1}_S} = ||R_S(\zeta)||_{2,k-1} + ||\beta_+(s-S)\mu||_{2,k-1} + ||\beta_+(s-S)\sigma||_{2,k-1}.\end{equation}

The $\bar{\partial}$-operator can then be recast in this set-up as a section
\begin{align}\label{F_section}\cal{F}: \tilde{\cal{Z}}^{2,k}_{S}(x_-, \bf{x}_+) &\to \cal{V}^{2,k-1}_{S}\\
(u, r, \varphi) &\mapsto \left(\bar{\partial}_{\tilde{J}_t}u, \frac{dr}{ds} + \Lambda(\varphi)r, \bar{\nabla}_I \varphi - \Lambda(\varphi)\varphi \right) \end{align}
of the vector bundle $\cal{V}^{2,k-1}_{S}$. Its zeroes within the closed subspace
\[\cal{Z}^{2,k}_{S}(x_-, \bf{x}_+) \subset \tilde{\cal{Z}}^{2,k}_{S}(x_-, \bf{x}_+)\]
where $r(s) \ge 0$ are then exactly the solutions to the $\bar{\partial}_{\tilde{J}_t}$ equations with limits at $x_-$ and $x_+$, such that $u(Z_S^+) \subset N$, with asymptotics of type $\lambda_+$ at $x_+$, and with $\varphi(s,t) \to \psi_+$ rather than $-\psi_+$. 

Moreover, the linearization $(D\cal{F})_{u}$ at a solution $u = (u, r, \varphi)$ is given by:
\begin{equation}\label{DF_formula}D\cal{F}_{u}(\xi, \rho, \vartheta) = \begin{pmatrix}D_u \xi \\
\frac{d\rho}{ds} + \Lambda(\varphi)\rho + r(s) \langle \varphi, \bar{\nabla}_I \vartheta \rangle_H + r(s) C^{\tau}_{(v, \varphi)}\pi_{TM}\xi \\
B^{\tau}_{(v, \varphi)} \pi_{TM} \xi + \Pi_V \bar{\nabla}_I \vartheta - \Lambda(\varphi)\vartheta)\end{pmatrix}\end{equation}
where $D_u$ is the linearization of the $\bar{\partial}$-operator, $v$ is the projection of $u$ to $M$ using $\beta_N$; $B^{\tau}_{(v, \varphi)}$ is given by the same formula as \eqref{Btauoperator}; and $C^{\tau}_{(v, \varphi)} : W^{2,k}(v^*TM) \to L^{2,k-1}([S,\infty), \bb{R})$ is given by
\begin{equation}\label{Ctauoperator}C^{\tau}_{(v, \varphi)} \xi = \langle \varphi, (\nabla_{\xi} I) \nabla_t \varphi\rangle_H + \langle \varphi, I F_{\nabla}(\xi, \partial_t v) \varphi\rangle_H.\end{equation}

\begin{proposition}The linearization $(D \cal{F})_u$ at a solution $(u, r, \varphi)$ is a Fredholm operators of the index
\[\ind((D \cal{F})_u) = \mu(u) - i(\lambda_+).\]\end{proposition}
\begin{proof}That the linearized operator is Fredholm follows from the proof of Proposition \ref{tau_index}. Moreover, by an excision argument comparing these operators to the linearized operators used in the $\tau$-model for polarization-twisted Floer theory earlier, it suffices to compute the index for the case when $\lambda_+ = \lambda_0$ is the smallest positive eigenvalue.

To do this, we will compare $D \cal{F}$ to the linearized $\bar{\partial}$-operator in the $\sigma$-model acting on the \emph{unweighted spaces}: there is a commutative diagram
\begin{equation}\begin{tikzcd}
W^{2,k}(u^* T\tilde{M}) \arrow[r, "D(\bar{\partial})"] \arrow[d] & L^{2,k-1}(u^* T\tilde{M}) \arrow[d]\\
T_u \tilde{\cal{Z}}^{2,k}_{S} \arrow[r, "D\cal{F}"] & V_S^{2,k-1}\end{tikzcd}\end{equation}
where the vertical arrows are the natural maps given by sending $\xi \mapsto (\xi, \rho, \vartheta)$ where for $\rho = \langle \varphi, \xi_E \rangle_H$ and $\vartheta = \frac{1}{r}(\xi_E - \rho \varphi)$ for $s \ge S$. However in this case, the vertical arrows are in fact isomorphisms of Banach spaces, from which we deduce the result.\end{proof}

We then say that a solution $(u, r, \varphi)$ is regular if $(D \cal{F})_u$ is surjective. Thankfully, this notion is again equivalent to regularity in the $\sigma$-model, utilizing Sobolev spaces with large exponential weights:

\begin{proposition}$D\cal{F}$ is surjective at a solution $(u, r, \varphi)$ if and only if the corresponding $\sigma$-model solution $(u, \phi)$ is regular for $\tilde{J}_t$, in the sense that the operator
\[D(\bar{\partial}_{\tilde{J}_t}): W^{2,k}_{\lambda_+} \to L^{2,k-1}_{\lambda_+}\]
on the Sobolev spaces with weights $e^{(\lambda_+ - \delta)s}$ at the positive end of the strip is surjective.\end{proposition}

\begin{proof}Since $D(\bar{\partial}_{\tilde{J}_t})$ and $D\cal{F}$ have the same index $\mu(u) + i(\lambda_+)$, it suffices to show their kernels are isomorphic. There certainly an injection
\[\ker(D(\bar{\partial}_{\tilde{J}_t}) \hookrightarrow \ker(D\cal{F})\]
given by
\[\xi \mapsto (\xi, \rho, \vartheta), \quad \rho = \langle \varphi, \xi_E \rangle_H, \quad \vartheta = \frac{1}{r}(\xi_E - \rho \varphi).\]
Conversely, if $(\xi, \rho, \vartheta) \in \ker(D\cal{F})$, we then have
\[\frac{d}{ds} \rho + \Lambda(\varphi)\rho + r(s)\eta(s,t) = 0\]
for $s \ge S$, where $\eta(s,t) = \langle \varphi, \bar{\nabla}_I \vartheta \rangle_H + \langle \varphi, I_t F_{\nabla}(\pi_{TM} \xi, \partial_t v) \varphi \rangle_H$. Now, since $(u, r, \varphi)$ is a solution, we must have $r(s) = r_0 e^{- \lambda_+ s} + R(s,t)$ where the error $R(s,t)$ satisfies $e^{(\lambda_+ + \delta)s}R(s,t) \to 0$; moreover we certainly have $\eta(s,t) \in L^{2,k}([S,\infty), \bb{R})$. We deduce that since $\Lambda(\varphi) \to \lambda_+$ (also with exponential speed), we must actually have
\[e^{(\lambda_+ - \delta)s}\rho \in L^{2,k}([S, \infty), \bb{R}).\]
Thus we see that since $\xi_E = \rho \varphi + r \vartheta$, we must have $\xi \in W^{2,k}_{\lambda_+}(u)$ and is thus an element of $\ker(D(\bar{\partial}_{\tilde{J}_t}))$, which completes the proof.\end{proof}

This set-up is entirely analogous in the case that $\bf{x}_- = (x_-, \lambda_-) \in \frak{C}_u$ and $\bf{x}_+ \in \frak{C}_o$, but just with the roles of the positive and negative end of the strip reversed. Writing $Z_{S}^-$ for the semi-infinite strip $(-\infty, S] \times [0,1]$, we obtain a Banach manifold of $\tilde{\cal{Z}}^{2,k}_{S}(\bf{x}_-, x_+)$ of tuples $(u, r, \varphi)$ where $u : Z \to \tilde{M}$ is a strip with boundaries on $\tilde{L}_0, \tilde{L}_1$ and limits at $x_-, x_+$, such that $u(Z_S^-) \subset N$ and on $Z_S^-$ we have $\pi_E u = r(s) \varphi(s,t)$, where $||\varphi||_H = 1$ and $\varphi \to \psi_-$ as $s \to \infty$. The moduli space of parametrized solutions is then cut out by a Fredholm section $\cal{F}$, whose linearization $D\cal{F}$ has index
\[\mu(u) + i(\lambda_-) + 1\]
and is surjective if the corresponding solution in the $\sigma$ model is itself regular.

The case that $\bf{x}_- = (x_-, \lambda_-) \in \frak{C}_u$ and $\bf{x}_+ = (x_+, \lambda_+) \in \frak{C}_s$ is more interesting. Here we must consider both semi-infinite ends of the strip $Z_S^-$ and $Z_S^+$ (we could also use different values of $S$ for the positive and negative end, but for simplicity let us not), and we define a similar Banach manifold $\tilde{\cal{Z}}^{2,k}_S(\bf{x}_-, \bf{x}_+)$ of tuples $(u, r, \varphi)$ where we now impose that both $u(Z_S^-) \subset N$ and $u(Z_S^+) \subset N$, with $r, \varphi$ being defined on $Z_S^- \sqcup Z_S^+$ and satisfying the same matching condition. We similarly obtain a Fredholm map $\cal{F}$, cutting out the moduli space of parametrized solutions, whose linearization has index
\[\mu(u) + i(\lambda_-) - i(\lambda_+) + 1\]
and is in particular surjective if and only if the corresponding $\sigma$ model solution is regular.

Consider now $(v, \varphi)$ a solution of the twisted equations in the $\tau$-model, with asymptotics to $\bf{x}_-$ and $\bf{x}_+$; recall that in the context of equivariant Floer theory, we case these boundary solutions. Now, fix a positive solution $r_0 : \bb{R} \to [0,\infty)$ of
\[\frac{d}{ds} r_0 + \Lambda(\varphi)r_0 = 0.\]
Since $\lambda_- < 0$ and $\lambda_+ > 0$, we must have $r_0(s) \to \infty$ exponentially fast as $s \to \pm \infty$. In particular, for sufficiently small $\eps > 0$, we have $\eps r_0(s)\varphi(s,t) \in D(E)$ for each $(s,t) \in Z$, and then in particular $u = (v, \eps r_0, \varphi)$ defines an element of the moduli space of ``interior'' trajectories $\cal{M}^o(\bf{x}_-, \bf{x}_+)$. It is then not difficult to check that for any compact subset $K \subset \cal{M}^{\partial}(\bf{x}_-, \bf{x}_+)$ of the moduli space of boundary trajectories, this assignment defines a smooth open embedding
\[G: K \times (0, \eps_0) \to \cal{M}^o(\bf{x}_-, \bf{x}_+).\]
In the next section, we will define what it means for a sequence of interior trajectories to converge to a boundary trajectory. A direct consequence of our definition will be that
\[\lim\limits_{\eps \to 0} G((v, \varphi), \eps) = (v, \varphi);\]
in particular we see that
\begin{proposition}$\cal{M}^o(\bf{x}_-, \bf{x}_+)$ is naturally the interior of a manifold with boundary, where the boundary is exactly $\cal{M}^{\partial}(\bf{x}_-, \bf{x}_+)$.\end{proposition}

\subsection{Compactness.} We turn our discussion to the compactification of our moduli spaces $\cal{M}(\bf{x}_-, \bf{x}_+)$ of non-invariant flows with prescribed asymptotics; where we must account for a sequence of non-invariant pseudoholomorphic strips $u_{\alpha}$ having as a limit a solution to the \emph{twisted} equations.

As before, an important tool will be control on the function $\Lambda(\varphi)(s)$ for a solution $\varphi$ of the ($\tau$-model) twisted equations. This is a priori not defined for a non-invariant solution $u$ of the $\tilde{J}_t$-holomorphic curve equation. However, fixing a tubular neighbourhood $N \cong D(E)$ as in Assumption \ref{localgeomII}, if for some $s \in \bb{R}$ we have $u(s, t) \in N$ for all $t \in [0,1]$, so that $\beta_N u(s,t) = (v(s,t), r(s)\varphi(s,t))$ a solution of the twisted equations, we can define
\begin{equation*}\Lambda(u)(s) = \Lambda(\varphi) = \langle \varphi, I_t \nabla_t \varphi \rangle_H\end{equation*}
at this $s$. Let $S(u) = \{s \in \bb{R}: u(s,[0,1]) \subset N\}$ be the subset of $\bb{R}$ on which $\Lambda$ is defined.

\begin{proposition}Suppose that $\tilde{J}_t$ is equivariantly regular. Then, there is a uniform bound
\begin{equation}|\Lambda(u)(s)| \le \Lambda_0\end{equation}
for all finite energy, non-invariant $\tilde{J}_t$-holomorphic maps $u : Z \to \tilde{M}$ and $s \in S(u)$.\end{proposition}

\begin{proof}Any finite energy $u : Z \to \tilde{M}$ must have positive and negative limits at some $x_-, x_+ \in \tilde{L}_0 \cap \tilde{L}_1$. Furthermore, in the event that say $x_+ \in L_0 \cap L_1$ is invariant, then for $s \gg 0$ we must have $\beta_N u = (v, r \varphi)$ where $\varphi \to \psi_+(t)$ as $s \to \infty$, for some $\psi_+ > 0$ an eigenvector of $I_{x_+} \frac{d}{dt}$ with eigenvalue $\lambda_+$. Observe that by the regularity of $\tilde{J}_t$, there is actually an a priori upper bound on $\lambda_+$: the dimension of the space of flows with these asymptotics is $\mu(u) - i(\lambda_+) - 1$, which must be non-negative. Moreover, the exactness assumption together with the fact that $\tilde{L}_0 \cap \tilde{L}_1$ is finite means that there is an a priori upper bound on the Maslov index $\mu(u)$ for pseudoholomorphic strips $u$; there is consequently there is an priori upper bound on $\lambda_+ > 0$. Likewise, there is an a priori bound on the possible eigenvalues $\lambda_- < 0$ which can appear as asymptotics of a strip $u(s,t)$ with $x_-$ being invariant.

Now, suppose for a contradiction there were a sequence $u_{\alpha} : Z \to \tilde{M}$ of finite energy $\tilde{J}_t$-holomorphic strips, and $s_{\alpha} \in S(u_{\alpha})$, such that $|\Lambda(u_{\alpha})(s_{\alpha})|$ increased without bound. By Gromov compactness, we can pass to a subsequence, such that $u_{\alpha}$ converges to a broken flow line $(u^1, u^2, \hdots, u^n)$ where $u^i \in \cal{M}(x^{i-1}, x^i)$ for $x^0 = x_-, x^1, \hdots, x^{n-1}, x^n = x_+ \in \tilde{L}_0 \cap \tilde{L}_1$. In particular, this means that there are sequences $\sigma^1_{\alpha} < \sigma^2_{\alpha} < \hdots < \sigma^n_{\alpha}$ such that each translate $\tau_{\sigma^i_{\alpha}}^* u_{\alpha} \to u^i$ locally, and so that for any sequence $\sigma'_{\alpha}$ with $\sigma^i_{\alpha} < \sigma'_{\alpha} < \sigma^{i+1}_{\alpha}$ and $\sigma^{i+1}_{\alpha} - \sigma'_{\alpha}, \sigma'_{\alpha} - \sigma^i_{\alpha} \to \infty$, we have $\tau^*_{\sigma'_{\alpha}} u_{\alpha} \to x^i$ locally.

After passing to a subsequence, we can presume either there exists an $i$ such that $\sigma^i_{\alpha} < s_{\alpha} < \sigma^{i+1}_{\alpha}$ for each $i$ and both sequences $s_{\alpha} - \sigma^i_{\alpha}, \sigma^{i+1}_{\alpha} - s_{\alpha}$ go to infinity; or that there is some $i$ so that $\sigma^i_{\alpha} - s_{\alpha}$ converges to a finite constant, say $s'$. In the former case, we must have $u_{\alpha}(s_{\alpha}, t) \to x^i$; by the assumption on $s_{\alpha}$, the $x^i$ must then be invariant. In the latter, we have $\tau^*_{s_{\alpha}} u_{\alpha} \to \tau^*_{s'} u^i$. The limit $u^i$ must be then invariant trajectory: if not, then since $u(s_{\alpha}, [0,1]) \subset N$, in the limit and after rescaling we must have $u^i(s', [0,1]) \subset N$, and then $\Lambda(u_{\alpha})(s_{\alpha}) \to \Lambda(u)(s')$, a contradiction of the assumption. Hence, we know $u^i$ is invariant in this case.

Now, pick a maximal consecutive string $j, \hdots, i, \hdots, k$ such that each $u^{j+1}, u^{j+2}, \hdots, u^k$ are invariant (this string may possibly be empty). There are then four cases to consider.

\emph{Case I. Suppose that $j = 1$ and $k = n$}: this means that all of the $u^i$ were contained in $M$. In particular, each $u_{\alpha}$ must be a non-invariant trajectory between invariant $x_-, x_+$; each of these must then be a trajectory between $\bf{x}_- = (x_-, \lambda_-)$ and $\bf{x}_+ = (x_+, \lambda_+)$ for some $\lambda_-, \lambda_+$ which are a priori bounded by the earlier observation. Moreover, for large enough $\alpha$, $u_{\alpha}$ is entirely contained within $N$; in particular it must be of the form $\beta_N^{-1}(v_{\alpha}, r_{\alpha}\varphi_{\alpha})$ for some solution of the twisted equations with limits at $\bf{x}_-, \bf{x}_+$. However we already proved in Proposition \ref{Lambdabound} that for such solutions, $\Lambda(\varphi_{\alpha})(s)$ is uniformly bounded; this contradicts our assumption.

\emph{Case II. Suppose that $j = 0$, but that $k < n$}: this means that $u^1, \hdots u^k$ are all invariant, but that $u^{k+1}$ is a non-invariant trajectory. Since each $u_{\alpha}$ is a non-invariant strip, with negative limit at an invariant $x^0 = x_-$, as above it must also have asymptotics of type $\lambda_-$ at $x_-$, for some $\lambda_-$ which is a priori bounded. There is also some $\lambda^k < 0$ such that $u^{k+1}$ has a negative limit at $x^k$ with asymptotics of type $\lambda^k$; consequently $\Lambda(u^{k+1})(s) \to \lambda^k$ as $s \to -\infty$. Moreover, there is an a priori bound on $\lambda^k$ from the above observation. In particular, we can find a sequence $s'_{\alpha}$ such that $\Lambda(u_{\alpha})(s'_{\alpha}) \to \lambda^k + \delta$ for some small $\delta$, with $\sigma^{k+1}_{\alpha} - s'_{\alpha} \to \infty$, and so that $s_{\alpha} < s'_{\alpha}$ for each $\alpha$. In particular, since $u^1, \hdots, u^k$ are all invariant, for large enough $\alpha$, we must have
\begin{equation*}u_{\alpha}((-\infty, s'_{\alpha}] \times [0,1]) \subset N.\end{equation*}
$u_{\alpha}$ can then be thought of as a twisted trajectory on $(-\infty, s'_{\alpha}]\times[0,1]$, with $\Lambda(u_{\alpha})(s'_{\alpha}) = \lambda^k + \delta$, and with $\Lambda(u_{\alpha})(s) \to \lambda_-$ as $s \to -\infty$. From here, the proof of Proposition \ref{Lambdabound} produces a uniform bound on $\Lambda$ for such twisted flows, giving a contradiction.

\emph{Case III. Suppose that $j > 1$, but that $k = n$}. This is entirely analogous to the previous case, but with positive and negative directions reversed.

\emph{Case IV. Suppose that $j > 1$ and that $k < n$}. In this case, we have non-invariant trajectories $u^{j-1}$ and $u^{k+1}$, and invariant $u^j, \hdots, u^k$. Again, $u^{j-1}$ must have invariant limit $x^{j-1}$ at $+\infty$ with asymptotics of type $\lambda^j$; whilst $u^{k+1}$ must have an invariant limit $x^k$ at $-\infty$ with asymptotics type $\lambda^k$, where $\lambda^j, \lambda^k$ are a priori bounded. We can then as above find sequences $s'_{\alpha}, s''_{\alpha}$ with $s'_{\alpha} < s_{\alpha} < s''_{\alpha}$ and $\Lambda(u_{\alpha})(s'_{\alpha}) \to \lambda^j + \delta'$, $\Lambda(u_{\alpha})(s''_{\alpha}) \to \lambda^k + \delta''$, for some small $\delta', \delta''$, and so that $u_{\alpha}([s'_{\alpha}, s''_{\alpha}]\times [0,1]) \subset N$. Again, the proof of Proposition \ref{Lambdabound} allows us to conclude that $\Lambda(u_{\alpha})$ is bounded on $[s'_{\alpha}, s''_{\alpha}]$, a contradiction.\end{proof}

\begin{corollary}\label{localconvergenceII}Suppose that $u_{\alpha}$ is some sequence of finite energy $\tilde{J}_t$-holomorphic curves, with a limit in the local topology $u_{\alpha} \to u$ such that $u$ is invariant. For each $R > 0$, and for sufficiently large $\alpha$, we can then write $u_{\alpha} = \beta_N^{-1}(v_{\alpha}, r_{\alpha} \varphi_{\alpha})$, where $v_{\alpha}, \varphi_{\alpha}$ is a twisted trajectory on $Z_R$. Then, there exists a twisted trajectory $(u, \varphi)$, and nonzero real numbers such that $(v_{\alpha}, \varphi_{\alpha}) \to (u, \varphi)$ uniformly with all derivatives on each $Z_R$.\end{corollary}

\begin{proof}This directly follows from the local compactness result for solutions of the twisted equations, Proposition \ref{localconvergence}, together with the uniform bound on $\Lambda(u_{\alpha})$ as above.\end{proof}

To deduce a ``global'' convergence result for interior trajectories, we will again use control over the total variation of $\Lambda_{u}(s)$. Even though this function is only defined on $S_u \subset \bb{R}$, we still define
\begin{equation}K(u) = \int_{S(u)} \left| \frac{d \Lambda(u)}{ds} \right|, \qquad K^+(u) = \int_{S(u)} \left( \frac{d \Lambda(u)}{ds} \right)^+.\end{equation}
We declare $K, K^+ = 0$ if $S(u)$ is empty. Likewise, for a sub-strip $Z' = [R_1,R_2] \subset Z$, we define $K_{Z'}(u), K_{Z'}^+(u)$ to be the same integral over $[R_1, R_2] \cap S(u)$, and is zero if $[R_1, R_2]$ and $S_u$ are disjoint. Similarly, $K, K^+$ are defined for broken trajectories by summing over the components, including where some of the components are boundary trajectories.

\begin{proposition}$K, K^+$ are uniformly bounded over all possibly broken trajectories between fixed $\bf{x}_-, \bf{x}_+$.\end{proposition}
\begin{proof}Such a trajectory has an a priori bounded number of components, and we already have the result for boundary trajectories from Proposition \ref{KboundI} so it suffices to prove it for a single, interior trajectory $u$. From the uniform bounds on $\Lambda$, it then suffices to prove it for just one of $K, K^+$: we prove it for $K^+$. Suppose there were a sequence $u_{\alpha}$ of interior trajectories such that $K^+(u_{\alpha})$ increases without bound. Pass to a subsequence with a downstairs broken limit $(u^1, u^2, \hdots, u^n)$ for $u$. We can then for each $\alpha$ find a partition of $Z$ as
\begin{equation*}Z = W^0_{\alpha} \cup Z^1_{\alpha} \cup W^1_{\alpha} \cup Z^2_{\alpha} \cup \hdots \cup Z^n_{\alpha} \cup W^n_{\alpha}\end{equation*}
such that $u(W^i_{\alpha})$ is contained in an arbitrarily small neighbourhood of some intersection point $x \in \tilde{L}_0 \cap \tilde{L}_1$, and so that the translations $\tau^*_{\sigma^i_{\alpha}} u_{\alpha}$ recentered on $Z^i_{\alpha}$ converge locally to $u^i$.

Solutions $(v, \phi)$ of the twisted equations in small neighbourhoods of $x \in L_0 \cap L_1$ have $\Lambda(\phi)$ non-increasing, by the proof of Proposition \ref{Lambdabound}, while non-invariant curves in small neighbourhoods of non-invariant $x \in \tilde{L}_0 \cap \tilde{L}_1$ have $K = K^+ = 0$ by definition. In particular, $K^+_{W^i_{\alpha}}(u_{\alpha}) = 0$ for each $i, \alpha$.

On the other hand, on each $Z^i_{\alpha}$, either the retranslated $u_{\alpha}$ have a limit to a non-invariant $u^i$, in which case $K^+_{Z^i_{\alpha}}(u_{\alpha}) \to K^+_{Z^i}(u_i)$, or they have a limit to an invariant $u^i$. In this second case, we can write $\beta_N u_{\alpha}|_{Z^i_{\alpha}} = (v_{\alpha}, r_{\alpha} \varphi_{\alpha})$ for sufficiently large $\alpha$, and there exists a twisted solution $(u^i, \varphi^i)$ such that $(v_{\alpha}, \varphi_{\alpha}) \to (u^i, \varphi^i)$, and so $K^+_{Z^i_{\alpha}}(u_{\alpha}) \to K^+_{Z^i}(u^i, \phi^i)$. In particular, we obtain uniform bounds on $K^+(u_{\alpha})$ over the whole of $Z$, a contradiction. \end{proof}

Observe that a blown-up solution is a constant solution if either: $u$ is an interior solution, with $u(s,t) = x \in \tilde{L}_0 \cap \tilde{L}_1$ for a non-invariant $x$, or if $(u, \varphi)$ is a twisted solution with $u(s,t) = x \in L_0 \cap L_1$ and $\varphi(s,t) = \psi(t)$ for some unit eigenvector $\psi(t)$ of $I \frac{d}{dt}$. Alternatively, in the $\sigma$ model, this means $\phi(s,t) = C e^{- \lambda s}\psi(t)$ where $\lambda$ is the eigenvalue of $\psi$.

The following lemma, an analogue of Lemma \ref{smallenergylemmaI}, allows us to control local convergence to constant solutions.

\begin{lemma}\label{smallenergylemmaII}Fix arbitrary compact sub-strips $Z_1, Z_2 \subset Z$, and constants $\eps > 0$, $\cal{E}_0 > 0$. Then there exists $\delta_1, \delta_2$ such that if $u : Z \to M$ is a non-invariant $\tilde{J}_t$-holomorphic strip  which satisfies
\begin{equation}\cal{E}(u) < \cal{E}_0, \qquad \cal{E}_{Z_2}(u) \le \delta_1, \qquad K_{Z_2}(u) \le \delta_2\end{equation}
then one of the following holds:
\begin{itemize}\item either $u|_{Z_1}$ is in an $\eps$-neighbourhood of a \emph{non-invariant constant solution};
\item or $u|_{Z_1}$ is entirely contained in the tubular neighbourhood $N$ of $M$, and after writing $\beta_N u|_{Z_2} = (v, r\varphi)$, the twisted solution $(v, \varphi)$ is $\eps$-close to a constant twisted solution in the $L^{2,k}$ norm.\end{itemize}\end{lemma}

\begin{proof}Again the proof is similar to that of Lemma \ref{smallenergylemmaI}. Suppose $u_{\alpha}$ were a sequence of non-invariant solutions with $\cal{E}_{Z_2}, K_{Z_2} \to 0$, of total energy bounded by $\cal{E}_0$, such that $u_{\alpha}|_{Z_1}$ is at least $\eps$-distance from any constant blown-up solution, in either of the two senses described above. By the ``downstairs'' local convergence, there is some subsequence that converges locally to a $\tilde{J}_t$-holomorphic strip $u$. This strip must have zero energy on $Z_2$, and hence by unique continuation we must have $u(s,t) = x$ for some $x \in \tilde{L}_0 \cap \tilde{L}_1$. If $x$ is non-invariant, then for sufficiently large $\alpha$, $u_{\alpha}|_{Z_1}$ is within $\eps$ of the constant solution at $x$, a contradiction.

If $x$ is invariant, then we are in the setting of Corollary \ref{localconvergenceII}. In particular, fixing some $Z_R$ containing both $Z_1, Z_2$, for sufficiently large $\alpha$ we have $u_{\alpha}(Z_R) \subset N$, and after writing $\beta_N u_{\alpha} = (v_{\alpha}, r_{\alpha} \varphi_{\alpha})$, there is a twisted solution $(x, \varphi)$ with $(v_{\alpha}, \varphi_{\alpha}) \to (x, \varphi)$. By assumption, we must have $K_{Z_2}(x, \varphi) = 0$, so in particular $\varphi$ must be a constant solution. In particular, for sufficiently large $\alpha$, the second of the two options holds, a contradiction.\end{proof}

This immediately allows us to deduce the main result of this section, by following exactly the same proof as in Section 2.3:

\begin{theorem}Let $u_{\alpha}$ be a sequence of interior trajectories from $\bf{x}_-$ to $\bf{x}_+$ with uniformly bounded energy. Then some subsequence Gromov converges to a broken, blown-up trajectory $(\bf{u}^1, \hdots, \bf{u}^n)$, where $\bf{u}^i \in \cal{M}(\bf{x}^i, \bf{x}^{i+1})$ is either an interior trajectory $u^i$, or a boundary twisted trajectory $(u^i, \varphi^i)$.\end{theorem}

Explicitly, Gromov convergence here means there are real numbers $\sigma_{\alpha}^1 < \hdots < \sigma_{\alpha}^n$ such that each translate $\tau_{-\sigma_{\alpha}^i}^*u_{\alpha}$ converges locally to $\bf{u}^i$, and for any sequences of real numbers $\rho_{\alpha}$ with $\sigma^i_{\alpha} - \rho_{\alpha} \to -\infty$ and $\sigma^{i+1}_{\alpha} - \rho_{\alpha} \to +\infty$ for some $i = 0, \hdots, n$ (where by convention $\sigma^0 = -\infty, \sigma^{n+1} = +\infty$), the translates $\tau_{-\rho_{\alpha}}^*u_{\alpha}$ converge locally to the constant solution at $\bf{x}^i$.

\begin{corollary}The moduli space of blown-up trajectories
\[\bar{\cal{M}}(\bf{x}_-, \bf{x}_+) = \bigcup\limits_{n \ge 0} \bigcup\limits_{\bf{x}^1, \hdots, \bf{x}^{n-1}} \cal{M}(\bf{x}_-, \bf{x}^1) \times \hdots \times \cal{M}(\bf{x}^{n-1}, \bf{x}_+)\]
with the topology of Gromov convergence is compact.\end{corollary}

In particular, since for a holomorphic strip $u : Z \to M$ contained entirely within the invariant set, the difference between the Maslov index as computed in $M$ and $\tilde{M}$ is precisely the spectral flow $\specflow(u)$ over $u$:
\[\mu_{\tilde{M}}(u) = \mu_M(u) + \specflow(u).\]
In particular, we see that if $u_{\alpha}$ is a sequence of non-invariant trajectories converging to a broken blown-up trajectory $\bf{u}^1, \hdots, \bf{u}^n$, we have for sufficiently large $\alpha$
\[\mu_{\tilde{M}}(u_{\alpha}) = \sum\limits_{u^i \text{ non-invariant}} \mu_{\tilde{M}}(u^i) + \sum\limits_{u^i \text{ invariant}}\left(\mu_M(u^i) + \specflow(u^i)\right).\]
Recalling that the definition of the index in $\tilde{M}$ of a twisted trajectory $(u, \varphi)$, this implies that the total $\tilde{M}$-index (meaning the sum of $\mu_{\tilde{M}}$ and $\ind_{\tilde{M}}$ over the non-invariant and invariant components respectively) is preserved by Gromov convergence.

\subsection{The compactified moduli space.}\label{stratalist}

We now prove gluing theorems for the moduli space of blown up trajectories, which allow us to prove that indeed each $\hat{C}, \check{C}, \bar{C}$ are indeed complexes, and the natural maps $i^*, j^*, \partial$ between them are chain maps.

To do this, consider a component of the compactified moduli space $\bar{\cal{M}}(\bf{x}_-, \bf{x}_+)$ of \emph{interior trajectories} which is of dimension one. We have shown so far that this is a compact topological space, stratified by one- and zero-manifolds; there are just two strata, with the one-manifolds forming the open stratum and the zero-manifolds forming the closed stratum. We will often refer to the zero-manifold strata as the \emph{boundary points} of the moduli space.

Let us now list such boundary points; there are four cases to consider.

\emph{Case I. If $\bf{x}_-, \bf{x}_+ \in \frak{C}_o$ are both interior critical points}, then we expect the boundary points to be of two forms:
\begin{enumerate}\item two component broken trajectories $(u^1, u^2) \in \cal{M}(\bf{x}_-, \bf{x}^1) \times \cal{M}(\bf{x}^1, \bf{x}_+)$ for $\bf{x}^1 \in \frak{C}_o$.
\item three component broken trajectories $(u^1, (u^2, \varphi^2), u^3) \in \cal{M}(\bf{x}_-, \bf{x}^1) \times \cal{M}(\bf{x}^1, \bf{x}^2) \times \cal{M}(\bf{x}^2, \bf{x}_+)$ where $\bf{x}^1 \in \frak{C}_s$ and $\bf{x}^2 \in \frak{C}_u$. In particular, $(u^2, \varphi^2)$ is a boundary (twisted) trajectory of index zero in $\tilde{M}$: this is the \emph{boundary obstructed} case.\end{enumerate}

\emph{Case II. If $\bf{x}_- \in \frak{C}_o, \bf{x}_+ \in \frak{C}_s$ are interior and boundary-stable respectively}, then we expect the boundary points to be of three forms:
\begin{enumerate}\item two component broken trajectories $(u^1, u^2) \in \cal{M}(\bf{x}_-, \bf{x}^1) \times \cal{M}(\bf{x}^1, \bf{x}_+)$ for $\bf{x}^1 \in \frak{C}_o$. Here, both $u^1$, $u^2$ are interior trajectories.
\item two component broken trajectories $(u^1, (u^2, \varphi^2)) \in \cal{M}(\bf{x}_-, \bf{x}^1) \times \cal{M}(\bf{x}^1, \bf{x}_+)$ for $\bf{x}^1 \in \frak{C}_s$. Here, $u^1$ is an interior trajectory, whilst $(u^2, \varphi^2)$ is a boundary trajectory.
\item three component broken trajectories $(u^1, (u^2, \varphi^2), u^3) \in  \cal{M}(\bf{x}_-, \bf{x}^1) \times \cal{M}(\bf{x}^1, \bf{x}^2) \times \cal{M}^o(\bf{x}^2, \bf{x}_+)$ where $\bf{x}^1 \in \frak{C}_s$ and $\bf{x}^2 \in \frak{C}_u$. Here, $u^1$ and $u^3$ are interior trajectories, while $(u^2, \varphi^2)$ is a boundary obstructed trajectory.\end{enumerate}

\emph{Case III. If $\bf{x}_- \in \frak{C}_u, \bf{x}_+ \in \frak{C}_o$ are boundary-unstable and interior respectively}, then similar to the above, we expect the boundary points to be of three forms:
\begin{enumerate}\item two component broken trajectories $(u^1, u^2) \in \cal{M}(\bf{x}_-, \bf{x}^1) \times \cal{M}(\bf{x}^1, \bf{x}_+)$ for $\bf{x}^1 \in \frak{C}_o$. Here, both $u^1$, $u^2$ are interior trajectories.
\item two component broken trajectories $((u^1, \varphi^1), u^2) \in \cal{M}(\bf{x}_-, \bf{x}^1) \times \cal{M}(\bf{x}^1, \bf{x}_+)$ for $\bf{x}^1 \in \frak{C}_u$. Here, $(u^1, \varphi^1)$ is a boundary trajectory, whilst $u^2$ is an interior trajectory.
\item three component broken trajectories $(u^1, (u^2, \varphi^2), u^3) \in  \cal{M}^o(\bf{x}_-, \bf{x}^1) \times \cal{M}(\bf{x}^1, \bf{x}^2) \times \cal{M}(\bf{x}^2, \bf{x}_+)$ where $\bf{x}^1 \in \frak{C}_s$ and $\bf{x}^2 \in \frak{C}_u$. Here, $u^1$ and $u^3$ are interior trajectories, while $(u^2, \varphi^2)$ is a boundary obstructed trajectory.\end{enumerate}

\emph{Case IV. If $\bf{x}_- \in \frak{C}_u, \bf{x}_+ \in \frak{C}_s$ are boundary unstable and boundary stable respectively}, then we expect the boundary points of the compactified moduli space $\bar{\cal{M}}^o(\frak{x}_-, \frak{x}_+$ of interior trajectories to be of five forms:
\begin{enumerate}\item unbroken boundary (twisted) trajectories  $(u, \varphi) \in \cal{M}^{\partial}(\bf{x}_-, \bf{x}_+)$, which we have already shown form a natural boundary of $\bar{\cal{M}}^o$.
\item two component broken trajectories $(u^1, u^2) \in \cal{M}(\bf{x}_-, \bf{x}^1) \times \cal{M}(\bf{x}^1, \bf{x}_+)$ for $\bf{x}^1 \in \frak{C}_o$. Here, both $u^1, u^2$ are interior trajectories.
\item two component broken trajectories $((u^1, \varphi^1), u^2) \in \cal{M}^o(\bf{x}_-, \bf{x}^1) \times \cal{M}(\bf{x}^1, \bf{x}_+)$ for $\bf{x}^1 \in \frak{C}_u$. Here, $(u^1, \varphi^1)$ is a boundary trajectory, and $u^2$ is an interior trajectory.
\item two component broken trajectories $(u^1, (u^2, \varphi^2)) \in  \cal{M}(\bf{x}_-, \bf{x}^1) \times \cal{M}^o(\bf{x}^1, \bf{x}_+)$ for $\bf{x}^1 \in \frak{C}_s$.  Here, $u^1$ is an interior trajectory, and $(u^2, \varphi^2)$ is a boundary trajectory.
\item three component broken trajectories $(u^1, (u^2, \varphi^2), u^3) \in  \cal{M}^o(\bf{x}_-, \bf{x}^1) \times \cal{M}(\bf{x}^1, \bf{x}^2) \times \cal{M}^o(\bf{x}^2, \bf{x}_+)$ where $\bf{x}^1 \in \frak{C}_s$ and $\bf{x}^2 \in \frak{C}_u$. Here, $u^1$ and $u^3$ are interior trajectories, while $(u^2, \varphi^2)$ is a boundary obstructed trajectory.\end{enumerate}

We then group the gluing results needed to elucidate the local structure of the compactified moduli space for the above cases into three types:
\begin{itemize}\item \emph{interior gluing}, for boundary strata corresponding to two component broken interior trajectories: these are cases (I.1), (II.1), (III.1) and (IV.2).
\item \emph{boundary-unobstructed gluing}, for two-component broken trajectories with one interior and one boundary component: these are cases (II.2), (III.2), (IV.3) and (IV.4).
\item \emph{boundary-obstructed gluing}, for three-component broken trajectories where the middle component is a boundary-obstructed trajectory, and the other two are interior trajectories: these are cases (I.2), (II.3), (III.3) and (IV.5).\end{itemize}

In the cases of interior gluing and boundary-unobstructed gluing, as well as case (IV.1) where no gluing is required at all, we will see that the compactified moduli space is locally a manifold with boundary. This will not be the case in the boundary-obstructed setting. Instead, we have the following structure (compare the more general Definition 19.5.3 of \cite{KronheimerMrowka07}), which if the reader so wishes can be thought of as a type Kuranishi structure adapted to this particular setting:

\begin{definition}\label{delta_structure}We say that a topological space $N$ is a \emph{one-dimensional space stratified by manifolds} if it is equipped with a closed subset $M^0 \subset N$ which is a zero-manifold, so that the complement $M^1 = N \backslash M^0$ has the structure of a one-manifold. Moreover, we say that $N$ has a \emph{codimension one smooth $\delta$ structure} at a point $p \in M^0$ if there is an open neighbourhood $W \subset N$ of $p$ which is disjoint from $M^0 \backslash \{p\}$, a continuous map
\[j : W \to (\bb{R}_{\ge 0})^2\]
and a continuous function
\[\delta : (\bb{R}_{\ge 0})^2 \to \bb{R}\]
such that:
\begin{itemize} \item $\delta > 0$ on $\{0\} \times \bb{R}_{>0}$ and $\delta < 0$ on $\bb{R}_{>0} \times \{0\}$, with $\delta = 0$ at $(0,0)$;
\item on $(\bb{R}_{>0})^2$, the function $\delta$ is smooth and transverse to zero, so in particular $\delta^{-1}(0) \backslash \{(0,0\}$ is a smooth one-manifold;
\item $j(p) = (0, 0)$, and $j|_{W \backslash \{p\}}$ is a diffeomorphism of $W \backslash \{p\} \subset M^0$ onto $\delta^{-1}(0) \cap U$ for some open neighbourhood $U$ of $(0,0) \in (\bb{R}_{\ge 0})^2$.\end{itemize}\end{definition}

For example, the one-manifold with boundary $\bb{R}_{\ge 0}$ has a $\delta$-structure along its boundary $\{0\}$, via the function $\delta(x,y) = y - x$. A simple example that is not a one-manifold with boundary is given by the union of any \emph{odd} number of copies of $\bb{R}_{\ge 0}$, each glued together at the point $0$. However more pathological examples exist, such as the union of the ``Hawaiian ear-rings'' with a single ray (see Section 16.5 of \cite{KronheimerMrowka07}). It is possible that more careful gluing analysis might rule out such examples from appearing in the moduli space of trajectories for generic $\tilde{J}_t$; it is an interesting question precisely which $\delta$-structures appear for generic data.

Despite the potential for pathological behaviour, a codimension one smooth $\delta$-structure is enough to ensure that the usual counting arguments used in Floer theory carry through. Indeed, from the above discussion, a point $p \in M^0$ with a smooth $\delta$ structure ``contributes an odd number'' to the count of boundary points of $N$.

\begin{proposition}Suppose $N$ is a compact one-dimensional space stratified by manifolds, with a codimension one smooth $\delta$ structure at each point of the zero stratum. Then the number of points of the zero stratum is even.\end{proposition}

For the proof, we refer the reader to Section 21 of \cite{KronheimerMrowka07}, which proves a considerably more general statement with coefficient ring $\bb{Z}$; Corollary 21.3.2 is exactly the above but with signs.

Let us now get on with proving that the compactified moduli space $\bar{\cal{M}}(\bf{x}_-, \bf{x}_+)$ of interior trajectories has either the structure of a manifold with boundary, or a smooth codimension one $\delta$-structure, at each additional point of the compactification. For the sake of brevity, we will not discuss interior gluing at all: the argument is a very minor modification of the usual gluing argument in ordinary Lagrangian Floer theory, only different in that one of the Banach manifolds $\tilde{\cal{Z}^{2,k}_S}$ must be used. Moreover, we will discuss each of boundary-unobstructed and boundary-obstructed gluing in exactly one of the cases that they appear; the other cases are essentially the same. Thus in what follows, we will carefully explain the gluing process in cases (II.2) and (I.2).

\subsection{Gluing boundary-unobstructed trajectories.}

Let us now deal gluing in case (II.2) of Section \ref{stratalist}. Hence, let $\bf{x}^0 = (x^0, \iota(x^0)) \in \frak{C}_o$, and $\bf{x}^1 = (x^1, \lambda^1) \in \frak{C}_s, \bf{x}^2 = (x^2, \lambda^2) \in \frak{C}_s$; since we are only looking to write down boundary compactifications of the one-dimensional moduli spaces of flows, we are taking a zero-dimensional component of the moduli spaces $\cal{M}(\bf{x}^0, \bf{x}^1)$ and $\cal{M}(\bf{x}^1, \bf{x}^2)$. As always, fix parametrized representatives, say $u^1$ and $(u^2, \varphi^2)$ for the solutions in question, and as before we will work entirely in the $\tau$ model.

Pick unit eigensolutions $\psi^1, \psi^2$ corresponding to $\lambda^1, \lambda^2$ at $x^1, x^2$, and assume that $\varphi^1 \to \psi^1$ as $s \to \infty$, and $\varphi^2 \to \psi^1, \psi^2$ respectively as $s \to \pm \infty$. Let us also fix the parametrization on $u^1$ so that $u^1([0, \infty) \times [0,1]) \subset N \cong D(E)$ the tubular neighbourhood of $M$ given by Assumption \ref{localgeomII}.

We will construct for sufficiently large $T$ an element in the Banach manifold
\[(u_T, r_T, \varphi_T) \in \tilde{\cal{Z}}^{2,k}_{S}(x^0, \bf{x}^2)\]
which is an approximate solution to $\cal{F}(u, r, \varphi) = 0$, where in fact we will take $S = -T$. We will then use the Newton-Picard method to construct an actual solution.

Write $u^1 = (u^1, r^1, \varphi^1) \in \cal{Z}^{2,k}_{S = 0}$, meaning that for $s \ge 0$, we have $\pi_E u^1(s,t) = r^1(s) \varphi^1(s,t)$ where $r^1 > 0$ and $||\varphi^1||_H = 1$; let us also write $v^1(s,t) = \pi_M u^1(s,t)$ for $s \ge 0$. Choose $T$ large enough so that $v^1([T, \infty) \times [0,1])$ and $u^2((-\infty, -T] \times [0,1])$ are each contained in a small neighbourhood $U$ of $x^1$ on each $E$ is trivialized; so that we can write for $s \ge T$ and $s \le T$ respectively
\[\varphi^1(s,t) = \psi^1(t) + \eta^1(s,t), \qquad \varphi^2(s,t) = \psi^1(t) + \eta^2(s,t)\]
where $\eta^1, \eta^2 \to 0$ uniformly in $t$ and exponentially in all derivatives as $s \to \infty$.

Take a smooth, non-decreasing cut-off function $\beta_+ : \bb{R} \to [0,1]$ such that $\beta_+(s) = 0$ for $s \le -1$ and $\beta_+(s) = 1$ for $s \ge 1$, and write $\beta_-(s) = 1- \beta_+(s)$. We then construct a ``preglued'' solution by taking:
\begin{equation}u_T(s,t) = \begin{cases}u^1(s+2T, t) & \text{ for } s \le -T \\
\beta_N^{-1}(v_T(s,t), r_T(s) \varphi_T(s,t)) & \text{ for } s \ge -T\end{cases}\end{equation}
where $v_T : Z_{-T}^+ \to M$, $r_T : [-T, \infty) \to \bb{R}$ and $\varphi_T \in C^{\infty}(Z_{-T}^+, v_T^* E)$ are defined as follows:
\[v_T(s,t) = \begin{cases}v^1(s + 2T, t) & \text{ for } -T \le s \le -1 \\
\exp_{x^1}\left(\beta_-(s) \exp_{x^1}^{-1}(v^1(s + 2T, t)) + \beta_+(s) \exp_{x^1}^{-1}(u^2(s-T, t))\right) & \text{ for } -1 \le s \le 1 \\
u^2(s - T, t) & \text{ for } s \ge 1;\end{cases}\]
\[r_T(s) = \begin{cases}r^1(s+T) & \text{ for } -T \le s \le -1 \\
\beta_-(s) r^1(s +2T) & \text{ for } -1 \le s \le 1 \\
0 & \text{ for } s \ge 1;\end{cases}\]
\[\varphi_T(s,t) = \begin{cases}\varphi^1(s + 2T, t) & \text{ for } s \le -1 \\
\frac{\psi^1(t) + \beta_-(s) \eta^1(s+2T,t) + \beta_+(s)\eta^2(s-T,t)}{||\psi^1(t) + \beta_-(s) \eta^1(s+2T,t) + \beta_+(s)\eta^2(s-T,t)||_H} & \text{ for } -1 \le s \le 1\\
\varphi^2(s - T, t) & \text{ for } s \ge 1\end{cases}\]
where $\eta^1, \eta^2$ are the error terms above. This defines an element of $\tilde{\cal{Z}}^{2,k}_S(x^0, \bf{x}^2)$ for $S = -T$. Indeed, by construction, there is a constant $C$ so that $r_T(s) \le C e^{-\lambda^1 T}$ for all $s \ge -T$; the preglued strip $u_T$ is constructed to be exponentially close to the invariant set $M$ on $Z_S^+$, and thus well within $N$.

\begin{remark}Although the constructed $(u_T, r_T, \varphi_T) \in \tilde{\cal{Z}}^{2,k}_S(x^0, \bf{x}^2)$ has $r_T(s) = 0$ for $s \ge 0$, the genuine solution to $\cal{F} = 0$ which we will construct by the Newton-Picard method will satisfy $r > 0$ everywhere.\end{remark}

Observe that
\[\cal{F}(u_T, r_T, \varphi_T) = \left(\bar{\partial} u_T, \frac{d}{ds} r_T + \Lambda(\varphi_T) r_T, \bar{\nabla}_I \varphi_T - \Lambda(\varphi_T)\varphi_T\right) \in \cal{V}^{2,k-1}_S\]
is supported in $[-1,1] \times [0,1]$. Moreover, the exponential decay of $v^1, r^1, \varphi^1$ and $u^2, \varphi^2$ imply that for some constants $C$ and $\delta > 0$ independent of $T$ we have
\begin{equation}\label{pregluing_error_II}||\cal{F}(u_T, r_T, \varphi_T)|| = ||\bar{\partial} v_T||_{L^{2,k-1}} + ||\frac{d}{ds} r_T + \Lambda(\varphi_T) r_T||_{L^{2,k-1}} + ||\bar{\nabla}_I \varphi_T - \Lambda(\varphi_T)\varphi_T)||_{L^{2,k-1}} < C e^{- \delta T}.\end{equation}
We seek to find a zero of $\cal{F}$ near $(u_T, r_T, \varphi_T) \in \tilde{\cal{Z}}^{2,k}_S$. As before, the first step is to extend the linearization $D\cal{F}$ of $\cal{F}$ away from the zero set. In the case of polarization-twisted Floer theory, this was essentially fixed for us by the choice of a linearization $D_u$ of the ordinary $\bar{\partial}$-operator. Here the situation is more delicate, since over $Z_S^+$ we need to ensure that the $u$- and $r, \varphi$- coordinates in the linearized operators obey the matching condition. However, by fixing an isomorphism $N \cong D(E)$, the $\bar{\partial}$ equation for strips in $N$ is exactly the twisted equations. Thus, using the time-dependent family of metrics inherited from $\tilde{J}_t$ to construct the linearization $D_u$ of $\bar{\partial}$ at some $u$ with $u(Z_S^+) \subset N$, writing $\pi_M u = v$ and $\pi_E u = r \varphi$, we can calculate that

\begin{equation}\label{Flinearization}D_{(u, r, \varphi)}(\xi, \rho, \vartheta) = \begin{pmatrix}D_u \xi \\
\frac{d\rho}{ds} + \Lambda(\varphi) \rho + r \langle \varphi, \bar{\nabla}_I \vartheta \rangle_H  + r \langle \bar{\nabla}_I \varphi, \vartheta \rangle_H + r C^{\tau}_{(v, \varphi)}\pi_{TM}\xi \\
B^{\tau}_{(v, \varphi)}\pi_{TM} \xi + \Pi_V \bar{\nabla}_I \vartheta - \Lambda(\varphi) \vartheta\end{pmatrix} \in \cal{V}^{2,k-1}_S\end{equation}
does in fact satisfy the matching condition, where $B^{\tau}_{(v,\varphi)}$ is given by \eqref{Btauoperator} and $C^{\tau}_{(v, \varphi)}$ is given by \eqref{Ctauoperator}.

We then consider the non-linear map
\[\cal{G}_{u_T} : B \subset T_{u_T}\tilde{\cal{Z}}^{2,k}_S \to \cal{V}^{2,k-1}_S\]
given by
\[(\xi, \rho, \vartheta) \mapsto \left(\Pi^{\exp_{u_T}(\xi)}_{u_T}\right)^{-1}\cal{F}\left(\exp_{u_T}(\xi), \rho, \Pi^{\exp_{v_T}(\xi_M)}_{v_T} \vartheta\right)\]
where $B$ is some ball in $T_{u_T}\tilde{\cal{Z}}^{2,k}_S$; the radius is chosen so that for $(\xi, \rho, \vartheta) \in B$, $\exp_{u_{T}}(\xi(s,t)) \in N$ for $s \ge S$ and thus $\exp_{u_{T}}(\xi)$ does in fact define an element of $\tilde{\cal{Z}}^{2,k}_S$. The radius of $B$ can be chosen to be fixed, independent of $T$: this is since $r_T(s) \le C e^{-\lambda^1 T}$ for all $s \ge S$.
If we write $\cal{G}_{u_T}$ as
\[\cal{G}_{u_T}(\xi, \rho, \vartheta) = \cal{F}(u_T, r_T, \varphi_T) + D_{(u_T, r_T, \varphi_T)}(\xi, \rho, \vartheta) + N(\xi, \rho, \vartheta)\]
then it is not difficult to prove a quadratic estimate on the error term:
\begin{lemma}For some constant $C$ independent of $T$, if $(\xi, \rho, \vartheta), (\xi', \rho', \vartheta') \in T_{u_T}\tilde{\cal{Z}}^{2,k}_S$ such that $||(\xi, \rho, \vartheta)||_{Z^{2,k}_S}, ||(\xi', \rho', \vartheta')||_{Z^{2,k}_S} < C$, then
\[||N(\xi, \rho, \vartheta) - N(\xi', \rho', \vartheta')||_{\cal{V}^{2,k-1}_S} \le C ||(\xi, \rho, \vartheta) - (\xi', \rho', \vartheta')||_{Z^{2,k}_S} ||(\xi, \rho, \vartheta) + (\xi', \rho', \vartheta')||_{Z^{2,k}_S}.\]\end{lemma}
Moreover, the linearized operators $D_{(u_T, r_T, \varphi_T)}$ admit uniformly bounded right inverses:
\begin{lemma}The operator $D_{(u_T, r_T, \varphi_T)}$ is Fredholm and surjective, of index $\mu(u_T) - i(\lambda^2) = 2$, and admits a smooth family of right inverses $Q_{u_T}$ with a uniform bound on the operator norm
\[||Q_{u_T}|| \le C\]
independent of $T$.\end{lemma}

Let us briefly sketch the construction of $Q_{u_T}$: it is more or less the same as the proof of Lemma \ref{bounded_right_inverse}, but with one additional ingredient arising from the loss of the radial coordinate in the twisted equations. Indeed, there is a natural ``breaking map'' defined by using the appropriate cut-off functions
\[\beta_T : V^{2,k-1}_{S=-T}(u_T) \to V^{2,k-1}_{S=0}(u^1) \oplus L^{2,k-1}((u^2)^* TM) \oplus V^{2,k-1}_{(u^2, \varphi^2)} \oplus L^{2,k-1}(\bb{R}, \bb{R})\]
where the first three coordinates of this map are defined as earlier, and the final coordinate given by
\[(\zeta, \mu, \sigma) \mapsto \beta_+(s-S) \mu(s-S)\]
remembers the auxiliary radial coordinate $\mu$ over $(u^2, \varphi^2)$. Likewise, there is a ``linear pregluing map''
\[\gamma_T : T_{u^1}\tilde{\cal{Z}}^{2,k}_{S=0}(x^0, \bf{x}^1) \oplus W^{2,k}((u^2)^*TM) \oplus T_{\varphi^2}\cal{S}^{2,k}_{\psi^1\psi^2}((u^2)^*E) \oplus L^{2,k}(\bb{R}, \bb{R}) \to T_{u_T}\tilde{\cal{Z}}^{2,k}_{S=-T}(x^0, \bf{x}^2)\]
where the final argument gives the radial coordinate $\rho$ over $s \ge S$ on $u_T$.

Now, by the regularity of $u^1 \in \cal{M}(x^0, \bf{x}^1)$ and $(u^2, \varphi^2) \in \cal{M}(\bf{x}^1, \bf{x}^2)$ we can choose right inverses $Q_{u^1}$ and $Q_{(u^2, \varphi^2)}$ for the linearized operators $D_{u^1}$, $D_{(u^2, \varphi^2)}$ respectively. Moreover, the linear map
\[L^{2,k}(\bb{R}, \bb{R}) \to L^{2,k-1}(\bb{R}, \bb{R}), \qquad \rho \mapsto \frac{d}{ds} \rho + \Lambda(\varphi^2)\rho\]
is in fact an isomorphism, since $\Lambda \to \lambda^1, \lambda^2$ as $s \to -\infty, +\infty$ respectively, which are both positive. Hence writing $Q_r$ for the inverse map, we can then define an approximate right inverse to $D_{u_T}$ by
\[Q_{u_T}^{approx} = \gamma_T \circ \begin{pmatrix}Q_{u^1} & & \\ & Q_{(u^2, \varphi^2)} & \\ & & Q_r\end{pmatrix} \circ \beta_T.\]
It is then not difficult to check that $D_{u_T} \circ Q^{approx}_{u_T} - \id$ has operator norm exponentially decaying in $T$, and thus we obtain an actual right inverse as $Q_{u_T} = Q^{approx} \circ (D \circ Q^{approx})^{-1}$.

As before, it will be helpful to make a more geometrically meaningful choice of right inverse, using local hypersurfaces. Again pick regular points $(s_1, t_1)$ and $(s_2, t_2)$ for $u^1, u^2$ respectively, such that $u^1(s_1, t_1)$ is outside the neighbourhood $N$ of the invariant set $M$, and choose small hypersurfaces $H^1 \subset \tilde{M}$ and $H^2 \subset M$ within the invariant set around $u^1(s_1, t_1)$ and $u^2(s_2, t_2)$. Assume again that these hypersurfaces are tranverse via $u^i$ to the lines $\bb{R}\times\{t_i\} \subset Z$, and that the exponential map on $T_{u^i(s_i, t_i)}H^i$ lands in $H^i$.

Then in the above construction, again choose unique right inverses $Q_{u_1}, Q_{(u^1, \varphi^1)}$ for $D_{u^1}, D_{(u^2, \varphi^2)}$ respectively which have image those vector fields $\xi$ on each $u^i$ with $\xi(s_i, t_i) \in TH^i$. Consequently, the image of the $Q_{u_T}$ constructed above will be exactly those vector fields $\xi$ along $u_T$ such that $\xi(s_1-2T, t_1) \in TH^1$, and $\pi_M \xi(s_2 +T, t_2) \in TH^2$.

By applying Floer's quantitative implicit function theorem, adapted to find solutions of $\cal{G}_{u_T} = 0$ just in the ball $B \subset T_{u_T} \tilde{\cal{Z}}^{2,k}_S$, we then obtain:

\begin{theorem}\label{gluing_theorem_II}For sufficiently large $T_0$ and small enough $\eps_0$, for each $T \ge T_0$ there is a unique solution
\[(\hat{u}_T, \hat{r}_T, \hat{\varphi}_T) = (\exp_{u_T}\xi_T, r_T + \rho_T, \Pi^{\exp_{u_T}}_{u_T} \vartheta_T)\]
to $\cal{F} = 0$ with $||\xi|| < \eps_0$ and such that $\xi_T(s_1 - 2T, t_1) \in H^1$ and $\pi_M \xi_T(s_2 + T, t_2) \in H^2$. Moreover this satisfies the bound
\begin{equation}\label{gluing_estimate_II}||(\xi_T, \rho_T, \vartheta_T)||_{Z^{2,k}_S} \le C||\cal{F}(u_T, r_T, \varphi_T)||_{V^{2,k-1}_S}\end{equation}
for some constant $C$. In particular this yields a smooth embedding
\[G: [T_0, \infty) \to \cal{M}(\bf{x}^0, \bf{x}^2)\]
with $G(T)$ converging in the Gromov topology to the broken trajectory $(u^1, (u^2, \varphi^2))$, and by further defining $G(\infty) = (u^1, (u^2, \varphi^2))$, we obtain a surjection onto a neighbourhood of this broken trajectory in $\bar{\cal{M}}(\bf{x}^0, \bf{x}^2)$.\end{theorem}

The constant $C$ in \eqref{gluing_estimate_II} can be taken to be supremum of the operator norms $||Q_{u_T}||$; the convergence to the broken trajectory is then forced by \eqref{gluing_estimate_II} as well as the earlier estimate \eqref{pregluing_error_II} on $\cal{F}(u_T, r_T, \varphi_T)$ coming from the exponential decay of $u^1, u^2$ and $\varphi^2$.

The smoothness here is more subtle than in Theorem \ref{gluing_theorem_I}, since the Banach manifold $\tilde{\cal{Z}}^{2,k}_{S}$ itself depends on $T$. However, by our construction, for large enough $T_0$, for every $T \in (T_0 - 1, T_0+1)$ and every $S \in (-T_0-1, -T_0+1)$, the pregluing $(u_{T}, r_{T}, \varphi_{T})$ defines an element of $\tilde{\cal{Z}}^{2,k}_{S}$ (this is since for $S$ and $T$ in this range, $u_{T}(S,t)$ is exponentially close to $M$, but for no $S$ do we have $u_T(S,t) \subset M$, and thus $u_T$ indeed determines $r_T, \varphi_T$). Similarly, for each such $T$, there are canonical vector space isomorphisms between each $T_{u_T} \tilde{\cal{Z}}^{2,k}_S$ and $\cal{V}^{2,k-1}_S$ for $S$ in this range, so that we can consider these to be fixed topological vector spaces with a smoothly varying family of equivalent metrics for $S \in (-T_0-1, -T_0+1)$.

The upshot is that neighbourhoods of each $(u_T, r_T, \varphi_T) \in \tilde{\cal{Z}}^{2,k}_{S = -T}$ form a fibration of Banach manifolds over $(T_0 - 1, T_0+1)$, with $(u_T, r_T, \varphi_T)$ a smooth section.  Since the map $\cal{G}_{u_T}$, its linearization $D_{u_T}$ and chosen right inverse $Q_{u_T}$ by construction also vary smoothly, we see that the gluing map must be smooth as well.

The proof of the injectivity and surjectivity of the gluing map then proceeds exactly as in the proof of Theorem \ref{gluing_theorem_I}. We should note that a bit more care is required to conclude that given a sequence of non-invariant trajectories $u_{\alpha}$ converging to the broken trajectory $(u^1, (u^2, \varphi^2))$ eventually lands in the image of the gluing map: we must upgrade the naive $C^1$ bounds to one in the more exotic metric on $\tilde{\cal{Z}}^{2,k}_S$. However, this metric is a linear admixture of the usual $L^{2,k}$ metric on vector fields on the one hand outside the fixed point set, and then near the fixed point set we use the $\tau$-model $L^{2,k}$ metric on $(\pi_M \xi, \vartheta)$ as well as the usual $L^{2,k}$ metric in the radial direction. The required bound then follows from Proposition \ref{longstrip_phicontrol} combined with its analogue in more traditional Floer theory.

\subsection{The extended moduli space.}\label{sec:extended_moduli_space} Let us now turn to gluing a three component broken trajectory together, as in case (I.2) of Section \ref{stratalist}. Hence, let $\bf{x}^0 = (x^0, \iota(x^0)) \in \frak{C}_o$, $\bf{x}^1 = (x^1, \lambda^1) \in \frak{C}_s$ (so that $\lambda^1 > 0$), $\bf{x}^2 = (x^2, \lambda^2) \in \frak{C}_u$ (so that $\lambda^2 < 0$), and $\bf{x}^3 = (x^3, \iota(x^3)) \in \frak{C}_o$. For simplicity we will only consider gluing onto the one-dimensional strata, so we will take discrete and regular elements of the corresponding moduli spaces $\cal{M}(\bf{x}^0, \bf{x}^1), \cal{M}(\bf{x}^1, \bf{x}^2)$ and $\cal{M}(\bf{x}^2, \bf{x}^3)$. As always, we will fix parametrized representatives for the solutions in question, and work in the $\tau$ model.

This case poses a new challenge, which necessitates the introduction of $\delta$ structures. We are gluing three trajectories $u^1 \in \cal{M}(\bf{x}^0, \bf{x}^1), (u^2, \varphi^2) \in \cal{M}(\bf{x}^1, \bf{x}^2)$ and $u^3 \in \cal{M}(\bf{x}^2, \bf{x}^3)$ together, so at first sight there are two gluing parameters $T_1, T_2$. However, what should be the target space of the gluing, $\cal{M}(\bf{x}^0, \bf{x}^3)$ is only one-dimensional, so we certainly cannot have two independent gluing parameters. Our solution to this is to define an \emph{extended moduli space} $E\cal{M}(\bf{x}^0, \bf{x}^3)$ of solutions with a certain mild discontinuity, which is one dimension higher than the actual moduli space $\cal{M}(\bf{x}^0, \bf{x}^3)$. We will then have a gluing map at $(u^1, (u^2, \varphi^2), u^3)$:
\[G: [T_0, \infty) \times [T_0, \infty) \to E\cal{M}(\bf{x}^0, \bf{x}^3)\]
together with a smooth function
\[\delta_0 : [T_0, \infty) \times [T_0, \infty) \to \bb{R}\]
which measures the discontinuity of $G(T_1, T_2)$; we will have $\delta_0(T_1, T_2) = 0$ if and only if $G(T_1, T_2)$ is a bona fide continuous trajectory. This will then give rise to a $\delta$-structure, as we will explain in the next section.

Let us get on with defining the extended moduli space. Fix representatives $x^0 \in \bf{x}^0, x^3 \in \bf{x}^3$, and fix a real number $S \in \bb{R}$.

\begin{definition}An \emph{extended solution with discontinuity at $S$} and with limits at $x^0, x^3$ is a tuple $(u_-, S, u_+)$ where:
\begin{itemize}\item $u_-, u_+$ are $\tilde{J}_t$ holomorphic maps
\[u_- : Z_S^- = (-\infty, S]\times[0,1] \to \tilde{M}, \qquad u_+ : Z_S^+ = [S,\infty)\times[0,1] \to \tilde{M}\]
such that $u_{\pm}(\cdot, i) \in \tilde{L}_i$ for $i = 0,1$, and $u_- \to x^0$ as $s \to -\infty$, and $u_+ \in x^3$ as $s \to +\infty$;
\item We have both $u_-(S, [0,1]), u_+(S, [0,1]) \subset N$, and moreover we have matching conditions
\[\pi_M u_-(S, t) = \pi_M u_+(S,t), \quad \frac{\pi_E u_-(S,t)}{||\pi_E u_-(S, \cdot)||_H} = \frac{\pi_E u_+(S,t)}{||\pi_E u_+(S, \cdot)||_H} \qquad \text{ for } t \in [0,1].\]\end{itemize}\end{definition}

Given an extended solution $(u_-, S, u_+)$, by the unique continuation properties of the twisted equations, we can find $S_1 < S < S_2$ such that $u_-([S_1, S]\times[0,1]) \subset N$ and $u_+([S, S_2]\times[0,1]) \subset N$, and a tuple $(v, r_-, r_+, \varphi)$ where:
\begin{itemize}\item $v$ is a $J_t$-holomorphic strip
\[v : Z_{S_1, S_2} = [S_1, S_2] \times [0,1] \to M\]
with $v(\cdot, i) \in L_i$ for $i = 0,1$;
\item $r_- : [S_1, S] \to \bb{R}_{>0}$ and $r_+ : [S, S_2] \to \bb{R}_{>0}$ are smooth positive real functions;
\item $\varphi \in C^{\infty}(Z_{S_1 S_2}, v^*E)$ is a smooth section of $v^* E$ such that $\varphi(\cdot, i) \in F_i$ for $i = 0,1$ and $||\varphi||_H = 1$;
\item $r_{\pm}$ and $\varphi$ solve the twisted equations
\[\frac{d}{ds} r_{\pm} + \Lambda(\varphi)r_{\pm} = 0, \qquad \bar{\nabla}_I \varphi - \Lambda(\varphi) \varphi;\]
\item we have matching conditions
\[\pi_M u_{-}(s,t) = v(s,t), \quad \pi_E u_-(s,t) = r_-(s) \varphi(s,t) \quad \text{ for } s \in [S_1, S]\]
as well as
\[\pi_M u_+(s,t) = v(s,t), \quad \pi_E u_+(s,t) = r_+(s) \varphi(s,t) \quad \text{ for } s \in [S, S_2].\]\end{itemize}

Indeed, we can think of an extended solution $(u_-, S, u_+)$ as a $\tau$-model solution with a discontinuity in the $r$ coordinate at $s = S$. Correspondingly, we can then define
\begin{equation}\label{delta_map}\delta(u_-, S, u_+) = r_+(S) - r_-(S).\end{equation}
In particular, the data of an extended solution $(u_-, S, u_+)$ with $\delta = 0$ is exactly that of a $\tilde{J}_t$-holomorphic strip $u$ with a choice of $S \in \bb{R}$ such that $u(S, [0,1]) \subset N$.

There is a natural $\bb{R}$ action on such tuples $(u_-, S, u_+)$ by translation, given by $\tau_{\sigma}^*(u_-, S, u_+) = (u_-(\cdot - \sigma, \cdot), S + \sigma, u_+(\cdot - \sigma, \cdot))$ for $\sigma \in \bb{R}$. Moreover, there is a second natural equivalence relation from unique continuity: for $S \le S'$, we say that $(u_-, S, u_+)$ is \emph{continuation equivalent} to $(u_-', S', u_+')$ if $u_+(Z_{S,S'}) \subset N$, $u_-'(Z_{S,S'}) \subset N$ and moreover we have $u_- = u_-'$ on $Z_{S}^-$, and $u_+ = u_+' = Z_{S'}^+$ (and similarly if $S' \le S$). Concretely in terms of the tuple $(v, r_-, r_+, \varphi)$ as above, for sufficiently small $\eps_0 > 0$ there are unique extensions
\[\tilde{r}_- : [S_1, S+\eps_0] \to \bb{R}, \qquad \tilde{r}_+ : [S-\eps_0, S_2] \to \bb{R}\]
solving $\frac{dr}{ds} + \Lambda r = 0$. Assuming $\eps_0$ is small enough, we must have $\tilde{r}_{\pm}\varphi \in D(E)$ on all of $[S-\eps_0, S+\eps_0] \times [0,1]$, from which we can construct $\tilde{J}_t$-holomorphic maps
\[\tilde{u}_-: Z_{S+\eps_0}^- \to \tilde{M}, \qquad \tilde{u}_+ : Z_{S-\eps_0}^+ \to \tilde{M}.\]
Hence, for each $S' \in (S-\eps_0, S+\eps_0)$, we have a unique extended solution $(\tilde{u}_-|_{Z_{S'}^-}, S', \tilde{u}_+|_{Z_{S'}^+})$ which is continuation equivalent to $(u_-, S, u_+)$.

\begin{definition}The \emph{moduli space of extended trajectories}
\[E \cal{M}(x^0, x^3)\]
is the quotient of the space of finite energy extended solutions $(u_-, S, u_+)$ with limits at $x^0, x^3$ by translation and continuation equivalence.\end{definition}

Observe that $\delta$ does \emph{not} descend to a function on this quotient space: while the value of $\delta$ is preserved by translation, it is not equal for continuation equivalent extended solutions $(u_-, S, u_+)$ and $(u'_-, S', u'_+)$. However, the \emph{sign} of $\delta$ is well-defined: $\delta(u_-,S,u_+) = 0$ (respectively $>0, <0$) if and only if $\delta(u'_-, S', u'_+) = 0$ (respectively $>0, <0$). In any case, the main purpose for which we have defined $E\cal{M}(x^0, x^3)$ is as the natural target of a gluing map, and we will see that this gluing map in fact provides a natural slice for the continuation equivalence relation; on this slice the function $\delta$ is certainly defined.

Let us now describe a Banach space set-up for extended solutions with limits at $x^0$ and $x^3$. Fix real numbers $S_1 < S < S_2$; we will moreover assume that $S - S_1 > 1$ and $S_2 - S > 1$ (this will certainly be true when we perform gluing). We then let
\begin{equation}E\tilde{\cal{Z}}^{2,k}_{S_1,S,S_2}(x^0, x^3)\end{equation}
to be the set of tuples $(u_-, u_+, v, r_-, r_+, \varphi)$ where:
\begin{itemize}\item $u_- : Z_S^- \to \tilde{M}$ and $u_+ : Z_S^+ \to \tilde{M}$ are class $L^{2,k}_{loc}$ maps with boundaries on $\tilde{L}_0, \tilde{L}_1$, and so that for $s \ll 0$ we can write $u_-(s,t) = \exp_{x^0}(\xi_-(s,t))$ where $\xi_- \in T_{x^0}\tilde{M}$ is class $L^{2,k}$, and likewise for $s \gg 0$ we have $u_+(s,t) = \exp_{x^3}(\xi_+(s,t))$ for $\xi_+ \in T_{x^3}\tilde{M}$ of class $L^{2,k}$;
\item $v : Z_{S_1 S_2} \to M$ is class $L^{2,k}_{loc}$ with boundaries on $L_0, L_1$;
\item $r_- \in L^{2,k}([S_1, S], \bb{R})$ and $r_+ \in L^{2,k}([S,S_2], \bb{R})$;
\item $\varphi$ is an $L^{2,k}_{loc}$ section of $v^*E$ over $Z_{S_1S_2}$ with boundaries on $F_0, F_1$ and so that $|| \varphi||_H = 1$ everywhere;
\item $u_-(Z_{S_1,S}) \subset N$ and $u_+(Z_{S,S_2}) \subset N$, and we have the matching conditions that
\[\pi_M u_{-}(s,t) = v(s,t), \quad \pi_E u_-(s,t) = r_-(s) \varphi(s,t) \quad \text{ for } s \in [S_1, S]\]
as well as
\[\pi_M u_+(s,t) = v(s,t), \quad \pi_E u_+(s,t) = r_+(s) \varphi(s,t) \quad \text{ for } s \in [S, S_2].\]\end{itemize}

This can be made into a Banach manifold: abbreviating $u = (u_-, u_+, v, r_-, r_+, \varphi)$, the tangent space $T_u E\tilde{\cal{Z}}^{2,k}_{S_1,S,S_2}$ is the vector space of tuples $(\xi_-, \xi_+, \xi_M, \rho_-, \rho_+, \vartheta)$ where
\begin{itemize}\item $\xi_- \in W^{2,k}(Z_S^-, u_-^*T\tilde{M})$ and $\xi_+ \in W^{2,k}(Z_S^+, u_+^* T \tilde{M})$;
\item $\xi_M \in W^{2,k}(Z_{S_1S_2}, v^*TM)$;
\item $\rho_- \in L^{2,k}([S_1, S], \bb{R})$ and $\rho_+ \in L^{2,k}([S,S_2], \bb{R})$;
\item $\vartheta \in W^{2,k}(Z_{S_1S_2}, v^*E)$ such that $\langle \varphi, \vartheta \rangle_H = 0$;
\item we have matching conditions
\[\pi_M \xi_- = \xi_M, \qquad \pi_E \xi_- = \rho_- \varphi + r_- \vartheta \qquad \text{ over } Z_{S_1S}\]
as well as
\[\pi_M \xi_+ = \xi_M, \qquad \pi_E \xi_+ = \rho_+ \varphi + r_+ \vartheta \qquad \text{ over } Z_{SS_2}.\]\end{itemize}
The norm is defined as follows: writing $\beta_+ : \bb{R} \to [0,1]$ for a real function with $\beta_+ = 0$ on $s \le 0$ and $\beta_+ = 1$ on $s \ge 1$, as well as $\beta_-(s) = \beta_+(-s)$, define a ``rescaled'' vector field $R_{S_1S_2}(\xi)$ by
\[R_{S_1S_2}(\xi) = \begin{cases}\xi_-(s,t) & \text{ for } s < S_1;\\
(\xi_M, (1- \beta_+(s-S_1)) \pi_E \xi_-) & \text{ for } S_1 \le s \le S;\\
(\xi_M, (1-\beta_-(s-S_2)) \pi_E \xi_+) & \text{ for } S \le s \le S_2; \\
\xi_+(s,t) & \text{ for } s > S_2.\end{cases}\]
We then define the norm on $T_u E\tilde{\cal{Z}}^{2,k}_{S_1,S,S_2}$ as
\begin{align*}||(\xi_-, \xi_+, \xi_M, \rho_-, \rho_+, \vartheta)||_{EZ^{2,k}_{S_1SS_2}} &= ||R_{S_1S_2}(\xi)||_{2,k} + ||\beta_+(s-S_1)\rho_+||_{2,k} + ||\beta_-(s-S_2) \rho_-||_{2,k} \\ &+ ||\beta_+(s-S_1)\beta_-(s-S_2) \vartheta||_{2,k}.\end{align*}

There is likewise a Banach vector bundle $E\cal{V}^{2,k-1}_{S_1,S,S_2}$ over $E\tilde{\cal{Z}}^{2,k}_{S_1,S,S_2}$ whose fibre is the space of those $(\zeta_-, \zeta_+, \zeta_M, \mu_-, \mu_+, \sigma)$ where
\begin{itemize}\item $\zeta_{\pm} \in L^{2,k-1}(Z_S^{\pm}, u_{\pm}^*T \tilde{M})$,
\item $\zeta_M \in L^{2,k-1}(Z_{S_1S_2}, v^*E)$,
\item $\mu_-, \mu_+ \in L^{2,k-1}([S_1,S],\bb{R}), L^{2,k-1}([S,S_2],\bb{R}$ respectively,
\item $\sigma \in L^{2,k-1}(Z_{S_1S_2}, v^*E)$ with $\langle \varphi, \sigma \rangle_H = 0$ everywhere\end{itemize}
subject to the same matching condition as above. The metric is defined analogously to above:
\begin{align*}||(\zeta_-, \zeta_+, \zeta_M, \mu_-, \mu_+, \sigma)||_{EV^{2,k-1}} &= ||R_{S_1S_2}(\zeta)||_{2,k-1} + ||\beta_+(s-S_1)\mu_+||_{2,k-1} \\ &+ ||\beta_-(s-S_2) \mu_-||_{2,k-1} + ||\beta_+(s-S_1)\beta_-(s-S_2) \sigma||_{2,k-1}.\end{align*}
This comes with a natural section $\cal{F}: E\tilde{\cal{Z}}^{2,k}_{S_1,S,S_2} \to E\cal{V}^{2,k-1}_{S_1,S,S_2}$ defined by
\begin{equation}(u_-, u_+, v, r_-, r_+, \varphi) \mapsto \left(\bar{\partial}_{\tilde{J}_t} u_-, \bar{\partial}_{\tilde{J}_t} u_+, \bar{\partial}_{J_t} v, \frac{dr_-}{ds} + \Lambda(\varphi)r_-, \frac{dr_+}{ds} + \Lambda(\varphi)r_+, \bar{\nabla}_I \varphi - \Lambda(\varphi)\varphi\right)\end{equation}
whose zeroes in the closed subset
\[E\cal{Z}^{2,k}_{S_1,S,S_2} = \{(u_-, u_+, v, r_-, r_+, \varphi): r_-, r_+ \ge 0\} \subset E\tilde{\cal{Z}}^{2,k}_{S_1,S, S_2}\]
are exactly the extended solutions $(u_-, S, u_+)$ with $u_-(Z_{S_1S}) \subset N$ and $u(Z_{SS_2}) \subset N$.

We will not describe a separate Fredholm and transversality set-up for the space of extended solutions. Indeed, we are only interested in this moduli space locally near a broken trajectory with boundary-obstructed components, and here the appropriate transversality results will be a consequence of the gluing estimates. So, let us now proceed in earnest towards proving a gluing theorem.

\subsection{Gluing in the boundary-obstructed case.}\label{sec:boundary_obstructed_gluing} Fix discrete and regular trajectories $u^1 \in \cal{M}(\bf{x}^0, \bf{x}^1), (u^2, \varphi^2) \in \cal{M}(\bf{x}^1, \bf{x}^2)$ and $u^3 \in \cal{M}(\bf{x}^2, \bf{x}^3)$, where $\bf{x}^0 \in \frak{C}_o \bf{x}^1 \in \frak{C}_s, \bf{x}^2 \in \frak{C}_u, \bf{x}^3 \in \frak{C}_o$ are in case (I.2) of Section \ref{stratalist}. As before, pick unit eigensolutions $\psi^1, \psi^2$ corresponding to $\lambda^1, \lambda^2$ at $x^1, x^2$, so that $\varphi^2 \to \psi^1, \psi^2$ as $s \to -\infty, +\infty$ respectively. Let us also assume that $u^1 \to x^0$ as $s \to -\infty$, and $u^3 \to x^3$ as $s \to +\infty$. For convenience, let us also fix the parametrizations on $u^1, u^3$ so that $u^1([0, \infty)\times[0,1]) \subset N$ and $u^3((-\infty, 0]\times[0,1]) \subset N$.

For a tuple $T = (T_1, T_2)$ where $T_1, T_2$ are sufficiently large, we construct a ``preglued'' extended solution $(u_T^-, 0, u_T^+)$ with discontinuity at $S = 0$, as follows.

Write $\pi_E u^1(s,t) = r^1(s) \varphi^1(s,t)$ for $s \ge 0$, where $r^1 > 0$ and $||\varphi^1||_H = 1$, and write $v^1(s,t) = \pi_M u^1(s,t)$: in other words, we have an element $(u^1, r^1, \varphi^1) \in \tilde{\cal{Z}}^{2,k}_{S=0}(x^0, \bf{x}_1)$. Likewise write $\pi_E u^3(s,t) = r^3(s) \varphi^3(s,t)$ for $s \le 0$, where $r^3 > 0$ and $||\varphi^3||_H = 1$, and write $v^3(s,t) = \pi_M u^3(s,t)$, so that $(u^3, r^3, \varphi^3) \in \tilde{\cal{Z}}^{2,k}_{S=0}(\bf{x}^2, x^3)$.

Choose open neighbourhoods $U^1$ of $x^1$ and $U^2$ of $x^2$ on which $E$ is trivialized, and choose $T_1, T_2$ large enough so that $v^1 \in U^1$ for $s \ge T_1$, $u^2 \in U^1$ for $s \le -T_1$, $u^2 \in U^2$ for $s \ge T_2$, and $v^3 \in U^2$ for $s \le -T_2$. Hence, we can write
\[\varphi^1(s,t) = \psi^1(t) + \eta^1_+(s,t) \text{ for } s \ge T_1, \qquad \varphi^2(s,t) = \psi^1(t) + \eta^2_-(s,t) \text{ for } s \le -T_1;\]
\[\varphi^2(s,t) = \psi^2(t) + \eta^2_+(s,t) \text{ for } s \ge T_2, \qquad \varphi^3(s,t) = \psi^2(t) + \eta^3_-(s,t) \text{ for } s \le -T_2\]
so that $\eta^1_+, \eta^2_-, \eta^2_+, \eta^3_- \to 0$ exponentially as $|s| \to \infty$.

Indeed, we construct a preglued map $(u^-_T, u^+_T)$ for $T = (T_1, T_2)$ as
\[u^-_T = \begin{cases}u^1(s+T_1, t) & \text{ for } s \le -2T_1; \\
\beta_N^{-1}(v_T, r^-_T \varphi_T(s,t)) & \text{ for } -2T_1 \le s \le 0\end{cases}\]
\[u^+_T = \begin{cases}\beta_N^{-1}(v_T, r^+_T \varphi_T(s,t)) & \text{ for } 0 \le s \le 2T_2; \\
u^3(s-3T_1, t) & \text{ for } s \le -2T_1\end{cases}\]
where $v_T : [-2T_1, 2T_2] \times [0,1] \to M$, $r_T^- : [-2T_1, 0] \to \bb{R}$, $r_T^+: [0, 2T_2] \to \bb{R}$ and $\varphi_T \in C^{\infty}(v^* E)$ are constructed as follows. Once again, take a smooth, non-decreasing function $\beta_+: \bb{R} \to [0,1]$ with $\beta_+(s) = 0$ for $s \le -1$ and $\beta_+(s) = 1$ for $s \ge 1$, and set $\beta_-(s) = 1-\beta_+(s)$. Then we have
\[v_T(s,t) = \begin{cases}v^1(s+3T_1, t) & \text{ for } -2T_1 \le s \le -T_1 - 1;\\
\exp_{x^1}(\beta_-(s+T_1) \exp^{-1}_{x^1}(v^1(s+3T_1, t)) & \text{ for } -T_1 - 1 \le s \le -T_1 + 1 \\
\qquad + \ \beta_+(s+T_1)\exp^{-1}_{x^1}(u^2(s,t))) & \\
u^2(s,t) & \text{ for } -T_1 + 1 \le s \le T_2 - 1;\\
\exp_{x^2}(\beta_-(s-T_2) \exp^{-1}_{x^2}(u^2(s,t)) & \text{ for } T_2 - 1 \le s \le T_2+1;\\
\qquad + \ \beta_+(s-T_2)\exp^{-1}_{x^2}(v^3(s-3T_2,t))) &\\
v^3(s-3T_2, t) & \text{ for } T_2 + 1 \le s \le 2T_2\end{cases}\]
\[r_T^-(s) = \begin{cases}r^1(s+3T_1) & \text{ for } -2T_1 \le s \le -T_1 - 1;\\
\beta_-(s+T_1) r^1(s+3T_1) & \text{ for } -T_1 -1 \le s \le -T_1 + 1\\
0 & \text{ for } -T_1 + 1 \le s \le 0\end{cases}\]
\[r_T^+(s) = \begin{cases}0 & \text{ for } 0 \le s \le T_2 - 1\\
\beta_+(s-T_2) r^3(s-3T_2) & \text{ for } T_2 -1 \le s \le T_2 + 1;\\
r^3(s-3T_2) & \text{ for } T_2 + 1 \le s \le 2T_2\end{cases}\]
\[\varphi_T(s,t) = \begin{cases}\varphi^1(s+3T_1, t) & \text{ for } -2T_1 \le s \le -T_1 - 1;\\
\frac{\psi^1(t) + \beta_-(s+T_1)\eta^1_+(s+3T_1,t) + \beta_+(s+T_1)\eta^2_-(s,t)}{||\psi^1(t) + \beta_-(s+T_1)\eta^1_+(s+3T_1,t) + \beta_+(s+T_1)\eta^2_-(s,t)||_H} & \text{ for } -T_1 - 1 \le s \le -T_1 + 1\\
\varphi^2(s,t) & \text{ for } -T_1 + 1 \le s \le T_2 - 1\\
\frac{\psi^2(t) + \beta_-(s-T_2)\eta^2_+(s,t) + \beta_+(s-T_2)\eta^3_-(s-3T_2, t)}{||\psi^2(t) + \beta_-(s-T_2)\eta^2_+(s,t) + \beta_+(s-T_2)\eta^3_-(s-3T_2, t)||_H} & \text{ for } T_2 - 1 \le s \le T_2 + 1;\\
\varphi^3(s-2T_2, t) & \text{ for } T_2 + 1 \le s \le 2T_2.
\end{cases}\]

Setting $S_1 = -2T_1$ and $S_2 = 2T_2$, with this construction $u_T = (u_T^-, u_T^+, v_T, r_T^-, r_T^+, \varphi_T)$ defines an element of the Banach manifold $E\tilde{\cal{Z}}^{2,k}_{S_1,0,S_2}$. Indeed, by construction $u^-_T(Z_{S_1,S})$ and $u^+_T(Z_{S,S_2})$ is exponentially close to the invariant set $M$, and thus well within $N$.

$\cal{F}(u_T) \in E\cal{V}^{2,k-1}_{S_1,0,S_2}$ is then by construction supported in the two strips $[-T_1-1, -T_1+1]\times[0,1]$ and $[T_2-1,T_2+1]\times[0,1]$. Moreover, the exponential decay of $v^1, u^2, v^3$, the exponential decay of $r^1, r^3$ and the exponential decay of $\eta^1_+, \eta^2_-, \eta^2_+, \eta^3_-$ ensures that for some small $\delta > 0$ and a constant $C$ independent of $T_1, T_2$ we have the estimate
\begin{equation}\label{pregluing_error_III}||\cal{F}(u_T)||_{EV^{2,k-1}} \le C e^{-\delta T_1} + C e^{-\delta T_2}.\end{equation}

We will apply the implicit function theorem to find zeroes of a non-linear map
\[\cal{G}_{u_T} : B \subset T_{u_T} E\tilde{\cal{Z}}^{2,k}_{S_1,0,S_2} \to E\cal{V}^{2,k-1}_{S_1,0,S_2}\]
where $B$ is some metric ball inside $T_{u_T} E\tilde{\cal{Z}}^{2,k}_{S_1,0,S_2}$, given by
\[\cal{G}_{u_T}(\xi_+, \xi_-, \xi_M, \rho_-, \rho_+, \vartheta) = \left(\Pi^{\exp{u_T}(\xi)}_{u_T}\right)^{-1}\cal{F}\left(\exp_{u_T^-}(\xi_-), \exp_{u_T^+}(\xi_+), \exp_{v_T}(\xi_M), \rho_-, \rho_+, \Pi^{\exp_{v_T}(\xi_M)}_{v_T} \vartheta\right).\]
As in the boundary-unobstructed gluing case, the radius of the ball $B$ can be chosen to be independent of $T_1, T_2$ by the a priori bounds on $r_T^-, r_T^+$.

Now, at $u_T \in E\tilde{\cal{Z}}^{2,k}_{S_1,0,S_2}$, consider the linear operator
\[D_{u_T} : T_{u_T}E\tilde{\cal{Z}}^{2,k}_{S_1,0,S_2} \to E\cal{V}^{2,k-1}_{S_1,0,S_2}\]
defined by
\[(\xi_-, \xi_+, \xi_M, \rho_-, \rho_+, \vartheta) \mapsto (\zeta_-, \zeta_+, \zeta_M, \mu_-, \mu_+, \sigma)\]
where
\[\zeta_- = D(\bar{\partial}_{\tilde{J}_t})_{u_T^-}(\xi_-), \qquad \zeta_+ = D(\bar{\partial}_{\tilde{J}_t})_{u_T^+}(\xi_+), \qquad D(\bar{\partial}_{J_t})_{v_T}(\xi_M)\]
are the usual linearized $\bar{\partial}$-operators, and
\[\mu_{\pm} = \frac{d}{ds}\rho_{\pm} + \Lambda(\varphi_T)\rho_{\pm} + r^{\pm}_T \langle \bar{\nabla}_I \varphi_T, \vartheta \rangle_H + r^{\pm}_T C^{\tau}_{(v_T, \varphi_T)} \xi_M\]
\[\sigma = B^{\tau}_{(v_T, \varphi_T)} \xi_M + \Pi_V \bar{\nabla}_I \vartheta - \Lambda(\varphi_T) \vartheta\]
are as in our earlier expression \eqref{Flinearization} for the linearized operator $D\cal{F}$ in the boundary unobstructed case (with $B^{\tau}, C^{\tau}$ as in \eqref{Btauoperator}, \eqref{Ctauoperator} respectively). Again, since over $N \cong D(E)$ the $\tilde{J}_t$ holomorphic curve equation reduces to the twisted equations for $(J_t, I_t, \nabla)$, we see that $(\zeta_-, \zeta_+, \zeta_M, \mu_-, \mu_+, \sigma)$ satisfy the matching conditions to define an element of $E\cal{V}^{2,k-1}_{S_1,0,S_2}$.

Observe that for large enough $T_1, T_2$, the operator $D_{u_T}$ is obtained by a gluing of three Fredholm operators on three strips together: on the first strip we have $D_{u^1}$, on the third strip we have $D_{u^3}$, and on the middle strip we have the diagonal operator
\[\begin{pmatrix}D_{(u^2, \varphi^2)} & 0 \\ 0 & D_{\Lambda}\end{pmatrix}\]
where $D_{(u^2, \varphi^2)}$ is the linearized operator from polarization-twisted Floer theory acting on the $\xi, \vartheta$ coordinates, and
\begin{align}\label{D_Lambda_operator} D_{\Lambda}: L^{2,k}((-\infty, 0], \bb{R}) \oplus L^{2,k}([0, \infty), \bb{R}) &\to L^{2,k-1}((-\infty, 0], \bb{R}) \oplus L^{2,k-1}([0, \infty), \bb{R}) \\
(\rho_-, \rho_+) &\mapsto \left(\frac{d\rho_-}{ds} + \Lambda(\varphi^2)\rho_-, \frac{d\rho_+}{ds} + \Lambda(\varphi^2)\rho_+\right).\end{align}
$D_{\Lambda}$ is moreover an isomorphism, since $\Lambda(\varphi^2)$ has limits $\lambda^1 > 0$ and $\lambda^2 < 0$ as $s \to -\infty, \infty$ respectively. We thus deduce that $D_{u_T}$ is Fredholm, and of index three.

Abbreviating $\xi = (\xi_-, \xi_+, \xi_M, \rho_-, \rho_+, \vartheta)$, we then have the following quadratic estimate on the error term $N(\xi) = \cal{G}_{u_T}(\xi) - D_{u_T}(\xi) \in E\cal{V}^{2,k-1}_{S_1,0,S_2}$:

\begin{lemma}For some small constant $C$ independent of $T_1, T_2$, if $\xi = (\xi_-, \xi_+, \xi_M, \rho_-, \rho_+, \vartheta)$ and $\xi' = (\xi_-', \xi_+', \xi_M', \rho_-', \rho_+', \vartheta')$ are in $B \subset T_{u_T} E \tilde{\cal{Z}}^{2,k}_{S_1,0,S_2}$ with $||\xi||_{EZ^{2,k}}, ||\xi'||_{EZ^{2,k}} \le C$, we have
\[||N(\xi) - N(\xi')||_{EV^{2,k-1}} \le C ||\xi - \xi'||_{EZ^{2,k}} ||\xi + \xi'||_{EZ^{2,k}}.\]\end{lemma}

As before, we also have

\begin{lemma}The Fredholm operator $D_{u_T}$ admits a smooth family of right inverses $Q_{u_T}$ with a uniform bound on the operator norm
\[||Q_{u_T}|| \le C\]
independent of $T = (T_1, T_2)$.\end{lemma}

The construction of $Q_{u_T}$ is the same as in previous sections, with the added discontinuity of the radial coordinate $r$. Indeed, we first construct an approximate right inverse, utilizing a natural ``breaking map''
\begin{align*}\beta_T : EV^{2,k-1}_{S_1,0,S_2} \quad \to & \quad V^{2,k-1}_{S=0}(u^1) \oplus L^{2,k-1}((u^2)^*TM) \oplus V^{2,k-1}_{u^2, \varphi^2} \\
&\oplus L^{2,k-1}((-\infty, 0], \bb{R}) \oplus L^{2,k-1}([0,\infty), \bb{R}) \oplus V^{2,k-1}_{S=0}(u^3)\end{align*}
and a ``linear pregluing map''
\begin{align*}\gamma_T: &T_{u^1} \tilde{\cal{Z}}^{2,k}_{S=0} \oplus W^{2,k}((u^2)^* TM) \oplus T_{\varphi^2}\cal{S}^{2,k}_{\psi^1 \psi^2}((u^2)^*E) \quad \to \quad T_{u_T}E\tilde{\cal{Z}}^{2,k}_{S_1,0,S_2} \\
&\oplus L^{2,k}((-\infty, 0], \bb{R}) \oplus L^{2,k}([0,\infty), \bb{R}) \oplus T_{u^3} \tilde{\cal{Z}}^{2,k}_{S=0}.\end{align*}
We also chose fixed right inverses $Q_{u^1}, Q_{(u^2, \varphi^2)}, Q_{u^3}$ for the linearized $\cal{F}$-operators at the regular solutions $u^1, (u^2, \varphi^2), u^3$, and write $Q_{\Lambda}$ for the inverse of the operator $D_{\Lambda}$ of \eqref{D_Lambda_operator}. The approximate inverse $Q^{approx}_{u_T}$ is then the composition
\[Q^{approx}_{u_T} = \gamma_T \circ \text{diag}(Q_{u^1}, Q_{(u^2, \varphi^2)}, Q_{\Lambda}, Q_{u^3}) \circ \beta_T.\]
The error $D_{u_T} Q^{approx}_{u_T} - \id$ then exponentially decays in $T_1, T_2$, allowing us to define the actual right inverse operator $Q_{u_T} = Q^{approx}(D_{u_T}Q^{approx})^{-1}$ for large enough $T_1, T_2$.
This shows $D_{u_T}$ is surjective; the uniform bound on the operator norm of $Q_{u_T}$ directly follows from the construction.

\begin{remark}In this boundary-obstructed case where $\lambda^1 > 0$ and $\lambda^2 < 0$, the failure of the operator
\[L^{2,k}(\bb{R}, \bb{R}) \to L^{2,k-1}(\bb{R}, \bb{R}), \qquad \rho \mapsto \frac{d}{ds}\rho + \Lambda(\varphi^2)\rho\]
to be surjective is the main reason why we use a gluing map that involves extended solutions with discontinuities; for the extended moduli space one instead works with the operator $D_{\Lambda}$ of \eqref{D_Lambda_operator}, which is an isomorphism. Indeed, even if we were to work with a single gluing parameter $T$ and chose $T_1(T), T_2(T)$ so that we can produce a preglued strip $u_T$ that had no discontinuity, this precise problem would prevent the above strategy from producing a bounded right inverse of $D_{u_T}$.\end{remark}

Once again, we will specify the right inverse $Q_{u_T}$ we use by choosing local hypersurfaces $H^1 \subset \tilde{M}, H^2 \subset M, H^3 \subset \tilde{M}$ around regular points $(s_1, t_1), (s_2, t_2)$ and $(s_3, t_3)$ of $u^1, (u^2, \varphi^2), u^3$ respectively, and which are appropriately transverse to $u^i$ and closed under the exponential map at $u^i(s_i, t_i)$.
Then in the above construction, choose $Q_{u^i}$ to be the unique right inverse of $D_{u^i}$ with image those vector fields whose value of $(s_i, t_i)$ lies in $TH^i$. The resulting right inverse $Q_{u_T}$ along the pregluing $u_T$ then has image precisely those $\xi_T = (\xi_-, \xi_+, \xi_M, \rho_-, \rho_+, \vartheta)$ such that $\xi_-(s_1- 3T_1, t) \in TH^1, \xi_M(s_2, t_2) \in TH^2$ and $\xi_+(s_3 + 3T_2, t) \in TH^3$.

By applying the Floer's quantitative implicit function theorem, we then construct a gluing map onto the extended moduli space:

\begin{theorem}\label{gluing_theorem_III}For sufficiently large $T_0$ and small enough $\eps_0$, for each $T_1, T_2 \ge \eps_0$, there is a unique extended solution
\[(\hat{u}_T^-, 0, \hat{u}_T^+) = \left(\exp_{u_T^-}\xi_T^-, \exp_{u_T^+}\xi_T^+, \exp_{v_T}\xi_{M, T}, r_T^- + \rho_T^-, r_T^+ + \rho_T^+, \Pi_{v_T}^{\exp_{v_T}\xi_{M, T}} \vartheta_T\right)\]
to $\cal{F} = 0$ such that $||(\xi_-, \xi_+, \xi_M, \rho_-, \rho_+, \vartheta)|| < \eps_0$ and with
\[\xi_-(s_1- 3T_1, t) \in TH^1,\quad \xi_M(s_2, t_2) \in TH^2,\quad \xi_+(s_3 + 3T_2, t) \in TH^3.\]
Moreover it satisfies the bound
\begin{equation}\label{gluing_estimate_III}||\xi_T||_{EZ^{2,k}_{S_1, 0, S_2}} \le C ||\cal{F}(u_T^-, 0, u_T^+)||_{EV^{2,k-1}_{S_1,0,S_2}}\end{equation}
for some constant $C$. In particular, this yields a smooth embedding
\[G: [T_0, \infty) \times [T_0, \infty) \to E\cal{M}(x^0, x^3)\]
with $G(T_1, T_2)$ converging in the Gromov topology to the broken trajectory $(u^1, (u^2, \varphi^2), u^3)$ as $T_1, T_2 \to \infty$. Moreover, after a change of coordinates $x_1 = \frac{1}{T_1 - T_0}, x_2 = \frac{1}{T_2 - T_0}$, the function
\[\delta : (\bb{R}_{>0})^2 \to \bb{R}\]
given by $\delta(x_1, x_2) = \delta(\hat{u}_T^-, 0, \hat{u}_T^+)$ extends continuously over $(\bb{R}_{\ge 0})^2$ so that $\delta > 0$ along $\{0\} \times \bb{R}_{> 0}$, $\delta < 0$ along $\bb{R}_{>0} \times \{0\}$, and $\delta(0,0) = 0$.

Finally, $G|_{\delta^{-1}(0)}$ is a surjection onto some open neighbourhood of the broken trajectory $(u^1, (u^2, \varphi^2), u^3) \in \bar{\cal{M}}(\bf{x}^0, \bf{x}^3)$, thus endowing this moduli space with a smooth codimension one $\delta$-structure.\end{theorem}

The smoothness of $G$ follows from a similar argument as in Theorem \ref{gluing_theorem_II}; the bound \eqref{gluing_estimate_III} follows since $Q_{u_T}$ is uniformly bounded, and this implies that $G(T_1, T_2)$ converges to $(u^1, (u^2, \varphi^2), u^3)$ in view of the estimate \eqref{pregluing_error_III}.

Injectivity of the gluing map follows by a similar argument exploiting our choice of local hypersurfaces, and surjectivity onto the space of actual solutions proceeds as before, with the crucial estimate coming from Proposition \ref{longstrip_phicontrol} (we do not have to prove surjectivity onto a neighbourhood of the broken trajectory in the extended moduli space, which in any case is more delicate due to the additional equivalence relation from extended solutions coming from unique continuation).

To show that $\delta$ extends continuously over $(\bb{R}_{\ge 0})^2$ (or equivalently, over $[T_0, \infty]^2$), let us hold $T_2$ constant and consider the limit
\[\lim_{T_1 \to \infty} G(T_1, T_2).\]
This limit is well-defined in the Gromov-Floer compactification of the \emph{extended} moduli space: namely, it has limit a broken trajectory $(v_{T_2}, (w_{T_2}^-, 0, w_{T_2}^+))$ where $v_{T_2} \in \cal{M}(\bf{x}^0, \bf{x}^1)$ and $(w_{T_2}^-, 0, w_{T_2}^+)$ is an extended solution defining a trajectory in $E\cal{M}(\bf{x}^1, \bf{x}^3)$. Strictly speaking earlier we only defined $E\cal{M}$ for interior critical points, but the definition here for $\bf{x}^1 \in \frak{C}_u$ is analogous: the new feature is that there are no nonzero $L^{2,k}$ solutions $r_- : (-\infty, 0] \to \bb{R}$ of
\[\frac{d}{ds}r_- + \Lambda(\varphi)r_- = 0;\]
in particular we automatically must have that $\delta(w_{T_2}^-, 0, w_{T_2}^+) = r_+(0) > 0$. It then suffices to show that $\delta(w_{T_2}^-, 0, w_{T_2}^+)$ varies continuously with $T_2$. However, by our construction, $(w_{T_2}^-, 0, w_{T_2}^+)$ must satisfy $\pi_M w(s_2, t_2) \in H^2$ and $w_{T_2}^+(s_3 + 3T_2, t_3) \in H^3$, and moreover by the gluing estimates \eqref{gluing_estimate_III}, \eqref{pregluing_error_III}, it is close to a pregluing of just the two trajectories $(u^2, \varphi^2)$ and $u^3$ in the $EZ^{2,k}$ norm. In particular, by the uniqueness of gluing, $(w_{T_2}^-, 0, w_{T_2}^+)$ must itself be the extended solution created by gluing $(u^2, \varphi^2)$ and $u^3$ with gluing parameter $T_2$ and a choice of right inverse prescribed by $H^2, H^3$: this operation is clearly continuous in $T_2$. Hence, $\delta$ extends continuously to $(\bb{R}_{\ge 0})^2$.

The final thing to note is that $\delta$ is indeed transverse to $0 \in \bb{R}$ over $(\bb{R}_{>0})^2$. Indeed, this is essentially immediate for dimension reasons: at a genuine solution $u$ considered inside the extended moduli space $E\cal{M}$, for a tangent vector $\xi = (\xi_-, \xi_+, \xi_M, \rho_-, \rho_+, \vartheta)$, it is not hard to see that $(d\delta)_u(\xi) = \rho_+(0) - \rho_-(0)$. In particular $\ker(d\delta) \subset T_u E\cal{M}$ exactly coincides with the tangent space $T_u \cal{M}$ of the ordinary moduli space, which is one dimension less than the dimension of $E\cal{M}$. Moreover, it is not hard to adjust the proof of the injectivity of the gluing map to prove $G$ is a local diffeomorphism onto $E \cal{M}$, from which we conclude that $(\delta \circ G)^{-1}(0)$ is transversely cut out away from $(0,0)$. This completes the proof of Theorem \ref{gluing_theorem_III}.

\section{K\"{u}nneth theorems}

\subsection{In polarization-twisted Floer theory.}
We now turn our attention towards proving a K\"{u}nneth theorem for polarized Floer theory. Suppose $M, M'$ are two exact symplectic manifolds, each with exact Lagrangians $L_0, L_1; L_0', L_1'$, equipped with polarization data $\frak{p} = (E, F_0, F_1)$ and $\frak{p}' = E', F_0', F_1'$. This gives rise to a polarization $\frak{p}\oplus \frak{p}' = (E \oplus E', F_0 \oplus F_0', F_1 \oplus F_1')$ on $M \times M'$ with the Lagrangians $L_0 \times L_0', L_1 \times L_1'$.

\begin{theorem}\label{kunnethI}There is an isomorphism in the derived category of $\bb{F}_2[t]$-modules
\begin{equation}K: CF(L_0, L_1; \frak{p}) \tensor^{\bb{L}}_{\bb{F}_2[t]} CF(L_0', L_1'; \frak{p}') \xrightarrow{\sim} CF(L_0 \times L_0', L_1 \times L_1'; \frak{p} \oplus \frak{p}').\end{equation}\end{theorem}

The need to state this equivalence on the level of the derived category comes not just from chain-level tensor product, but also because the map comes by exploiting the $A_{\infty}$-equivalence of Section 2.1 between $A_{\bb{F}_2}$ and $(A[t], d_{borel})$ for $A$ a free $\bb{F}_2[\bb{Z}/2]$-complexes of finite type. Indeed, in the course of the proof, we will define a $\bb{F}_2$-bilinear map
\begin{equation}\label{kunnethmap}\kappa: CF(L_0, L_1, \frak{p}) \tensor CF(L_0', L_1', \frak{p}') \to CF(L_0 \times L_0', L_1 \times L_1'; \frak{p} \oplus \frak{p}')\end{equation}
which is only $\bb{F}_2[t]$-bilinear up to homotopy.

As a special case of Theorem \ref{kunnethI}, suppose that $M', L_0', L_1' = \{pt\}$, and that $E' = \bb{C}, F_0' = \bb{R}, F_1' = i \bb{R}$:

\begin{corollary}\label{stabilization_invariance}Polarized Floer cohomology is invariant under stabilization: writing $\frak{p}\oplus \bb{C} = (E \oplus \bb{C}, F_0 \oplus \bb{R}, F_1 \oplus i \bb{R})$, we have
\begin{equation}HF(L_0, L_1; \frak{p}) \cong HF(L_0, L_1; \frak{p}\oplus \bb{C})\end{equation}\end{corollary}

Before we prove Theorem \ref{kunnethI}, let us make an important observation: if $I_t, I_t'$ are regular complex structures on $E, E'$ respectively, the block diagonal complex structure $I \oplus I' = \begin{pmatrix}I_t & 0 \\ 0 & I_t'\end{pmatrix}$ on $E \oplus E'$ need not be regular, even if the eigenvalues of $I \frac{d}{dt}$ and $I' \frac{d}{dt}$ are distinct (so that the assumptions of \ref{localgeomI} still hold in the product).

This is an artefact of the asymptotic conditions imposed on the twisted equations. Indeed, assume that $\lambda_-, \lambda_+$ are eigenvalues of $I \frac{d}{dt}$ at $x_-, x_+ \in L_0 \cap L_1$ respectively, and additionally take $x_-', x_+' \in L_0' \cap L_1'$. Then $\lambda_-, \lambda_+$ are also eigenvalues of $(I \oplus I') \frac{d}{dt}$ at $(x_-, x_-'), (x_+, x_+')$ respectively. Given pseudoholomorphic strips $u, u'$ on $M, M'$ with limits at $x_{\pm}, x_{\pm}'$, twisted solutions in $E \oplus E'$ over $u \times u'$ with asymptotics of type $\lambda_-, \lambda_+$ are exactly given by pairs $(\phi, \phi')$ of twisted solutions on $E, E'$ respectively, such that $\phi$ has asymptotics of type $\lambda_-, \lambda_+$, and that $\phi'$ has asymptotics of type $\lambda_-', \lambda_+'$ for some $\lambda_-' < \lambda_-$ and $\lambda_+' > \lambda_+$.

However, if
\begin{equation}\mu(u') + \specflow(u') + i(\lambda_-') - i(\lambda_+') < 0 \end{equation}
then the linearized equation for such $(u', \phi')$ has negative index, and so the moduli space can never be transversely cut out. This can happen even if the expected dimension of twisted flows in $M \times M', E \oplus E'$ between $\lambda_-, \lambda_+$ is non-negative, thus yielding an index-theoretic obstruction to the regularity of $I \oplus I'$.

Now, take the diagonal time-dependent almost complex structure $J''_t = J_t \oplus J'_t$ on $M \times M'$, and choose a regular complex structure $I''$ on $E \oplus E'$. We can assume that in an arbitrarily small neighbourhood of $(L_0 \times L_0') \cap (L_1 \times L_1')$, this agrees to the diagonal almost complex structure $I \oplus I'$, since the only solutions to the twisted equations $(u, \phi)$ with $u(Z)$ contained within such neighbourhoods are the solutions with constant $u$, which are automatically regular. 

To define the map \eqref{kunnethmap}, we also need to choose a $Z$-dependent complex structure $\{\tilde{I}_{s, t}\}_{s, t \in \bb{R} \times [0,1]}$ on $E \oplus E'$, such that
\begin{equation}\tilde{I}_{s, t} = \begin{cases}I''_t & \text{ for } s \ll 0\\
I_t \oplus I'_t & \text{ for } s \gg 0.\end{cases}\end{equation}
We also choose a $Z$-dependent $\tilde{I}_{s,t}$-Hermitian connection on $E \oplus E'$, which we will just call $\nabla$ and hope there is minimal confusion with the connections already chosen.

Fix $\bf{x}_+ = (x_+, \lambda_+) \in \frak{C}(L_0, L_1; \frak{p})$, $\bf{x}'_+ = (x'_+, \lambda'_+) \in \frak{C}(L_0', L_1'; \frak{p}')$ and $\bf{x}''_- = (x''_-, \lambda''_-) \in \frak{C}(L_0 \times L_0', L_1 \times L_1'; \frak{p}\oplus \frak{p}')$. Consider the equations for a map $\tilde{u} : Z \to M \times M'$ and a section $\tilde{\phi}$ of $E \oplus E'$, which we will often refer to as the $\tilde{I}_{s,t}$-continuation equation:
\begin{equation}\label{tildeIstcontinuation}\partial_s \tilde{u} + J''_t \partial_t \tilde{u} = 0, \qquad \nabla_s \tilde{\phi} + \tilde{I}_{s,t} \nabla_t \tilde{\phi} = 0\end{equation}
subject to the usual boundary conditions that $\tilde{u}(\cdot, i) \in L_i \times L_i'$ and $\tilde{\phi}(\cdot, i) \in F_i \oplus F_i'$ for $i = 0,1$. We impose that $\tilde{u}$ has finite energy with limits $\tilde{u} \to x_+ \times x'_+ \in (L_0 \times L_0') \cap (L_1 \times L_1')$ as $s \to \infty$, and $\tilde{u} \to x''_- \in (L_0 \times L_0') \cap (L_1 \times L_1')$ as $s \to -\infty$. We also impose the following asymptotics on $\tilde{\phi}$:
\begin{equation}\label{kunnethasymptotics}\tilde{\phi}(s,t) \sim \begin{cases}(C e^{- \lambda_+ s} \psi_+(t), C' e^{- \lambda'_+} \psi'_+(t)) \in E_{x_+} \oplus E'_{x'_+} &\text{ as } s \to \infty, \\
C'' e^{- \lambda''_- s} \psi''_-(t) \in (E \oplus E')_{x''_-} &\text{ as } s \to -\infty.\end{cases}\end{equation}
This can be set up as a Fredholm problem as follows: choose a small enough $\delta > 0$ (so that there are no other eigenvalues in the intervals $(\lambda - \delta, \lambda), (\lambda' - \delta, \lambda')$ and $(\lambda'', \lambda''+\delta)$). Then take two weight functions $w, w': \bb{R} \to \bb{R}$ such that
\begin{equation}w(s) = \begin{cases}(\lambda'' + \delta) s & \text{for } s \ll 0\\ (\lambda - \delta) s & \text{for } s \gg 0;\end{cases} \qquad w'(s) = \begin{cases}(\lambda'' + \delta) s & \text{for } s \ll 0\\ (\lambda' - \delta) s & \text{for } s \gg 0.\end{cases}\end{equation}
Given solutions $u, u'$ to the $J-$ and $J'$-holomorphic strip equations respectively on $M, M'$ with the relevant boundary conditions, we then construct Sobolev spaces
\begin{align*}W^{2,k}_{\kappa}(u^*E, u'^*E') &= \{(\phi, \phi') \in W^{2,k}_{loc}(u^*E \oplus u'^* E') : e^{w(s)}\phi \in W^{2,k}, e^{w'(s)}\phi' \in W^{2,k}\} \\
L^{2,k-1}_{\kappa}(u^*E, u'^*E') &=  \{(\phi, \phi') \in L^{2,k-1}_{loc}(u^*E \oplus u'^* E') : e^{w(s)}\phi \in L^{2,k-1}, e^{w'(s)}\phi' \in L^{2,k-1}\}.\end{align*}
In this case, $\bar{\nabla}_{I''}: W^{2,k}_{\kappa} \to L^{2,k-1}_{\kappa}$ is a Fredholm map, whose kernel consists of solutions $\tilde{\phi}$ to \eqref{tildeIstcontinuation} with the asymptotics \eqref{kunnethasymptotics}.

We then can define a moduli space $\cal{M}^{\kappa}(\bf{x}''_-, \bf{x}_+, \bf{x}'_+)$ of nonzero solutions $(\tilde{u}, \tilde{\phi})$ to \eqref{tildeIstcontinuation} with these asymptotics, modulo rescaling $\tilde{\phi}$. For generic choice of $\tilde{I}_{s,t}$, this will be a smooth manifold of dimension
\begin{equation}\mu(\tilde{u}) + \specflow_{\tilde{I}_{s,t}}(\tilde{u}) + i(\lambda_-'') - i(\lambda_+) - i(\lambda_+').\end{equation}
where $\specflow_{\tilde{I}_{s,t}}(\tilde{u})$ is the spectral flow of the family of operators $\tilde{I}_{s,t}(\tilde{u}(s,t))\nabla_{\partial/\partial t}$.

Such a trajectory $\tilde{u}, \tilde{\phi}$ can degenerate into a ``broken triangle'' configuration, meaning the following four pieces of data:
\begin{itemize}\item a sequence of trajectories
\begin{equation}\label{brokentrianglei}(u_1, \phi_1) \in \cal{M}(\bf{x}_1, \bf{x}_2), (u_2, \phi_2) \in \cal{M}(\bf{x}_2, \bf{x}_3), \hdots, (u_k, \phi_k) \in \cal{M}(\bf{x}_k, \bf{x}_+)\end{equation}
solving the $(J, I)$-twisted equations on $(M, E)$;
\item a sequence of trajectories
\begin{equation}\label{brokentriangleii}(u'_1, \phi'_1) \in \cal{M}(\bf{x}'_1, \bf{x}'_2), (u'_2, \phi'_2) \in \cal{M}(\bf{x}'_2, \bf{x}'_3), \hdots, (u'_{\ell}, \phi'_{\ell}) \in \cal{M}(\bf{x}'_{\ell}, \bf{x}'_+)\end{equation}
solving the $(J', I')$-twisted equations on $(M', E')$;
\item a single solution
\begin{equation}\label{brokentriangleiii}(\tilde{u}, \tilde{\phi}) \in \cal{M}^{\kappa}(\bf{x}''_-, \bf{x}_+, \bf{x}'_+)\end{equation}
for the $\tilde{I}_{s,t}$-continuation equation \eqref{tildeIstcontinuation} on $(M \times M', E \oplus E')$;
\item a sequence of trajectories
\begin{equation}\label{brokentriangleiv}(u''_m, \phi''_m) \in \cal{M}(\bf{x}''_-, \bf{x}''_m), (u''_{m-1}, \phi''_{m-1}) \in \cal{M}(\bf{x}''_m, \bf{x}''_{m-1}), \hdots, (u''_1, \phi''_1) \in \cal{M}(\bf{x}''_2, \bf{x}''_1)\end{equation}
solving the $(J\times J', I'')$-twisted equations on $(M\times M, E \oplus E')$.\end{itemize}

By adding in all the broken triangle configurations, we can compactify the moduli space $\cal{M}^{\kappa}$. In particular, we obtain a map by counting the zero-dimensional part of $\cal{M}^{\kappa}$:
\begin{align}\kappa: CF(L_0, L_1, \frak{p}) \tensor CF(L_0', L_1', \frak{p}') &\to CF(L_0 \times L_0', L_1 \times L_1'; \frak{p} \oplus \frak{p}')\\
\bf{x}_+ \tensor \bf{x}'_+ &\mapsto \sum\limits_{\bf{x}''_-, [\tilde{u}]: \dim \cal{M}^{\kappa}_{[\tilde{u}]} = 0} \# \cal{M}(\bf{x}''_-, \bf{x}_+, \bf{x}_+') \cdot \bf{x}''_.\end{align}
By considering the boundaries of the one-dimensional components of $\cal{M}^{\kappa}$ and proving an appropriate gluing result, we see that $\kappa$ is a chain map, i.e.
\begin{equation}d \kappa(\bf{x}, \bf{x}') + \kappa(d \bf{x}, \bf{x}') + \kappa(\bf{x}, d \bf{x}') = 0.\end{equation}
It is, however, not bilinear on the nose for the module structures.

To remedy this, by distinguishing between the two choices of unit eigenvector for each $\bf{x}_+, \bf{x}'_+, \bf{x}''_-$, we can lift $\kappa$ to a map
\begin{equation}\tilde{\kappa}: \widetilde{CF}(L_0, L_1, \frak{p}) \tensor \widetilde{CF}(L_0', L_1', \frak{p}') \to \widetilde{CF}(L_0 \times L_0', L_1 \times L_1'; \frak{p} \oplus \frak{p}').\end{equation}
This is again a chain map, but moreover it is actually a $\bb{F}_2[\bb{Z}/2]$-module map, since solutions to \eqref{tildeIstcontinuation} modulo $\bb{R}_{> 0}$ come in pairs $(\tilde{u}, \tilde{\phi}), (\tilde{u}, -\tilde{\phi})$. For simplicity, let us write
\begin{equation*}A = \widetilde{CF}(L_0, L_1, \frak{p}), \quad A' = \widetilde{CF}(L'_0, L'_1, \frak{p}'), \quad A'' = \widetilde{CF}(L_0\times L'_0, L_1 \times L_1', \frak{p}\oplus \frak{p}')\end{equation*}
which are all free $\bb{F}_2[\bb{Z}/2]$-complexes of finite type; in particular we have quasi-isomorphisms of $A_{\infty}$-modules
\begin{align*}F: CF(L_0, L_1, \frak{p}) &= A_{\bb{F}_2} \xrightarrow{\sim} (A[t], d_{borel})\\
F': CF(L'_0, L'_1, \frak{p}') &= A'_{\bb{F}_2} \xrightarrow{\sim} (A'[t], d_{borel})\\
F'': CF(L_0\times L_0', L_1 \times L_1', \frak{p} \oplus \frak{p}') &= A''_{\bb{F}_2} \xrightarrow{\sim} (A''[t], d_{borel}).\end{align*}
Proposition 2.3 gives an identification $A[t] \tensor^{\bb{L}}_{\bb{F}_2[t]} A'[t] \cong (A\tensor A')[t]$, and the assignment $A \mapsto (A[t], d_{borel})$ is functorial, so we also have a map of $\bb{F}_2[t]$-complexes
\begin{equation}\tilde{\kappa}_{borel}: ((A \tensor A')[t], d_{borel}) \to (A''[t], d_{borel}).\end{equation}
\begin{definition}The Kunneth map $K$ is then the composition
\begin{equation}K: A_{\bb{F}_2} \tensor^{\bb{L}}_{\bb{F}_2[t]} A'_{\bb{F}_2} \xrightarrow{F \tensor F'} A[t] \tensor^{\bb{L}}_{\bb{F}_2[t]} A'[t] \xrightarrow{(2.3)} (A \tensor A')[t] \xrightarrow{\tilde{\kappa}_{borel}} A''[t] \xrightarrow{(F'')^{-1}} A''_{\bb{F}_2}\end{equation}
where the last map $(F'')^{-1}$ is to be thought of as a formal inverse of the quasi-isomorphism $F''$ in the derived category of $\bb{F}_2[t]$-modules.\end{definition}

\begin{remark}Both geometrically and algebraically, the more natural K\"{u}nneth map is $\tilde{\kappa}$, on the level of $\bb{F}_2[\bb{Z}/2]$ modules; however we have kept our presentation to $\bb{F}_2[t]$-modules to stay in line with traditional presentations of equivariant cohomology.\end{remark}

Although the construction of $K$ as a chain-level $\bb{F}_2[t]$-map is more involved, on cohomology it just produces $\kappa$: if $\alpha \in H(A_{\bb{F}_2})$ and $\alpha' \in H(A'_{\bb{F}_2})$ are cohomology classes, then
\begin{equation*}\kappa(\alpha \tensor \alpha') = K(\alpha \tensor \alpha') \in H(A''_{\bb{F}_2}).\end{equation*}

To see that $K$ in fact a quasi-isomorphism, observe that the symplectic action on $M \tensor M'$ provides a filtration of each of the above complexes, and each of the maps in the above composition necessarily preserve this filtration (this is essentially tautological except for the map $\tilde{\kappa}_{borel}$; however recall that the `untwisted' part of the equation defining $\kappa$ is exactly the pseudoholomorphic curve equation in $M \times M'$). It thus suffices to prove that we have a quasi-isomorphism on the subquotients, each of which corresponds to the symplectic action of an intersection point of $(L_0 \times L'_0) \cap (L_1 \times L'_1)$. Since there are no pseudoholomorphic strips connecting different intersection points of the same action, it suffices to prove that $K$ is a quasi-isomorphism for the case when $M = M' = \pt$.

This can be verified explicitly; recall that each $CF(\pt, \frak{p}) \cong \bb{F}_2[t,t^{-1}]$ as a $\bb{F}_2[t]$-module, with zero differential. Choose any $\bf{x}_+ \in CF(\pt, \frak{p})$, $\bf{x}'_+ \in CF(\pt, \frak{p}')$; there is then a unique $\bf{x}''_- \in CF(\pt, \frak{p}\oplus \frak{p}')$ such that $\dim \cal{M}^{\kappa}(\bf{x}''_-, \bf{x}_+, \bf{x}'_+) = 0$.

Any solution $(\tilde{u}, \tilde{\phi})$ of the $\tilde{I}_{s,t}$-equations has $\tilde{u}$ constant. The moduli space $\cal{M}^{\kappa}(\bf{x}''_-, \bf{x}_+, \bf{x}'_+)$ is then exactly the projectivization of the vector space
\begin{equation*}\ker(\bar{\nabla}_{\tilde{I}_{s,t}} : W^{2,k}_{\kappa} \to L^{2,k-1}_{\kappa})\end{equation*}
which in this case is one-dimensional, since $\ind(\bar{\nabla}_{\tilde{I}_{s,t}}) = 1$ for this choice of asymptotics $\bf{x}_+, \bf{x}'_+, \bf{x}''_-$, and for a regular $\tilde{I}_{s,t}$ we have $\coker(\bar{\nabla}) = 0$. In particular, $\cal{M}^{\kappa}$ is just a single point; and so
\begin{equation*}\kappa(\bf{x}_+, \bf{x}'_+) = \bf{x}''_-.\end{equation*}
In particular this implies that $\kappa$ must be given by
\begin{equation*}\bb{F}_2[t,t^{-1}] \tensor \bb{F}_2[t, t^{-1}] \to \bb{F}_2[t,t^{-1}], \qquad t^i \tensor t^j \mapsto t^{i+j+n}\end{equation*}
where the fixed integer $n$ is the spectral flow of the family of operators $\tilde{I}_{s,t}\frac{d}{dt}$. As a consequence, $K$ is an isomorphism, which concludes the proof of Theorem \ref{kunnethI}.

\subsection{In equivariant Floer theory.}
Now, suppose that $\tilde{M}'$ is another exact symplectic manifold with a symplectic involution $\iota'$, with exact Lagrangians $\tilde{L}_0', \tilde{L}_1'$ preserved by $\iota'$, satisfying the same conditions as before. Write $M', L_0', L_1'$ for the corresponding fixed point sets, which has polarization data $\frak{p}' = (E', F_0', F_1')$ given by the respective normal bundles of $M', L_0', L_1'$ to $\tilde{M}', \tilde{L}_0', \tilde{L}_1'$.

There is then a $\bb{Z}/2$ action on $\tilde{M} \times \tilde{M}'$ given by $\iota \times \iota'$, preserving the Lagrangians $\tilde{L}_0 \times \tilde{L}_0', \tilde{L}_1 \times \tilde{L}_1'$. The fixed point sets are respectively $M \times M', L_0 \times L_0', L_1 \times L_1'$, on which we have the product polarization data $\frak{p}\oplus \frak{p}' = (E \oplus E', F_0 \oplus F_0', F_1 \oplus F_1')$.

\begin{theorem}\label{kunnethII}There is a quasi-isomorphism of $A_{\infty}$-modules over $\bb{F}_2[t]$
\begin{equation}CF_{\bb{Z}/2}(\tilde{L}_0, \tilde{L}_1) \tensor^{\bb{L}}_{\bb{F}_2[t]} CF_{\bb{Z}/2}(\tilde{L}_0', \tilde{L}_1') \xrightarrow{\sim} CF_{\bb{Z}/2}(\tilde{L}_0 \times \tilde{L}_0', \tilde{L}_1 \times \tilde{L}_1').\end{equation}\end{theorem}

For brevity we will only state and prove this theorem for the flavor of equivariant cohomology coming from the $\check{C}$-complex: indeed the proof of the $\hat{C}$-complex version is analogous, and we have already proved the theorem for the $\bar{C}$-complex version as Theorem \ref{kunnethI}. Indeed, this proof closely tracks the earlier one.

Choose equivariant time-dependent almost complex structures $\tilde{J}_t, \tilde{J}'_t, \tilde{J}''_t$ on $\tilde{M}, \tilde{M}', \tilde{M}\times \tilde{M}'$ respectively. Assume that these are each equivariantly regular in the sense of Definition \ref{equivarianttransversality}, and that these all satisfy the conditions of Assumption \ref{localgeomII}. As before, choose as well a $Z$-dependent almost complex structure $\tilde{J}_{s,t}$ on $\tilde{M} \times \tilde{M}'$, such that
\begin{equation}\tilde{J}_{s,t} = \begin{cases}\tilde{J}''_t & \text{ for } s \ll 0 \\
\tilde{J}_t \times \tilde{J}'_t & \text{ for } s \gg 0.\end{cases}\end{equation}
In particular, along the fixed point set, we have induced complex structures $I_t, I'_t, I''_t$ on $E, E', E\oplus E'$ respectively, as well as an induced $Z$-dependent complex structure $\tilde{I}_{s,t}$.

Then, for $\bf{x}_+ \in \frak{C}(\tilde{L}_0, \tilde{L}_1), \bf{x}'_+ \in \frak{C}(\tilde{L}'_0, \tilde{L}'_1), \bf{x}''_- \in \frak{C}(\tilde{L}_0 \times \tilde{L}'_0, \tilde{L}_1 \times \tilde{L}'_1)$, we construct a moduli space $\cal{M}^{\kappa}(\bf{x}''_-, \bf{x}_+, \bf{x}'_+)$ in the following fashion, depending on whether the generators are interior, boundary-stable or boundary-unstable:

\emph{If each $\bf{x}''_-, \bf{x}_+, \bf{x}'_+ \in \frak{C}_o$ are interior generators}, then $\cal{M}^{\kappa}$ is the moduli space of solutions of the $\tilde{J}_{s,t}$-continuation map equation
\begin{equation}\label{tildeJstcontinuation}\partial_s \tilde{u} + \tilde{J}_{s,t}\partial_t \tilde{u} = 0\end{equation}
on $\tilde{M} \times \tilde{M}'$, with the usual boundary conditions, subject to the that asymptotics one of the two points of $\bf{x}''_-$ at $-\infty$, and one of the four points of $\bf{x}_+ \times \bf{x}'_+$ at $+\infty$, modulo the $\bb{Z}/2$ action given by $\iota \times \iota'$.

\emph{If each $\bf{x}''_-, \bf{x}_+, \bf{x}'_+ \in \frak{C}_s \cup \frak{C}_u$ are boundary critical points}, then we have the moduli space which was used to define the earlier Kunneth map
\begin{equation}\cal{M}^{\kappa, \partial}(\bf{x}''_-, \bf{x}_+, \bf{x}'_+)\end{equation}
of solutions to the $\tilde{I}_{s,t}$-continuation equations, modulo the $\bb{R}^*$ rescaling action, with boundary conditions as in \eqref{kunnethasymptotics}. These ``boundary moduli spaces'' are to be distinguished with another ``interior moduli space'', which exists only for the case that $\bf{x}''_- \in \frak{C}_u$ and $\bf{x}_+, \bf{x}'_+ \in \frak{C}_s$
\begin{equation}\cal{M}^{\kappa, o}(\bf{x}''_-, \bf{x}_+, \bf{x}'_+)\end{equation}
of interior (i.e. non-$\iota \times \iota'$-invariant) solutions $\tilde{u}$ to the $\tilde{J}_{s,t}$-continuation map equation \eqref{tildeJstcontinuation} above, modulo the $\bb{Z}/2$-action given by $\iota \times \iota'$, with the usual boundary conditions. Writing $\bf{x}''_- = (x''_-, \lambda''_-), \bf{x}_+ = (x_+, \lambda_+), \bf{x}'_+ = (x'_+, \lambda'_+)$, we impose the asymptotics that $\tilde{u} \to x''_-$ as $s \to -\infty$, and $\tilde{u} \to x_+ \times x'_+$ as $s \to +\infty$, with decay in the normal direction to the fixed point set given by
\begin{align}\pi_{E \oplus E'} \tilde{u}(s,t) \sim \begin{cases} C''_- e^{-\lambda''_- s}\psi''_-(t) & \text{ as } s \to -\infty\\
(C_+ e^{-\lambda_+ s} \psi_+(t), C'_+ e^{-\lambda'_+ s} \psi'_+(t)) \in E_{x_+} \oplus E'_{x'_+} & \text{ as } s \to +\infty\end{cases}\end{align}
where $\psi''_-, \psi_+, \psi'_+$ are eigensolutions corresponding to $\lambda''_-, \lambda_+, \lambda'_+$ respectively.

\emph{In the remaining cases, where at least one of $\bf{x}''_-, \bf{x}_+, \bf{x}'_+$ is an interior generator and at least one is a boundary generator}, we define $\cal{M}^{\kappa}$ to be the moduli space of solutions $\tilde{u}$ to the $\tilde{J}_{s,t}$-continuation equation, modulo $\bb{Z}/2$, asymptotics being a mixture of the two cases above. For example, when $\bf{x}''_- = \{x''_-, (\iota\times \iota')(x''_-)\} \in \frak{C}_o$ is interior, $\bf{x}_+ = (x_+, \lambda_+) \in \frak{C}_s$ is boundary-stable, and $\bf{x}'_+ = \{x'_+, \iota'(x'_+)\}$ is also interior, we impose that $\tilde{u} \to x''_-$ or $(\iota \times \iota')(x''_-)$ as $s \to -\infty$, and $\tilde{u} \to x_+ \times x'_+$ or $x_+ \times \iota'(x'_+)$ as $s \to +\infty$, together with decay in the normal direction to $M \subset \tilde{M}$ given by
\begin{equation*}\pi_E \tilde{u}(s,t) \sim C_+ e^{-\lambda_+ s} \psi_+(t) \qquad \text{as } s \to +\infty\end{equation*}
where $\psi_+$ is an eigensolution corresponding to $\lambda_+$. The other cases are entirely analogous.

There is a subtlety arising from certain constant solutions to the $\tilde{J}_{s,t}$-continuation equation. The constant solutions are exactly just $\tilde{u} = x\times x'$ for $x, x' \in \tilde{L}_0 \times \tilde{L}'_0 \cap \tilde{L}_1 \times \tilde{L}'_1$. Suppose that $x$ is invariant, while $x'$ is non-invariant; consequently $x \times x'$ will also be non-invariant, and there is a $\bb{Z}/2$-related pair of constant solutions at $x \times x'$ and $x \times \iota'(x')$. Even though this pair of constant solutions have $\pi_E \tilde{u}$ identically zero as $s \to +\infty$, we are to regard this solution as defining a single element of $\cal{M}^{\kappa}(\bf{x}'', \bf{x}, \bf{x}')$ where $\bf{x} = (x, \lambda_0)$ for $\lambda_0$ \emph{the lowest positive eigenvalue} of $I\frac{d}{dt}$ at $x$, and $\bf{x}' = \{x', \iota'(x')\}, \bf{x}'' = \{x \times x', x \times \iota'(x')\}$. For higher eigenvalues we declare the moduli space to be empty. Likewise, if $x'$ is invariant but $x$ is non-invariant, we similarly regard the pair of constant solutions as a single element of $\cal{M}^{\kappa}$ for $\bf{x}' = (x', \lambda'_0)$, where $\lambda'_0$ is the lowest positive eigenvalue.

\begin{remark}Indeed, we could circumvent this issue with the constant solutions by adding a small $Z$-dependent Hamiltonian perturbation to the $\tilde{J}_{s,t}$-equation, compactly supported on $Z$, to remove the constant solutions from consideration. Solutions of the perturbed equation would satisfy the  $\tilde{J} \times \tilde{J}'$ equation for $s \gg 0$, and would generically be non-constant, and decay in the normal direction $E$ to $M \subset \tilde{M}$ as $e^{-\lambda_0 s}\psi_0(t)$ for $\lambda_0$ the smallest positive eigenvalue. We could then consider a family of such Hamiltonian perturbations over $[0, \eps)$, with the original unperturbed equation at $0$. Since the original unperturbed equation defines a transversely cut out moduli space (in the normal sense of continuation moduli spaces for holomorphic curves on $\tilde{M} \times \tilde{M}'$), we could then prove a gluing result to get also a family of moduli spaces over $[0,\eps)$, which would be a topological submersion over $[0, \eps)$. In particular, since there is precisely one solution over $0$, there must also be just one solution for any sufficiently small Hamiltonian perturbation.\end{remark}

We can set each of these up as a Fredholm problem as before using Sobolev spaces with exponential weights (where the weights must be placed separately on $E$ and $E'$ on the positive infinite end of the strip). For generic choices of $\tilde{J}_{s,t}$, the moduli spaces are cut out transversely and are thus smooth manifolds. Moreover, these moduli spaces are only non-compact in that they can degenerate into a ``broken triangle'' configuration given by an element of one of the $\cal{M}^{\kappa}$, together with three broken trajectories on each of $\tilde{M}$, $\tilde{M}'$ and $\tilde{M} \times \tilde{M}'$ respectively as described earlier in \eqref{brokentrianglei}-\eqref{brokentriangleiv}.

Consequently, by counting the (interior) zero-dimensional moduli spaces, we obtain eight $\bb{F}_2$-homomorphisms:
\begin{align*}& \kappa_{ooo}: C_o \tensor C'_o \to C''_o, \quad \kappa_{uoo}: C_o \tensor C'_o \to C''_u, \\
&\kappa_{oos}: C_o \tensor C'_s \to C''_o, \quad \kappa_{uos} : C_o \tensor C'_s \to C''_u, \\
&\kappa_{oso}: C_s \tensor C'_o \to C''_o, \quad \kappa_{uso} : C_s \tensor C'_o \to C''_u, \\
&\kappa_{oss} : C_s \tensor C'_s \to C''_o, \quad \kappa_{uss}: C_s \tensor C'_s \to C''_u\end{align*}
where we have written $C_o, C'_o, C''_o$ and so on to distinguish between the free vector spaces generated by $\frak{C}_o(\tilde{L}_0, \tilde{L}_1)$, $\frak{C}_o(\tilde{L}'_0, \tilde{L}'_1)$ and $\frak{C}_o(\tilde{L}_0\times\tilde{L}'_0, \tilde{L}_1 \times \tilde{L}'_1)$ respectively. Similarly we have another eight $\bb{F}_2$-homomorphisms
\begin{equation}\kappa^{\partial}_{sss}, \kappa^{\partial}_{uss}, \kappa^{\partial}_{ssu}, \kappa^{\partial}_{usu}, \kappa^{\partial}_{sus}, \kappa^{\partial}_{uus}, \kappa^{\partial}_{suu}, \kappa^{\partial}_{uuu}\end{equation}
counting the boundary moduli spaces: these are all just the different components of the earlier map $\kappa$ of \eqref{kunnethmap}, but with the positive (stable) and negative (unstable) eigenvalues distinguished.

These can then be assembled into a chain map, $\kappa$,
\[\kappa : \check{C} \tensor \check{C}' \to \check{C}''\]
which counts ``total index zero'' possibly broken triangle configurations, to account for the  \emph{boundary-obstructed} cases. The are now four different types of boundary obstructed components of a broken triangle. The first are those counted by $d^{\partial}_{su}$ as before on each of $C_u \to C_s$, $C'_u \to C'_s$ and $C''_u \to C''_s$; as before these can be thought of as having ``formal index 0'' as blown up trajectories, despite having index 1 as boundary trajectories.

But we can also have boundary-obstructed triangles. Two types of such triangles are counted by the operators $\kappa^{\partial}_{sus}$ and $\kappa^{\partial}_{ssu}$; the triangles counted by these are of formal index -1. Furthermore, we also have those counted by $\kappa^{\partial}_{suu}$, which are of formal index -2.

With this in mind, we then define
\[\kappa : (C_o \tensor C'_o) \oplus (C_o \tensor C'_s) \oplus (C_s \tensor C'_o) \oplus (C_s \tensor C'_s) \to C''_o \tensor C''_s\]
as the \emph{transpose} of the matrix
\[\begin{pmatrix}\kappa_{ooo} & d^{\partial}_{su}\kappa_{uoo} + \kappa^{\partial}_{suu}(d_{uo}\tensor d_{uo}) \\
\kappa_{oos} & d^{\partial}_{su}\kappa_{uos} + \kappa^{\partial}_{sus}(d_{uo} \tensor \id) \\
\kappa_{oso} & d^{\partial}_{su}\kappa_{uso} + \kappa^{\partial}_{ssu}(\id \tensor d_{uo}) \\
\kappa_{oss} & d^{\partial}_{su}\kappa_{uss} + \kappa^{\partial}_{sus}(d_{us} \tensor \id) + \kappa^{\partial}_{ssu}(\id \tensor d_{us}) + \kappa^{\partial}_{suu}(d_{us} \tensor d_{us})
\end{pmatrix}.\]
This then defines a chain map:
\begin{equation*}\kappa: CF_{\bb{Z}/2}(\tilde{L}_0, \tilde{L}_1) \tensor_{\bb{F}_2} CF_{\bb{Z}/2}(\tilde{L}'_0, \tilde{L}'_1) \to CF_{\bb{Z}/2}(\tilde{L}_0 \times \tilde{L}'_0, \tilde{L}_1 \times \tilde{L}'_1).\end{equation*}
We will omit the proof that this is a chain map. As always, it is based on considering compactifications of one-dimensional moduli spaces  by broken configurations, and showing the number of such additional points in the compactification is even. Due to the presence of boundary-obstructed trajectories and triangles, the compactified moduli space is not a manifold with boundary, but rather carries a $\delta$-structure, which still allows us to deduce the appropriate algebraic relations. The reader should compare Section 25 of \cite{KronheimerMrowka07}.

To incorporate the $\bb{F}_2[t]$-module structures, we lift this map to the level of the free $\bb{F}_2[\bb{Z}/2]$-complexes following the scheme used in the proof of Theorem \ref{kunnethI}. From this we obtain a map in the derived category of $\bb{F}_2[t]$-modules
\begin{equation}K : CF_{\bb{Z}/2}(\tilde{L}_0, \tilde{L}_1) \tensor^{\bb{L}}_{\bb{F}_2[t]} CF_{\bb{Z}/2}(\tilde{L}'_0, \tilde{L}'_1) \to CF_{\bb{Z}/2}(\tilde{L}_0 \times \tilde{L}'_0, \tilde{L}_1 \times \tilde{L}'_1).\end{equation}
compatible with $\kappa$ in the sense that if $\alpha \in HF_{\bb{Z}/2}(\tilde{L}_0, \tilde{L}_1)$ and $\alpha' \in HF_{\bb{Z}/2}(\tilde{L}'_0, \tilde{L}'_1)$ are cohomology classes, then $K(\alpha \tensor \alpha') = \kappa(\alpha \tensor \alpha')$.

To see that $K$ is a quasi-isomorphism, again note that each side can be filtered by the symplectic action on $\tilde{M} \tensor \tilde{M}'$, and $K$ preserves this filtration. It thus suffices to check that $K$ is an isomorphism on the subquotients, in other words on the complexes corresponding to a pair of intersection points $x \in \tilde{L}_0 \cap \tilde{L}_1, x' \in \tilde{L}'_0 \cap \tilde{L}'_1$ and their conjugates by $\iota$ and $\iota'
$, with differentials and Kunneth maps computed using only the constant holomorphic curves on $\tilde{M} \tensor \tilde{M}'$.
\begin{itemize}\item In the case that $x, x'$ are each invariant critical points, we are asking for the map
\begin{equation*}K: \bb{F}_2[t] \tensor^{\bb{L}}_{\bb{F}_2[t]} \bb{F}_2[t] \to \bb{F}_2[t]\end{equation*}
to be an isomorphism. This map is induced from the map $\kappa$
\begin{equation*}\bb{F}_2[t] \tensor \bb{F}_2[t] \to \bb{F}_2[t]\end{equation*}
which by the same computation as in the proof of Theorem \ref{kunnethI}, is given by $t^i \tensor t^j \mapsto t^{i+j+n}$ where $n$ is the spectral flow of $\tilde{I}_{s,t}\frac{d}{dt}$. However by assuming that $\tilde{J}_{s,t}$ is sufficiently close to the diagonal almost complex structure $J_t \times J'_t$, at least in a small neighbourhood of $x \times x'$, we can assume this spectral flow is zero, and so $K$ is an isomorphism.
\item In the case that $x$ is invariant but $x'$ is non-invariant, or the other way round, the product $x\times x'$ is also non-invariant, and we are asking for the map
\begin{equation*}K: \bb{F}_2[t] \tensor^{\bb{L}}_{\bb{F}_2[t]} \bb{F}_2 \to \bb{F}_2\end{equation*}
to be an isomorphism. Again this map is induced from
\begin{equation*}\kappa: \bb{F}_2[t] \tensor \bb{F}_2 \to \bb{F}_2.\end{equation*}
However by the description of the contribution to $\cal{M}^{\kappa}$ from the constant solutions, we see that this map exactly sends $1 \tensor 1 \mapsto 1$ and $t^i \tensor 1 \mapsto 0$ for $i > 0$, from which we see that $K$ is an isomorphism.
\item In the remaining case that both $x, x'$ are non-invariant, we have two corresponding pairs of non-invariant points of $\tilde{M} \times \tilde{M}'$, namely $\{x\times x', \iota(x) \times \iota(x')\}$ and $\{x \times \iota'(x'), \iota(x) \times x'\}$. We ask for the map
\begin{equation*}K : \bb{F}_2 \tensor^{\bb{L}}_{\bb{F}_2[t]} \bb{F}_2 \to \bb{F}_2 \oplus \bb{F}_2\end{equation*}
to be an isomorphism. This we cannot see on the level of $\kappa$ (which is, in this case, zero), however its lift to free $\bb{F}_2[\bb{Z}/2]$-complexes, which is a map
\begin{equation*}\bb{F}_2[\bb{Z}/2] \tensor_{\bb{F}_2} \bb{F}_2[\bb{Z}/2] \to \bb{F}_2[\bb{Z}/2] \oplus \bb{F}_2[\bb{Z}/2]\end{equation*}
is certainly an isomorphism, and thus $K$ is as well.\end{itemize}

\section{Comparison to the Seidel-Smith model}

We now compare our definition of $HF_{\bb{Z}/2}(\tilde{L}_0, \tilde{L}_1)$, which we refer to as the ``Kronheimer-Mrowka model'', to the model due to Seidel-Smith in \cite{SeidelSmith10}, which goes via the Borel construction. We begin by briefly reviewing it.

\subsection{Review of the construction.} Consider the fibration
\begin{equation}\label{borelfibre}\tilde{M}_{borel} = \tilde{M} \times_{\bb{Z}/2} S^{\infty} \to \bb{R}P^{\infty}\end{equation}
which has fibre $\tilde{M}$. Take the standard Morse function $h$ on $\bb{RP}^{\infty}$ which pulls back to $h(z_1, z_2, \hdots) = z_1^2 + 2 z_2^2 + \hdots$ on $S^{\infty}$, and a family $\{\tilde{J}_z\}_{z \in \bb{RP}^{\infty}}$ of time-dependent almost complex structures on the fibres $\tilde{M}_z$ of \eqref{borelfibre}, each of contact type at infinity. These are not required to be equivariant; however they are locally constant around each critical point $z \in \bb{RP}^{\infty}$, and at these critical points we require that they are regular in the usual sense of non-equivariant Floer theory of $(\tilde{L}_0, \tilde{L}_1)$. We then define
\begin{equation}CF^{SS}_{\bb{Z}/2}(\tilde{L}_0, \tilde{L}_1) = \Pi_{z \in \Crit(h)} CF(\tilde{L}_0, \tilde{L}_1; \tilde{J}_z)\end{equation}
with a differential $d^{SS}$ that counts $(v,u)$, for $v : \bb{R} \to \bb{RP}^{\infty}$ and $u: Z \to \tilde{M}$ with $u(\cdot, i) \subset \tilde{L}_i$ for $i = 0,1$, which solve the coupled equations
\begin{equation}\frac{dv}{ds} + \nabla h(v(s)) = 0, \qquad \partial_s u + J_z(t,u) \partial_t u = 0\end{equation}
with asymptotics at $(z_-, x_-)$ and $(z_+, x_+)$ for $z_{\pm} \in \Crit(h)$ and $x_{\pm} \in \tilde{L}_0 \cap \tilde{L}_1$.

By counting trajectories $(v, u)$ such that $v(0)$ is constrained to lie in a generically chosen codimension one hyperplane of $\bb{R}P^{\infty}$, we obtain a $\bb{F}_2[t]$-module structure. This can be made more explicit: there is a canonical self-embedding $\tau: \bb{R}P^{\infty} \to \bb{R}P^{\infty}$ sending $(z_1, z_2, \hdots)$ to $(0, z_1, z_2, \hdots)$, satisfying $\tau^* h = h + 1$. We can choose the family $\{\tilde{J}_z\}$ of almost complex structures to satisfy $\tilde{J}_z = \tilde{J}_{\tau(z)}$ everywhere, in particular we have the same $\tilde{J}$ at each critical point of $h$. In particular this allows us to write the chain complex as
\begin{equation}CF_{\bb{Z}/2}^{SS}(\tilde{L}_0, \tilde{L}_1) = CF(\tilde{L}_0, \tilde{L}_1)\tensor \bb{F}_2[[t]]\end{equation}
with differential $d^{(0)} + d^{(1)}t + d^{(2)}t^2 + \hdots$, where $d^{(0)}$ is the ordinary differential on $CF(\tilde{L}_0, \tilde{L}_1; \tilde{J})$.

Consequently, there is a spectral sequence with $E_2$-page $HF(\tilde{L}_0, \tilde{L}_1)\tensor\bb{F}_2[[t]]$ converging to $HF_{\bb{Z}/2}^{SS}(\tilde{L}_0, \tilde{L}_1)$: this is \emph{not} a natural feature of the Kronheimer-Mrowka model.

We will not directly compare the Kronheimer-Mrowka model of equivariant Floer cohomology with Seidel-Smith's, in part because whilst our model is naturally finite rank over $\bb{F}_2[t]$, Seidel-Smith's theory is necessarily defined over $\bb{F}_2[[t]]$. Instead, we will work with finite rank truncations of Seidel-Smith's theory: we work instead over the standard $\bb{R}P^{n-1} \subset \bb{R}P^{\infty}$ consisting of just those points of the form $(z_1, \hdots, z_{n}, 0, 0, \hdots)$, and define
\begin{equation}CF^{(n)}_{\bb{Z}/2}(\tilde{L}_0, \tilde{L}_1) = \bigoplus\limits_{z \in \Crit(h)\cap \bb{R}P^{n-1}}CF(\tilde{L}_0, \tilde{L}_1; \tilde{J}_z)\end{equation}
with the same differential $d^{SS}$. This is naturally a quotient complex of $CF^{SS}$, and there are natural maps
\begin{equation}CF^{(n+1)}_{\bb{Z}/2}(\tilde{L}_0, \tilde{L}_1) \to CF^{(n)}_{\bb{Z}/2}(\tilde{L}_0, \tilde{L}_1).\end{equation}
We can recover $CF^{SS}$ as the inverse limit $\lim\limits_{\leftarrow} CF^{(n)}_{\bb{Z}/2}$.

To minimize confusion, for this subsection we will write $CF^{KM}_{\bb{Z}/2}$ and $HF^{KM}_{\bb{Z}/2}$ for the Kronheimer-Mrowka model of equivariant cohomology (which elsewhere we just refer to as $CF_{\bb{Z}/2}$ and $HF_{\bb{Z}/2}$).

\begin{theorem}\label{KMSScomparison}There is a quasi-isomorphism of $\bb{F}_2[t]$-modules
\begin{equation}CF^{(n)}_{\bb{Z}/2}(\tilde{L}_0, \tilde{L}_1) \cong CF^{KM}_{\bb{Z}/2}(\tilde{L}_0, \tilde{L}_1) \tensor^{\bb{L}}_{\bb{F}_2[t]} \frac{\bb{F}_2[t]}{(t^{n})}.\end{equation}
Moreover, these isomorphisms are compatible with the natural maps $CF^{(n+1)}_{\bb{Z}/2} \to CF^{(n)}_{\bb{Z}/2}$ on the one hand, and $\bb{F}_2[t]/(t^{n+1}) \to \bb{F}_2[t]/(t^{n})$ on the other. In particular there is an isomorphism in the inverse limit
\begin{equation}HF^{SS}_{\bb{Z}/2}(\tilde{L}_0, \tilde{L}_1) \cong HF^{KM}_{\bb{Z}/2}(\tilde{L}_0, \tilde{L}_1)^{\wedge}\end{equation}
where $HF^{KM}_{\bb{Z}/2}(\tilde{L}_0, \tilde{L}_1)^{\wedge}$ is the completion of $HF^{KM}_{\bb{Z}/2}(\tilde{L}_0, \tilde{L}_1)$ in $t$.\end{theorem}

In particular, we immediately deduce the inequality of Theorem \ref{smith_inequality}:
\begin{align*}\dim_{\bb{F}_2} HF(\tilde{L}_0, \tilde{L}_1) &\ge \rank_{\bb{F}_2[[t]]}HF^{SS}_{\bb{Z}/2}(\tilde{L}_0, \tilde{L}_1) \\
&= \rank_{\bb{F}_2[t]}HF^{KM}_{\bb{Z}/2}(\tilde{L}_0, \tilde{L}_1) \\
&= \rank_{\bb{F}_2[t,t^{-1}]} HF_{tw}(L_0, L_1; \frak{p})\end{align*}
by using the equivalence after localization of Theorem \ref{localization_theorem} of $HF^{KM}_{\bb{Z}/2}(\tilde{L}_0, \tilde{L}_1)$ and $HF_{tw}(L_0, L_1; \frak{p})$.

\subsection{Idea of the proof.} Suppose $\tilde{X}$ is a smooth manifold with a $\bb{Z}/2$-action, with fixed point set $X$. Recall that to implement the Kronheimer-Mrowka model here, we first take the oriented real blowup $\tilde{Y}$ of $\tilde{X}$ along $X \subset \tilde{X}$, and then take the now free $\bb{Z}/2$-quotient, $Y$. An invariant Morse function $f$ on $\tilde{X}$ then induces a flow on $Y$ which is everywhere tangent to the boundary; we then apply the formalism of Morse theory for manifolds with boundary.

On the other hand, the (truncated) Seidel-Smith model is to take the fibration
\begin{equation}\tilde{X}^{(n)} = X \times_{\bb{Z}/2} S^{n-1} \to \bb{R}P^{n-1}\end{equation}
and couple the Morse theory of $\bb{R}P^{n-1}$ with Morse theory in each of the fibres.

To interpolate between the two, consider the ball $B^{n}$ with the $\bb{Z}/2$-action given by $-1$. Put a $\bb{Z}/2$-equivariant Morse function on $B^{n}$, which is equal to the standard Morse function $h$ used earlier on its boundary $S^{n-1}$, and such that its flow is everywhere tangent to the boundary, and with a single interior critical point at $0$, which we take to be a minimum.

Explicitly, choose a smooth function $a : \bb{R}_{\ge 0} \to \bb{R}$ such that: (i) $a(r) = \frac{1}{2}r^2$ for sufficiently small $r$, (ii) $a(r)$ is increasing for $0 < r < 1$; (iii) $a(r) = 1 - (r-1)^2$ in a neighbourhood of $r = 1$. Then the function $g : B^n \to \bb{R}$ given by
\begin{equation}g(z) = \frac{1}{||z||^2}a(||z||)h(z),\end{equation}
where $h(z) = z_1^2 + 2z_2^2 + \hdots + n z_n^2$, has the desired properties. We will soon want to also extend $g$ to all of $\bb{R}^n$. To do this, use the same formula as above, and impose two further conditions on $a(r)$: (iv) $a(r)$ is decreasing for $1 < r < 2$; and (v) $a(r) = -r^2$ for $r \ge 2$.

Take the product $\tilde{X} \times B^{n}$. This also has a $\bb{Z}/2$-action, with fixed point set exactly $X \times 0$. Suppose $\tilde{Z}$ is the oriented real blow-up of $\tilde{X} \times B^{n}$ along $X \times 0$, and let $Z$ be the free quotient of this by $\bb{Z}/2$. The map $Z \to B^{n}/(\bb{Z}/2)$ is exactly the truncated Borel fibration $\tilde{X}^{(n)} \to \bb{R}P^{n-1}$ over the boundary, and over $0$ it is just the Kronheimer-Mrowka model manifold $Y$.

Now, put a $\bb{Z}/2$-equivariant gradient-like flow on $\tilde{Z}$ which covers the Morse flow on $B^{n}$, such that on the fibre over $0$ it agrees with the Kronheimer-Mrowka flow, and on boundary $\tilde{X} \times S^{n-1} \to S^{n-1}$ agrees with the Seidel-Smith flow. We can then define a map from (a stabilized version of) the Kronheimer-Mrowka chain complex to the Seidel-Smith chain complex, by counting pairs of trajectories on $\tilde{Z}$ which cover Morse flows in $B^{n}$ connecting the boundary to the minimum at $0$. 

\subsection{Implementation in Floer theory.} Consider again the Morse function $g : \bb{R}^n \to \bb{R}$ now extending to all $\bb{R}^n$ as explained above. For $\eps > 0$, take the graph of $\eps g$, $\Gamma_{\eps} = \{(z, \eps dg_z) \in T^*\bb{R}^n\}$, as a Lagrangian in the symplectic manifold $T^* \bb{R}^n$. It is equivariant for the $\bb{Z}/2$ action on $T^* \bb{R}^n$ given by multiplication by $-1$, since $g$ is; and the same is true for the zero section $\bb{R}^n \subset T^* \bb{R}^n$. We are interested in the Kronheimer-Mrowka model equivariant chain complex
\begin{equation}CF_{\bb{Z}/2}^{KM}(\bb{R}^n, \Gamma_{\eps}).\end{equation}
The fixed point set of the $\bb{Z}/2$-action is just a single point $\{0\} \subset T^* \bb{R}^n$, which is also an intersection of $\bb{R}^n$ and $\Gamma_{\eps}$. The non-invariant intersection points are precisely the non-invariant critical points of $g$, which are $\pm y_1 = (\pm 1, 0, \hdots, 0), \hdots, \pm y_n = (0, \hdots, 0, \pm 1)$. At $0$, the Lagrangian subspace $T_0 \Gamma_{\eps}$ is the graph of the diagonal matrix with entries $\eps, 2\eps, \hdots, n \eps$; accordingly the standard almost complex structure $J$ on $T^* \bb{R}^n$ satisfies the conditions of Assumption \ref{localgeomI}, and the lowest $n$ positive eigenvalues of $J \frac{d}{dt}$ are
\begin{equation*}\lambda_0 = \tan^{-1}(\eps), \lambda_1 = \tan^{-1}(2\eps), \hdots, \lambda_{n_1} = \tan^{-1}(n \eps).\end{equation*}
Write $\bf{x}_0, \bf{x}_1, \hdots$ for the ``boundary-stable'' generators of $CF^{KM}_{\bb{Z}/2}(\bb{R}^n, \Gamma_{\eps})$, the first $n$ of which are given by the eigenvalues above; write $\bf{y}_1, \hdots, \bf{y}_n$ for the ``interior'' generators given by the pairs $\{y_i, -y_i\}$.

Observe that the complex is naturally graded, with $|\bf{y}_i| = i$ and $|\bf{x}_i| = i$. Moreover, the symplectic action of a point of $\bb{R}^n \cap \Gamma_{\eps}$ is exactly the value of $g$ at the corresponding critical point. Since $g(y_i) = \eps i > 0$ while $g(0) = 0$, the generators $\bf{y}_i$ form a subcomplex $B$, and the quotient complex $A$ is given by the generators $\bf{x}_i$ and the differentials between them. The total complex $CF^{KM}_{\bb{Z}/2}(\bb{R}^n, \Gamma_{\eps})$ is then recovered as $\Cone(A \to B)$.

If $\phi_t$ is the Hamiltonian flow on $T^* \bb{R}^n$ induced by $g$, take the almost complex structures $J_t = (\phi_t)_* J (\phi_t)^{-1}_*$; these are $\bb{Z}/2$-equivariant. Then, finite energy Morse gradient flows $v : \bb{R} \to \bb{R}^n$ of $g$ with endpoints at two critical points of $g$ give rise to finite energy $J_t$-holomorphic strips $u: Z \to T^*\bb{R}^n$ with boundaries on $\bb{R}^n$ and $\Gamma_{\eps}$, and endpoints at the corresponding points of $\bb{R}^n \cap \Gamma_{\eps}$, via
\begin{equation*}v(s) \mapsto u(s,t) = \phi_t v(s).\end{equation*}
Moreover, by a result of Floer \cite{Floer89}, if $\eps$ is sufficiently small then this is actually a bijection onto all the finite energy $J_t$-holomorphic strips (although Floer's result is for compact Lagrangians, our assumptions on the geometry of $\Gamma_{\eps}$ at infinity confine holomorphic strips with endpoints in $\bb{R}^n \cap \Gamma_{\eps}$ to a compact set).

We can then completely characterize $CF_{\bb{Z}/2}(\bb{R}^n, \Gamma_{\eps}) = Cone(A \to B)$ as an $\bb{F}_2[t]$-module complex. Between each of $y_{i+1}, -y_{i+1}$ and $y_i, -y_i$ there is exactly one flow line of $g$, showing that $d \bf{y}_i = 0$ and $T \bf{y}_i = y_{i+1}$; thus $B \cong \bb{F}_2[t]/(t^n)$ (with a shift by $-1$). Any blown-up trajectory with endpoints at $\bf{x}_{i+1}$ and $\bf{x}_i$ is supported on a constant strip at $0$ downstairs; hence $A \cong \bb{F}_2[t]$. Note that we can directly identify $A = CF_{\bb{Z}/2}(\bb{R}^n, \Lambda)$, where $\Lambda$ is the linear Lagrangian subspace $T_0 \Gamma_{\eps} \subset T^* \bb{R}^n$.

Finally, of all the flow lines from $y_i$ to $0$, all but the one along the $i$-th coordinate axis are asymptotically tangent to the subspace spanned by the first $(i-1)$-coordinates, and likewise for the flow lines from $-y_i$ to $0$. In particular, $d \bf{x}_{i-1} = \bf{y}_i$, and so we conclude
\begin{equation}CF^{KM}_{\bb{Z}/2}(\bb{R}^n, \Gamma_{\eps}) = \Cone\left(\bb{F}_2[t] \to \bb{F}_2[t]/(t^n)\right).\end{equation}
This is of course quasi-isomorphic to $\bb{F}_2[t]$ (which a shift by $-n$ in grading), but the above description will be important in what comes next.

\begin{proof}[Proof of Theorem \ref{KMSScomparison}]Take $\tilde{M}, \tilde{L}_0, \tilde{L}_1$ with a $\bb{Z}/2$-action as before, with fixed point sets $M, L_0, L_1$ respectively, and consider the product $\tilde{M} \times T^* \bb{R}^n$ with product Lagrangians $\tilde{L}_0 \times \bb{R}^n$ and $\tilde{L}_1 \times \Gamma_{\eps}$. This has a $\bb{Z}/2$-action, with fixed point set $M \times \{0\}$. 

Equip $\tilde{M} \times T^* \bb{R}^n$ with a time-dependent, equivariant almost complex structure $\tilde{J}^{prod}_t$ in the following fashion. First, take the $S^{n-1}$-parametrized family of almost complex structures $J_z = \{J_t\}_z$ on $\tilde{M}$ obtained by pulling back those on the fibres of $\tilde{M}_{borel} \to \bb{R}P^{n-1}$ to $S^{n-1}$. Extend this family to $z \in \bb{R}^n$, in such a way as to preserve the equivariance $J_{-z} = \iota^* J$, and such that
\begin{itemize}\item for some small $\delta > 0$, $J_z = J_{z/||z||}$ whenever $z > 1 - \delta$
\item for $||z|| < 2\delta$, $J_z = \{\tilde{J}_t\}$, where $\tilde{J}$ is an equivariant, equivariantly regular, time-dependent almost complex structure on $\tilde{M}$.\end{itemize}

Again take some small tubular neighbourhood $U$ of $M \subset \tilde{M}$ as in Assumption \ref{localgeomII}. Then writing $z$ for the base coordinate in $T^* \bb{R}^n$, we require
\begin{equation*}\label{Jprod}\tilde{J}^{prod} = J_z \times J^{std}\end{equation*}
on $(\tilde{M} \times T^* \bb{R}^n) \backslash (U \times D^*_{\delta}\{||z|| < \delta\})$, where $J^{std}$ is the standard almost complex structure on $T^* \bb{R}^n$ and $D^*_{\delta}$ is the bundle of covectors of norm less than $\delta$.

We would ideally want \eqref{Jprod} to hold on the entirety of $\tilde{M} \times T^* \bb{R}^n$, since then any $\tilde{J}^{prod}$-holomorphic strip would project to a $J^{std}$-holomorphic strip in $T^* \bb{R}^n$, and in particular we could filter the complex $CF^{KM}_{\bb{Z}/2}(\tilde{L}_0 \times \bb{R}^n, \tilde{L}_1 \times \Gamma_{\eps}$ by the symplectic action on $\bb{R}^n \cap \Gamma_{\eps})$. However, we cannot assume that $\tilde{J}^{prod}$ takes the form \eqref{Jprod} near the fixed point set $M \times 0$, as then it may not be regular in the presence of spectral flow on $M$, as in Section 2.7.

Instead, we can rescale the symplectic form on $\tilde{M}$ so that the difference between the largest and the smallest symplectic action of points of $\tilde{L}_0 \cap \tilde{L}_1$ is less than $\eps / 2$. The product almost complex structures $J_z \times J^{std}$ of \eqref{Jprod} will still be compatible for this rescaled symplectic form on $\tilde{M} \times T^* \bb{R}^n$. Moreover, if we assume that $\tilde{J}^{prod}$ is sufficiently close to $\tilde{J} \times J^{std}$ over $U \times D^*_{\delta}\{||z|| < \delta\}$, it will be compatible both for the original symplectic form on $\tilde{M} \times T^* \bb{R}^n$, as well as the rescaled form.

In particular this means that for this choice of $\tilde{J}^{prod}$, $CF^{KM}_{\bb{Z}/2}(\tilde{L}_0 \times \bb{R}^n, \tilde{L}_1 \times \Gamma_{\eps})$ is still filtered by the symplectic action on $T^*\bb{R}^n$. Consequently, there is a subcomplex $B^{prod}$ formed by generators lying over the non-invariant points of $\bb{R}^n \cap \Gamma_{\eps}$, and a quotient complex $A^{prod}$ formed by just the generators coming from $(\tilde{L}_0 \cap \tilde{L}_1) \times \{0\}$. The differential then yields a map $A^{prod} \to B^{prod}$, such that $CF^{KM}_{\bb{Z}/2}(\tilde{L}_0 \times \bb{R}^n, \tilde{L}_1 \times \Gamma_{\eps}) = \Cone(A^{prod} \to B^{prod})$.

We can directly identify $B^{prod} \cong CF^{(n)}_{\bb{Z}/2}(\tilde{L}_0, \tilde{L}_1)$ from the description of $\tilde{J}^{prod}$ as well as the identification of Floer trajectories in $T^* \bb{R}^n$ bounded by $\bb{R}^n$ and $\Gamma_{\eps}$ with Morse flows of $g$.

As for $A^{prod}$, observe that since on $\delta < ||z|| < 2 \delta$ we have $\tilde{J}^{prod} = \tilde{J} \times J^{std}$, by a maximum principle argument applied to the projection to $T^* \bb{R}^n$, any Floer trajectory with endpoints on $\tilde{M} \times \{0\}$ must in fact be confined to $\tilde{M} \times T^*\{||z|| \le \delta\}$. In particular, again  writing $\Lambda$ for $T_0 \Gamma_{\eps} \subset T^*\bb{R}^n$, we can identify $A^{prod}$ with $CF^{KM}_{\bb{Z}/2}(\tilde{L}_0 \times \bb{R}^n, \tilde{L}_1 \times \Lambda)$. There is then a Kunneth isomorphism
\begin{equation}CF^{KM}_{\bb{Z}/2}(\tilde{L}_0, \tilde{L}_1) \tensor^{\bb{L}}_{\bb{F}_2[t]} A \xrightarrow{\sim} A^{prod}\end{equation}
recalling that $A = CF^{KM}_{\bb{Z}/2}(\bb{R}^n, \Lambda) \cong \bb{F}_2[t]$.

On the other hand, Theorem \ref{kunnethII} yields an equivalence in the derived category 
\begin{equation}K: CF^{KM}_{\bb{Z}/2}(\tilde{L}_0, \tilde{L}_1) \tensor^{\bb{L}}_{\bb{F}_2[t]} CF^{KM}_{\bb{Z}/2}(\bb{R}^n, \Gamma_{\eps}) \to CF^{KM}_{\bb{Z}/2}(\tilde{L}_0 \times \bb{R}^n, \tilde{L}_1 \times \Gamma_{\eps}).\end{equation}
As before, $CF^{KM}_{\bb{Z}/2}(\bb{R}^n, \Gamma_{\eps}) = \Cone(A \to B) = \Cone(\bb{F}_2[t] \to \bb{F}_2[t]/(t^n))$. Moreover by making a similar careful choice of $Z$-dependent almost complex structure used to define $K$, we see that $K$ preserves the filtration by the action on $T^* \bb{R}^n$. In particular, we obtain a commutative diagram in the derived category
\begin{equation}\label{KMSScommsquare}\begin{tikzcd}
CF^{KM}_{\bb{Z}/2}(\tilde{L}_0, \tilde{L}_1) \tensor^{\bb{L}}_{\bb{F}_2[t]} B \arrow[r] \arrow[d, dashed] &
CF^{KM}_{\bb{Z}/2}(\tilde{L}_0, \tilde{L}_1) \tensor^{\bb{L}}_{\bb{F}_2[t]} CF^{KM}_{\bb{Z}/2}(\bb{R}^n, \Gamma_{\eps}) \arrow[r] \arrow[d, "K"] &
CF^{KM}_{\bb{Z}/2}(\tilde{L}_0, \tilde{L}_1) \tensor^{\bb{L}}_{\bb{F}_2[t]} A \arrow[d, "K"] \\
CF^{(n)}_{\bb{Z}/2}(\tilde{L}_0, \tilde{L}_1) \cong B^{prod} \arrow[r] & CF^{KM}_{\bb{Z}/2}(\tilde{L}_0 \times \bb{R}^n, \tilde{L}_1 \times \Gamma_{\eps}) \arrow[r] & A^{prod}
\end{tikzcd}\end{equation}
where the rows are exact triangles, and the vertical two arrows are the Kunneth maps, which are quasi-isomorphisms. In particular the induced dashed vertical arrow is also a quasi-isomorphism. Since $B \cong \bb{F}_2[t]/(t^n)$, this is the desired equivalence
\begin{equation}\label{KMSSequivalence}CF^{KM}_{\bb{Z}/2}(\tilde{L}_0, \tilde{L}_1) \tensor^{\bb{L}}_{\bb{F}_2[t]} \bb{F}_2[t]/(t^n) \xrightarrow{\sim} CF^{(n)}_{\bb{Z}/2}(\tilde{L}_0, \tilde{L}_1).\end{equation}
The naturality in $n$ statement in the theorem then follows from a similar argument, by considering the action filtration on $T^*\bb{R}^{n+1}$. First, we must ensure that $\eps$ is chosen small enough so that the holomorphic strips in both $T^*\bb{R}^{n+1}$ and $T^*\bb{R}^n$ with boundaries on $\bb{R}^{n+1}, \Gamma^{n+1}_{\eps}$ and $\bb{R}^n, \Gamma^n_{\eps}$ are in bijection with Morse gradient trajectories, where $\Gamma^{n+1}_{\eps}, \Gamma^n_{\eps}$ are the Lagrangians used earlier (indeed, decreasing $\eps$ gives rise to Lagrangians isotopic through via a $\bb{Z}/2$-equivariant Hamiltonian flow).

Observe that the generator $\bf{y}_{n+1}$ of $T^*\bb{R}^{n+1}$ spans a subcomplex of $CF_{\bb{Z}/2}(\bb{R}^{n+1}, \Gamma^{n+1}_{\eps})$ whose quotient is isomorphic to $CF_{\bb{Z}/2}(\bb{R}^n, \Gamma^{n}_{\eps})$, where we have written $\Gamma^n_{\eps} \subset T^*\bb{R}^n$ for the Lagrangian given earlier. On the larger subcomplex $B$ of generators with positive action, this quotient map is exactly the quotient map $\bb{F}_2[t]/(t^{n+1}) \to \bb{F}_2[t]/(t^n)$.

Likewise since $CF^{KM}_{\bb{Z}/2}(\tilde{L}_0 \times \bb{R}^{n+1}, \tilde{L}_1 \times \Gamma^{n+1}_{\eps})$ is filtered by the action on $T^* \bb{R}^{n+1}$, it has a subcomplex spanned by generators arising from the points of $(\tilde{L}_0 \cap \tilde{L}_1) \times \{y_{n+1}, -y_{n+1}\}$, whose quotient is isomorphic to $CF^{KM}_{\bb{Z}/2}(\tilde{L}_0 \times \bb{R}^{n}, \tilde{L}_1 \times \Gamma^{n}_{\eps})$.

We should note that the identification of these two quotient complexes does itself involve another use of the Kunneth formula in the form of a stabilization result: the lowest subquotient of the filtration on $CF^{KM}_{\bb{Z}/2}(\bb{R}^{n+1}, \Gamma^{n+1}_{\eps})$ is exactly $CF^{KM}_{\bb{Z}/2}(\bb{R}^{n+1}, \Lambda^{n+1})$ where $\Lambda^{n+1} = T_0 \Gamma^{n+1}_{\eps}$. However under the identification $T^* \bb{R}^{n+1} = T^* \bb{R}^n \times \bb{C}$, we have
\begin{equation}CF^{KM}_{\bb{Z}/2}(\bb{R}^{n+1}, \Lambda^{n+1}) = CF^{KM}_{\bb{Z}/2}(\bb{R}^n \times \bb{R}, \Lambda^{n+1} \times e^{i \tan^{-1}(\eps(n+1))} \bb{R}) \xrightarrow{\sim} CF^{KM}_{\bb{Z}/2}(\bb{R}^n, \Lambda^n).\end{equation}
Putting this all together, and noting that the Kunneth isomorphism preserves the filtrations by the action on $T^* \bb{R}^{n+1}$, we see that the diagram \eqref{KMSScommsquare} is natural in $n$, thus completing the proof of Theorem \ref{KMSScomparison}.\end{proof}

\begin{remark}For cleanness of exposition we have stated the above result on the level of the derived category of $\bb{F}_2[t]$-complexes; however this raises potential problems arising from the non-uniqueness of dashed map in \eqref{KMSScommsquare}. This can be rectified by working with the chain level category of $A_{\infty}$-modules, at the cost of requiring \eqref{KMSScommsquare} to be a homotopy coherent rather than homotopy commutative diagram. To obtain this upgrade, we could work instead at the level of $\bb{F}_2[\bb{Z}/2]$-complexes (recall that this is the most natural setting for the Kunneth map in any case): here the analogue of \eqref{KMSScommsquare} commutes on the nose. In particular, there is a distinguished choice for the map \eqref{KMSSequivalence}.\end{remark}

\section{Steenrod squares}

\subsection{The setting.}
Suppose that $M$ is any exact symplectic manifold with convex boundary, and $L_0, L_1 \subset M$ two exact Lagrangians intersecting transversely, either compact or both conical and disjoint at infinity. We can then form the product $\tilde{M} = M \times M$, which is again an exact symplectic manifold, carrying a symplectic involution $\iota$ which exchanges the two factors, the invariant set of which is the diagonal $M$. Likewise, $\tilde{L}_0 = L_0 \times L_0$ and $\tilde{L}_1 = L_1 \times L_1$ are exact Lagrangian submanifolds of $\tilde{M}$ preserved set-wise by the $\bb{Z}/2$-action, with fixed point sets the original Lagrangians $L_0$ and $L_1$ respectively.

Moreover, the normal bundle $N_{M \subset M \times M}$ of the invariant set is symplectically isomorphic to its tangent bundle $TM$. Moreover, this isomorphism can be arranged to send the normal bundles $N_{L_i \subset L_i \times L_i}$ to the Lagrangian subbundles $TL_i$ of $TM|_{L_i}$ for $i = 0, 1$. In particular, taking the polarization data $\frak{t} = (TM, TL_0, TL_1)$, we obtain a localization map
\begin{equation}HF_{\bb{Z}/2}(L_0 \times L_0, L_1 \times L_1) \to HF(L_0, L_1, \frak{t})\end{equation}
which is moreover an isomorphism after inverting $t$ by Theorem \ref{localization_theorem}.

This situation has an important special feature: if $J_t$ is a regular time-dependent almost complex structure on $M$ for $L_0$ and $L_1$, then $\tilde{J}_t = J_t \times J_t$ is regular on $M \times M$ for $L_0 \times L_0, L_1 \times L_1$. It is also manifestly $\bb{Z}/2$-equivariant.

Observe that under the identification $N_{M \subset \tilde{M}} \cong TM$, the complex structure $\tilde{J}_t|_{N_{M}}$ is identified with $J_t$. We can then choose a symplectic connection $\nabla$ on $TM$ such that the pair $(J_t, \nabla)$ is equivariantly regular in the sense of Definition \ref{regularCS}. Moreover, we could further specify that our regular choice of $J_t$ is exactly the standard complex structure on $\bb{C}^n$ in given Darboux charts around the intersection points of $L_0 \cap L_1$, the Lagrangians $L_0, L_1$ are given by linear Lagrangian subspaces; the geometric hypotheses of \ref{localgeomI} and \ref{localgeomII} will then be satisfied by $\tilde{M}$. Moreover, by perturbing $J_t$, we can also achieve equivariant transversality in the sense of Definition \ref{equivarianttransversality}, and thus $\tilde{J}_t$ can be used to define $HF_{\bb{Z}/2}(L_0 \times L_0, L_1 \times L_1)$.

Let us now make a brief digression, to discuss the algebraic consequences of this.

\subsection{Equivariant and regular almost complex structures.}
Just for this subsection, take $\tilde{M}$ any symplectic manifold with $\bb{Z}/2$ action and fixed point set $M$ as in Section 3 (not necessarily $\tilde{M} = M \times M$), with equivariant Lagrangians $\tilde{L}_0, \tilde{L}_1 \subset \tilde{M}$ having fixed point sets $L_0, L_1$.

Suppose that $\tilde{M}$ has a time-dependent almost complex structure $\tilde{J}$ which is $\bb{Z}/2$-equivariant, and is also regular in the sense of ordinary Floer theory on $\tilde{M}$.

This situation is highly atypical. It has the immediate consequence that given any $x_-, x_+ \in L_0 \cap L_1$, the moduli space $\cal{M}_M(x_-, x_+)$ of Floer trajectories in $M$ from $x_-$ to $x_+$ must be of dimension at most that of $\cal{M}_{\tilde{M}}(x_-, x_+)$, the moduli space of trajectories in $\tilde{M}$. In particular for such a trajectory $u$, the spectral flow $\specflow(u)$ of the operator $I \frac{d}{dt}$ used in the polarization-twisted Floer theory of the normal bundle $N_{M \subset \tilde{M}}$ must be non-negative, since it is equal to the difference in the Maslov indices $\mu_{\tilde{M}}(u) - \mu_{M}(u)$.

In particular, the involution $\iota$ on $\tilde{M}$ induces an involution on the chain complex $CF(\tilde{L}_0, \tilde{L}_1)$, which we will also call $\iota$ and hope no confusion is caused. As in Section 2.1, we can then form an algebraic version of the Borel construction, and produce a chain complex computing the group homology:
\begin{equation}C(\bb{Z}/2, CF(\tilde{L}_0, \tilde{L}_1)) = CF(\tilde{L}_0, \tilde{L}_1)[t]\end{equation}
with differential
\begin{equation}d_{borel} = d + t(\id + \iota).\end{equation}

Now consider Borel fibration $\tilde{M}_{borel} \to \bb{R}P^{\infty}$ used in the Seidel-Smith model of equivariant cohomology. Observe that since $\tilde{J}_t$ is equivariant and regular, by putting the almost complex structure $\tilde{J}_t$ on each of the fibres of $\tilde{M}_z$ for $z \in \bb{R}P^{\infty}$ we obtain a chain complex that computes $CF^{SS}_{\bb{Z}/2}(\tilde{L}_0, \tilde{L}_1)$. In fact this effectively ``decouples'' the Morse flow on $\bb{R}P^{\infty}$ and the Floer equation in the fibres, and we obtain an isomorphism of chain complexes
\begin{equation*}CF^{SS}_{\bb{Z}/2}(\tilde{L}_0, \tilde{L}_1) \cong (CF(\tilde{L}_0, \tilde{L}_1)[[t]], d_{borel})\end{equation*}
where the two terms of $d_{borel} = d + t(\id + \iota)$ count the discrete pairs $(v, u)$ of a Morse trajectory $v$ on $\bb{R}P^{\infty}$ and a $\tilde{J}_t$-holomorphic strip in $\tilde{M}$ such that $v$ is constant in $\bb{R}P^{\infty}$, or $u$ is constant in $\tilde{M}$, respectively.

Theorem \ref{KMSScomparison} then yields an isomorphism
\begin{equation*}H(CF(\tilde{L}_0, \tilde{L}_1)[[t]], d_{borel}) \cong HF^{KM}_{\bb{Z}/2}(\tilde{L}_0, \tilde{L}_1)^{\wedge}\end{equation*}
however this isomorphism is quite inexplicit (both geometrically, as it involves at several steps counting solutions of continuation-style equations, and algebraically). In this particular setting however, we can produce a direct equivalence going the other way, defined over $\bb{F}_2[t]$ instead of $\bb{F}_2[[t]]$:

\begin{proposition}Suppose we have an equivariant and regular $\tilde{J}_t$ as above, which is furthermore equivariantly transverse in the sense of Definition \ref{equivarianttransversality}. Then there is a map of $\bb{F}_2[t]$-complexes
\begin{equation}G: CF(\bb{Z}/2, CF(\tilde{L}_0, \tilde{L}_1)) \to CF^{KM}_{\bb{Z}/2}(\tilde{L}_0, \tilde{L}_1)\end{equation}
which, after completion of both complexes in $t$, is a quasi-inverse for the map \eqref{KMSSequivalence}.\end{proposition}

For each pair of non-invariant intersection points of $\tilde{L}_0 \cap \tilde{L}_1$, choose a distinguished representative; write $\frak{C}_{non}$ for the set they form. Likewise write $\frak{C}_{inv}$ for the set of intersection points of $L_0 \cap L_1$, i.e. the invariant intersections of $\tilde{L}_0 \cap \tilde{L}_1$, so that $\frak{C}_{non} \sqcup \iota(\frak{C}_{non}) \sqcup \frak{C}_{inv}$ forms a basis of $CF(\tilde{L}_0, \tilde{L}_1)$. For $y \in \frak{C}_{non}$, write $\bf{y} = \{y, \iota(y)\}$ for the corresponding generator of $\frak{C}_o$, and for $x \in \frak{C}_{inv}$, write $\bf{x}_0 = (x, \lambda_0(x))$ for the generator of $\frak{C}_s$ corresponding to the lowest positive eigenvalue of $I \frac{d}{dt}$ at $x$.

Also choose, for every eigenvalue $\lambda$ of $I \frac{d}{dt}$ at $x$, a distinguished choice of unit eigenvector; in particular this specifies an endomorphism $T$ of $CF^{KM}_{\bb{Z}/2}(\tilde{L}_0, \tilde{L}_1)$ defining the $\bb{F}_2[t]$-module structure.

\begin{definition}The map $G : CF(\bb{Z}/2, CF(\tilde{L}_0, \tilde{L}_1)) \to CF^{KM}_{\bb{Z}/2}(\tilde{L}_0, \tilde{L}_1)$ is defined by setting, for each $n \ge 0$ and $x \in \tilde{L}_0 \cap \tilde{L}_1$:
\begin{equation}\label{G_formula}
G(t^n x) = \begin{cases}T^n \bf{x} & \text{ if } x \in \frak{C}_{non};\\
0 & \text{ if } x \in \iota(\frak{C}_{non})\\
T^n \bf{x}_0 & \text{ if } x \in \frak{C}_{inv}\end{cases}\end{equation}\end{definition}

\begin{remark}Apart from $M \subset M \times M$, the other major situation where equivariant and regular complex structures naturally arise is in the presence of a stable normal trivialization, as in \cite{SeidelSmith10}. There is then a natural question, beyond the scope of this discussion, of whether the composite of $G$, the localization map $HF^{KM}_{\bb{Z}/2}(\tilde{L}_0, \tilde{L}_1) \to HF_{tw}(L_0, L_1; \frak{p})$ and the isomorphism with $HF(L_0, L_1)[t,t^{-1}]$ induced by the trivialization of the polarization agrees with the localization map constructed by Seidel-Smith.\end{remark}

By construction $G$ is $\bb{F}_2[t]$-linear. First, we will show that $G$ is a chain map. We first make a couple of observations about the dimension zero moduli spaces counted by the differential both on $CF(\tilde{L}_0, \tilde{L}_1)$ and on $CF^{KM}_{\bb{Z}/2}(\tilde{L}_0, \tilde{L}_1)$:

First, If $x \in \frak{C}_{inv}$ and $y \in \tilde{L}_0 \cap \tilde{L}_1$, by equivariant transversality the moduli space of $\tilde{J}_t$-holomorphic strips in a fixed homotopy class $[u]$ with endpoints $y$ and $x$ and with asymptotics of type $\lambda_+$ where $\lambda_+ > \lambda_0(x)$ at the positive end form a submanifold of strictly positive codimension in the moduli space of all Floer trajectories from $y$ to $x$. In particular, if $y$ is non-invariant, the dimension zero components of the moduli spaces $\cal{M}(y, x)$ and $\cal{M}(\bf{y}, \bf{x}_0)$ coincide.

Likewise, for $x \in \frak{C}_{inv}$ and $y \in \frak{C}_{non}$, if $\bf{x}_{-1}$ is the element of $\frak{C}_u$ corresponding to the highest negative eigenvalue, the dimension zero components of $\cal{M}(x, y)$ and $\cal{M}(\bf{x}_{-1}, \bf{y})$ coincide.

If $x, x' \in \frak{C}_{inv}$, then the dimension zero part of $\cal{M}(x_-, x_+)$ can be partitioned into the invariant trajectories contained in $M$, and pairs of non-invariant trajectories. Similarly to the above, the $\bb{Z}/2$-quotient of the non-invariant part of the moduli space is exactly the dimension zero part of $\cal{M}(\bf{x}_{-1}, \bf{x}_0)$.

On the other hand, if $u$ is an invariant trajectory of index one in $\tilde{M}$, by equivariant transversality it is also index one in $M$, and thus $\specflow(u) = 0$. In particular, the unique $i \in \bb{Z}$ such that the expected dimension $\cal{M}(\bf{x}_i, \bf{x}'_0)$ is dimension zero is $i = 0$. Moreover, by the regularity of the twisted equations, the operator $\bar{\nabla}_I(u)$ in the twisted equation must then always be surjective, taken on Sobolev spaces with exponential weights; thus its kernel is dimension one. In particular, the downstairs trajectory $u$ supports a unique twisted trajectory up to rescaling. Hence, the dimension zero invariant trajectories in $\cal{M}(x, x')$ are in bijection with the moduli space of twisted trajectories $\cal{M}^{\partial}(\bf{x}_0, \bf{x}'_0)$.

Moreover, consider the dimension zero part of the moduli space of twisted trajectories $\cal{M}^{\partial}(\bf{x}_i, \bf{x}'_0)$ between some $(x, \lambda_i)$ and $(x', \lambda_0)$. For such a trajectory $(u, \phi)$, we know from the dimension formula that
\begin{equation*}\mu_M(u) + \specflow(u) + i - 1 = 0.\end{equation*}
We know that $\mu(u), \specflow(u) \ge 0$. If $\mu(u) = 0$, then $x = x'$ and $u$ is a constant trajectory; the only discrete flows are for $i = 1$ and there are exactly two of them. If $\mu_M(u) = 1$ and $\specflow(u) = 0$, then we also have $\mu_{\tilde{M}}(u) = 1$, and this flow must be of the form we just described. Otherwise, we must have $i < 0$, and thus $\bf{x}_i \in \frak{C}_u$.

By the same argument, the only dimension zero moduli spaces of boundary obstructed trajectories (i.e. a twisted trajectory in $\cal{M}^{\partial}(\bf{x}_i, \bf{x}'_j)$ with $i \ge 0$ and $j < 0$) occur when $x = x'$ and $i = 0, j = -1$; such trajectories have constant $u$, and each such moduli space has exactly two of them.

\begin{proposition}$G$ is a chain map.\end{proposition}

\begin{proof}It suffices to prove $G(d_{borel}(x)) = d_{KM}G(x)$ where $d_{KM}$ is the Kronheimer-Mrowka differential, in the three cases.

\emph{Case I: Suppose that $x \in \frak{C}_{inv}$ is an invariant intersection point}. Then \begin{equation*}d_{borel}(x) = dx = \sum\limits_{y \in \frak{C}_{non}} \# \cal{M}_{\tilde{M}}(y, x)(y + \iota(y)) + \sum\limits_{y \in \frak{C}_{inv}} \# \cal{M}_{M}(y, x)\cdot y\end{equation*}
noting that contributions from non-invariant trajectories from an invariant $y$ to $x$ come in pairs and thus cancel out; thus
\begin{equation*}G(d_{borel} x) = \sum\limits_{y \in \frak{C}_{non}} \# \cal{M}_{\tilde{M}}(y, x) \cdot \bf{y} + \sum\limits_{y \in \frak{C}_{inv}} \# \cal{M}_{M}(y, x) \cdot \bf{y}_0.\end{equation*}
However, this is exactly $d_{KM}(\bf{x}_0)$ by the observations above (noting that since boundary-obstructed trajectories come in pairs, contributions from broken trajectories to $d_{KM}$ cancel out).

\emph{Case II: Suppose that $x \in \frak{C}_{non}$ is a distinguished, non-invariant intersection point}. Then we have
\begin{equation*}dx = \sum\limits_{y \in \frak{C}_{non}}\left(\# \cal{M}_{\tilde{M}}(y, x) \cdot y + \# \cal{M}_{\tilde{M}}(\iota(y), x) \cdot \iota(y)\right) + \sum\limits_{y \in \frak{C}_{inv}} \# \cal{M}_{\tilde{M}}(y, x) \cdot y\end{equation*}
and hence
\begin{equation*}\label{Gdx}G(dx) = \sum\limits_{y \in \frak{C}_{non}}\# \cal{M}_{\tilde{M}}(y, x) \cdot \bf{y} + \sum\limits_{y \in \frak{C}_{inv}} \# \cal{M}_{\tilde{M}}(y, x) \cdot \bf{y}_0\end{equation*}
On the other hand, $G(t \iota(x)) = 0$, while
\begin{equation*}G(t x) = T \bf{x} = \sum\limits_{y \in \frak{C}_{non}} \# \cal{M}_{\tilde{M}}(\iota(y), x) \cdot \bf{y} + \sum\limits_{y \in \frak{C}_{inv}} \# \cal{M}_{\tilde{M}}(y, x) \cdot \bf{y}_0\end{equation*}
The first term here is since $T$ only counts trajectories between a ``distinguished'' and an ``undistinguished'' choice of intersection point. The second term deserves some explanation: the total index one broken trajectories from some $\tilde{y}_i \in \frak{C}_s$ to $\bf{x}$ must have a boundary obstructed component; by the earlier observation we must have $i = 0$, the interior component of the trajectory must be from $\bf{y}_{-1}$ to $\bf{x}$, and the boundary obstructed component must be one of the two trajectories from $\bf{y}_0$ to $\bf{y}_{-1}$. In particular, there are exactly two broken trajectories for each discrete downstairs trajectory in $\cal{M}(y, x)$, and moreover exactly one of them is counted by $T$. Hence,
\begin{equation*}G(d_{borel}(x)) = \sum\limits_{y \in \frak{C}_{non}}\left(\# \cal{M}_{\tilde{M}}(y, x) + \#\cal{M}_{\tilde{M}}(\iota(y), x)\right) \cdot \bf{y} = \sum\limits_{\bf{y} \in \frak{C}_o} \# \cal{M}(\bf{y}, \bf{x}) \cdot \bf{y}.\end{equation*}
However this is exactly $d_{KM}(\bf{x})$. Indeed, this differential has no terms in $\bf{x} \in \frak{C}_s$, since such contributions must come from broken trajectories with a boundary-obstructed component, but as observed these come in pairs.

\emph{Case III: Similarly suppose $x \in \iota(\frak{C}_{non})$ is an ``undistinguished'' non-invariant intersection point}. Then $G(dx)$ is computed by the same formula \eqref{Gdx} as before. On the other hand, $G(tx) = 0$, while
\begin{equation*}G(t\iota(x)) = T\bf{x} = \sum\limits_{y \in \frak{C}_{non}} \# \cal{M}_{\tilde{M}}(y, x) \cdot \bf{y} + \sum\limits_{y \in \frak{C}_{inv}} \# \cal{M}_{\tilde{M}}(y, x) \cdot \bf{y}_0\end{equation*}
by the same reasoning as earlier, where now $T$ counts trajectories between the ``distinguished'' points $y \in \frak{C}_{non}$ and the ``undistinguished'' point $x$. Consequently, $G(d_{borel}x) = 0$, but we also have $d_{KM}G(x) = 0$ since $G(x) = 0$.
This concludes the proof that $G$ is a chain map.\end{proof}

To see that $G$ is a quasi-inverse for equivalence of Theorem \ref{KMSScomparison}, observe that each of $CF(\tilde{L}_0, \tilde{L}_1)[t]$ and $CF^{KM}_{\bb{Z}/2}(\tilde{L}_0, \tilde{L}_1)$ is filtered by the symplectic action of $\tilde{L}_0 \cap \tilde{L}_1$. The map $G$ preserves this filtration by construction.

On the other hand, the map \eqref{KMSSequivalence} also preserves this filtration. To see this is a bit more delicate, since the continuation-style equations used to define this produce maps which are a priori only filtered for the symplectic action on $(\tilde{L}_0 \cap \tilde{L}_1) \cap (\bb{R}^n \cap \Gamma_{\eps})$. However, by choosing $\eps$ to be much smaller than all the differences between action values of $\tilde{L}_0 \cap \tilde{L}_1$, we can achieve this (recall that in fact the equivalence is also filtered for the action on $\bb{R}^n \cap \Gamma_{\eps}$, by choosing complex structures which are compatible for both the original symplectic form on $\tilde{M} \times T^* \bb{R}^n$, and one created by rescaling $\omega_{\tilde{M}}$ to be very small).

It then suffices to consider the two cases when $\tilde{M}$ is two disjoint points interchanged by $\bb{Z}/2$, and when $\tilde{M} = \bb{C}^n$ and $\tilde{L}_0, \tilde{L}_1$ are linear Lagrangians, with $\bb{Z}/2$ action by $-1$, corresponding to filtration levels coming from a non-invariant pair and a single invariant intersection point respectively. In the former case, $G$ is exactly the map
\begin{equation*}\bb{F}_2[\bb{Z}/2][t] \to \bb{F}_2\end{equation*}
defined by $1 \mapsto 1$ and $\iota \mapsto 0$ along with any nonzero power of $t$. On the other hand, for each truncation level $n$, $F$ is an $A_{\infty}$-module map $\{F^i\}_{i \ge 1}$ whose first term $F^i$ is the map
\begin{equation*}\bb{F}_2 \to \bb{F}_2[\bb{Z}/2]\tensor_{\bb{F}_2} \bb{F}_2[t]/(t^n)\end{equation*}
defined by $1 \mapsto (1 + \iota)$, where recall the right hand side has differential $t(1+\iota)$. After reducing $G$ modulo $t^n$, we see that these two maps are inverses on cohomology (in fact, they are chain homotopy inverses).

In the second case, $G$ is just the identity map
\begin{equation*}\bb{F}_2[t] \to \bb{F}_2[t]\end{equation*}
On the other hand, for each truncation level $n$, $F$ is the canonical map
\begin{equation*}\bb{F}_2[t] \tensor_{\bb{F}_2[t]} \bb{F}_2[t]/(t^n) \to \bb{F}_2[t]/(t^n)\end{equation*}
which are clearly inverse after reducing the first map modulo $t^n$.

We have thus proven that for each $n$,
\begin{equation*}G \text{ mod } t^n : (CF(\tilde{L}_0, \tilde{L}_1)\tensor_{\bb{F}_2}\bb{F}_2[t]/(t^n), d_{borel}) \to CF^{KM}_{\bb{Z}/2}(\tilde{L}_0, \tilde{L}_1) \tensor_{\bb{F}_2[t]}^{\bb{L}} \bb{F}_2[t]/(t^n)\end{equation*}
is an inverse on cohomology for \eqref{KMSSequivalence}, completing the proof.

\subsection{The Floer-theoretic Steenrod square}

Let us return to the setting $\tilde{M} = M \times M, \tilde{L}_0 = L_0 \times L_0, \tilde{L}_1 = L_1 \times L_1$ of before. Using the almost complex structure $\tilde{J}_t = J_t \times J_t$, we have an isomorphism of chain complexes
\begin{equation}CF(L_0 \times L_0, L_1 \times L_1) \cong CF(L_0, L_1) \tensor CF(L_0, L_1).\end{equation}
This is just the usual Kunneth theorem for ordinary Floer cohomology, which sends an intersection point $(x, y) \in (L_0 \times L_0, L_1 \times L_1)$ to $x \tensor y$ (it is a chain map since the only Floer trajectories $(u, u')$ on $M \times M$ which come in zero-dimensional moduli spaces are those where one of $u, u'$ is constant). This isomorphism is clearly equivariant for the $\bb{Z}/2$-action induced from that on $M \times M$ for the left complex, and the action which interchanges the two tensor factors on the right.

Moreover, it is well known (see for instance \cite{Seidel15}) that for any complex of $\bb{F}_2$-vector spaces $V$, the function
\begin{equation}V \to C(\bb{Z}/2, V \tensor V), \qquad x \mapsto x \tensor x\end{equation}
defines a chain map after inverting $t$ (which, if over a larger characteristic zero field, would be Frobenius-linear). If $V$ is finite dimensional, moreover yields an isomorphism
\begin{equation}V \tensor \bb{F}_2[t, t^{-1}] \xrightarrow{\sim} H(\bb{Z}/2, V \tensor V)[t^{-1}].\end{equation}
Applying this to the case $V = CF(L_0, L_1)$, we obtain

\begin{definition}The Floer-theoretic total Steenrod square is the map
\begin{equation}St : HF(L_0, L_1) \to HF_{tw}(L_0, L_1; \frak{t})\end{equation}
where $\frak{t}$ is the polarization $(TM, TL_0, TL_1)$ defined on the chain level by the composition of the ``algebraic square''
\begin{equation}CF(L_0, L_1) \to C(\bb{Z}/2, CF(L_0, L_1) \tensor CF(L_0, L_1));\end{equation}
the Kunneth equivalence $C(\bb{Z}/2, CF(L_0, L_1) \tensor CF(L_0, L_1)) \cong C(\bb{Z}/2, CF(L_0 \times L_0, L_1 \times L_1)$; the comparison map $G$ of the previous section
\begin{equation}G: C(\bb{Z}/2, CF(L_0 \times L_0, L_1 \times L_1) \to CF_{\bb{Z}/2}(L_0 \times L_0, L_1 \times L_1);\end{equation}
and the localization map
\begin{equation}CF_{\bb{Z}/2}(L_0 \times L_0, L_1 \times L_1) \to CF_{tw}(L_0, L_1; \frak{t}).\end{equation}
Over a different characteristic zero field, it is Frobenius-linear, and in the presence of gradings in the sense of \cite{Seidel00}, it doubles the degree. In particular, a trivialization of $\frak{t}$ induces an isomorphism $HF_{tw}(L_0, L_1; \frak{t}) \cong HF(L_0, L_1)[t, t^{-1}]$, as well as a grading on $HF(L_0, L_1)$. We thus obtain individual Steenrod operations
\begin{equation}St^i : HF^*(L_0, L_1) \to HF^{*+i}(L_0, L_1)\end{equation}
defined as the $t^{|x| - i}$-term of $St(x)$.\end{definition}

In view of the construction \eqref{G_formula} of the map $G$, the chain level Steenrod operation
\[CF(L_0, L_1) \to CF_{tw}(L_0, L_1; \frak{p})\]
admits a simple description. For an intersection point $x \in L_0 \cap L_1$, take the corresponding generator $\bf{x}_0 = (x, \lambda_0) \in \frak{C}_s$ of lowest positive eigenvalue $\lambda_0 > 0$. Then the total Steenrod square $St(x)$ is given by the the image of $\bf{x}_0$ in $CF_{tw}(L_0, L_1; \frak{p})$ under the localization map. Turning to the formula \eqref{localization_formula} for this map, we see that
\[St(x) = d_{us}\bf{x}_0 = \sum_{\bf{y} \in \frak{C}_u} \# \cal{M}^o(\bf{y}, \bf{x}_0) \ \bf{y}\]
counts the index one \emph{interior} trajectories connecting boundary-unstable generators $\bf{y}$ to $\bf{x}_0$.

Recall that the action filtration on $L_0 \cap L_1$ induces a spectral sequence computing the polarization-twisted Floer cohomology starting at the ordinary Floer cohomology $HF(L_0, L_1; \xi)$ with coefficients in the local system $\xi$ on $\cal{P}_M(L_0, L_1)$ with fibres $\bb{F}_2[t,t^{-1}]$, and monodromy given by $t^{\specflow(u)}$. However for the polarization $\frak{t}$, the spectral flow $\specflow(u)$ of the operator $I \frac{d}{dt}$ is equal to the Maslov index $\mu(u)$, and since the differentials in ordinary Floer cohomology count strips of Maslov index one, we have
\begin{equation}CF(L_0, L_1; \xi) = (CF(L_0, L_1)\tensor \bb{F}_2[t,t^{-1}], td)\end{equation}
where $d$ is the ordinary Floer differential. 

\begin{proposition}The Floer-theoretic total Steenrod square defines an isomorphism
\begin{equation}HF(L_0, L_1) \tensor \bb{F}_2[t, t^{-1}] \xrightarrow{\sim} HF_{tw}(L_0, L_1; \frak{t}).\end{equation}
In particular, the spectral sequence from $HF(L_0, L_1; \xi)$ to $HF_{tw}(L_0, L_1; \frak{t})$ degenerates at the $E_2$ page.\end{proposition}

\begin{proof}The isomorphism statement is immediate, since each map in the composition defining $Sq$ is an isomorphism after inverting $t$. In particular, this implies that $HF_{tw}(L_0, L_1; \frak{t})$ is a free $\bb{F}_2[t,t^{-1}]$-module of rank $\dim_{\bb{F}_2} HF(L_0, L_1)$. However, the $E_2$ page of the spectral sequence, $HF(L_0, L_1; \xi)$, is also a free $\bb{F}_2[t,t^{-1}]$-module of this same rank, since both the kernel and image of the differential $td$ on $CF(L_0, L_1)\tensor \bb{F}_2[t,t^{-1}]$ are exactly the kernel and image of the ordinary differential $d$. The spectral sequence must then degenerate.\end{proof}

\begin{remark}In the case when $L_0$ and $L_1$ intersect in just two points, and there is an $n$-dimensional moduli space $\cal{M}$ of Floer trajectories connecting them, the sole possible differential in the spectral sequence is in the $E_{n+1}$ page. By Proposition \ref{porteous_calculation}, this differential is
\begin{equation}w_n(-T\cal{M})\cdot [\cal{M}]\end{equation}
which is in fact zero for all closed $n$-manifolds $\cal{M}$, since $\cal{M}$ can always be immersed into $\bb{R}^{2n-1}$. This recovers the result above.\end{remark}

\section{Double covers of three-manifolds}

\subsection{Heegaard Floer theory review.}

Let us first give an extremely incomplete review of Heegaard Floer homology. Suppose $Y$ is a closed, oriented three-manifold; recall that a Heegaard splitting of $Y$ is a handlebody decomposition $Y = U_0 \cup_{\Sigma} U_1$ where $\Sigma$ is a genus $g$, closed oriented surface. The attaching circles of $U_0, U_1$ determine a Heegaard diagram $(\Sigma, \bm{\alpha}, \bm{\beta})$, where $\bm{\alpha} = \{\alpha_1, \hdots, \alpha_g\}$ and $\bm{\beta} = \{\beta_1, \hdots, \beta_g\}$ are two sets of closed embedded curves on $\Sigma$, such that the $\bm{\alpha}$ curves are pairwise disjoint, as are the $\bm{\beta}$ curves, and such that each $\bm{\alpha}$, $\bm{\beta}$ span a $g$-dimensional subspace of $H_1(\Sigma)$; we also assume each $\alpha_i, \beta_j$ intersect transversely. Another useful viewpoint is that $Y$ carries a self-indexing Morse function $f : Y \to [0,3]$ with one critical point each of index 0 and 3, and $g$ critical points each of index 1 and 2. Then, $\Sigma = f^{-1}(3/2)$, and the $\bm{\alpha}, \bm{\beta}$ circles are the intersection of $\Sigma$ with the ascending and descending disks of the critical points of index $1, 2$ respectively.

A basepoint $z \in \Sigma$ not lying on any of the $\bm{\alpha}$ or $\bm{\beta}$ circles defines a pointed Heegaard diagram $\cal{D} = (\Sigma, \bm{\alpha}, \bm{\beta}, z)$. A choice of complex structure $j$ on $\Sigma$ then makes the $g$-fold symmetric product $\Sym^g(\Sigma - \{z\})$ into a complex manifold; by equipping it with the appropriate symplectic form it becomes an exact symplectic manifold with convex boundary. The $\bm{\alpha}$ and $\bm{\beta}$ curves then define two totally real submanifolds $\bb{T}_{\alpha} = \alpha_1 \times \hdots \times \alpha_g$ and $\bb{T}_{\beta} = \beta_1 \times \hdots \times \beta_g$ of $\Sym^g(\Sigma - \{z\})$, which can moreover be chosen to be exact Lagrangians; for details on choosing such a symplectic form we refer the reader to \cite{Perutz}, \cite{Hendricks12}. For our purposes, we will only consider the $\widehat{HF}$ flavor of Heegaard Floer homology, which can be taken to be the Lagrangian Floer homology inside $\Sym^g(\Sigma - \{z\})$:
\begin{equation}\widehat{HF}(Y) = HF(\bb{T}_{\alpha}, \bb{T}_{\beta}).\end{equation}

We will also have to consider 2-pointed Heegaard diagrams $\cal{D} = (\Sigma^g, \bm{\alpha}, \bm{\beta}, \{z_1, z_2\})$ where $z_1, z_2$ are two distinct basepoints, and $\bm{\alpha} = \{\alpha_1, \hdots, \alpha_{g+1}\}, \bm{\beta} = \{\beta_1, \hdots, \beta_{g+1}\}$ are now two families of $g+1$ pairwise disjoint curves; these arise from self-indexing Morse functions $f: Y \to [0,3]$ with exactly two minima and two maxima. However by deleting small disks in $\Sigma$ around $z_1, z_2$ and attaching either ends of a cylinder to the resulting circle boundaries, we obtain a (one-)pointed Heegaard diagram for the connect sum $Y \# (S^1 \times S^2)$. In particular, if we take the Lagrangians $\bb{T}_{\alpha} = \alpha_1 \times \hdots \times \alpha_{g+1}$ and $\bb{T}_{\beta} = \beta_1 \times \hdots \times \beta_{g+1}$ of $\Sym^{g+1}(\Sigma - \{z_1, z_2\}$, their Lagrangian Floer homology computes
\begin{equation}\widehat{HF}(Y \# (S^1 \times S^2)) = HF(\bb{T}_{\alpha}, \bb{T}_{\beta}).\end{equation}
In light of the connect sum formula for $\widehat{HF}$ (as proved in \cite{OzsvathSzabo04b}), this is isomorphic to $\widehat{HF}(Y) \tensor V$, where $V = \widehat{HF}(S^1 \times S^2) \cong H^*(S^1)$ is a two-dimensional vector space.

There is another use of a second base-point in a Heegaard diagram, to define invariants of knots in $Y$ (for us, a knot will mean an oriented, null-homologous embedded closed curve). Given a Heegaard diagram $(\Sigma, \bm{\alpha}, \bm{\beta})$ for $Y = U_0 \cup_{\Sigma} U_1$, choose two disjoint base points $w, z \in \Sigma - \cup \alpha_i - \cup \beta_j$. Choose an arc from $z$ to $w$ in $\Sigma$ disjoint from $\bm{\alpha}$, and push it off into $U_0$; also choose an arc from $w$ to $z$ disjoint from $\bm{\beta}$, and push it off into $U_1$. The concatenation defines a knot $K \subset Y$. Alternatively, given a self-indexing Morse function $f$ inducing the diagram, connect the minimum to the maximum via the flow line of $f$ passing through $w$, and then the maximum back to the minimum by the flow line passing through $z$. Indeed, any knot $K \subset Y$ can be presented in this fashion.

We then take the symplectic manifold $\Sym^g(\Sigma - \{w, z\})$, which again carries two Lagrangians $\bb{T}_{\alpha}, \bb{T}_{\beta}$. The knot Floer homology $\widehat{HFK}(Y, K)$ (\cite{OzsvathSzabo04c}, \cite{Rasmussen03}) can then be defined as
\begin{equation}\widehat{HFK}(Y, K) = HF(\bb{T}_{\alpha}, \bb{T}_{\beta})\end{equation}
where this time we take the Lagrangian Floer homology inside $\Sym^g(\Sigma - \{w, z\})$.

\subsection{Double covers.} Let $\tilde{Y} \to Y$ be an unbranched double cover of closed oriented three-manifolds. Observe that a pointed Heegaard diagram $\cal{D} = (\Sigma^g, \bm{\alpha}, \bm{\beta}, z)$ for $Y$ induces a 2-pointed Heegaard diagram for $\tilde{Y}$
\begin{equation}\tilde{\cal{D}} = (\tilde{\Sigma}, \tilde{\bm{\alpha}}, \tilde{\bm{\beta}}, \{\tilde{z}_1, \tilde{z}_2\})\end{equation}
where $\tilde{\Sigma}^{2g-1} \to \Sigma^g$ is a double cover of surfaces, $\tilde{\bm{\alpha}} = \{\tilde{\alpha}_1, \hdots,  \tilde\alpha_{2g}\}$ and $\tilde{\bm{\beta}} = \{\tilde{\beta}_1, \hdots, \tilde{\beta}_{2g}\}$ are the preimages of $\bm{\alpha}$ and $\bm{\beta}$, and $\{\tilde{z}_1, \tilde{z}_2\}$ are the preimages of $z$. In terms of a self-indexing Morse function $f : Y \to [0,3]$, this is the Heegaard diagram corresponding to the Morse function $\tilde{Y} \to Y \to [0,3]$.

The symmetric product
\begin{equation}\tilde{M} = \Sym^{2g}(\tilde{\Sigma} - \{\tilde{z}_1, \tilde{z}_2\})\end{equation}
carries a natural $\bb{Z}/2$-action induced from the deck transformations on $\tilde{\Sigma}$, and the exact symplectic structure can be chosen to make this an action by symplectomorphisms. Clearly the Lagrangian tori $\bb{T}_{\tilde{\alpha}}$, $\bb{T}_{\tilde{\beta}}$ are preserved set-wise by this action. Moreover, the fixed point set can be identified with
\begin{equation*}M = \Sym^g(\Sigma - \{z\})\end{equation*}
such that the fixed points of $\bb{T}_{\tilde{\alpha}}, \bb{T}_{\tilde{\beta}}$ are exactly $\bb{T}_{\alpha}, \bb{T}_{\beta}$; this identification is through the map
\begin{align}\Sym^g(\Sigma - \{z\}) &\to \Sym^{2g}(\tilde{\Sigma} - \{\tilde{z}_1, \tilde{z}_2)\\
(x_1, \hdots, x_g) &\mapsto (\tilde{x}_1, \iota(\tilde{x}_1), \hdots, \tilde{x}_g, \iota(\tilde{x}_g))\end{align}
where $\tilde{x}_i, \iota(\tilde{x}_i)$ are the two preimages of $x_i$.

We then have the following fact, which we will prove in the next section:

\begin{proposition}\label{normaltangentpolarizationI}The normal polarization $\frak{p} = (E, F_0, F_1)$ of $\Sym^g(\Sigma - \{z\}) \subset \Sym^{2g}(\tilde{\Sigma} - \{z_1, z_2\}$ with its Lagrangians $\bb{T}_{\alpha} \subset \bb{T}_{\tilde{\alpha}}, \bb{T}_{\beta} \subset \bb{T}_{\tilde{\beta}}$ is in fact isomorphic to the tangential polarization $\frak{t} = (T\Sym^g(\Sigma - \{z\}), T\bb{T}_{\alpha}, T\bb{T}_{\beta})$ of $\Sym^g(\Sigma - \{z\})$.\end{proposition}

\begin{remark}This is similar in spirit to Lemma 7.2, Proposition 7.3 of Hendricks in \cite{Hendricks12}, which concerns genus zero Heegaard diagrams with many base-points. Our method of proof however will be quite different. Hendricks makes use of Macdonald's computation of the Chern classes of symmetric products of surfaces \cite{Macdonald62} to deduce that the stable complex normal bundle $E$ is stably trivial, and then appeals to the algebraic topology (in particular, the K-theory) of these spaces to deduce automatically from compatible trivializations of the Lagrangian normal bundles $F_0, F_1$. 

While the Chern class computation does carry through to our case to show that $E$ and $TM$ are stably isomorphic, Proposition 5.2 of \cite{Hendricks12} breaks down in this setting, meaning that we cannot then automatically produce compatible stable isomorphisms of $F_i$ and $TL_i$ (indeed, there are stable isomorphisms of $E$ and $TM$ for which there is no compatible isomorphism between $F_0$ and $TL_0$). While a more careful calculation of a relative Chern class in this case may be possible, instead we have opted for a direct construction of the required isomorphism.\end{remark}

Given Proposition \ref{normaltangentpolarizationI}, we can now give a proof of Theorem \ref{HF_inequality_doublecover}.

\begin{proof}From Theorem \ref{smith_inequality}, we have an inequality
\[2 \dim_{\bb{F}_2} \widehat{HF}(\tilde{Y}) =  \dim_{\bb{F}_2} HF(\bb{T}_{\tilde{\alpha}}, \bb{T}_{\tilde{\beta}}) \ge \rank_{\bb{F}_2[t,t^{-1}]} HF(\bb{T}_{\alpha}, \bb{T}_{\beta}; \frak{p}).\]
However from the above we have an equivalence of polarization data $\frak{p} \cong \frak{t} = (TM, TL_0, TL_1)$; moreover the total Steenrod square furnishes an isomorphism $HF(\bb{T}_{\alpha}, \bb{T}_{\beta}) \tensor \bb{F}_2[t,t^{-1}] \xrightarrow{\sim} HF_{tw}(\bb{T}_{\alpha}, \bb{T}_{\beta}; \frak{t})$. In particular we have
\[\rank_{\bb{F}_2[t,t^{-1}]} HF(\bb{T}_{\alpha}, \bb{T}_{\beta}; \frak{p}) = \rank_{\bb{F}_2[t,t^{-1}]} HF(\bb{T}_{\alpha}, \bb{T}_{\beta}; \frak{t}) = \dim_{\bb{F}_2} HF(\bb{T}_{\alpha}, \bb{T}_{\beta})\]
from which the result follows.\end{proof}

An analogue of this result in knot Floer homology also holds for double covers $\tilde{Y} \to Y$ which are branched along a knot $K \subset Y$. Indeed, given a Heegaard diagram $(\Sigma^g, \bm{\alpha}, \bm{\beta}, w, z)$ for $K$, we obtain a Heegaard diagram $(\tilde{\Sigma}^{2g}, \tilde{\bm{\alpha}}, \tilde{\bm{\beta}}, w, z)$ for $K \subset \tilde{Y}$, where $\tilde{\Sigma}^{2g} \to \Sigma^g$ is a double cover branched over $\{w, z\}$, with $\tilde{\bm{\alpha}}, \tilde{\bm{\beta}}$ the pre-images of $\bm{\alpha}, \bm{\beta}$, and $\tilde{w}, \tilde{z} \in \tilde{\Sigma}$ the branch points.

With an appropriate choice of symplectic structure,
\begin{equation}\tilde{M} = \Sym^{2g}(\tilde{\Sigma}-\{w,z\})\end{equation}
again carries a symplectic $\bb{Z}/2$-action, whose fixed point set is $M = \Sym^g(\Sigma - \{w, z\})$. The Lagrangians $\bb{T}_{\tilde{\alpha}}, \bb{T}_{\tilde{\beta}}$ are again preserved set-wise, with fixed point sets $\bb{T}_{\alpha}, \bb{T}_{\beta}$. Moreover, the analogue of Proposition \ref{normaltangentpolarizationI} also holds:

\begin{proposition}\label{normaltangentpolarizationII} The normal polarization $\frak{p} = (E, F_0, F_1)$ of $\Sym^g(\Sigma - \{w, z\}) \subset \Sym^{2g}(\tilde{\Sigma} - \{\tilde{w}, \tilde{z}\})$ with its Lagrangians $\bb{T}_{\alpha} \subset \bb{T}_{\tilde{\alpha}}, \bb{T}_{\beta} \subset \bb{T}_{\tilde{\beta}}$ is isomorphic to the tangential polarization $\frak{t} = (T\Sym^g(\Sigma - \{w, z\}), T\bb{T}_{\alpha}, T\bb{T}_{\beta})$ of $\Sym^g(\Sigma - \{w, z\})$.\end{proposition}

The proof of Theorem \ref{HF_inequality_branched} then proceeds exactly as above.

\begin{remark}Heegaard Floer homology $\widehat{HF}(Y)$ has a natural grading by the set $H^2(Y, \bb{Z})$ of spin$^c$ structures on $Y$, and moreover in  knot Floer homology has an additional integer graded called the Alexander grading. We fully expect the inequalities of Theorems \ref{HF_inequality_doublecover} and \ref{HF_inequality_branched} to also apply in each appropriate gradings; for the sake of brevity we have left this to the reader.\end{remark}


\subsection{Proof of Propositions \ref{normaltangentpolarizationI} and \ref{normaltangentpolarizationII}.}

Let us first explain work in the unbranched case. Since $p: \tilde{Y} \to Y$ is a double cover of $Y$ that is induced from a $\bb{Z}$-fold cover, there is a continuous function
\begin{equation}f : Y \to S^1\end{equation}
and an identification of $\tilde{Y}$ with the bundle of square roots $\{\pm \sqrt{f(y)}\}_{y \in Y}$ for $f$; in other words there is a function
\begin{equation}\tilde{f} : \tilde{Y} \to S^1\end{equation}
such that $\tilde{f}(\tilde{y}_1), \tilde{f}(\tilde{y}_2)$ are the two square roots of $f(y)$ whenever $\{\tilde{y}_1, \tilde{y}_2\} = p^{-1}(y)$ are the two preimages of $y \in Y$.

We will for this section only consider the functions $f, \tilde{f}$ restricted to $\Sigma, \tilde{\Sigma}$. The salient fact is that since the $\bm{\alpha}$ and $\bm{\beta}$ curves are null-homotopic in $Y$, for each $i$ and $j$ the restrictions
\begin{equation*}f|_{\alpha_i} : \alpha_i \to S^1, \qquad f|_{\beta_j} : \beta_j \to S^1\end{equation*}
are homotopic to constant maps. Likewise $\tilde{f}$ restricted to the $\tilde{\bm{\alpha}}$ and $\tilde{\bm{\beta}}$ curves is homotopic to constant maps.

We first prove a simpler version of the statement involving only the second symmetric power. There is likewise a $\bb{Z}/2$ action on $\Sym^2(\tilde{\Sigma})$, with fixed point set isomorphic to $\Sigma$; we will just refer to it as $\Sigma$. Suppose that the $\tilde{\bm{\alpha}}$ and $\tilde{\bm{\beta}}$ curves are numbered so that $\tilde{\alpha}_i, \tilde{\alpha}_{g+i}$ are the two pre-images of $\alpha_i$ under $p$, and likewise $\tilde{\beta}_j, \tilde{\beta}_{g+j}$ are the two pre-images for $\beta_j$. Under this numbering, each $\tilde{\alpha}_i \times \tilde{\alpha}_{g+i}, \tilde{\beta}_j \times \tilde{\beta}_{g+j} \subset Sym^2(\tilde{\Sigma})$ are preserved set-wise by the $\bb{Z}/2$ action, with fixed point sets $\alpha_i, \beta_j \subset \Sigma$ respectively; by abuse of notation we will write $\alpha_i \subset \tilde{\alpha}_i \times \tilde{\alpha}_{g+i}$ for this fixed point set and likewise for $\beta_j$.

\begin{proposition}\label{normaltangentpolarizationIII}There is an isomorphism of complex line bundles $T\Sigma \xrightarrow{\phi} N_{\Sigma \subset \Sym^2(\tilde{\Sigma})}$. Moreover, for each curve $\alpha_i$, the image $\phi(T\alpha_i)$ is homotopic through totally real sub-bundles to $N_{\alpha_i \subset \tilde{\alpha}_i \times \tilde{\alpha}_{g+i}}$, and likewise for $\beta_j$. \end{proposition}

\begin{proof}To produce $\phi$, it suffices to describe a continuous linear map of complex vector bundles
\begin{equation}\psi : T\Sigma \to T\Sym^2(\tilde{\Sigma})|_{\Sigma}\end{equation}
of maximal rank whose image is everywhere transverse to $T\Sigma \subset T\Sym^2(\tilde{\Sigma})|_{\Sigma}$.

Indeed, for $x \in \Sigma$, consider its two preimages $\{\tilde{x}_1, \tilde{x}_2\} = p^{-1}(x)$. The tangent space $T_x Sym^2(\tilde{\Sigma})$ is identified with $T_{\tilde{x}_1}\tilde{\Sigma} \oplus T_{\tilde{x}_2}\tilde{\Sigma}$, under which $T_x\Sigma \subset T_x\Sym^2(\tilde{\Sigma})$ is given by the map $v \mapsto (dp_{\tilde{x}_1}^{-1}(v), dp_{\tilde{x}_2}^{-1}(v))$. Then take the linear map
\begin{align*}\psi_x: T_x \Sigma & \to T_{\tilde{x}_1}\tilde{\Sigma} \oplus T_{\tilde{x}_2}\tilde{\Sigma} \\
v & \mapsto (\tilde{f}(\tilde{x}_1) dp_{\tilde{x}_1}^{-1}(v), (\tilde{f}(\tilde{x}_2) dp_{\tilde{x}_2}^{-1}(v)).\end{align*}
This yields the desired map of vector bundles, since $\tilde{f}(\tilde{x}_1) = - \tilde{f}(\tilde{x}_2)$.

For the statement about $\alpha_i \subset \tilde{\alpha}_i \times \tilde{\alpha}_{g+i}$, observe there is a homotopy of $\tilde{f}$ to the constant function $1$ on $\tilde{\alpha}_i$ and $-1$ on $\tilde{\alpha}_{g+i}$, through functions $g$ on $\tilde{\alpha}_i \sqcup \tilde{\alpha}_{g+i}$ such that $g(\tilde{x}) = -g(\iota(\tilde{x}))$ where $\iota$ is the deck transformation. We then obtain a homotopy of $\psi|_{\alpha_i}$, through linear maps of maximal rank and image transverse to $T\Sigma$, to the map
\begin{equation*}T\Sigma|_{\alpha_i} \to T\Sym^2(\tilde{\Sigma})|_{\alpha_i}, \qquad v \mapsto (dp|_{\tilde{\alpha}_i}^{-1}(v), -dp|_{\tilde{\alpha}_{g+i}}^{-1}(v))\end{equation*}
which in turn specifies a homotopy from $\phi(T\alpha_i)$ to $N_{\alpha_i \subset \tilde{\alpha}_i \times \tilde{\alpha}_{g+i}}$. The case of the $\beta_j$ curves is the same.\end{proof}

We now turn ourselves to the inclusion $\Sym^g(\Sigma) \subset \Sym^{2g}(\tilde{\Sigma})$. We will in fact prove a stronger statement than Proposition \ref{normaltangentpolarizationI} by working on the entirety of $\Sym^g(\Sigma)$  rather than just $\Sym^g(\Sigma-\{z\})$. We first need suitable replacements for the functions $f, \tilde{f}$.

\begin{lemma}\label{gfunctions}There exist continuous maps $g_1, g_2, \hdots, g_{2n}: \tilde{\Sigma}^{\times 2g} \to S^1$ satisfying the following properties:
\begin{itemize}\item the collection is $S_{2g}$-equivariant, meaning for all permutations $\sigma \in S_{2g}$, $i = 1, 2, \hdots, 2g$, and $(\tilde{x}_j) \in \tilde{\Sigma}^{\times 2g}$, we have $g_{\sigma(i)}(\tilde{x}_1, \hdots, \tilde{x}_{2g}) = g_i(\tilde{x}_{\sigma(1)}, \hdots, \tilde{x}_{\sigma(2g)})$;
\item if $\pi_i : \tilde{\Sigma}^{\times 2g} \to \tilde{\Sigma}$ is the $i$-th projection for $i=1, \hdots, 2g$, for some small $\eps > 0$ the functions $g_i$ and $\tilde{f} \circ \pi_i$ are uniformly $\eps$-close; 
\item moreover, away from some small open neighbourhood of the fat diagonal $\Delta_{\tilde{\Sigma}^{\times 2g}} \subset \tilde{\Sigma}^{\times 2g}$ of points where $\tilde{x}_i = \tilde{x}_j$ for some $i \not= j$, we have $g_i = \tilde{f} \circ \pi_i$;
\item there is some small $\delta > 0$ such that whenever $d(\tilde{x}_i, \tilde{x}_j) < \delta$ in some Riemannian metric on $\tilde{\Sigma}$, we have
\begin{equation*}g_i(\hdots, \tilde{x}_i, \hdots, \tilde{x}_j, \hdots) = g_j(\hdots, \tilde{x}_i, \hdots, \tilde{x}_j, \hdots).\end{equation*}\end{itemize}\end{lemma}

\begin{proof}Observe that the functions $\tilde{f} \circ \pi_i$ satisfy all the properties except the last, which is at least satisfied on the fat diagonal $\Delta$. However, by using an equivariant version of the tubular neighbourhood theorem strata by strata for $\Delta$, we can find a neighbourhood $U_1$ of $\Delta$ and and an $S_{2g}$-equivariant deformation retraction
\begin{equation}r : U \to \Delta.\end{equation}
We then obtain an $S_{2g}$-equivariant map $q: \Sigma^{\times 2g} \to \Sigma^{\times 2g}$ which is equal to the identity outside $U_1$, and on some smaller open neighbourhood $U_{1/2} \subset U_1$, is equal $r|_{U_{1/2}}$. By arranging for $U_{1/2}$ to be sufficiently small, we can ensure that $q$ and the identity are as $C^0$-close as desired. We can then define the functions $g_i$ as $\tilde{f} \circ \pi_i \circ q$.\end{proof}

We can now give the proof of Proposition \ref{normaltangentpolarizationI}.

\begin{proof}[Proof of \ref{normaltangentpolarizationI}.]
Again, we will produce a map of complex vector bundles
\begin{equation}\psi: T\Sym^g(\Sigma) \to T\Sym^{2g}(\tilde{\Sigma})|_{\Sym^g(\tilde{\Sigma})}\end{equation}
which is everywhere of maximal rank and transverse to $T\Sym^g(\Sigma) \subset T\Sym^{2g}(\tilde{\Sigma})$.

We first describe the map in the complement $\Sym^g(\Sigma)\backslash \Delta_{\Sym^g(\Sigma)}$ of the fat diagonal in $\Sym^g(\Sigma)$. Indeed, if $\bf{x} = (x_1, \hdots, x_g)$ is a collection of distinct points of $\Sigma$, with preimages $\{\tilde{x}_i, \tilde{x}_{g+i}\} = p^{-1}(x_i)$  in $\tilde{\Sigma}$, we have identifications of the tangent spaces
\begin{equation*}T_{\bf{x}} \Sym^g(\Sigma) \cong \bigoplus\limits_{i=1}^g T_{x_i} \Sigma, \qquad T_{\bf{x}} \Sym^{2g}(\tilde{\Sigma}) \cong \bigoplus_{i=1}^{2g} T_{\tilde{x}_i} \tilde{\Sigma}\end{equation*}
under which the embedding $T\Sym^g(\Sigma) \subset T\Sym^{2g}(\tilde{\Sigma})$ is given by the map
\begin{equation*}(v_1, \hdots, v_g) \mapsto \left(dp_{\tilde{x}_1}^{-1}(v_1), \hdots, dp_{\tilde{x}_g}^{-1}(v_g), dp_{\tilde{x}_{g+1}}^{-1}(v_1), \hdots, dp_{\tilde{x}_{2g}}^{-1}(v_g)\right).\end{equation*}
According to the above identification of the tangent spaces, we can then describe $\psi_{\bf{x}}$ as
\begin{equation*}\psi_{\bf{x}}(v_1, \hdots, v_g) = \left(g_1(\tilde{\bf{x}}) dp_{\tilde{x}_1}^{-1}(v_1), \hdots, g_g(\tilde{\bf{x}}) dp_{\tilde{x}_g}^{-1}(v_g), g_{g+1}(\tilde{\bf{x}}) dp_{\tilde{x}_{g+1}}^{-1}(v_1), \hdots, g_{2g}(\tilde{\bf{x}}) dp_{\tilde{x}_{2g}}^{-1}(v_g)\right).\end{equation*}
This is well-defined by the $S_{2g}$-equivariance of the functions $g_i$. By the second property of Lemma \ref{gfunctions}, the ratios $g_{i}(\tilde{\bf{x}})/g_{g+i}(\tilde{\bf{x}})$ are uniformly bounded away from $1 \in S^1$, and so the image of $\psi_{\bf{x}}$, which is clearly of maximal rank, is transverse to $T\Sym^g(\Sigma)$.

We now explain how to extend $\psi$ over the fat diagonal. Indeed, over the fat diagonal, there is still a direct sum decomposition of $T\Sym^g(\Sigma)$ and $T\Sym^{2g}(\tilde{\Sigma})|_{T\Sym^g(\Sigma)}$ over the distinct points of each unordered tuple, into pieces of the form $T_{(x)^n} \Sym^n (\Sigma)$ and $T_{(\tilde{x}_1)^n, (\tilde{x}_2)^n} \Sym^{2n}(\tilde{\Sigma})$, where $x \in \Sigma$ has pre-images $\tilde{x}_1, \tilde{x}_2$, and $(x)^n$ denotes the $x$ as a point of $Sym^n (\Sigma)$ with multiplicity $n$, for some $1 \le n \le g$. We claim that by taking the direct sum of the maps given by
\begin{align*}T_{(x)^n} \Sym^n (\Sigma) &\to T_{(\tilde{x}_1)^n}\Sym^n (\tilde{\Sigma}) \oplus T_{(\tilde{x}_2)^n} T\Sym^n (\tilde{\Sigma}) \cong T_{(\tilde{x}_1)^n, (\tilde{x}_2)^n} \Sym^{2n}(\tilde{\Sigma})\\
v & \mapsto \left(h_1 dp_{(\tilde{x}_1)^n}^{-1}(v), h_2 dp_{(\tilde{x}_2)^n}^{-1}(v)\right)\end{align*}
we obtain a continuous extension of $\psi$ over the fat diagonal, where $h_1 = g_1(\tilde{\bf{x}}) = \hdots = g_n(\tilde{\bf{x}})$ and $h_2 = g_{g+1}(\tilde{\bf{x}}) = \hdots = g_{g+n}(\tilde{\bf{x}})$ are unit complex numbers. These maps are clearly of maximal rank and transverse to $T_{(x)^n}\Sym^n(\Sigma)$, since $h_1/h_2$ is bounded away from $1$.

Indeed, choose a local holomorphic chart around $x \in \Sigma$, identifying some open neighbourhood of $x$ with the unit disc $D \subset \bb{C}$. Then there is some open neighbourhood $V$ of $\Sym^{n}(\Sigma)$ around $(x)^n$, and a holomorphic chart $\varphi: V \to \bb{C}^n$, such that there is a commutative diagram
\begin{equation}\label{Symncharts}\begin{tikzcd}V \backslash \Delta \arrow[hookrightarrow]{d} & D^n\backslash\Delta \arrow{l} \arrow{d}{F} \\
V \arrow{r}{\varphi} & \bb{C}^n\end{tikzcd}\end{equation}
where $D^n \backslash \Delta \to V \backslash \Delta$ is the order $n!$ holomorphic covering that sends an ordered tuple of disjoint points $(z_1, \hdots, z_n)$ to the corresponding point of $\Sym^n(\Sigma)$ under the local chart around $x \in \Sigma$, and $F$ is the map $(P_1, \hdots, P_n)$ given by the elementary symmetric polynomials
\begin{equation}P_1 = z_1 + z_2 + \hdots + z_n, \quad \hdots, \quad P_n = z_1 z_2 \hdots z_n.\end{equation}

Likewise around $((\tilde{x}_1)^n, (\tilde{x}_2)^n) \in \Sym^{2n}(\tilde{\Sigma})$, we have a chart on the open set given by $\tilde{V}_1 \times \tilde{V}_2$, where $\tilde{V}_1, \tilde{V}_2$ are open sets around $(\tilde{x}_1)^n, (\tilde{x}_2)^n \in \Sym^n(\tilde{\Sigma})$ which can in turn be identified with $V$. We then have a commutative diagram
\begin{equation}\label{Sym2ncharts}\begin{tikzcd}(\tilde{V}_1 \times \tilde{V}_2)\backslash \Delta \arrow{r}{\sim} \arrow[hookrightarrow]{d} & (V\backslash \Delta) \times (V \backslash \Delta) \arrow[hookrightarrow]{d} & (D^n \backslash \Delta) \times (D^n \backslash \Delta) \arrow{l} \arrow{d}{F \times F} \\
\tilde{V}_1 \times \tilde{V}_2 \arrow{r}{\sim} & V \times V \arrow{r}{\varphi \times \varphi} & \bb{C}^n \times \bb{C}^n\end{tikzcd}\end{equation}
The top rows of \eqref{Symncharts}, \eqref{Sym2ncharts} give charts on small open subsets of the complement of the fat diagonal; according to the trivializations of $T\Sym^n(\Sigma)$ and $T\Sym^{2n}(\tilde{\Sigma})$ induced by these charts, the map $\psi$ is given by the $2n \times n$ matrix
\begin{equation}\begin{pmatrix}g_1 & & & 0 \\ & g_2 & & \\ & & \ddots & \\ 0 & & & g_n \\ g_{g+1} & & & 0 \\ & g_{g+2} & & \\ & & \ddots & \\ 0 & & & g_{g+n}\end{pmatrix}.\end{equation}
However, by the final property of Lemma \ref{gfunctions}, if we take the chart $V$ to be sufficiently small, we must have $h_1(\tilde{\bf{x}}) = g_1(\tilde{\bf{x}}) = \hdots = g_n(\tilde{\bf{x}})$ and $h_2(\tilde{\bf{x}}) = g_{g+1}(\tilde{\bf{x}}) = \hdots = g_{g+n}(\tilde{\bf{x}})$ for some functions $h_1, h_2$; in particular $\psi$ is given by $(h_1 \id, h_2 \id)$ in these charts.

The right vertical maps of \eqref{Symncharts}, \eqref{Sym2ncharts} provide transition functions to the charts given by the bottom rows. In particular, viewed in the trivializations of $T\Sym^n(\Sigma)$ and $T\Sym^{2n}(\tilde{\Sigma})$ induced by the charts $\varphi$ and $\varphi \times \varphi$, the map $\psi$ is given by
\begin{equation}\begin{pmatrix}\left(\frac{\partial P_i}{\partial z_j}\right)_{ij} & 0 \\ 0 & \left(\frac{\partial P_i}{\partial z_j}\right)_{ij}\end{pmatrix} \begin{pmatrix}h_1 \id_{n \times n} \\ \\ h_2 \id_{n \times n}\end{pmatrix} \left(\frac{\partial P_i}{\partial z_j}\right)^{-1}_{ij} = \begin{pmatrix}h_1 \id_{n \times n} \\ \\ h_2 \id_{n \times n}\end{pmatrix} \end{equation}
which clearly extends continuously over the fat diagonal as claimed.

This proves the existence of an isomorphism of complex vector bundles
\begin{equation}\phi: T\Sym^g(\Sigma) \to N_{\Sym^g(\Sigma) \subset \Sym^{2g}(\tilde{\Sigma})}.\end{equation}
To see that under this isomorphism, $T\bb{T}_{\alpha}$ and $T\bb{T}_{\beta}$ are homotopic through Lagrangian subbundles to $N_{\bb{T}_{\alpha}}, N_{\bb{T}_{\beta}}$, observe that by construction, away from some small neighbourhood of the fat diagonal, the isomorphism $\phi$ is given as the direct sum of $g$ copies of the isomorphism of Proposition \ref{normaltangentpolarizationIII}. In particular, we can take the direct sum of the homotopies of real-line-sub-bundles of $N_{\Sigma \subset \Sym^2(\tilde{\Sigma})}$ between $\phi(T\alpha_i)$ and $N_{\alpha_i \subset \tilde{\alpha}_i \times \tilde{\alpha}_{g+i}}$ to obtain the desired homotopy of Lagrangian subbundles over $\bb{T}_{\alpha}$, and likewise for $\bb{T}_{\beta}$.\end{proof}

The proof of Proposition \ref{normaltangentpolarizationII} is more or less identical. In this case, since $K$ is null-homotopic, after choosing a Seifert surface we have $S^1$-valued functions $f, \tilde{f}$ defined on the knot complements $Y \backslash K$ and $\tilde{Y}\backslash K$ respectively, yielding $S^1$-valued functions on $\Sigma - \{w, z\}$ and $\tilde{\Sigma}- \{\tilde{w}, \tilde{z}\}$ respectively, whose restrictions to the $\alpha, \beta$ and $\tilde{\alpha}, \tilde{\beta}$ curves are null-homotopic. The same construction as above then yields the result.

\bibliographystyle{plain}
\bibliography{Equivariant-Localization-Bibliography.bib}

\end{document}